\documentclass[10pt]{amsart}
\usepackage{geometry}                
\usepackage{graphicx}
\usepackage{dsfont}
\usepackage{amssymb}
\usepackage{epstopdf}
\usepackage{amsmath}
\usepackage{fullpage}
\usepackage{enumerate}
\usepackage{color,subcaption}
\usepackage{amsthm}
\usepackage{placeins}
\usepackage{array}
\usepackage{bm}
\usepackage{blkarray}
\usepackage{multirow}
\usepackage{float}
\usepackage{morefloats}
\usepackage{pgfplots}
\usepackage{algorithm,algorithmic}
\usepackage[abs]{overpic}

\input{macros.sty}
\usepackage{tikz}
\usetikzlibrary{calc,positioning}

\usepackage{amsopn}
\let\oldtocsection=\tocsection

\let\oldtocsubsection=\tocsubsection

\let\oldtocsubsubsection=\tocsubsubsection

\renewcommand{\tocsection}[2]{\hspace{0em}\oldtocsection{#1}{#2}\textbf}
\renewcommand{\tocsubsection}[2]{\hspace{1em}\oldtocsubsection{#1}{#2}}
\renewcommand{\tocsubsubsection}[2]{\hspace{2em}\oldtocsubsubsection{#1}{#2}}

\newcommand{\bigzero}{\mbox{\normalfont\Large\bfseries 0}}
\DeclareMathOperator*{\argmin}{arg\,min}

\newcommand{\dv}{\text{\rm div}}

\renewcommand{\o}{\text{\rm o}}

\renewcommand{\d}{\text{\rm d}}

\newcommand{\tr}{\text{\rm tr}}

\newcommand{\diam}{\text{\rm diam}}
\newcommand{\e}{\varepsilon}
\newcommand{\I}{\text{\rm I}}
\newcommand{\Id}{\text{\rm Id}}
\newcommand{\Lag}{\text{\rm Lag}}
\newcommand{\Vor}{\text{\rm Vor}}
\newcommand{\Ker}{\text{\rm Ker}}

\renewcommand{\span}{\text{\rm span}}

\newcommand{\narc}{n_{\text{\rm arc}}}
\newcommand{\bpsi}{{\bm\psi}}
\newcommand{\bxi}{{\bm\xi}}

\newcommand{\bs}{{\bm\s}}

\newcommand{\btheta}{{\bm\theta}}
\newcommand{\bvphi}{{\bm\varphi}}
\newcommand{\bzeta}{{\bm\zeta}}
\newcommand{\bnu}{{\bm\nu}}
\newcommand{\bbeta}{{\bm\beta}}
\newcommand{\blambda}{{\bm\lambda}}
\newcommand{\calP}{{\mathcal P}}
\newcommand{\calK}{{\mathcal K}}
\newcommand{\calL}{{\mathcal L}}
\newcommand{\calZ}{{\mathcal Z}}
\newcommand{\calC}{{\mathcal C}}
\newcommand{\calE}{{\mathcal E}}

\newcommand{\calN}{{\mathcal N}}
\newcommand{\calT}{{\mathcal T}}
\newcommand{\calW}{{\mathcal W}}
\newcommand{\Vol}{\text{\rm Vol}}

\newcommand{\Winfty}{W^{1,\infty}(\mathbb{R}^d,\mathbb{R}^d)}
\newcommand{\R}{{\mathbb R}}

\newcommand{\s}{\mathbf{s}}
\newcommand{\G}{\mathbf{G}}
\newcommand{\bQ}{\mathbf{Q}}
\newcommand{\bF}{\mathbf{F}}
\newcommand{\be}{\mathbf{e}}
\newcommand{\bp}{\mathbf{p}}
\newcommand{\h}{\mathbf{h}}

\newcommand{\f}{\mathbf{f}}
\newcommand{\p}{\mathbf{p}}
\newcommand{\bz}{\mathbf{0}}
\newcommand{\bt}{\mathbf{t}}
\newcommand{\m}{\mathbf{m}}
\renewcommand{\b}{\mathbf{b}}
\newcommand{\bVsp}{\mathbf{V}(\mathbf{s},\bm{\psi})}
\newcommand{\bLag}{\mathbf{Lag}(\mathbf{s},\bm{\psi})}

\newcommand{\bVor}{\mathbf{Vor}}
\newcommand{\bzero}{\bm{0}}
\newcommand{\n}{\mathbf{n}}
\newcommand{\q}{\mathbf{q}}
\renewcommand{\r}{\mathbf{r}}
\renewcommand{\c}{\mathbf{c}}
\newcommand{\x}{\mathbf{x}}
\newcommand{\y}{\mathbf{y}}
\renewcommand{\a}{\mathbf{a}}
\renewcommand{\u}{\mathbf{u}}
\renewcommand{\v}{\mathbf{v}}
\newcommand{\g}{\mathbf{g}}

\newcommand{\Per}{\text{\rm Per}}

\usepackage{soul,xcolor}
\setstcolor{violet}

\begin{document}
\newtheorem{theorem}{Theorem}[section]
\newtheorem{problem}{Problem}[section]
\newtheorem{remark}{Remark}[section]
\newtheorem{example}{Example}[section]
\newtheorem{definition}{Definition}[section]
\newtheorem{lemma}{Lemma}[section]
\newtheorem{corollary}{Corollary}[section]
\newtheorem{proposition}{Proposition}[section]
\newtheorem{propdef}{Definition--Proposition}[section]
\numberwithin{equation}{section}

\title{A Lagrangian shape and topology optimization framework based on semi-discrete optimal transport}
\author{
C. Dapogny\textsuperscript{1}, B. Levy\textsuperscript{2} and E. Oudet\textsuperscript{1}
}

\maketitle
\begin{center}
\emph{\textsuperscript{1} Univ. Grenoble Alpes, CNRS, Grenoble INP\footnote{Institute of Engineering Univ. Grenoble Alpes}, LJK, 38000 Grenoble, France},\\
\emph{\textsuperscript{2} Centre Inria de Saclay, Universit\'e Paris Saclay, Laboratoire de Math\'ematiques d’Orsay}.\\
\end{center}
 
\begin{abstract}
This article revolves around shape and topology optimization, in the applicative context where the objective and constraint functionals depend on the solution to a physical boundary value problem posed on the optimized domain. 
We introduce a novel numerical framework based on modern concepts from computational geometry, optimal transport and numerical analysis. Its pivotal feature is a representation of the optimized shape by the cells of an adapted version of a Laguerre diagram. 
Although such objects are originally described by a collection of seed points and weights, 
recent results from optimal transport theory suggest a more intuitive parametrization in terms of the seed points and measures of the associated cells.
The discrete polygonal mesh of the shape induced by this diagram is a convenient support for the deployment of the Virtual Element Method for the numerical solution of the boundary value problem characterizing the physics at play
and for the evaluation of the functionals of the domain involved in the optimization problem.  
The sensitivities of the latter are derived next; at first, we calculate their derivatives with respect to the positions of the vertices of the Laguerre diagram by shape calculus techniques; 
a suitable adjoint methodology is then developed to express them in terms of the seed points and cell measures of the diagram. 
The evolution of the shape through the optimization process is realized by first updating the design variables  according to these sensitivities 
and then reconstructing the diagram thanks to efficient algorithms from computational geometry.
Our shape and topology optimization strategy is versatile: it can be applied to a whole gammut of physical situations (such as thermal and structural mechanics) and optimal design settings (including single- and multi-phase problems). It is Lagrangian and body-fitted by essence, 
and it thereby benefits from all the assets of an explicit, meshed representation of the shape at each stage of the process. 
Yet, it naturally handles dramatic motions of the optimized shape, including topological changes, in a very robust fashion.
These features, among others, are illustrated by a series of 2d numerical examples.
\end{abstract} 

\bigskip
\bigskip
\hrule
\tableofcontents
\vspace{-0.5cm}
\hrule
\bigskip
\bigskip

\section{Introduction and related works}

\noindent Since the early implementations of optimal design,
the identification of a means to represent the shape $\Omega \subset \R^d$ ($d=2$ or $3$ in practice) 
which is amenable to all the operations of the numerical resolution has been a recurring issue -- especially when physical applications are addressed, where the behavior of the optimized shape is modeled by boundary value problems, such as the conductivity or the linear elasticity equations. 
It is indeed notoriously difficult to reconcile accurate simulations of these boundary value problems with a robust description of the evolution of $\Omega$ through the optimization process.
The historical ``Lagrangian'' approaches for shape optimization hinge on a (simplicial) mesh of $\Omega$ which is updated between the iterations of the workflow; 
while this representation naturally offers a convenient support for mechanical computations via the Finite Element Method,
it makes it difficult to track the deformations of $\Omega$ in a robust way, especially when they include topological changes \cite{mohammadi2010applied,pironneau1982optimal}. 
More recent solutions to carry out this thorny evolution operation resort to implicit, Eulerian representations of the shape $\Omega$. The most popular practice in this category is certainly to use the level set method \cite{allaire2004structural,osher2001level,sethian2000structural,wang2003level}. The latter brings into play a fixed mesh $\calT$ of a large ``hold-all'' domain $D$: the optimized design $\Omega \subset D$ is encoded into a so-called ``level set function'' $\phi : D \to \R$, discretized on $\calT$ in practice, which is negative (resp. positive) inside $\Omega$ (resp. inside $D \setminus \overline \Omega$). Arbitrary updates of $\Omega$ are accounted for via ``simple'' updates of $\phi$ -- for instance, small deformations of the boundary $\partial \Omega$ translate into an advection or a Hamilton-Jacobi equation for $\phi$.
Unfortunately, this flexibility has a price: since no mesh of the shape $\Omega$ is available, physical boundary value problems posed on $\Omega$ have to be replaced 
by approximate counterparts, posed on the total domain $D$, which are solved on the fixed mesh $\calT$. 
This practice often leverages a fictitious domain method, such as the ersatz material approximation in structural mechanics, see e.g. \cite{allaire2002shape}; in any event, it results in an inaccurate numerical solution. 
The main competing Eulerian strategy for optimal design, the so-called phase-field method, shares fairly similar spirit, assets and drawbacks \cite{blank2012phase,bourdin_chambolle_2003,takezawa2010shape}.  
On the extreme side of this line of thinking, density-based topology optimization methods \cite{bendsoe2013topology}, such as the popular Solid Isotropic Material with Penalization (SIMP) method in structural mechanics (see also \cite{borrvall2003topology} about artificial porosity methods in the context of fluid mechanics), downright relax the optimal design problem in terms of a ``blurred'' version $\rho$ of the characteristic function of the optimized design, defined on (the mesh $\calT$ of) $D$ and taking values in the interval $[0,1]$. Intuitively, $\rho$ encodes the ``local density'' of material: regions where $\rho(\x) \equiv 1$ (resp. $\rho(\x) \equiv 0$) are full (resp. devoid) of material and the intermediate values $\rho(\x) \in (0,1)$ represent regions filled by a microscopic arrangement of material and void in proportions $\rho(\x)$ and $(1-\rho(\x))$. The physical problem under scrutiny is approximated and reformulated to accommodate such a description of the design, notably by endowing regions where $\rho$ takes intermediate values with fictitious material properties.
Last but not least, and without entering into details, let us eventually mention that several of the aforementioned approaches have been combined, see e.g. \cite{allaire2011topology,allaire2013mesh,allaire2014shape} for a concurrent use of the level set method with simplicial remeshing algorithms, or \cite{christiansen2014topology} for a quite sophisticated Lagrangian method using remeshing. 
 
 The present article introduces a novel framework for shape and topology optimization which is based on concepts and techniques from computational geometry and optimal transport.
 This construction elaborates on the previous work \cite{levy2022partial} of one of the authors, taking place in the context of computational fluid dynamics, 
 and on its recent applications \cite{levy2024monge,nikakhtar2022optimal,nikakhtar2024displacement} in the field of cosmology, see also \cite{de2015power} about related ideas. 
The cornerstone of our strategy is a representation of the optimized shape $\Omega$ by the cells of a (modified version of a) Laguerre diagram, 
a generalization of the well-known notion of Voronoi diagram in computational geometry \cite{aurenhammer1987power,boots2009spatial}.
Although the natural definition of such objects involves a collection $\s = \left\{\s_1,\ldots,\s_N \right\}\in \R^{dN}$ of seed points and associated weights $\bpsi = \left\{ \psi_1,\ldots,\psi_N\right\} \in \R^N$, 
we leverage recent results from optimal transport theory to parametrize them by the seed points $\s$ and the measures $\bnu = \left\{\nu_1,\ldots,\nu_N\right\} \in \R^N$ of the associated cells.
In the present applications, the shape $\Omega$ is iteratively optimized with respect to an objective and constraint functions $J(\Omega)$ and $\G(\Omega)$. Their evaluation and that of their derivatives, which in turn determine the velocity field governing the evolution of $\Omega$ through the process, raise the need to solve one or several boundary value problems posed on $\Omega$; this is a major additional difficulty with respect to the aforementioned works \cite{levy2022partial,levy2024monge,nikakhtar2022optimal,nikakhtar2024displacement}, 
where heuristic simplifications made it possible to express the velocity field for $\Omega$ in terms of geometric quantities of the diagram only. 
In order to solve these boundary value problems, we take advantage of the tessellation of $\Omega$ into convex polytopes induced by its diagram representation; this offers a convenient support for the use of the Virtual Element Method -- a recent variant of the classical Finite Element Method adapted to polygonal meshes, see e.g. \cite{antonietti2022virtual}. 
The calculation of the sensitivities of the optimization criteria $J(\Omega)$ and $\G(\Omega)$ with respect to the seed points $\s$ and cell measures $\bnu$ of the diagram is realized in two steps:
we first apply shape calculus techniques to compute their sensitivities with respect to the vertices of the diagram and we next employ a suitable adjoint method to reformulate them in terms of $\s$ and $\bnu$.
A descent direction for the optimization problem is then inferred from this information, 
and the parameters $\s$ and $\bnu$ of the Laguerre diagram for $\Omega$ are updated; the new diagram associated to the resulting seed points and weights is generated, revealing the next iterate of the shape $\Omega$. 
This strategy is Lagrangian by nature since the evolution of $\Omega$ is realized by tracking the motion of the defining seed points; it consistently features an exact, body-fitted description of the shape $\Omega$. 
In the meantime, it leaves the room for a robust account of dramatic deformations between iterations, including topological changes: the deformations of (the mesh of) $\Omega$ are never tracked explicitly; rather, the shape is rediscovered after each iteration, once the defining Laguerre diagram has been calculated anew.

As reflected by the previous discussion, one originality of our approach is that it leverages recent concepts and techniques pertaining to quite different fields: 
shapes are represented via Laguerre diagrams, a notion from computational geometry and optimal transport; the Virtual Element Method from numerical analysis is used for the solution of the physical boundary value problems at play on polygonal meshes; eventually, fine techniques from shape and topology optimization are used to orchestrate the optimal design process. 

To the best of our knowledge, the present work is the first one using all these ingredients together even though, admittedly some of them have already been combined in the literature. Voronoi-like diagrams, polygonal meshes, and numerical methods dedicated to this type of structures 
have been considered in the perspective of optimal design or more generally interface motion tracking, as we now briefly overview, without exhaustivity. 
In \cite{galindo2010molecular}, the cells of a Voronoi diagram are used to represent the many individual entities of a granular medium. The collection is subjected to the viscous, elastic and contact forces between cells, which are modeled by a finite-dimensional system of ordinary differential equations.
The contributions \cite{springel2010pur,springel2011hydrodynamic} arise in the field of fluid dynamics; 
the 2d polygonal mesh induced by a Voronoi diagram is used as computational support for the finite volume calculations involved in the simulation of the flow, notably for the reconstruction of the fluxes at the interface between cells. Although this discretization is not fitted with the surface of the fluid, the defining seed points evolve according to the flow velocity.
A high-order version of this strategy was proposed in \cite{gaburro2020high}, see also \cite{despres2024lagrangian} for another contribution in this setting. 
Still in the context of fluid mechanics, the article \cite{de2015power} associates a Laguerre diagram to a collection of fluid particles; 
its weights are adjusted to take into account the compressibility or incompressibility of the fluid.
A simple, approximate numerical solver based on discrete differential operators allows to compute the velocity of the fluid.
This framework has been recently revisited in \cite{qu2022power}, where entropy-regularized transport plans are used to ease the calculation of the Laguerre cell attached to each particle. 
It can be adapted to the case of free surfaces, owing to a modified version of this entropy-regularized optimal transport formulation and of the Sinkhorn algorithm used to calculate the transport plan. 
Voronoi diagrams appear more systematically as the pivotal ingredient of the so-called Voronoi Implicit Interface Method introduced in \cite{saye2011voronoi,saye2012analysis}, 
see also \cite{zaitzeff2019voronoi} for a mathematical analysis in the context of multiphase motion driven by curvature.
Briefly, a partition of a fixed domain into multiple phases is represented via the Voronoi diagram generated by a set of seed objects (and not mere seed points). 
The evolution of these phases is achieved by first calculating the domains induced by the $\e$-offset of the collection of edges (in 2d) or faces (in 3d) of the diagram, 
then by realizing their evolution via the ``classical'' level set method. 
The updated collection of phases is retrieved as the Voronoi diagram of the resulting objects.
Beyond the field of computational physics, Voronoi diagrams have inspired popular formats for shapes in applicative disciplines such as computer graphics, see e.g. the article \cite{allain2015efficient} where the motion of a shape is tracked by introducing a clipped Voronoi tessellation of a template shape, whose cells are rigidly deformed to match a target.

Closer to the setting of the present article, the specific features of Voronoi diagrams have inspired various original representations of the shape in optimal design. 
The recent contribution \cite{feng2022cellular} is dedicated to trusses, i.e. networks of bars. 
The considered structures are represented by the edges of a Voronoi diagram contained in a fixed ``hold-all'' domain $D \subset \R^d$; 
when it comes to calculating their elastic behavior, a density function $\rho: D \to [0,1]$ is generated, featuring a continuous interpolation between the value $1$ inside the bars and that $0$ in void.
This function is parametrized by the seed points of the diagram, and the optimal design problem is solved in terms of these variables. 
A quite similar strategy is used in the article \cite{li2023explicit}, dealing with the optimization of foams.

The use of a polygonal computational mesh of the shape $\Omega$ or a larger hold-all domain $D$ is not extraneous to the field of shape and topology optimization, either. 
For instance, the contributions \cite{talischi2010polygonal,antonietti2017virtual,chi2020virtual} rely on a density topology optimization method, where the hold-all domain $D$ is equipped with a fixed polygonal mesh serving as the support for all the mechanical computations of the workflow, which are performed by means of the Virtual Element Method.
The added value of the use of the Virtual Element Method on a general polygonal mesh with respect to the more classical practice of the Finite Element Method on a simplicial mesh hinges on its reported superior stability properties. 
To the best of our knowledge, only a few recent contributions leverage the greater flexibility of general polygonal meshes over simplicial meshes to propose shape and topology optimization methods featuring a body-fitted discretization of the optimized shape. In \cite{nguyen2019level,nguyen2022novel}, the authors rely on the level set method; a level set function $\phi: D \to \R$ for the shape $\Omega$ is defined on a fixed hexahedral mesh of the computational domain $D$ and at each iteration, the $0$ isosurface of $\phi$ is explicitly discretized in this mesh, thus revealing a polygonal mesh of the shape. Mechanical computations are performed on this computational support for $\Omega$ thanks to a moving least-square approximation, which allows to calculate the shape derivative of the minimized function of the domain. 
A similar strategy is used in \cite{ferro2023level}, where a polygonal mesh of the computational domain $D$ is used and the level set function $\phi: D \to \R$ for the shape $\Omega$ is represented as a discontinuous Finite Element function. 
Again, at each iteration of the process, the $0$ level set of $\phi$ for the shape $\Omega$ is explicitly discretized in the mesh of $D$, which results in a new, non conforming polygonal mesh of the latter, 
on which Finite Element analyses can be conducted by a discontinuous Galerkin solver. The topological derivative of the minimized function of the domain is calculated and the level set function $\phi$ is updated from this information. Let us eventually mention the related, very recent work \cite{fernandes2024level} where a very fine simplicial mesh of $D$ is used to update the level set function, and a coarse polygonal mesh is generated when it comes to mechanical computations, whose elements are agglomerations of the simplices of the fine mesh.

To conclude this brief overview of related works, let us point out that, from the conceptual viewpoint, 
our shape and topology optimization framework shares similarities with that developed in \cite{laurain2021analysis,birgin2023sensitivity}. 
In there, a decomposition of the computational domain $D$ into $K$ phases is described by a so-called minimization diagram. The latter is attached to a collection $\left\{\phi_k \right\}_{k=1,\ldots,K}$ of scalar functions $\phi_k : D \to \R$, in such a way that
each phase $k=1,\ldots,K$ is made of those points $\x \in D$ where $\phi_k(\x)$ is minimum among $\phi_1(\x), \ldots, \phi_k(\x)$. 
The motion of the boundary between each phase is accounted for by a quite similar mechanism as in the context of the level set method.
We also mention the recent contribution \cite{birgin2024reconstruction} about the reconstruction of a potential in an inverse problem, which is piecewise constant with respect to the cells of a Voronoi diagram. 

The present article focuses on the exposition of our shape and topology optimization framework in two space dimensions, 
where its main ingredients can be conveniently described with a minimum level of technicality. 
There is no conceptual obstruction to the three-dimensional extension of this work, but only an increase in the implementation burden.
Indeed, the chief tools involved -- notably, the calculation of Laguerre diagrams and the Virtual Element Method -- are already available in this situation. 
To emphasize this fact, our presentation takes place in the general, $d$ dimensional case whenever it is possible without jeopardizing with clarity. 
 On a different note, as we have already noted, our strategy leverages concepts and algorithms pertaining to disciplines which seldom appear together in the literature: 
 computational geometry, via the handling of Laguerre diagrams, 
numerical analysis for the simulation of physical phenomena (and notably, the Virtual Element Method), and optimal design for the notions of shape and topological derivatives, the adjoint method, etc.
In order to ensure a relatively self-contained presentation, we provide a reasonable amount of background material about these ingredients, including proofs when they help intuition or when they are not readily available in the precise context of interest.
%
%
Our implementation essentially relies on open-source libraries, and we aim to make it available in a short future, together with associated tutorials.
We believe that the methods introduced in this work will find useful applications beyond shape optimization, e.g. in computational physics.
 
 The remainder of this article is organized as follows. 
The next \cref{sec.presso} sets the scene of the shape optimization problems analyzed in this article.
In particular, we introduce the main notation used throughout and recall a few basic notions about shape optimization.
\cref{sec.discLag} is devoted to the numerical representation of our shapes in terms of the seed points and cell measures of a Laguerre diagram;
 the general shape optimization workflow adopted in this context is then presented.
 The subsequent \cref{sec.geomcomp,sec.mechcomp,sec.deropt} are more in-depth presentations of the main operations involved therein. 
 \cref{sec.geomcomp} deals with the geometric computations attached to Laguerre diagrams involved in our framework, and notably their construction, cell regularization, resampling, etc. \cref{sec.mechcomp} sketches the implementation of the Virtual Element Method for the solution of the conductivity equation and  the linear elasticity system. Eventually, \cref{sec.deropt} details the calculation of the objective and constraint functionals featured in our shape optimization problem and their sensitivities with respect to the parameters of our shape representation, namely the seed points and cell measures of the representing diagram. 
 A series of numerical examples is presented and discussed in \cref{sec.num} to showcase the main features of our approaches. 
 In \cref{sec.concl}, we draw a few conclusions about this strategy and we outline natural perspectives for future work. 
 The article ends with a series of appendices, devoted to the proofs of some technical results which are not completely new, but are adaptations of existing results, 
 or known facts which are not easily found under the needed form in the literature. It also contains several useful implementation details.
 
\section{Presentation of the shape optimization problem}\label{sec.presso}

\noindent Our main purpose in this article is to optimize shapes, that is, bounded, Lipschitz domains $\Omega$ of $\R^d$, where $d=2,3$ in practice. 
We consider optimization problems of the general form:
\begin{equation}\label{eq.sopb} 
\tag{\textcolor{gray}{P}}
\min \limits_{\Omega} \: J(\Omega) \:\: \text{ s.t. } \G(\Omega) = 0,
\end{equation}
where  $J(\Omega)$ is an objective criterion,
and $\G(\Omega) := (G_1(\Omega),\ldots,G_p(\Omega))$ is a collection of $p$ scalar-valued equality constraint functionals. 

Elementary examples of shape functionals are the volume $\Vol(\Omega)$ and the perimeter  $\Per(\Omega)$, defined by:
\begin{equation}\label{eq.volper}
\Vol(\Omega) = \int_\Omega \:\d \x , \text{ and } \Per(\Omega) = \int_{\partial \Omega} \:\d s.
\end{equation}
More involved instances about $J(\Omega)$ and $\G(\Omega)$ depend on the physical behavior of $\Omega$ via a so-called state function $u_\Omega$, 
which is the solution to a boundary value problem posed on $\Omega$.

After setting a few notations about calculus in \cref{sec.notation}, we describe in \cref{sec.twophys} the applicative situations considered in the present article, namely those of conductive media and of elastic structures.
Eventually, in \cref{sec.sotuto}, we recall a few facts about differentiation with respect to the domain. 

\subsection{Notations}\label{sec.notation}


\noindent Throughout this article, vectors $\x$ in a Euclidean space $\R^N$ appear in bold face, and their components are denoted by $\x= (x_1,\ldots,x_N)$. Moreover, we rely on the following conventions:
\begin{itemize}
\item For $N \geq 1$, the Euclidean space $\R^N$ is equipped with its canonical basis $\left\{\be_i \right\}_{i=1,\ldots,N}$.
The attached Euclidean inner product is denoted by $\langle \cdot , \cdot \rangle$, and sometimes only with a $\cdot$, and the norm is denoted by $\lvert \cdot \lvert$. 
We also introduce the supremum norm:
$$ \lvert \x \lvert_\infty := \sup\limits_{i=1,\ldots,N} \lvert x_i \lvert, \quad \x \in \R^N.$$
\item The open ball in $\R^N$ with center $\x \in \R^N$ and radius $r >0$ is denoted by $B(\x,r)$.
\item Let $f : \R^N \to \R$ be a (smooth enough) scalar function and let $\x \in \R^N$;
we denote by $\frac{\partial f}{\partial \x}(\x)$ the derivative of $f$ at $\x$, i.e. $\frac{\partial f}{\partial \x}(\x): \R^N \to \R$ is the linear mapping giving rise to the expansion:
$$ f(\x + \widehat \x) = f(\x) + \frac{\partial f}{\partial \x}(\x)(\widehat \x) + \o(\widehat \x), \text{ where } \frac{\o(\widehat \x)}{|\widehat \x|} \to 0 \text{ as } \widehat \x \to 0.$$
The gradient $\nabla f(\x) \in \R^N$ of $f$ at $\x$ is the Riesz representative of this linear mapping, that is:
$$\forall \widehat \x \in \R^N, \quad \langle \nabla f(\x) , \widehat \x \rangle = \frac{\partial f}{\partial \x}(\x) (\widehat \x) .$$
When needed, we shall explicitly indicate in the operator $\nabla_\x$ the variables $\x$ involved in the derivative. 
\item Still considering a (smooth enough) scalar function $f: \R^N \to \R$, the Hessian matrix $\left[\nabla^2 f(\x)\right]$ of $f$ at $\x$ is the $N \times N$ matrix with entries:
$$ \left[\nabla^2 f(\x) \right]_{ij} = \frac{\partial^2 f}{\partial x_i\partial x_j}(\x), \quad i,j=1,\ldots,N. $$
\item For $M,N \geq 1$, the derivative of a smooth enough vector-valued function $\bpsi = (\psi_1,\ldots,\psi_M): \R^N \to \R^M$ is the linear mapping $\frac{\partial \bpsi}{\partial \x}(\x) : \R^N \to \R^M$ featured in the expansion
$$ \bpsi(\x + \widehat \x) = \bpsi(\x) + \frac{\partial \bpsi}{\partial \x}(\x)(\widehat \x) + \o(\widehat \x), \text{ where } \frac{\o(\widehat \x)}{|\widehat \x|} \to 0 \text{ as } \widehat \x \to 0.$$
The $M \times N$ matrix $\nabla \bpsi (\x)$ associated to this derivative -- sometimes also denoted by $\left[\nabla \bpsi(\x)\right]$ for clarity -- is that with entries:
$$ \left[ \nabla \bpsi(\x)\right]_{ij} = \frac{\partial \psi_i}{\partial x_j}(\x), \quad i=1,\ldots,M, \:\: j = 1,\ldots,N.$$
As a result, the following relation holds true:
$$ \forall \widehat \x \in \R^N, \quad \frac{\partial \bpsi}{\partial \x}(\x)(\widehat \x) = \left[\nabla\bpsi(\x) \right] \widehat \x,$$
where the right-hand side features the matrix-vector product between the $M \times N$ matrix $ \left[\nabla \bpsi(\x) \right]$ and the vector $\widehat \x \in \R^N$.
Note the (classical and hopefully harmless) ambiguity of these notations in the case $M =1$ where $\nabla \psi (\x)$ is a vector with size $N$ and $\left[ \nabla \bpsi(\x) \right]$ is the transpose $1 \times N$ matrix.
\end{itemize}
More specific notations will be introduced when needed, later in the article.

\subsection{Mathematical models for conductive media and elastic structures}\label{sec.twophys}
\noindent In this section, we introduce the two physical applications of our shape optimization problems. We refer to classical treaties about continuum mechanics such as \cite{temam2005mathematical} for further details, see also \cite{gould1994introduction} about linear elasticity.

\subsubsection{Optimal design of conductive media}\label{sec.conduc}

\noindent Let the shape $\Omega$ stand for an electric conductor in $\R^d$ ($d=2$ or $3$). 
It is filled by a material with smooth conductivity $\gamma \in \calC^\infty(\R^d)$, 
satisfying the following ellipticity condition: 
$$ \exists \:\: 0 <  \underline\gamma \leq \overline\gamma < \infty  \:\: \text{ s.t. } \forall \x \in \R^d, \quad \underline\gamma \leq \gamma(\x) \leq \overline\gamma. $$
The boundary $\partial \Omega$ is maintained at fixed potential $0$. 
Introducing a source $f \in \calC^\infty(\R^d)$, the voltage potential $u_\Omega$ inside $\Omega$ 
is the unique solution in $H^1_0(\Omega)$ to the following boundary value problem:
\begin{equation}\label{eq.conduc}
\tag{\textcolor{gray}{Cond}}
\left\{
\begin{array}{cl}
-\dv(\gamma \nabla u_\Omega) = f & \text{in } \Omega, \\ 
u_\Omega = 0 & \text{on } \partial \Omega.
\end{array}
\right.
\end{equation}
The associated variational formulation reads: 
\begin{equation}\label{eq.varfcond}
\text{Search for } u_\Omega \in H^1_0(\Omega)\:\:  \text{ s.t. }  \forall v \in H^1_0(\Omega), \quad \int_\Omega \gamma \nabla u_\Omega \cdot \nabla v \:\d \x = \int_\Omega f v \:\d \x.
 \end{equation}

In this context, a typical objective function measuring the performance of $\Omega$ is the compliance
\begin{equation}\label{eq.complianceConduc}
 C(\Omega) = \int_\Omega \gamma \lvert \nabla u_\Omega \lvert^2 \:\d \x = \int_\Omega f u_\Omega \:\d \x,
 \end{equation}
 which appraises the energy stored inside the shape $\Omega$, or equivalently, the work done by the source.   
One may also be interested in minimizing the first eigenvalue $\lambda_\Omega^1$ (resp. the $k^{\text{th}}$ eigenvalue $\lambda_\Omega^k$) of the operator $-\dv(\gamma\nabla\cdot)$ on $\Omega$ equipped with Dirichlet boundary conditions: $\lambda_\Omega^1$ is the smallest positive number $\lambda$ (resp. the $k^{\text{th}}$ smallest positive number) for which there exists a non trivial function $u \in H^1_0(\Omega)$ satisfying
\begin{equation}\label{eq.conducev}
\left\{
\begin{array}{cl}
-\dv(\gamma \nabla u) = \lambda u & \text{in } \Omega, \\ 
u = 0 & \text{on } \partial \Omega.
\end{array}
\right.
\end{equation}
From the physical viewpoint, $\lambda_\Omega^1$ accounts for the minimum energy of a potential $u$ with unit squared amplitude, as reflected by the classical Rayleigh quotient formula: 
\begin{equation}\label{eq.evlap}
 \lambda_\Omega^1 = \min\limits_{u \in H^1_0(\Omega), \atop u \neq 0} \frac{\displaystyle\int_\Omega \gamma \lvert\nabla u \lvert ^2 \:\d \x}{\displaystyle\int_\Omega u^2 \:\d \x}.
 \end{equation}

\begin{remark}\label{rem.2phaseconduc}
We shall also consider a two-phase variant of this setting, 
where the shape $\Omega$ represents one phase made of a material with conductivity $\gamma_1 \in \calC^\infty(\R^d)$ within a fixed domain $D$ made of a material with different conductivity $\gamma_0 \in \calC^\infty(\R^d)$. 
One then aims to optimize the repartition $D = \Omega \cup (D \setminus \overline \Omega)$ of both materials within $D$ with respect to, e.g. the compliance of the total device $D$; see \cref{sec.numconduc2phase} for a numerical example in this setting.
\end{remark}

\subsubsection{Optimization of the shape of elastic structures}\label{sec.elas}

\noindent In this section,
$\Omega \subset \R^d$ stands for a mechanical structure, whose boundary is divided into three disjoint pieces:
$$ \partial \Omega = \overline{\Gamma_D} \cup \overline{\Gamma_N} \cup \overline \Gamma.$$
In this decomposition,
\begin{itemize}
\item The shape $\Omega$ is clamped on $\Gamma_D$; 
\item Surface loads $\g: \Gamma_N \to \R^d$ are applied on $\Gamma_N$; 
\item The region $\Gamma$ is free of applied efforts. 
\end{itemize}
In practice, $\Gamma_D$ and $\Gamma_N$ are often imposed by the context, 
so that only the free boundary $\Gamma$ is subject to optimization. 
Introducing (smooth) body forces $\f: \R^d \to \R^d$ representing e.g. gravity, the displacement $\u_\Omega$ of $\Omega$ belongs to the  space 
$$H^1_{\Gamma_D}(\Omega)^d := \left\{ \u \in H^1(\Omega)^d ,\:\: \u = \bz \text{ on } \Gamma_D \right\}.$$
It is the unique solution in the latter to the linearized elasticity system:
\begin{equation}\label{eq.elas}
\tag{\textcolor{gray}{Elas}}
\left\{
\begin{array}{cl}
-\dv (Ae(\u_\Omega)) = \f & \text{in } \Omega, \\
\u_\Omega = \bz & \text{on } \Gamma_D, \\
Ae(\u_\Omega) \n = \g & \text{on } \Gamma_N, \\
Ae(\u_\Omega) \n = \bz & \text{on } \Gamma.
\end{array}
\right.
\end{equation}
Here, $e(\u) := \frac12(\nabla \u + \nabla \u^T)$ is the linearized strain tensor induced y a vector field $\u:  \Omega \to \R^d$. 
The material properties of the elastic medium are encoded in the Hooke's tensor $A$, which reads: 
$$\text{For all symmetric matrix } \xi \in \R^{d\times d}, \quad  A\xi = 2\mu \xi + \lambda\tr(\xi) \I,$$
where $\lambda >0$ and $\mu > 0$ are the Lam\'e coefficients of the material.
The boundary value problem \cref{eq.elas} rewrites, under variational form: 
\begin{equation}\label{eq.varfelas}
\text{Search for } \u_\Omega \in H^1_{\Gamma_D}(\Omega)^d \:\:  \text{ s.t. }\forall \v \in H^1_{\Gamma_D}(\Omega)^d, \:\:
 \int_\Omega Ae(\u_\Omega) : e(\v) \:\d \x =\int_\Omega \f \cdot \v \:\d \x + \int_{\Gamma_N} \g \cdot \v \:\d s.
\end{equation}

The shape optimization problem \cref{eq.sopb} under scrutiny in the present context of mechanical engineering may for instance consist in minimizing the compliance $J(\Omega) = C(\Omega)$ of the structure as a means to maximize its rigidity,
\begin{equation}\label{eq.complianceElas}
 C(\Omega) := \int_\Omega Ae(\u_\Omega): e(\u_\Omega) \:\d \x = \int_\Omega \f \cdot \u_\Omega \:\d \x + \int_{\Gamma_N} \g \cdot \u_\Omega \:\d s,
 \end{equation}
under a volume constraint $G(\Omega) = \Vol(\Omega) - V_T$, where $V_T$ is a volume target.

\begin{remark}
For simplicity of the presentation, we have assumed that all the data of the models of the previous and present sections (the domain $\Omega$, the forces $\f$, $\g$, etc.) are smooth. This framework is by no means the minimal one guaranteeing the well-posedness of these models, and the validity of our developments. 
\end{remark}

\subsection{Shape optimization in a nutshell}\label{sec.sotuto}

\noindent The mathematical analysis and numerical resolution of a shape optimization problem of the form \cref{eq.sopb} hinges on the sensitivities of the objective and constraint functions $J(\Omega)$ and $\G(\Omega)$ with respect to the domain $\Omega$ -- 
a notion which can be appraised in various ways. 
In this work, we mainly rely on the boundary variation method of Hadamard, see for instance \cite{allaire2007conception,henrot2018shape,murat1976controle,sokolowski1992introduction}. 
The latter is based on variations of a given shape $\Omega \subset \R^d$ of the form:
$$ \Omega_\btheta := (\Id + \btheta)(\Omega), \:\: \btheta \in \Winfty, \:\: || \btheta ||_{\Winfty} < 1.$$
Loosely speaking, $\Omega_\btheta$ is obtained by moving the points of $\Omega$ according to the ``small'' vector field $\btheta$, see \cref{fig.hadamard}. 

\begin{figure}[!ht]
\centering
\includegraphics[width=0.5\textwidth]{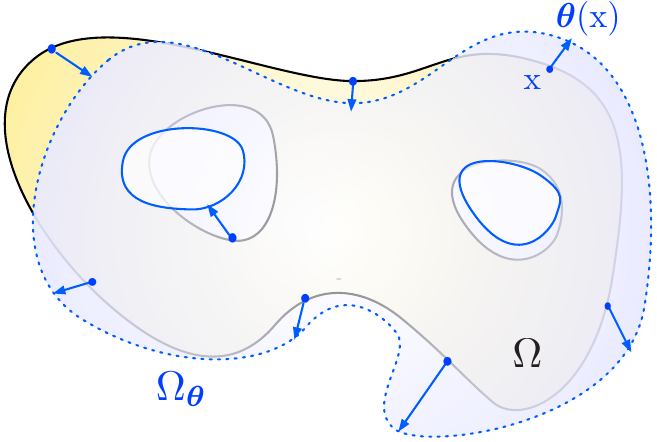}
\caption{\it Variation $\Omega_\btheta$ of a shape $\Omega$ in the sense of Hadamard's method.}
\label{fig.hadamard}
\end{figure}

One function $F(\Omega)$ of the domain is called shape differentiable at a particular shape $\Omega$ when the underlying mapping $\btheta \mapsto F(\Omega_\btheta)$, 
defined from a neighborhood of $\bz$ in $\Winfty$ into $\R$, is Fr\'echet differentiable at $\bz$. Its derivative $\btheta \mapsto F^\prime(\Omega)(\btheta)$ is called the shape derivative of $F$ at $\Omega$
and it satisfies the following expansion:
$$ F(\Omega_\btheta) = F(\Omega) + F^\prime(\Omega)(\btheta) + \o(\btheta), \text{ where } \frac{\o(\btheta)}{|| \btheta ||_{\Winfty}} \xrightarrow{\btheta \to \bz} 0.$$
Intuitively, a ``small'' perturbation of $\Omega$ following a ``descent direction'' $\btheta$, i.e. a deformation such that $F^\prime(\Omega)(\btheta) < 0$, for a short pseudo-time step $\tau >0$ produces a shape $\Omega_{\tau\btheta}$ with a ``better'' performance (i.e. a lower value), as measured in terms of $F$: 
$$F(\Omega_{\tau\btheta}) \approx F(\Omega) + \tau F^\prime(\Omega)(\btheta) < F(\Omega) .$$

Simple examples of shape derivatives are those of the volume and perimeter functionals defined in \cref{eq.volper}. 
Under mild assumptions that we omit for brevity, these read: 
\begin{equation}\label{eq.Volprime}
 \Vol^\prime(\Omega)(\btheta) =\int_\Omega \dv(\btheta) \:\d \x = \int_{\partial \Omega} \btheta\cdot \n \:\d s, \text{ and }  \Per^\prime(\Omega)(\btheta) = \int_{\partial \Omega} \kappa \: \btheta \cdot \n \:\d s,
 \end{equation}
where $\kappa : \partial \Omega \to \R$ is the mean curvature of $\partial \Omega$. 


The analysis of the ``physical'' functionals of interest in this article is more intricate, as they involve the solution $u_\Omega$ to a boundary value problem posed on $\Omega$, see e.g. \cref{eq.complianceConduc,eq.complianceElas}. 
Nevertheless, the classical adjoint method allows to calculate their shape derivatives in closed form, in terms of $u_\Omega$ and possibly of an adjoint state $p_\Omega$, solution to a boundary value problem similar to that for $u_\Omega$. 
Although it is non trivial, this technical issue is well-known in the literature, 
and we do not enter into details, referring to e.g. \cite{allaire2020survey,lions1971optimal}, and also \cite{plessix2006review} for a comprehensive introduction. 

\begin{remark}
The adjoint method will be deployed (and detailed) in \cref{sec.vertoseeds} below to achieve another purpose, that of converting a derivative with respect to the vertices of a Laguerre diagram into derivatives with respect to its seed points and cell measures.
\end{remark}

For further reference, let us note that the shape derivative of a generic shape functional $F(\Omega)$ has often two different, albeit equivalent structures: 
\begin{itemize}
\item A volume form
\begin{equation}\label{eq.Fpvol}
 F^\prime(\Omega)(\btheta) = \int_\Omega (\bt_\Omega \cdot \btheta + S_\Omega : \nabla \btheta )\:\d \x ,
 \end{equation}
featuring an integral over the domain $\Omega$ of the deformation $\btheta$ and its derivative, and some vector- and matrix-valued functions $\bt_\Omega : \to \R^d$ and $S_\Omega :\Omega \to \R^{d\times d}$ depending on the function $F(\Omega)$.
\item A surface form
\begin{equation}\label{eq.Fpsurf}
 F^\prime(\Omega)(\btheta) = \int_{\partial \Omega} v_\Omega \:\btheta \cdot \n \:\d s,
 \end{equation}
involving a scalar field $v_\Omega: \partial \Omega \to \R$ which depends on $F(\Omega)$.
\end{itemize}
The surface expression \cref{eq.Fpsurf} is often deemed more appealing: it indeed reflects the natural fact that the shape derivative of a ``regular enough'' objective function depends only on the normal component of the deformation $\btheta$ on $\partial \Omega$, see for instance \cite{delfour2011shapes,henrot2018shape} about the so-called Structure theorem for shape derivatives. Besides, it easily lends itself to the identification of a descent direction for $F(\Omega)$, as letting $\btheta = -v_\Omega \n$ on $\partial \Omega$ straightforwardly ensures that $F^\prime(\Omega)(\btheta) < 0$. 
However, the volume form \cref{eq.Fpvol} is usually more suitable for mathematical analysis, see e.g. \cite{hiptmair2015comparison,giacomini2017volumetric}. 

\begin{remark}\label{rem.subsetWinfty}
The deformations $\btheta$ considered in the practice of the method of Hadamard are often restricted to a subset of $\Winfty$, which is for instance made of vector fields with higher regularity, or vector fields vanishing on a fixed, non optimizable region of space.    
\end{remark}

\begin{remark}\label{rem.topder}
The notion of shape derivative is not the only means to understand differentiation with respect to the domain. 
The alternative concept of topological derivative appraises the sensitivity of a function $F(\Omega)$ with respect to the nucleation of a ``small'' hole inside $\Omega$. 
It relies on variations of $\Omega$ of the form:
$$ \Omega_{\x,r} :=  \Omega \setminus \overline{B(\x,r)}, \text{ where } \x \in \Omega \text{ and } r \ll 1.$$
The function $F(\Omega)$ is then said to have a topological derivative $\d_T F(\Omega)(\x)$ at $\x \in \Omega$ if the following expansion holds: 
$$ F(\Omega_{\x,r}) = F(\Omega) + r^d \d_T F(\Omega)(\x) + \o(r^d), \text{ where } \frac{\o(r^d)}{r^d} \xrightarrow{r\to0}0.$$
We refer to \cite{garreau2001topological,sokolowski1999topological} for the seminal contributions about topological derivatives and to \cite{amstutz2021introduction,novotny2012topological} for recent overviews.
\end{remark}

\section{Discretization and evolution of shapes using modified Laguerre diagrams}\label{sec.discLag}

\noindent This section details the representation of shapes adopted in our framework.
The first \cref{sec.Lagrep} introduces basic notions about Laguerre diagrams, in their original acceptation and in the modified version introduced in \cite{levy2022partial}.
In the next \cref{sec.Lagrefvol}, we explain how such objects can be conveniently parametrized by the seed points and the measures of their cells.
Finally, \cref{sec.Lagalgo} sketches our optimization strategy, whose most critical operations are more extensively described in the next parts of the article.

\subsection{Representation of shapes via Laguerre diagrams}\label{sec.Lagrep}

\noindent 
Let $D \subset \R^d$ be a fixed bounded and Lipschitz domain containing all the shapes of interest.
We start with the classical definitions of Voronoi and Laguerre diagrams, the latter being more general, weighted versions of the former; we refer to \cite{aurenhammer1987power,boots2009spatial} about these classical notions from computational geometry. 

\begin{definition}\label{def.classLag}
\noindent \begin{itemize}
\item Let $\s = \left\{ \s_i \right\}_{i=1,\ldots,N} \in \R^{dN}$ be a collection of seed points.
The Voronoi diagram $\bVor(\s)$ induced by $\s$ is the following decomposition of $\overline D$:
\begin{equation}\label{eq.defVoronoi}
\overline D = \bigcup\limits_{i=1}^N  \Vor_i(\s),
\end{equation}
where for $i=1,\ldots,N$, the Voronoi cell $\Vor_i(\s)$ is defined by: 
\begin{equation}\label{eq.defVori}
\Vor_i(\s) := \left\{ \x \in \overline D, \:\: | \x - \s_i |^2 \leq |\x -\s_j|^2,\:\: \forall j \neq i \right\}.
 \end{equation}
\item Let $\s = \left\{ \s_i \right\}_{i=1,\ldots,N} \in \R^{dN}$ be a collection of seed points and $\bpsi = \left\{\psi_i \right\}_{i=1,\ldots,N} \in \R^{N}$ be a set of associated weights.
The (classical) Laguerre diagram $\bLag$ attached to $\s$ and $\bpsi$ is the following decomposition of $\overline D$:
\begin{equation}\label{eq.defLaguerre}
\overline D = \bigcup\limits_{i=1}^N  \Lag_i(\s,\bpsi),
\end{equation}
where for $i=1,\ldots,N$, the Laguerre cell $\Lag_i(\s,\bpsi)$ reads: 
\begin{equation}\label{eq.defLagi}
\Lag_i(\s,\bpsi) := \left\{ \x \in \overline D, \:\: | \x - \s_i |^2 - \psi_i \leq |\x -\s_j|^2 - \psi_j ,\:\: \forall j \neq i \right\}.
 \end{equation}
 \end{itemize}
\end{definition}

The concept of Laguerre diagram paves the way to a way of representing a shape $\Omega \subset D$, as a subset of the cells of a diagram of the form \cref{eq.defLaguerre,eq.defLagi}, associated to suitably chosen seed points $\s$ and weights $\bpsi$:
\begin{equation}\label{eq.decompOmSubset} 
\overline\Omega =  \bigcup\limits_{i=1}^K  \Lag_i(\s,\bpsi), \text{ where } K < N,
\end{equation}
and the cells $\Lag_i(\s,\bpsi)$, $i=K+1,\ldots , N$ form a decomposition of the void phase $D \setminus \overline \Omega$. 
In addition to this option, we shall also rely on another point of view, introduced in \cite{levy2022partial}, which leverages a slightly modified version of the notion of Laguerre diagram. 

\begin{definition}\label{def.modLag}
Let $\Omega \subset D$ be a shape; two sets $\s = \left\{ \s_i \right\}_{i=1,\ldots,N} \in \R^{dN}$ and $\bpsi = \left\{\psi_i \right\}_{i=1,\ldots,N} \in \R^{N}$ of seed points and weights
 generate a (modified) Laguerre diagram $\bVsp$ for $\Omega$ if the following decomposition holds:
\begin{equation}\label{eq.decompOm}
\overline \Omega = \bigcup\limits_{i=1}^N V_i(\s,\bpsi),
\end{equation}
 where each closed cell $V_i(\s,\bpsi)$ is obtained by intersection of the corresponding Laguerre cell $\Lag_i(\s,\bpsi)$ in \cref{eq.defLagi} with the ball centered at $\s_i$ with radius $\psi_i^{1/2}$:
\begin{equation}\label{eq.defVipsi}
V_i(\s,\bpsi) = \Lag_i(\s,\bpsi)  \cap \overline{B(\s_i,\psi_i^{1/2})}.
 \end{equation}
 \end{definition}
 
 Both types of decomposition \cref{eq.decompOmSubset,eq.decompOm} will be used in the present work to represent a shape $\Omega \subset D$. The analyses of their mathematical properties are very similar, 
 and since ``classical'' diagrams of the form \cref{eq.defLaguerre,eq.defLagi} are more familiar in the literature (although not in the present shape and topology optimization context), we essentially focus our presentation on the modified diagrams of \cref{def.modLag}, pointing out the differences with their classical counterparts when relevant.  
\par\medskip 

The next statement draws the main geometric features of the representation \cref{eq.decompOm} and sets related notations used throughout the sequel, see \cref{fig.Vipsi}. 
For simplicity, it is only provided in the case of two space dimensions.
Its elementary proof is completely similar to that of the corresponding statement for ``classical'' Laguerre diagrams \cref{eq.defLaguerre}, about which we refer to e.g. \cite{aurenhammer1987power},
and it is omitted for brevity. 

\begin{definition}\label{propdef.Lag} 
Let $d=2$, and let $\Omega \subset D$ be a shape defined via the diagram \cref{eq.decompOm} associated to seed points $\s = \left\{ \s_i \right\}_{i=1,\ldots,N}$ and weights $\bpsi = \left\{\psi_i \right\}_{i=1,\ldots,N}$. 
The following facts hold true:
\noindent \begin{enumerate}[(i)]
\item The cells $V_i(\s,\bpsi)$ in \cref{eq.defVipsi} are closed, bounded  and convex subsets of $\R^2$. 
\item The intersection $V_i(\s,\bpsi) \cap V_j(\s,\bpsi)$ between two different cells $i \neq j$ is either empty, or a point or a (closed) straight segment $\be_{ij}$ orthogonal to $\s_i\s_j$.
\item When the intersection $V_i(\s,\bpsi) \cap V_j(\s,\bpsi)$ is a segment with positive length, $V_j(\s,\bpsi)$ is called a neighbor of $V_i(\s,\bpsi)$. 
The index set of the neighbors of the $i^{\text{th}}$ cell is denoted by $\calN_i \subset \left\{ 1,\ldots, i-1,i+1,\ldots,N \right\}$.
\item We denote by $\calE \subset \left\{1,\ldots,N \right\}^2$ the set of neighboring pairs of cells in the diagram \cref{eq.decompOm}, that is, $(i,j) \in \calE$ if $j \in \calN_i$ (and so $i \in \calN_j$).
\item For each $i=1,\ldots,N$, the boundary $\partial V_i(\s,\bpsi)$ can be decomposed as:
$$ \partial V_i(\s,\bpsi) = \left( \bigcup\limits_{j \in \calN_i} \be_{ij}  \right) \cup \calC_i \cup \calL_i,$$
where
\begin{itemize}
\item For $j \in \calN_i$, $\be_{ij} := \partial V_i(\s,\bpsi)  \cap \partial V_j(\s,\bpsi)$ is a line segment; 
\item $\calC_i = \partial \Omega \cap \partial V_i(\s,\bpsi) $ is a (possibly empty) collection of circular arcs; 
\item $\calL_i = \partial D \cap \partial V_i(\s,\bpsi) $ is a (possibly empty) collection of line segments. 
\end{itemize} 
\item The endpoints of the (straight or curvilinear) pieces of this representation are called the vertices of the diagram \cref{eq.decompOm}.
We denote by  $\q := \left\{ \q_j \right\}_{j=1,\ldots,M} \in \R^{2M}$ the collection of these $M$ vertices. 
\end{enumerate}
\end{definition}

\begin{figure}[!ht]
\centering
\includegraphics[width=0.7\textwidth]{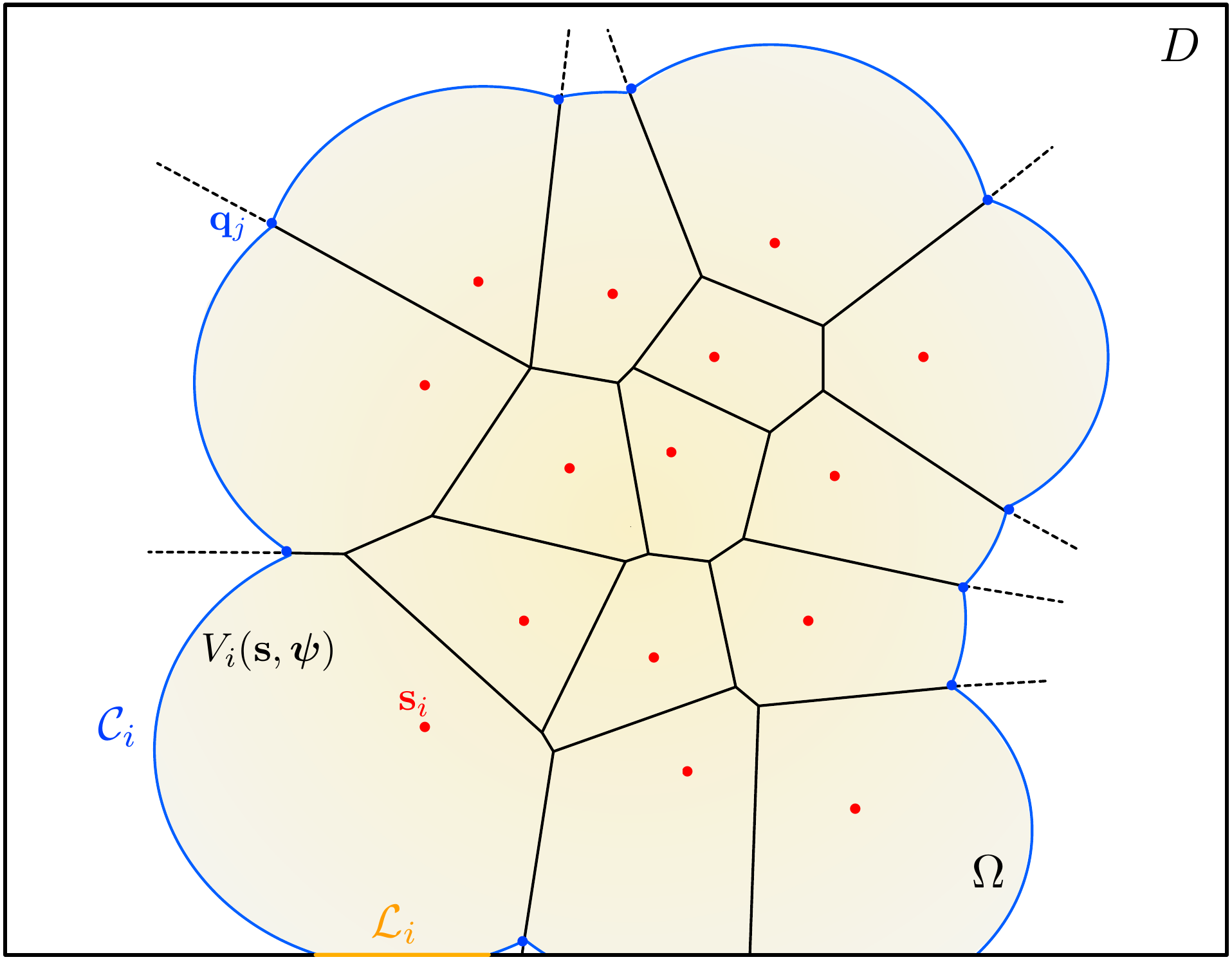}
\caption{\it Representation \cref{eq.decompOm} of the shape $\Omega \subset D$ as the modified Laguerre diagram \cref{eq.decompOm} associated to the seed points $\s = \left\{ \s_1,\ldots,\s_N \right\}$ and weights $\bpsi = \left\{ \psi_1,\ldots,\psi_N \right\}$.}
\label{fig.Vipsi}
\end{figure}

\begin{remark}\label{rem.outcells}
Although the seed points of a Voronoi diagram obviously belong to the interior of their respective cells, the defining seed points of a Laguerre diagram 
may lie far outside;
actually, they may even lie outside the computational domain $D$ although, for implementation stability purposes, we shall impose that they remain inside the latter, see \cref{sec.velext}.

To illustrate this point, let us consider a collection of seed points $\s = \left\{ \s_i \right\}_{i=1,\ldots,N} \in \R^{dN}$.
Let $\h$ be an arbitrary vector in $\R^d$; we introduce the collection of seed points $\widetilde{\s} = \left\{\s_i + \h \right\}_{i=1,\ldots,N}$ resulting from a common translation of $\s_1,\ldots,\s_N$ by $\h$ and the weight vector $\widetilde{\bpsi}$ given by $\widetilde{\psi}_i= 2 \s_i \cdot \h$, $i=1,\ldots,N$. Then, for any index $i=1,\ldots,N$ and any point $\x \in D$, a simple calculation reveals that:   
$$
\begin{array}{lcl}
\x \in \Lag_i(\widetilde{\s},\widetilde{\bpsi}) & \Leftrightarrow & \lvert \x - \h - \s_i \lvert^2 - 2 \s_i \cdot \h  \: \leq \: \lvert \x - \h - \s_j \lvert^2 - 2\s_j \cdot \h \quad \forall j \neq i \\[2mm]
& \Leftrightarrow & \lvert \x - \s_i \lvert^2 - 2\h\cdot(\x - \s_i) + \lvert \h \lvert ^2  - 2 \s_i \cdot \h  \: \leq \: \lvert \x - \s_j \lvert^2 - 2\h\cdot(\x - \s_j) + \lvert \h \lvert^2  - 2 \s_j \cdot \h \quad \forall j \neq i \\[2mm]  
& \Leftrightarrow & \lvert \x - \s_i \lvert ^2 \: \leq \: \lvert \x - \s_j \lvert ^2  \quad \forall j \neq i \\[2mm]  
& \Leftrightarrow & \x \in \Vor_i(\s).
\end{array}
$$
In other terms, the Voronoi tessellation $\mathbf{Vor}(\s)$ of $D$, whose seed points $\s_1,\ldots,\s_N$ belong to the interior of their respective cells, has the same cells as the Laguerre diagram $\mathbf{Lag}(\widetilde{\s},\widetilde{\bpsi})$ whose seed points lie arbitrarily far away.
\end{remark}

\begin{remark}\label{rem.s0}
Although, strictly speaking, a decomposition of the form \cref{eq.decompOm} is not a classical Laguerre diagram, it can be understood as the diagram associated to the collection of ``seed objects'' 
$\s \cup \left\{ \s_0 \right\}$ and weights $\bpsi \cup \left\{\psi_0\right\}$, where the additional object $\s_0$ is the whole domain $D$ and $\psi_0 = 0$, see \cite{levy2022partial}. 
A straightforward adaptation of the \cref{def.classLag} of a Laguerre diagram to allow seed objects (and not just seed points) shows indeed that the cells of this new diagram are exactly those $V_i(\s,\bpsi)$ in \cref{eq.defVipsi} for $i=1,\ldots,N$, with the following additional ``void cell'', associated to $\s_0$:
\begin{equation}\label{eq.V0}
V_0(\s,\bpsi) := \left\{ \x \in \overline D, \:\: \forall i = 1,\ldots, N, \quad 0 \leq |\x - \s_i |^2 - \psi_i \right\}.
\end{equation}
Hence, up to a small abuse of terminology, we shall refer to \cref{eq.decompOm} as a (modified) Laguerre diagram for $\Omega$.
\end{remark}

\begin{figure}
\begin{center}
\begin{tabular}{ccc}
\begin{minipage}{0.3\textwidth}
\begin{overpic}[width=1.0\textwidth]{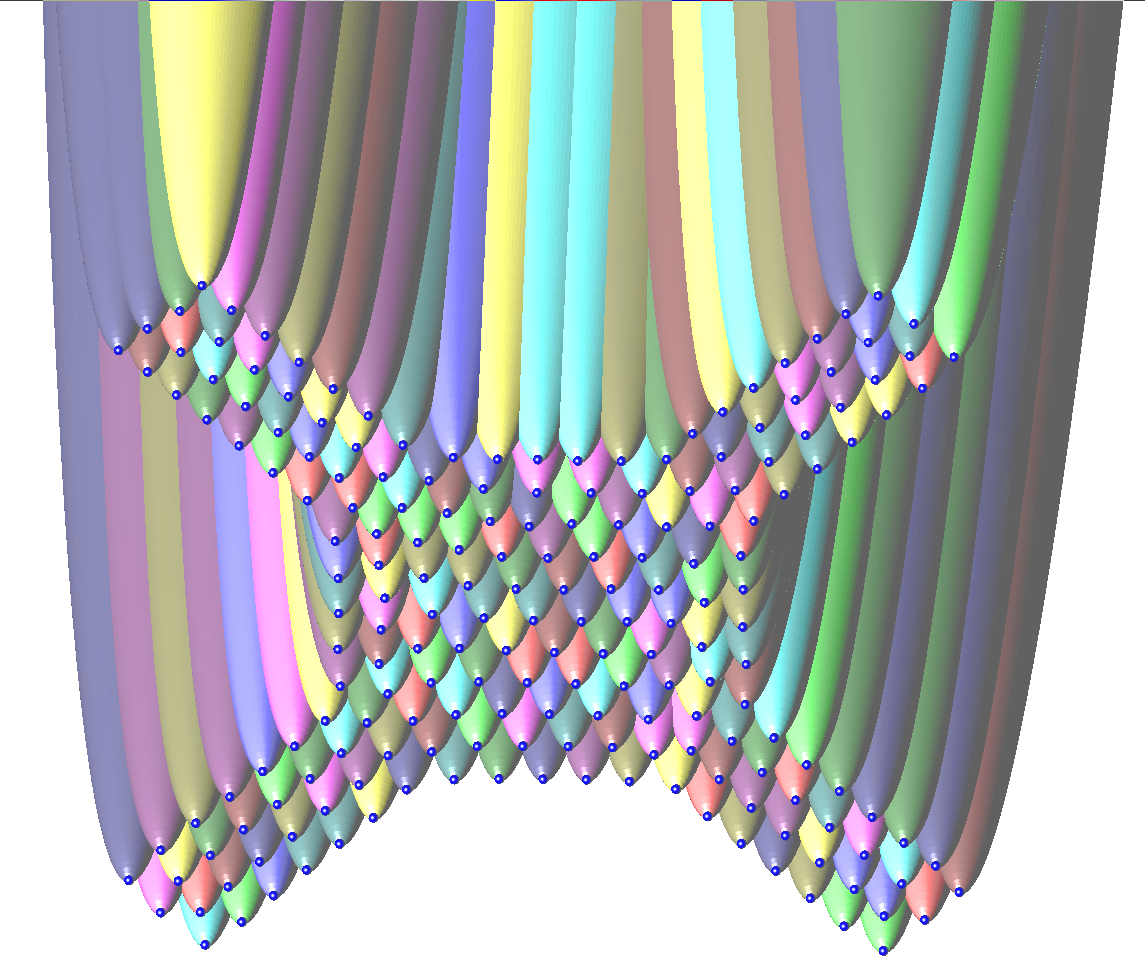}
\put(10,-3){\fcolorbox{black}{white}{a}}
\end{overpic}
\end{minipage}
 & 
 \begin{minipage}{0.3\textwidth}
\begin{overpic}[width=1.0\textwidth]{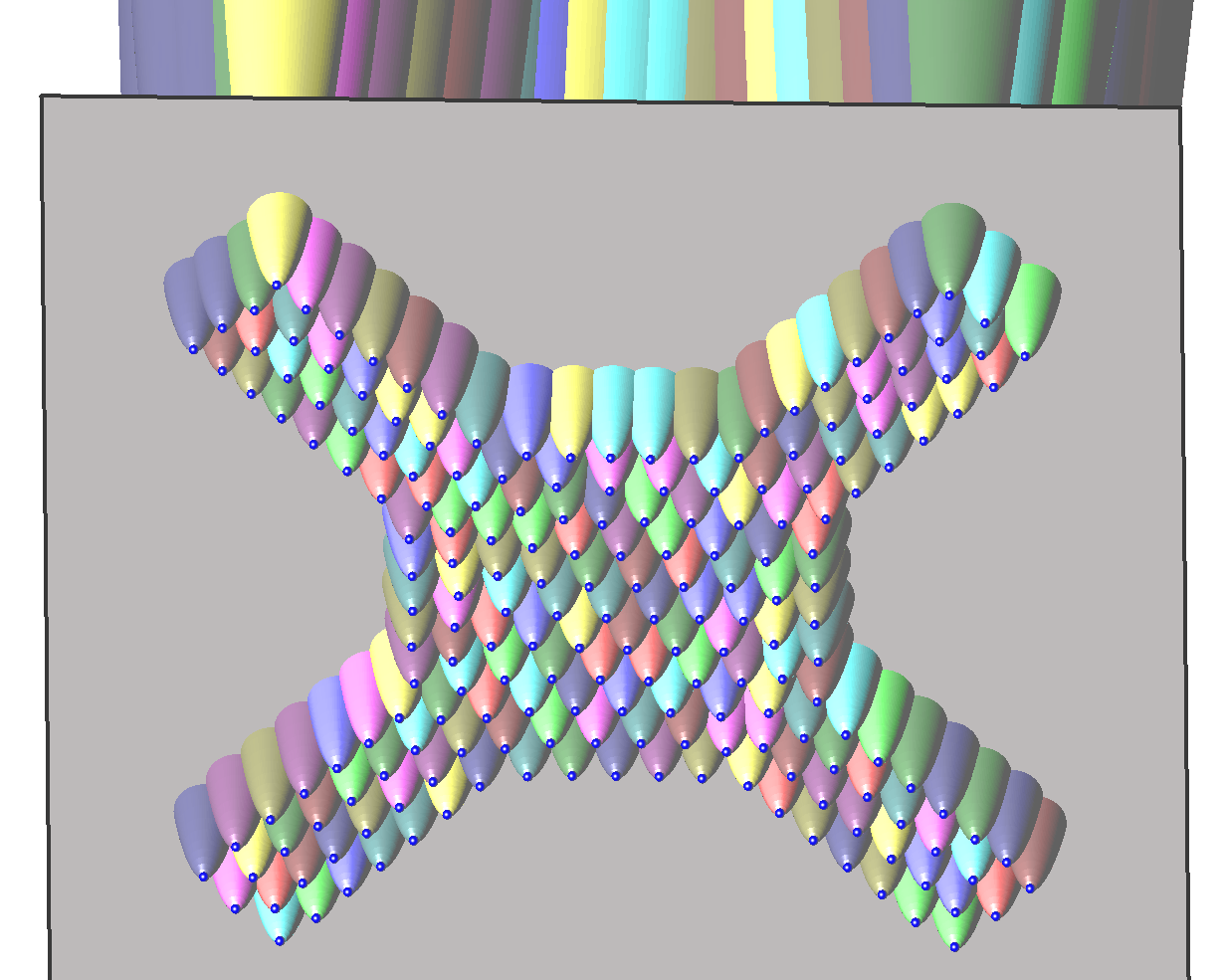}
\put(10,-3){\fcolorbox{black}{white}{b}}
\end{overpic}
\end{minipage} 
&
\begin{minipage}{0.3\textwidth}
\begin{overpic}[width=1.0\textwidth]{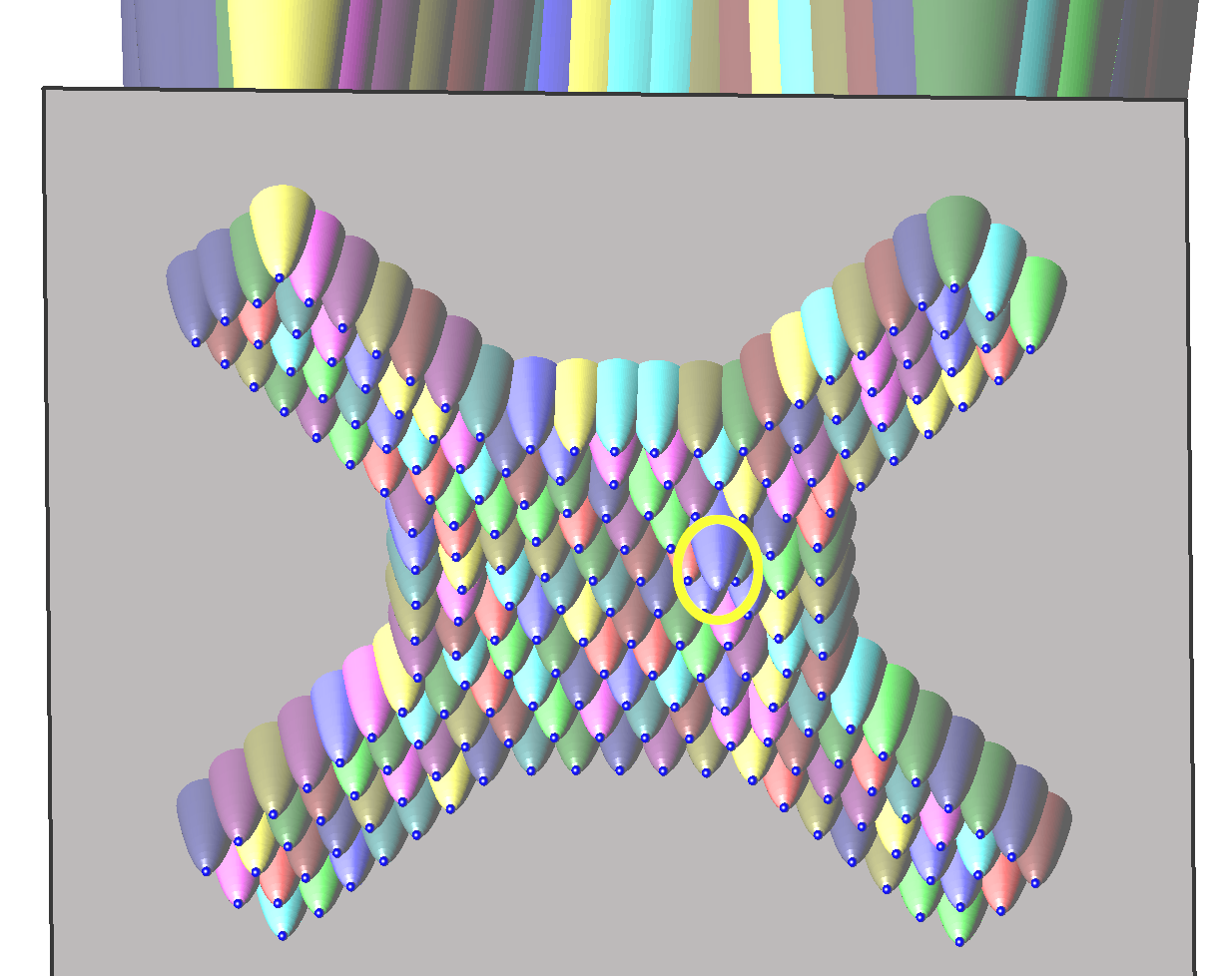}
\put(10,-3){\fcolorbox{black}{white}{c}}
\end{overpic}
\end{minipage} 
\\
\\
\\
\begin{minipage}{0.3\textwidth}
\begin{overpic}[width=1.0\textwidth]{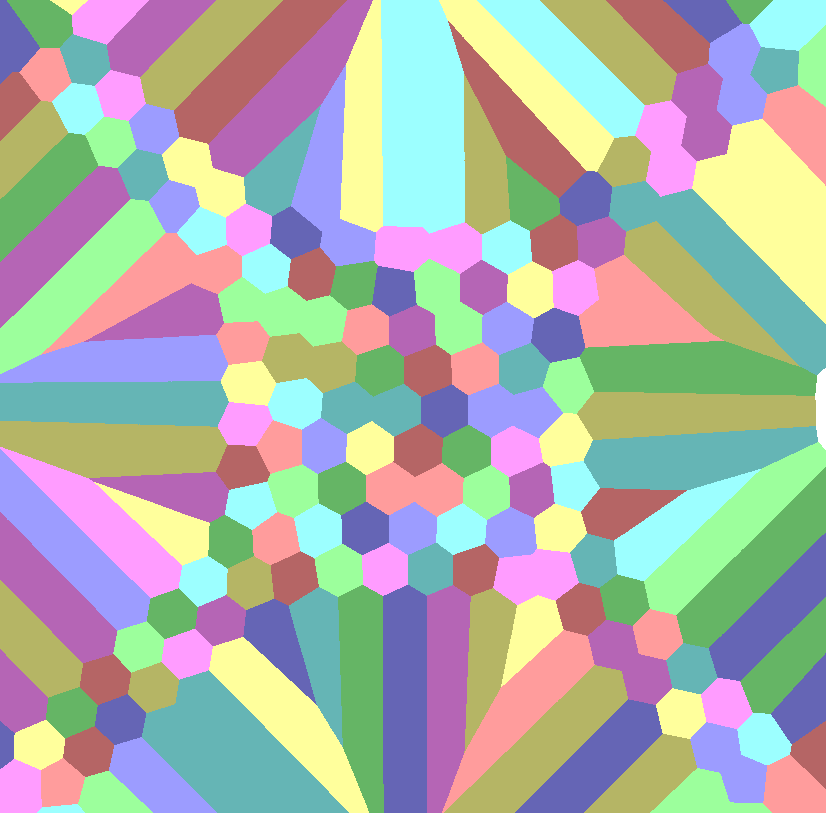}
\put(10,-3){\fcolorbox{black}{white}{d}}
\end{overpic}
\end{minipage} & 
\begin{minipage}{0.3\textwidth}
\begin{overpic}[width=1.0\textwidth]{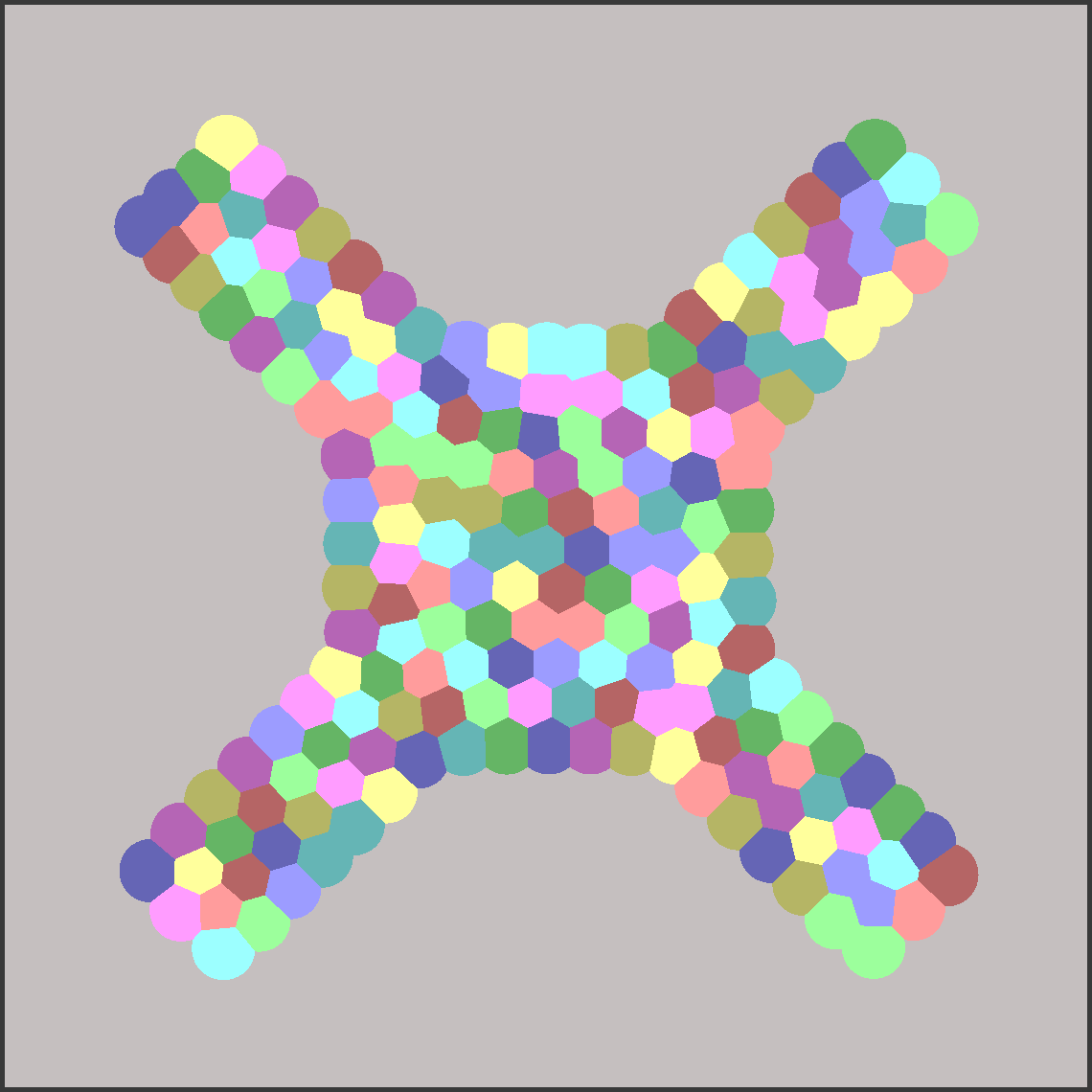}
\put(10,-3){\fcolorbox{black}{white}{e}}
\end{overpic}
\end{minipage}
& 
\begin{minipage}{0.3\textwidth}
\begin{overpic}[width=1.0\textwidth]{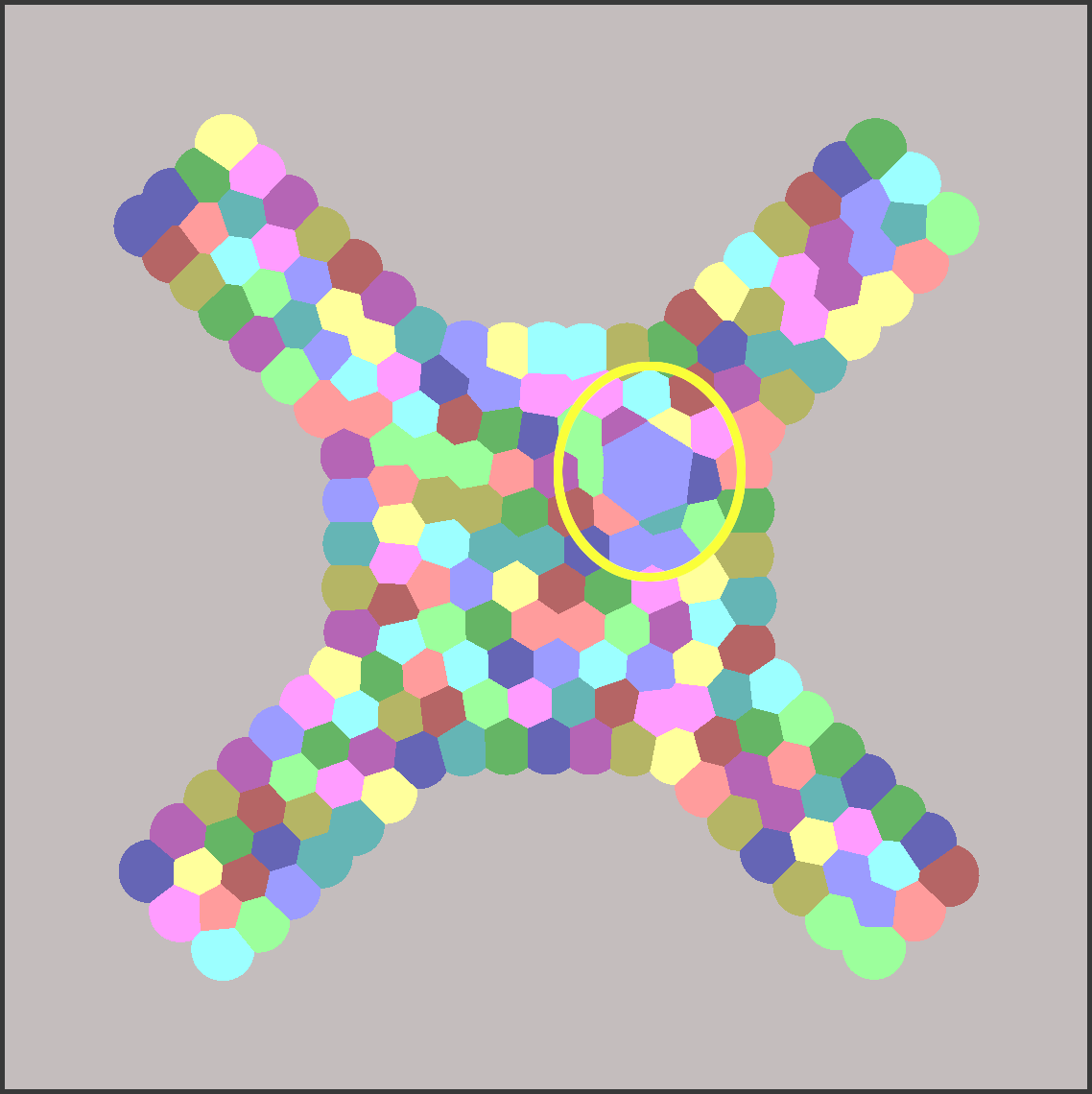}
\put(10,-3){\fcolorbox{black}{white}{f}}
\end{overpic}
\end{minipage}
 \end{tabular}
\end{center}
\caption{\it (a) Two-dimensional Voronoi diagrams (and ``classical'' Laguerre diagrams) can be described by collections of three-dimensional paraboloids with axes aligned with the third coordinate vector; (d) The 2d diagram is obtained by ``seeing this collection from below''; (b) Modified Laguerre diagrams are obtained by inserting the 3d lower half-space in the diagram, and (d) The intersections between the paraboloids and this half-space create circular arcs at the boundaries of the modified cells; (c) The weights $\psi_i$ of a
(classical or modified) Laguerre diagram correspond to shifts of the parabolas along the vertical axis; (f) This results in a change of the area of the cells.}
    \label{fig.minimization}
\end{figure}

An insightful geometric viewpoint about classical and modified Laguerre diagrams \cref{eq.defLaguerre,eq.defLagi} and \cref{eq.decompOm,eq.defVipsi} is to consider them as minimization diagrams of a family of functions, as we now briefly explain in the case $d=2$ for simplicity, see \cref{fig.minimization} for an illustration. 
The Voronoi diagram $\bVor(\s)$ of a set of seed points $\s \in \R^{2N}$ is generated by the 3d paraboloids given by the graphs of the functions $\x \mapsto p_i(\x) = \lvert \x - \s_i \lvert^2$, in the sense that:
$$ \text{For each } i=1,\ldots,N, \quad \Vor_i(\s) = \left\{ \x \in \overline D \text{ s.t. } i \in \argmin\limits_{j=1,\ldots,N} p_j(\x) \right\},$$
i.e. $\Vor_i(\s)$ is the set of points $\x \in \overline D$ where $p_i(\x)$ is minimal among $p_1(\x),\ldots,p_N(\x)$.
The ``classical'' Laguerre diagram $\bLag$ involving weights $\bpsi \in \R^N$ is obtained by ``shifting'' each of the aforementioned 3d paraboloids by the corresponding weight $\psi_i$ in the negative $\be_3$ direction.
Eventually, a corresponding representation of the modified diagram $\bVsp$ is obtained by adding the function $p_0(\x) = 0$ to the minimization diagram, i.e. by intersecting each of the 3d paraboloids with the lower half-space $\left\{ \x = (x_1,x_2,x_3) \in \R^3 \text{ s.t. }x_3 \leq 0 \right\}$.

\subsection{Parametrization of the diagram representation via the seed points and cell measures}\label{sec.Lagrefvol}

\noindent The effect of the weights $\bpsi$ on the shapes of the cells $V_i(\s,\bpsi)$ as encoded in \cref{eq.defVipsi} is implicit and difficult to handle. 
It turns out that the decomposition \cref{eq.decompOm} of $\Omega$ can be equivalently characterized by the seed points $\s \in \R^{dN}$ and the vector 
$$\bnu :=\left\{ \nu_i \right\}_{i=1,\ldots,N} \in \R^N, \quad \nu_i = \lvert V_i(\s,\bpsi)\lvert$$ 
gathering the measures of the cells. 
This section discusses this convenient alternative viewpoint which will be used throughout the following. 
It relies on optimal transport; for the convenience of the reader, we recall a few useful facts about this theory, 
referring to \cite{ levy2015numerical,levy2018notions,peyre2019computational,santambrogio2015optimal,villani2009optimal} for more exhaustive treatments.\par\medskip

Let $X$ stand for the compact set $\overline D$. 
For $\x,\y \in X$, we denote by $c(\x,\y) = \lvert \x - \y \lvert^2$ the quadratic ground cost on $X$, 
which intuitively represents the cost of moving one unit of mass from the position $\x$ to $\y$.
When $\mu$ and $\nu$ are positive measures on $X$ with equal mass $\mu(X) = \nu(X)$, the Monge formulation of the optimal transport problem of $\mu$ onto $\nu$ consists in the search for a mapping $T: X \to X$ realizing the transfer of the mass of $\mu$ onto $\nu$ at minimal cost, that is:
\begin{equation}\label{eq.OTMonge}
\tag{\textcolor{gray}{OT-M}}
 \min \left\{ \int_{X} c(\x, T(\x)) \:\d \mu(\x), \:\: T : X \to X \text{ is } (\mu, \nu)\text{-measurable with } T_\# \mu = \nu \right\},
 \end{equation}
where $T_\# \mu$ is the push-forward of the measure $\mu$ by $T$, i.e.
$$ \text{For all } \nu\text{-measurable subset } B \subset X, \quad T_\#\mu(B)= \mu(T^{-1}(B)).$$
Unfortunately, the problem \cref{eq.OTMonge} is often ill-posed; 
for instance, an admissible mapping $T$ in \cref{eq.OTMonge} may easily fail to exist. 
One therefore usually considers instead the so-called Kantorovic relaxation of \cref{eq.OTMonge}:
\begin{equation}\label{eq.OTKanto}
\tag{\textcolor{gray}{OT-K}}
 \min \left\{ \int_{X \times X} c(\x,\y) \:\d \pi(\x,\y), \:\: \pi \text{ is a positive measure on } X \times X \text{ s.t. } \pi_1 = \mu, \: \pi_2 = \nu \right\}.
 \end{equation}
 Here, $\pi_1$ and $\pi_2$ denote the first and second marginals of the measure $\pi$ on $X \times X$:
 $$ \text{For all } \pi_1, \pi_2\text{-measurable subsets } A \subset X, \: B\subset X,  \quad \pi_1(A) = \pi(A\times X) \:\:  \text{and }  \pi_2(B) = \pi(X\times B).$$
 In \cref{eq.OTKanto}, the minimization is realized over positive measures $\pi$ on $X\times X$ (also called transport plans or couplings) with marginals $\mu$ and $\nu$. Intuitively, when $A$ and $B$ are $\mu$- and $\nu$- measurable subsets of $X$, $\pi(A\times B)$ is the ``quantity of mass'' transferred from $A$ to $B$.
The problem \cref{eq.OTKanto} is often better behaved than \cref{eq.OTMonge} from the mathematical viewpoint; for instance, it is immediate to see that it is convex. \par\medskip

Let us specialize this abstract setting to our purpose. 
Let $\s = \left\{ \s_i \right\}_{i=1,\ldots,N} \in \R^{dN}$ be a collection of seed points, and let $\bpsi = \left\{ \psi_i \right\}_{i=1,\ldots,N}$ be a weight vector. 
According to \cref{rem.s0} we let $\psi_0 := 0$ be the weight for the seed object $\s_0 = D$. 
Likewise, when $\bnu = \left\{\nu_i \right\}_{i=1,\ldots,N}$ is a collection of volume targets such that 
$$\forall i =1,\ldots, N, \quad \nu_i \geq 0, \quad \text{ and } \quad \sum_{i=0}^N \nu_i  \leq |D|,$$
we denote by $\nu_0 := |D| -  \sum_{i=1}^N \nu_i$ the volume of the corresponding ``void phase'' $V_0(\s,\bpsi)$ in \cref{eq.V0}.  
We then aim to apply the aforementioned optimal transport theory to the situation where $\mu := \mathds{1}_D \:\d x$ is the restriction to $X$ of the $d$-dimensional Lebesgue measure and \begin{equation}\label{eq.nuOT}
\nu := \sum_{i=1}^N \nu_i \delta_{\s_i} +\nu_0 \frac{1}{|D|} \mathds{1}_D \:\d x
\end{equation}
is the discrete measure induced by the objects $\left\{\s_i \right\}_{i=1,\ldots,N} \cup \left\{\s_0 \right\}$. In this perspective,
we shall often invoke the following assumptions about $\s$ and $\bpsi$ to rule out very degenerate situations as far as the diagram $\bVsp$ is concerned: 
\begin{equation}\label{eq.Gen0}
\tag{\textcolor{gray}{G1}} \text{All the cells } V_i(\s,\bpsi) \text{ have positive Lebesgue measure, $i=0,1,\ldots,N$ (in particular, they are non empty).}
\end{equation}
\begin{equation}\label{eq.Gen1}
\tag{\textcolor{gray}{G2}} \text{No three distinct seed points } \s_i, \: \s_j, \: \s_k \text{ are aligned.}
\end{equation}
\begin{multline}\label{eq.Gen2}
\tag{\textcolor{gray}{G3}} \text{For each vertex } \q \text{ at the intersection between 2 neighboring cells } V_i(\s,\bpsi), \: V_j(\s,\bpsi), \:\: i,j=1,\ldots,N, \\ \text{ and the void phase } V_0(\s,\bpsi), \:\:  \q,\: \s_i \text{ and } \s_j \text{ are not aligned.}
\end{multline}
\begin{multline}\label{eq.Gen3}
\tag{\textcolor{gray}{G4}} \text{Each vertex } \q \text{ is at the intersection of at most 3 distinct cells }\\ 
V_i(\s,\bpsi), \: V_j(\s,\bpsi), V_k(\s,\bpsi), \:\: i,j,k=0,1,\ldots,N.
\end{multline}
\begin{multline}\label{eq.Gen4}
\tag{\textcolor{gray}{G5}} \text{The boundary } \partial D \text{ of the computational domain is smooth at each vertex } \q \text{ lying on } \partial D. \\  \text{Such a point is moreover at the intersection of at most 2 distinct cells } 
V_i(\s,\bpsi), \: V_j(\s,\bpsi), \:\: i,j=0,\ldots,N.
\end{multline}
\begin{multline}\label{eq.Gen5}
\tag{\textcolor{gray}{G6}} \text{For each vertex } \q \text{ at the intersection of } \partial D \text{ and 2 distinct cells } V_i(\s,\bpsi), \: V_j(\s,\bpsi), \: i,j=1,\ldots,N,\\ 
\text{the line } \partial V_i(\s,\bpsi) \cap \partial V_j(\s,\bpsi) \text{ is not orthogonal to } \n(\q).
\end{multline}
\begin{multline}\label{eq.Gen6}
\tag{\textcolor{gray}{G7}} \text{For each vertex } \q \text{ at the intersection between } \partial D, \text{ one cell } V_i(\s,\bpsi) \text{ and the void } V_0(\s,\bpsi),\\ 
\text{the line } \q\s_i \text{ is not orthogonal to } \n(\q).
\end{multline}

The desired description of shapes $\Omega$ in terms of the seed points $\s$ and cell measures $\bnu$ of an associated diagram \cref{eq.decompOm} is enabled by the following result. Its proof, which is an adaptation of those of Theorems 37 and 40 in \cite{merigot2021optimal}, 
is postponed to \cref{app.existunique}.

\begin{theorem}\label{prop.psi}
Let $D \subset \R^d$ be a bounded, Lipschitz domain and $X := \overline D$. 
Let $\s = \left\{\s_i\right\}_{i=1,\ldots,N} \in \R^{dN}$ be a set of seed points in $D$,
and let $\bnu =\left\{ \nu_i \right\}_{i=1,\ldots,N} \in \R^N$ be a vector of associated cell measures, satisfying: 
\begin{equation}\label{eq.assumnu}
\forall i =1,\ldots, N, \quad \nu_i > 0, \quad \text{ and } \quad \sum_{i=1}^N \nu_i  < |D|.
\end{equation}
The following statements hold true:
\begin{enumerate}[(i)]
\item There exists a unique weight vector $\bpsi^* \equiv \bpsi^*(\s,\bnu) \in \R^N$ such that:
\begin{equation}\label{eq.volPsi}
\forall i = 1,\ldots, N , \quad  |V_i (\s,\bpsi^*) |= \nu_i, \text{ and so } \left\lvert V_0 (\s,\bpsi^*) \right\lvert =  \nu_0.
\end{equation}
\item The vector $\bpsi^*(\s,\bnu)$ is one solution to the maximization problem
\begin{equation}\label{eq.maxK}
 \max \limits_{\bpsi \in \R^{N}} K(\s,\bnu,\bpsi), 
 \end{equation}
where the so-called Kantorovic functional $K : \R^{dN}_{\s} \times \R^N_{\bnu} \times \R^N_{\bpsi} \to \R$ is defined by:
\begin{equation}\label{eq.defK}
 K(\s,\bnu,\bpsi) := \sum\limits_{i=1}^N \int_{V_i(\s,\bpsi)} \left(|\x - \s_i |^2 - \psi_i  \right) \d \x + \sum\limits_{i=1}^N \nu_i \psi_i .
 \end{equation}
\item Assuming that the genericity conditions \cref{eq.Gen0,eq.Gen1,eq.Gen2,eq.Gen3,eq.Gen4,eq.Gen5,eq.Gen6} hold true, this solution is unique.
Moreover, $K$ is smooth on an open neighborhood of $(\s,\bnu, \bpsi^*(\s,\bnu))$ in $\R^{dN}_{\s} \times \R^N_{\bnu} \times \R^N_{\bpsi}$, and $\bpsi^*(\s,\bnu)$ is the unique solution to the equation:
\begin{equation}\label{eq.implicitPsi}
 \bF(\s, \bnu, \bpsi^*(\s,\bnu)) = \bzero, \text{ where } \bF: \R^{dN}_\s \times \R^N_\bnu \times \R^N_\bpsi \to \R^N \text{ is defined by } \bF(\s,\bnu,\bpsi) := \nabla_\bpsi K(\s,\bnu,\bpsi).
 \end{equation}
\item There exists a unique optimal transport mapping $T_{\bpsi}$ between the measures $\mu$ and $\nu$ in the Monge formulation \cref{eq.OTMonge}, which is given by:
\begin{equation}\label{eq.Tpsi}
\text{For a.e. } \x \in X, \quad T_{\bpsi}(\x) = \left\{
\begin{array}{cl}
\s_i & \text{if } \x \in V_i(\s,\bpsi^*(\s,\bnu)), \:\: i=1,\ldots,N,\\
\x & \text{otherwise}.
\end{array}
\right.
\end{equation}
\end{enumerate}
\end{theorem}

\begin{remark}\label{rem.suffconddiff}
The uniqueness statement $(iii)$ in the above theorem actually holds true under weaker assumptions than the collection \cref{eq.Gen0,eq.Gen1,eq.Gen2,eq.Gen3,eq.Gen4,eq.Gen5,eq.Gen6}. It only requires that the mapping $\bpsi \mapsto K(\s,\bnu,\bpsi)$ be differentiable, i.e. loosely speaking that the area $\lvert V_i(\s,\bpsi) \lvert$ of each cell be differentiable with respect to $\bpsi$, while \cref{eq.Gen0,eq.Gen1,eq.Gen2,eq.Gen3,eq.Gen4,eq.Gen5,eq.Gen6} imply that all the individual vertices of the cells of $\bVsp$ are differentiable with respect to $\bpsi$, see \cref{rem.diffverArea} below about this point. 
\end{remark}

\begin{remark}
The situation where $\sum_{i=1}^N \nu_i = |D|$ corresponds to the ``classical'' setting where the cell $V_0(\s,\bpsi)$ is empty and the diagram $\bVsp$ coincides with the Laguerre diagram $\bLag$. 
Then, there still exists a collection of weights $\bpsi$ fulfilling the conclusions of \cref{prop.psi}, 
but the latter is only unique up to the addition of a common factor to all the weights $\psi_i$, $i=1,\ldots,N$ \cite{merigot2021optimal}. 
\end{remark}

\begin{remark}\label{rem.linearOT}
Our representation \cref{eq.decompOm} of shapes and their deformations echoes to the framework of ``linearized optimal transport'', recently proposed in \cite{delalande2021quantitative,merigot2020quantitative,wang2013linear}. Briefly, the latter advocates to parametrize measures $\nu$ on $X$ by the unique optimal transport mapping between a fixed reference measure $\mu$ on $X$ and $\nu$, i.e. the solution to \cref{eq.OTMonge}. By comparison, our parametrization \cref{eq.decompOm} of shapes $\Omega \subset D$ boils down to considering the particular class \cref{eq.Tpsi} of optimal transport mapping between the uniform measure $\mu = \frac{1}{\lvert D\lvert} \mathds{1}_D \:\d \x$ on $X$ and measures $\nu$ of the form \cref{eq.nuOT}. 
\end{remark}

\subsection{General sketch of the Laguerre diagram-based shape and topology optimization method}\label{sec.Lagalgo}

\noindent This section outlines our algorithmic strategy for the resolution of a generic shape and topology optimization problem of the form \cref{eq.sopb}, 
in which the objective and constraint functionals $J(\Omega)$ and $\G(\Omega)$ depend on the optimized shape $\Omega$ via the solution $u_\Omega$ to a boundary value problem posed on $\Omega$, as in the examples of \cref{sec.twophys}.
The method is summarized in \cref{algo.lagevol}; its main steps are illustrated on \cref{fig.illusalgo} and they are more thoroughly described in the next sections of the article.

\begin{algorithm}[ht]
\caption{Laguerre diagram based shape and topology optimization algorithm for \cref{eq.sopb}.}
\label{algo.lagevol}
\begin{algorithmic}[0]
\STATE \textbf{Input:} Initial shape $\Omega^0$ characterized via \cref{eq.decompOm,eq.defVipsi} by the datum of:
\begin{itemize}
\item The collection of seed points $\s^0 = \left\{\s^0_1,\ldots,\s_N^0 \right\}$; 
\item The vector $\bnu^0 :=\left\{\nu_1^0, \ldots,\nu_N^0 \right\}$ gathering the measures of the cells.
\end{itemize}
\FOR{$n=0,...,$ until convergence}
\STATE \begin{enumerate}
\item Calculate the unique weight vector $\bpsi^n \in \R^N$ such that:
$$\forall i =1,\ldots,N, \quad \lvert V_i(\s^n,\bpsi^n) \lvert = \nu_i^n. $$ 
\item (Optionally) Modify this diagram by resampling and Lloyd regularization. 
\item Construct a polygonal mesh $\calT^n$ for $\Omega^n$ from the diagram $\mathbf{V}(\s^n,\bpsi^n)$. 
\item Calculate the solutions $u_{\Omega^n}$, $p_{\Omega^n}$ to the state and adjoint equations on the mesh $\calT^n$.
\item Calculate a descent direction $(\h^n,\widehat\bnu^n)$ for the problem \cref{eq.sopb}, that is:
\begin{itemize} 
\item A collection $\h^n = \left\{ \h_i^n \right\}_{i=1,\ldots,N} \in \R^{dN}$ of vectors attached to the seed points $\s^n$;
\item An update vector $ \widehat{\bnu}^n \in \R^N$ for the measures. 
\end{itemize}
\item Update the seed points and measure vectors $\s^n$ and $\bnu^n$ as:
$$ \s_i^{n+1} =\s_i^n + \tau^n \h_i^n, \text{ and } \bnu^{n+1} = \bnu^n + \rho^n  \widehat{\bnu}^n,$$
where $\tau^n > 0$ and $\rho^n > 0$ are suitable descent steps.
\end{enumerate}
\ENDFOR
\RETURN Seed points $\s^n$ and measures $\bnu^n$ representing the optimized shape $\Omega^n$ via \cref{eq.decompOm,eq.defVipsi}.
\end{algorithmic}
\end{algorithm}

\begin{figure}[!ht]
\centering
\begin{tabular}{ccc}
\begin{minipage}{0.49\textwidth}
\begin{overpic}[width=1.0\textwidth]{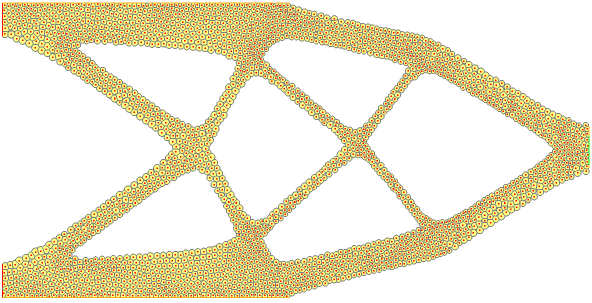}
\put(0,-3){\fcolorbox{black}{white}{a}}
\end{overpic}
\end{minipage}
 & 
 \begin{minipage}{0.49\textwidth}
\begin{overpic}[width=1.0\textwidth]{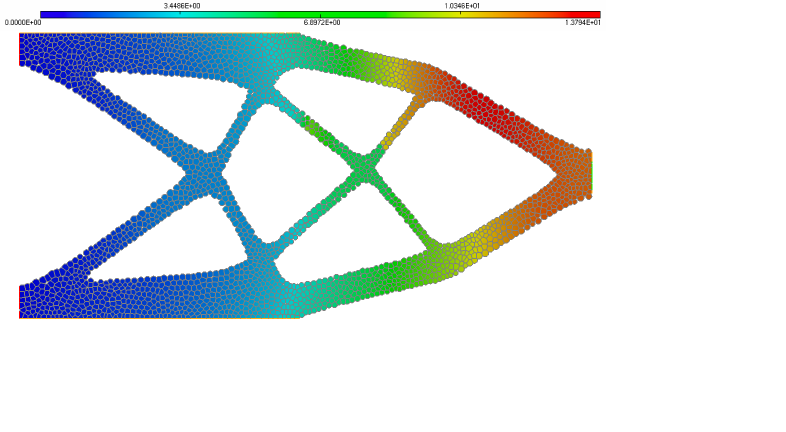}
\put(0,-3){\fcolorbox{black}{white}{b}}
\end{overpic}
\end{minipage} 
\\
\\
\begin{minipage}{0.49\textwidth}
\begin{overpic}[width=1.0\textwidth]{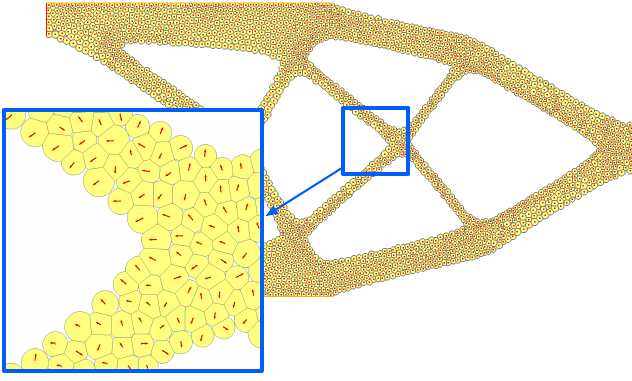}
\put(0,-3){\fcolorbox{black}{white}{c}}
\end{overpic}
\end{minipage} & 
\begin{minipage}{0.49\textwidth}
\vspace{-0.5cm}
\begin{overpic}[width=1.0\textwidth]{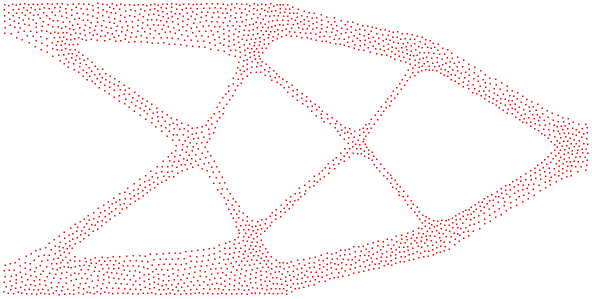}
\put(0,-3){\fcolorbox{black}{white}{d}}
\end{overpic}
\end{minipage}
 \end{tabular}
 \par\medskip
 \begin{center}
 \begin{minipage}{0.49\textwidth}
\begin{overpic}[width=1.0\textwidth]{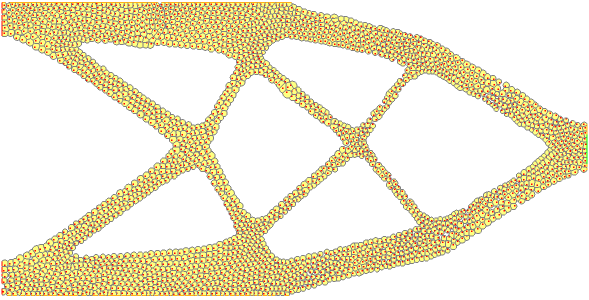}
\put(0,-3){\fcolorbox{black}{white}{e}}
\end{overpic}
\end{minipage} 
\end{center}
\caption{\it Illustration of the main steps of \cref{algo.lagevol}; (a) At the beginning of each iteration, the weight vector $\bpsi^n$ is computed, and the diagram $\mathbf{V}(\s^n,\bpsi^n)$ for $\Omega^n$ is constructed (the seed points are depicted in red); (b) The state equation for $u_{\Omega^n}$ is solved by the Virtual Element Method on the polygonal mesh $\calT^n$ of $\Omega^n$ induced by $\mathbf{V}(\s^n,\bpsi^n)$; (c) A descent direction is calculated as a vector field $\h^n$ for the seed points and an update $\bnu^n$ for the cell measures; (d) The seed points $\s^{n+1}$ of the new diagram for $\Omega^{n+1}$ are obtained by moving those in $\s^n$ along $\h^n$; the cell measures $\bnu^{n+1}$ are obtained likewise; (e) The new weight vector $\bpsi^{n+1}$ and the diagram $\mathbf{V}(\s^{n+1},\bpsi^{n+1})$ for $\Omega^{n+1}$ are computed.}
\label{fig.illusalgo}
\end{figure}

At each iteration $n = 0, \ldots$, we label with an $^n$ superscript the actual instances of the various objects at stake.
The shape $\Omega^n \subset D$ is consistently described by the seed points $\s^n \in \R^{dN}$ and the vector $\bnu^n \in \R^N$ of the cell measures of a diagram $\mathbf{V}(\s^n, \bpsi^n)$ of the form \cref{eq.decompOm,eq.defVipsi}. 
The iteration starts with the calculation of the unique weight vector $\bpsi^n \in \R^N$ such that the measure of each cell $V_i(\s^n,\bpsi^n)$ equals $\nu_i$, see \cref{sec.seedstoLag}. 
The diagram associated to these data is computed, as described in \cref{sec.seedstoLag}.
A number of post-processing operations are conducted on the latter, such as regularization, resampling, etc., see \cref{sec.lloyd,sec.resample}. 
Note that these operations may alter the number $N$ of seed points -- a detail which is omitted in the present sketch for simplicity.
A polygonal mesh $\calT^n$ of $\Omega^n$ is then constructed from this diagram, see again \cref{sec.geomcomp}. 
Then, the solution $u_{\Omega^n}$ to the mechanical system of interest (and that $p_{\Omega^n}$ to the adjoint system) is calculated on the mesh $\calT^n$; this task relies on the Virtuel Element Method, and its treatment is described in \cref{sec.mechcomp}. 
The derivatives of the objective and constraint functions $J(\Omega)$ and $\G(\Omega)$ are evaluated, at first with respect to the vertices $\q^n$ of the polygonal mesh $\calT^n$ -- see \cref{sec.derver} -- then with respect to the seed points and cell measures $\s^n$ and $\bnu^n$ defining the representation \cref{eq.decompOm} of $\Omega^n$, see \cref{sec.vertoseeds}. 
A descent direction  for the considered optimization problem is inferred thanks to a constrained optimization algorithm under the form of update vectors for seeds and cell measures, see \cref{sec.velext,sec.nullspace}. 
This procedure is iterated until convergence.

\section{Geometric computations on $\Omega$}\label{sec.geomcomp}

\noindent  
This section provides a few details about the non trivial operations from algorithmic geometry involved in \cref{algo.lagevol}.
In \cref{sec.seedstoLag}, we discuss the calculation of the unique weight vector $\bpsi^*(\s,\bnu)$ supplied by \cref{prop.psi}, guaranteeing that the diagram $\mathbf{V}(\s,\bpsi^*(\s,\bnu))$ complies with the measure constraints \cref{eq.volPsi}; we notably describe the computation of this diagram. In \cref{sec.lloyd}, we present a variant of the well-known Lloyd's algorithm aimed at improving the aspect of the cells of the diagram $\bVsp$, in the perspective of realizing accurate mechanical computations. Eventually, in \cref{sec.resample}, we sketch simple numerical recipes to adjust the local density of cells and remove ``small'' components disconnected from the main structure in the course of the iterative process. 

\subsection{Calculation of the diagram $\mathbf{V}(\mathbf{s},\mathbf{\nu})$ associated to seed points $\mathbf{s}$ and cell measures $\mathbf{\nu}$}\label{sec.seedstoLag}

\noindent Throughout this section, $\s \in \R^{dN}$ and $\bnu \in \R^N$ are collections of seed points and weights satisfying the assumptions of \cref{prop.psi}; notably, they satisfy \cref{eq.Gen0,eq.Gen1,eq.Gen2,eq.Gen3,eq.Gen4,eq.Gen5,eq.Gen6}. We describe the calculation of the unique weight vector $\bpsi^*(\s,\bnu)$ guaranteeing that each cell $i=1,\ldots,N$ of the modified diagram $\mathbf{V}(\s,\bpsi^*(\s,\bnu))$ has measure $\nu_i$, together with the practical construction of this diagram. 
Our strategy is based on \cref{prop.psi}, whereby $\bpsi^*(\s,\bnu) \in \R^N$ is characterized as the unique solution to the non linear equation
\cref{eq.implicitPsi}, that we rewrite below for convenience: 
\begin{equation}\label{eq.implicitPsi2}
 \bF(\s, \bnu, \bpsi^*(\s,\bnu)) = 0, \text{ where } \bF(\s,\bnu,\bpsi) := \nabla_\bpsi K(\s,\bnu,\bpsi).
 \end{equation}
This feature paves the way to a Newton-Raphson algorithm \cite{DBLP:journals/corr/KitagawaMT16}, based on the differentiability of $\bF(\s,\bnu,\cdot)$, and thus on the second order differentiability of the Kantorovic functional $K(\s,\bnu,\cdot)$ with respect to weights $\bpsi$, which hold true since $\s$ and $\bnu$ are generic in the sense of \cref{eq.Gen0,eq.Gen1,eq.Gen2,eq.Gen3,eq.Gen4,eq.Gen5,eq.Gen6}. 

This numerical strategy is sketched in \cref{algo.Newton} and it is illustrated in \cref{fig.POT_KMT}. Briefly, the algorithm starts with a suitable initial guess $\bpsi^0$. At each iteration $n=0,\ldots$, whose corresponding objects are labelled with an $^n$ superscript, we first construct the diagram $\mathbf{V}(\s,\bpsi^n)$. The entries of the gradient $\nabla_{\bpsi}K(\s,\bnu,\bpsi^n) \in \R^N$ and Hessian matrix $[\nabla^2_{\bpsi} K(\s,\bnu,\bpsi^n)] \in \R^{N\times N}$ of the Kantorovic functional are then calculated from this geometric support. The update step $\bp^n \in \R^N$ for the weight vector is then obtained by the solution of the following Newton linear system with size $N$:
\begin{equation}\label{eq.Newtonpn}
\left[\nabla_{\bpsi} \bF(\s,\bnu,\bpsi^n) \right] \p^n = -\bF (\s,\bnu,\bpsi^n).
\end{equation}
Eventually, a descent parameter $\alpha^n >0$ is chosen to update the weight vector $\bpsi^n$ into the next iterate $\bpsi^{n+1}$. We provide a little more details about the main steps of this process in the next subsections.

\begin{algorithm}[!ht]
\caption{Computation of the diagram $\bVsp$ associated to given seed points $\s$ and cell measures $\bnu$}
\label{algo.Newton}
\begin{algorithmic}[0]
\STATE \textbf{Inputs:}
\begin{itemize}
\item Fixed collection $\s= \left\{\s_1,\ldots,\s_N\right\}$ of seed points.
\item Fixed collection $\bnu= \left\{\nu_1,\ldots,\nu_N\right\}$ of cell measures.
\item Initial weight vector $\bpsi^0 = \{ \psi_1^0, \ldots, \psi_N^0\}$ such that for all $i=1,\ldots,N$, $V_i(\bs,\bpsi^0) \neq \emptyset$.
\item Tolerance parameter $\e_{\text{Newt}}$ about the fulfillment of \cref{eq.implicitPsi2}.
\end{itemize}
\FOR{$n=0,...,$ until convergence}
\STATE 
\begin{enumerate}
\item Compute the Laguerre diagram $\mathbf{V}(\s,\bpsi^n)$ 
\item Calculate the gradient $\bF(\s,\bnu,\bpsi^n) = \nabla_{\bpsi} K(\s,\bnu,\bpsi^n)$ of the functional $K(\s,\bnu,\cdot)$ at $\bpsi^n$.
\item \textbf{if} {$\lvert \bF(\s,\bnu,\bpsi^n) \lvert_\infty < \e_{\text{Newt}} $} \textbf{then exit loop}
\item Calculate the $N \times N$ Hessian matrix $[\nabla_{\bpsi}\bF(\s,\bnu,\bpsi^n)] = [\nabla^2_{\bpsi} K (\s,\bnu,\bpsi^n)]$ of $K(\s,\bnu,\cdot)$ at $\bpsi^n$.
\item Calculate the solution $\bp^n \in \R^N$ to the linear system \cref{eq.Newtonpn}.
\item Select a suitable descent parameter $\alpha^n >0$.
\item Update the weight vector as: $\bpsi^{n+1} = \bpsi^n + \alpha^n \bp^n$.
\end{enumerate}
\ENDFOR
\RETURN Diagram $\bVsp$ satisfying $\lvert V_i(\s,\bpsi^n)\lvert = \nu_i$, for all $i=1,\ldots,N$, up to precision $\e_{\text{Newt}}$.
\end{algorithmic}
\end{algorithm}

\begin{figure}[!ht]
\begin{center}
\begin{minipage}{\textwidth}
\begin{overpic}[width=1.0\textwidth]{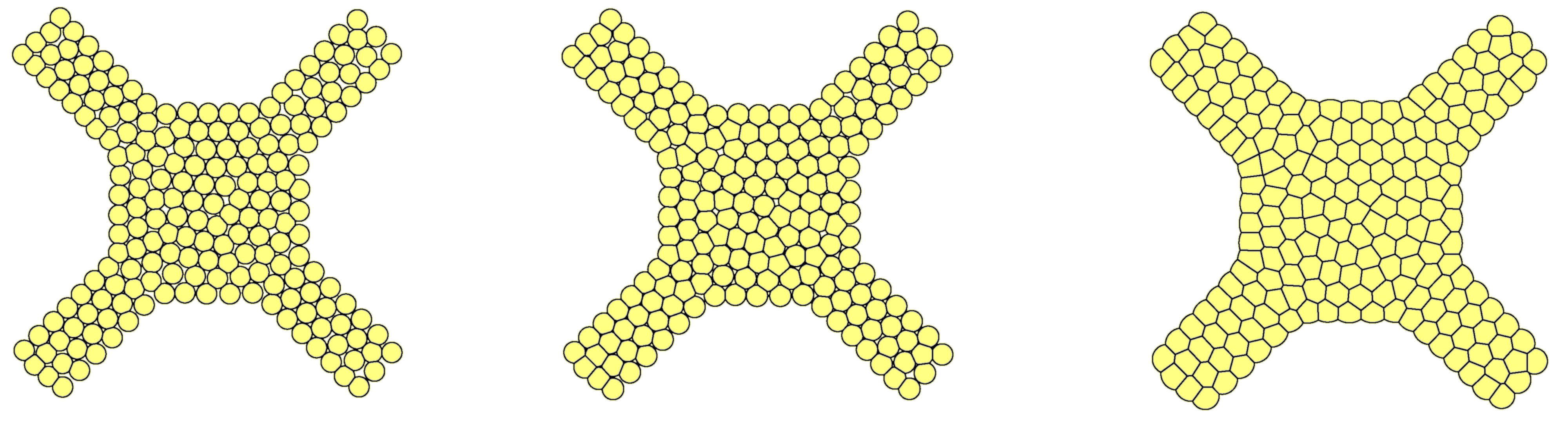}
\put(50,0){\fcolorbox{black}{white}{n=0}}
\put(215,0){\fcolorbox{black}{white}{n=3}}
\put(390,0){\fcolorbox{black}{white}{n=5}}
\end{overpic}
\end{minipage}
\end{center}
\vspace{5mm}
\begin{center}
\begin{minipage}{\textwidth}
\begin{overpic}[width=1.0\textwidth]{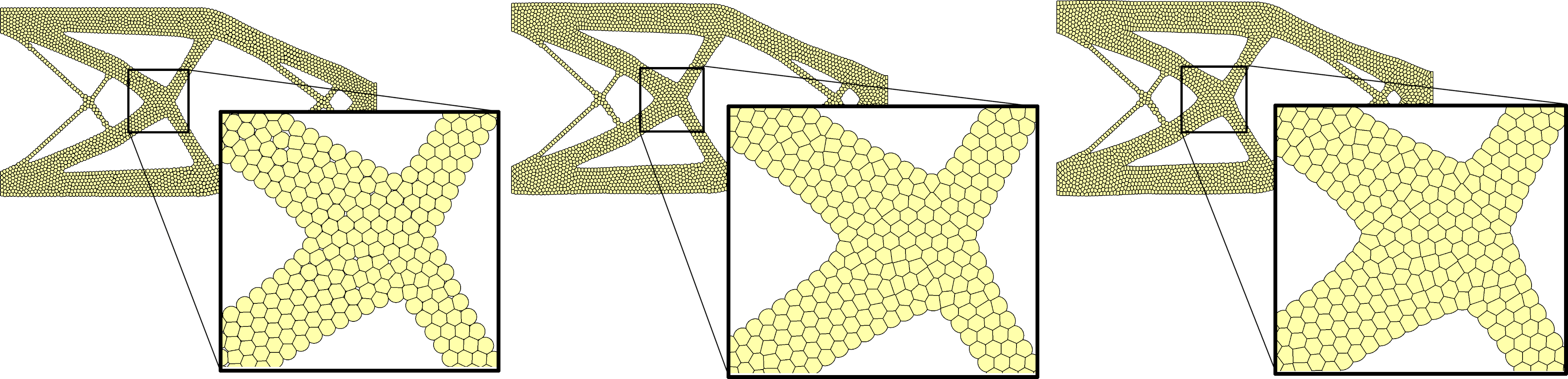}
\put(10,40){\fcolorbox{black}{white}{n=0}}
\put(175,40){\fcolorbox{black}{white}{n=4}}
\put(330,35){\fcolorbox{black}{white}{n=7}}
\end{overpic}
\end{minipage}
\end{center}


\caption{\it Three different Newton iterations in the calculation of the modified diagram $\bVsp$ in \cref{sec.seedstoLag}. The weight vector $\bpsi$ is initialized with a value $\bpsi^0$ ensuring that none of the cells $V_i(\s,\bpsi^0)$ is empty, and it is iteratively updated by Newton iterations. This procedure appraises contacts between neighboring cells, and at convergence, $\bpsi^n$ is such that the measure of each cell $\lvert V_i(\bs, \bpsi^n)\lvert$ corresponds to the prescribed value $\nu_i$, up to the precision parameter $\e_{\text{Newt}}$.}
\label{fig.POT_KMT}
\end{figure}

\subsubsection{Initialization}\label{sec.Laginit}

\noindent According to \cref{prop.psi}, the initial weight vector $\bpsi^0$ must be such that all the cells $V_i(\s,\bpsi^0)$, $i=0,1,\ldots,N$ (i.e. including the ``void'' phase $V_0(\s,\bpsi^0)$ given by \cref{eq.V0}) are non empty. In particular, this property is meant to ensure that the Kantorovic mapping $K(\s,\bnu,\cdot)$ is differentiable at $\bpsi^0$, see \cref{rem.suffconddiff}.

One sufficient condition for this to hold is to select a common positive value $\overline\psi > 0$ for all the $\psi^0_i$, $i=1,\ldots,N$. Indeed, the classical Laguerre diagram $\mathbf{Lag}(\s,\bpsi^0)$ then coincides with the Voronoi diagram $\mathbf{Vor}(\s)$ associated to the seed points $\s$, where each cell $\Vor_i(\s)$ contains at least its associated seed point $\s_i$. In turn, each modified cell $V_i(\s,\bpsi^0) = \Lag_i(\s,\bpsi^0) \cap \overline{B(\s_i,\psi_i^{1/2})}$ is non empty, $i=1,\ldots,N$, and since the common value $\overline\psi$ is arbitrary, it is easy to guarantee that $V_0(\s,\bpsi^0)$ is also non empty. In practice, we choose $\bpsi^0 = \left\{1, 1, \ldots, 1\right\}$ as initial weight vector. 

If all the target measures $\nu_i$ are equal to a common value $\overline\nu$, a simple heuristic suggests an initialization procedure with ``better'' practical behavior. Intuitively, if a seed point $\s_i$ lies ``sufficiently far away'' from all the other seeds $\s_j$, $j \neq i$,
then the classical Laguerre cell $\Lag_i(\bs, \bpsi)$ contains the ball $\overline{B(\s_i,\psi_i^{1/2})}$; thus the modified cell $V_i(\s,\bpsi)$ coincides with this ball. Going further, it is easy to show that if the seed points $\s_i$ are ``well separated" from one another in the sense that: 
\begin{equation}\label{eq.sepprop}
\lvert \s_j - \s_i \lvert > \sqrt{\frac{\nu_i}{\pi}} + \sqrt{\frac{\nu_j}{\pi}}\:\: \text{ in 2d, and }\:\: \lvert \s_j - \s_i \lvert  > \left(\frac{3 \nu_i}{4\pi}\right)^{\frac13} + \left(\frac{3 \nu_j}{ 4\pi}\right)^{\frac13} \:\:\text{ in 3d,}
\end{equation}
then the optimal transport problem \cref{eq.implicitPsi} is trivial, its solution being given by: 
\begin{equation}\label{eq.heurinit} 
\psi_i = \left\{
\begin{array}{cl}
\frac{\overline\nu}{\pi} &\text{if }d = 2, \\[0.2em]
\left(\frac{3 \overline\nu}{ 4 \pi}\right)^{\frac23} & \text{if } d = 3,
\end{array}
\right.
\quad i=1,\ldots,N,
\end{equation}
and the balls $\overline{B(\s_i,\psi_i^{1/2})}$ are disjoint from one another.

In general, the data $\s$ and $\bnu$ do not satisfy the separation property \cref{eq.sepprop}, and the cells of the diagram $\bVsp$ ``interact'' through their common boundaries, see \cref{fig.POT_KMT}. Then, several iterations of \cref{algo.Newton} are needed to achieve convergence (typically 5 to 10), but in practice, the heuristic initialization \cref{eq.heurinit} saves a couple of Newton iterations.

Last but not least, when the target measures $\nu_i$ are different, the choice \cref{eq.heurinit} still proves quite efficient when used with the average value $\overline\nu := \frac{1}{N}\sum_{i=1}^N \nu_i$.

\subsubsection{Computation of Laguerre diagrams}\label{sec.compLaguerre}

\noindent This section briefly describes the construction of the modified diagram $\bVsp$ associated to given seed points $\s \in \R^{dN} $ and weights $\bpsi \in \R^N$. We consider the 2d case $d=2$ for simplicity, and refer to \cite{levy2022partial} and the bibliography therein for further details, including the treatment of the 3d case. The task under scrutiny is decomposed into two steps.\par\medskip

\noindent \textit{Step 1: We compute the classical Laguerre diagram $\bLag$ associated to $\s$ and $\bpsi$.}

Several algorithms are available to achieve this task, whose outcome is illustrated on \cref{fig.VoroDelau} (a). In the present article, we rely on the iterative insertion algorithm of Bowyer and Watson proposed in \cite{DBLP:journals/cj/Bowyer81,journals/cj/Watson81}, see also \cite{frey2007mesh} -- a strategy originally intended to construct Voronoi diagrams (or their dual, Delaunay triangulations), which can handle more general Laguerre diagrams up to simple adaptations, see \cite{aurenhammer1987power}. 

The Bowyer-Watson algorithm proceeds by iterative insertion of the seed points $\s_1,\ldots,\s_N$. Each stage $n=1,\ldots,N$ of the process starts with the datum of
the Laguerre diagram $\mathbf{Lag}(\widehat{\s}^{n-1},\widehat{\bpsi}^{n-1})$ associated to the subcollections $\widehat{\s}^{n-1} := \left\{ \s_1,\ldots,\s_{n-1}\right\}$ and $\widehat{\bpsi}^{n-1} := \left\{\psi_1,\ldots,\psi_{n-1} \right\}$ of seed points and weights, with an obvious adaptation of this setting in the case $n=1$. This diagram is stored under the dual form of a so-called regular triangulation $\calT^{n-1}$ of $D$, see \cref{fig.VoroDelau} (b): the vertices of $\calT^{n-1}$ are exactly $\s_1,\ldots,\s_{n-1}$ and its triangles are those connecting the 3-uples of seed points $\left\{\s_i, \s_j, \s_k \right\} \subset \widehat{\s}^{n-1}$ (with disjoint indices $i,j,k$) whose associated cells share a common vertex $\q$. Working with this dual structure is convenient since all dual cells (i.e. the triangles of $\calT^{n-1}$) have the same number of vertices, whereas storing directly the Laguerre cells of $\mathbf{Lag}(\widehat{\s}^{n-1},\widehat{\bpsi}^{n-1})$ would require a more complicated data structure, since the latter have different numbers of vertices. 

\begin{figure}[!ht]
\begin{center}
\begin{minipage}{\textwidth}
\begin{overpic}[width=1.0\textwidth]{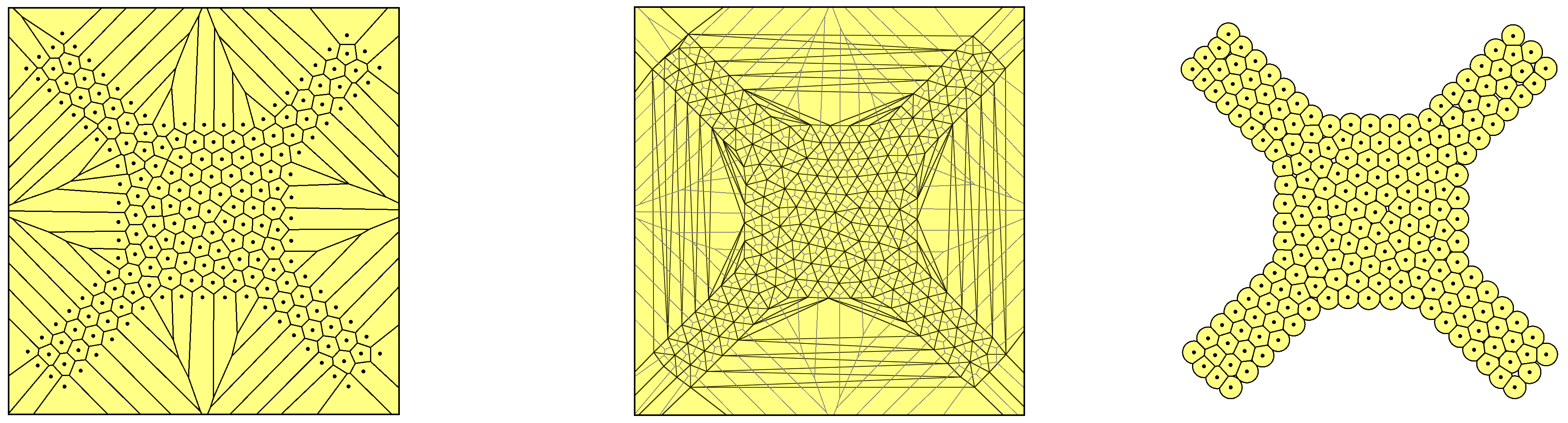}
\put(10,0){\fcolorbox{black}{white}{a}}
\put(200,0){\fcolorbox{black}{white}{b}}
\put(380,0){\fcolorbox{black}{white}{c}}
\end{overpic}
\end{minipage} 
\end{center}
\caption{\it Computation of the modified Laguerre diagram $\bVsp$ in \cref{sec.compLaguerre}: (a) ``Classical'' Laguerre diagram $\bLag$; (b) Associated dual regular triangulation used to store $\bLag$ internally; (c) Modified diagram $\bVsp$ obtained from $\bLag$ after clipping each cell $\Lag_i(\s,\bpsi)$ by the ball with center $\s_i$ and radius $\sqrt{\psi_i}$.}
\label{fig.VoroDelau}
\end{figure}

The algorithm then inserts the seed $\s_n$ into $\mathbf{Lag}(\widehat{\s}^{n-1},\widehat{\bpsi}^{n-1})$, see \cref{fig.BowyerWatson}.
\begin{figure}[!ht]
    \centering
    \begin{minipage}{\textwidth}
    \hspace{30mm}
    \begin{overpic}[height=45mm]{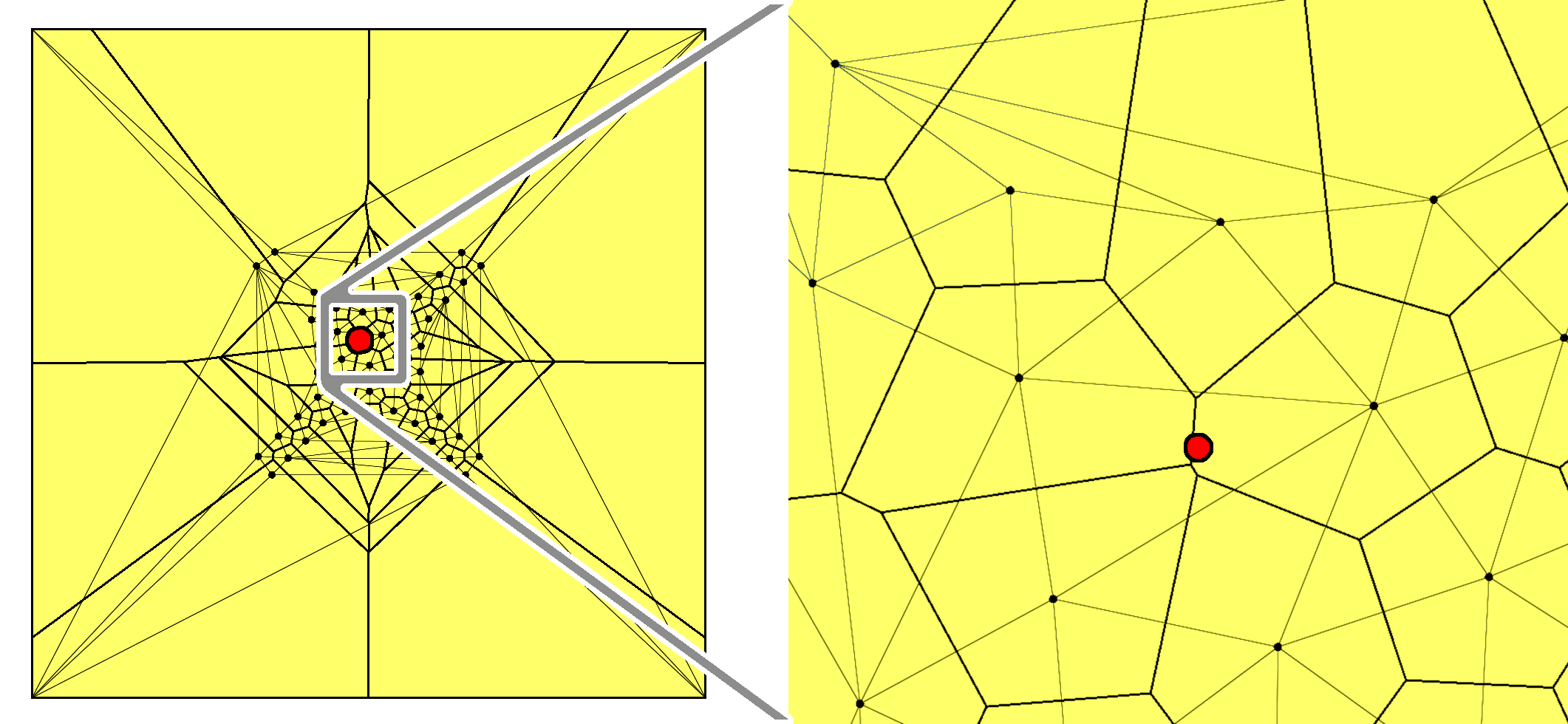}
        \put(150,0){\fcolorbox{black}{white}{a}}
    \end{overpic}
    \end{minipage}
    \\[5mm]
    \begin{minipage}{\textwidth}
    \hspace{15mm}
    \begin{overpic}[height=45mm]{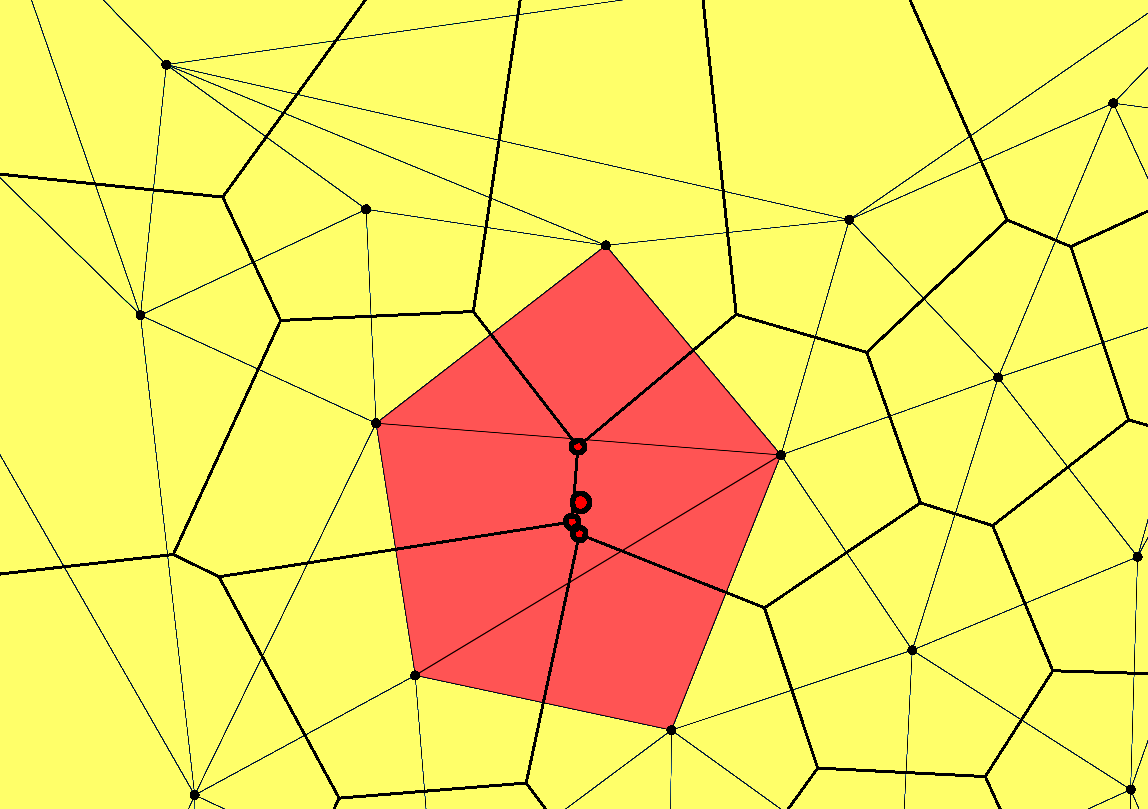}
        \put(10,0){\fcolorbox{black}{white}{b}}
    \end{overpic}
    \hspace{5mm}
    \begin{overpic}[height=45mm]{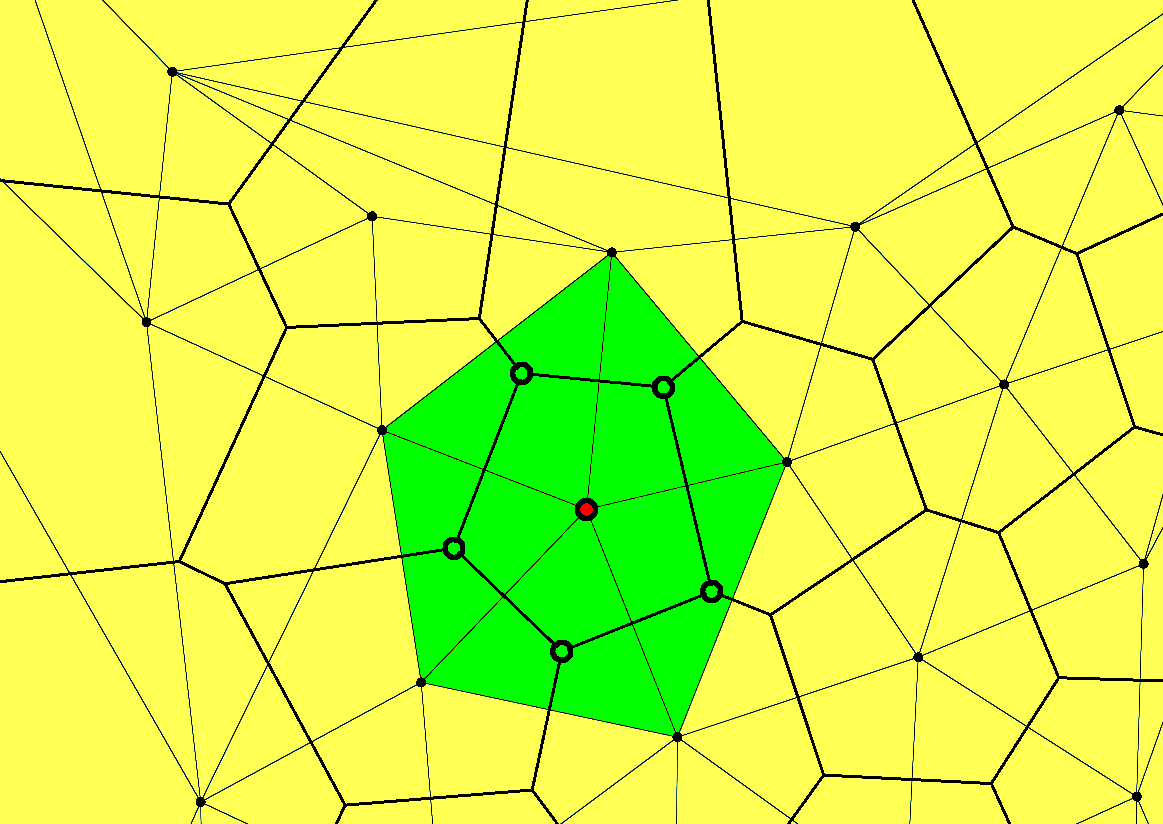}
        \put(10,0){\fcolorbox{black}{white}{c}}
    \end{overpic}    
    \end{minipage} 
     \caption{\it Iterative construction of the classical Laguerre diagram $\bLag$ with the Bowyer-Watson algorithm in \cref{sec.compLaguerre}; (a) Insertion of a new seed $\s_n$ (in red) into the diagram $\mathbf{Lag}(\widehat{\s}^{n-1},\widehat{\bpsi}^{n-1})$ under construction; (b) Some of the vertices of the current diagram (in red) cannot exist in the new diagram because they lie inside the Laguerre cell of the new seed $\s_n$; (c) The cavity formed by the associated dual triangles (in red) is removed and the vertices of the new Laguerre cell are dual to triangles connecting the new seed to those at the border of the cavity (in green).}
     \label{fig.BowyerWatson}
\end{figure}
To achieve this, one first identifies the vertices $\q$ of $\mathbf{Lag}(\widehat{\s}^{n-1},\widehat{\bpsi}^{n-1})$ that are in conflict with the presence of $\s_n$, in the sense that they cannot exist in the updated diagram $\mathbf{Lag}(\widehat{\s}^{n},\widehat{\bpsi}^{n})$ because they lie inside the Laguerre cell $\Lag_n(\widehat{\s}^{n},\widehat{\bpsi}^{n})$ of the new seed point $\s_n$ in the latter.
Such a conflicting vertex $\q$, defined by the intersection of the three Laguerre cells attached to the seeds $\s_i$, $\s_j$, $\s_k$, is characterized by the following relation:
\begin{equation}\label{eq.incircle}
 | \q - \s_n | - \psi_n < R_{ijk}, \text{ where } 
 R_{ijk} := | \q - \s_i |^2 - \psi_i =
 | \q - \s_j |^2 - \psi_j =
 | \q - \s_k |^2 - \psi_k.
\end{equation}
The equalities featured in \cref{eq.incircle} yield a linear system characterizing the coordinates of $\q$, see also \cref{app.verseeds}. By solving the latter with the classical Cramer's formulas, substituting the resulting values into the above condition and symmetrizing the determinant by row manipulations, \cref{eq.incircle} rewrites:
\begin{equation}\label{eq.incircle2}
    C(\s_i, \s_j, \s_k, \s_n, \psi_i, \psi_j, \psi_k, \psi_n) := 
    \left|
        \begin{array}{cccc}
        s_{i,1} & s_{i,2} & s_{i,1}^2 +s_{i,2}^2 - \psi_i & 1 \\[2mm]
        s_{j,1} & s_{j,2} & s_{j,1}^2 +s_{j,2}^2 - \psi_j & 1 \\[2mm]
        s_{k,1} & s_{k,2} & s_{k,1}^2 +s_{k,2}^2 - \psi_k & 1 \\[2mm]
        s_{n,1} & s_{n,2} & s_{n,1}^2 +s_{n,2}^2 - \psi_n & 1        
        \end{array}
    \right| > 0.
\end{equation}
One then constructs the cavity $\calC^{n-1}$, defined as the collection of the dual triangles in $\calT^{n-1}$ associated to the vertices in conflict with $\s_n$, see the red triangles in \cref{fig.BowyerWatson} (b). With a small abuse of notations, we use the same notation for this collection of triangles and the closed subset of $\overline D$ formed by their reunion.
It can be shown that $\calC^{n-1}$ is a connected set containing $\s_n$. Hence the identification of $\calC^{n-1}$ starts by finding one triangle $T \in \calT^{n-1}$ containing $\s_n$, and then proceeds by propagation through the neighbors of $T$ while they are in conflict with $\s_n$, in the sense that \cref{eq.incircle2} holds, see \cref{fig.BowyerWatson} (b).
Finally, the triangles in the cavity $\calC^{n-1}$ are removed from the dual triangulation $\calT^{n-1}$, and the new triangles dual to the vertices of the new Laguerre cell $\Lag_n(\widehat{\s}^n,\widehat{\bpsi}^n)$ are constructed. Each of them connects $\s_n$ with an edge of the boundary of $\calC^{n-1}$, see \cref{fig.BowyerWatson} (c). This procedure is iterated until all the seed points in $\s$ have been inserted. \par\medskip

The Bowyer-Watson algorithm crucially hinges on geometric predicates, that is, procedures taking combinatorial decisions based on the fulfillment of inequalities between the coordinates of some geometric objects. Notably, the conflict test \cref{eq.incircle2} is based on a predicate; moreover, the aforementioned search for a dual triangle in $\calC^{n-1}$ containing the inserted seed point $\s_n$ requires another predicate, measuring the relative orientation of three points. 
The implementation of both predicates requires a particular attention, in order to overcome the lack of precision of standard computer arithmetics. They amount to evaluate the sign of a polynomial function evaluated at the coordinates of the seed points $\s_i$ and weights $\psi_i$. Such a quantity can be exactly computed at a reasonable computational cost thanks to a combination of different techniques, see \cite{DBLP:conf/compgeom/Shewchuk96,DBLP:journals/cad/Levy16} and references therein. 

The above Bowyer-Watson procedure also rests on the assumption that the seed points $\s$ and the weights $\bpsi$ are in a generic configuration (see again \cref{prop.psi} and \cref{rem.suffconddiff}). In practice, this assumption guarantees that the determinant in the characterization \cref{eq.incircle2} of the dual triangles in the cavity $\calC^{n-1}$ does not vanish -- a situation where $\s_n$ would lie on the boundary of this cavity, making the procedure invalid. However, ``pathologic'' configurations where this determinant vanishes happen in practice, and they require a specific treatment. For instance, if $\bpsi = \bz$ (and then $\bLag$ is the Voronoi diagram $\mathbf{Vor}(\s)$), the predicate \cref{eq.incircle2} evaluates whether the added seed point $\s_n$ lies inside the disk circumscribed to the triangle with vertices
$\s_i$, $\s_j$, $\s_k$; the determinant in \cref{eq.incircle2} then equals zero if the four seed points
$\s_i, \s_j, \s_k, \s_n$ are cocyclic. In this case, two triangulations of the cavity $\calC^{n-1}$ based on $\s_n$ are possible, and one would need to consistently choose among them. 

One possibility to treat situations where the determinant $C(\s_i, \s_j, \s_k, \s_n, \psi_i, \psi_j, \psi_k, \psi_n)$ vanishes is to evade from singular configurations thanks to the symbolic perturbation approach \cite{DBLP:journals/tog/EdelsbrunnerM90}: we conceptually perturb each weight $\psi_i$ as $\psi_i + \e^i$, where $\e \ll 1$ is a very small parameter, and we consider the expansion of the associated determinant in terms of the powers of $\e$:
\begin{multline}
    C(\s_i, \s_j, \s_k, \s_l, \psi_i + \e^i, \psi_j + \e^j, \psi_k + \e^k, \psi_n + \e^n) = 
    \left|
        \begin{array}{cccc}
        s_{i,1} & s_{i,2} & s_{i,1}^2 +s_{i,2}^2 - \psi_i - \e^i & 1 \\[2mm]
        s_{j,1} & s_{j,2} & s_{j,1}^2 +s_{j,2}^2 - \psi_j - \e^j & 1 \\[2mm]
        s_{k,1} & s_{k,2} & s_{k,1}^2 +s_{k,2}^2 - \psi_k - \e^k & 1 \\[2mm]
        s_{n,1} & s_{n,2} & s_{n,1}^2 +s_{n,2}^2 - \psi_n - \e^n& 1     
        \end{array}
    \right| \\[1em]
    \quad = C(\s_i, \s_j, \s_k, \s_l, \psi_i, \psi_j, \psi_k, \psi_n)
      - \e^i \left|
       \begin{array}{ccc}
           s_{j,1} & s_{j,2} & 1 \\
           s_{k,1} & s_{k,2} & 1 \\
           s_{n,1} & s_{n,2} & 1 \\
       \end{array}
     \right|
     + \e^j \left|
       \begin{array}{ccc}
           s_{i,1} & s_{i,2} & 1 \\
           s_{k,1} & s_{k,2} & 1 \\
           s_{n,1} & s_{n,2} & 1 \\
       \end{array}
     \right|\\[1em]
     - \e^k \left|
       \begin{array}{ccc}
           s_{i,1} & s_{i,2} & 1 \\
           s_{j,1} & s_{j,2} & 1 \\
           s_{n,1} & s_{n,2} & 1 \\
       \end{array}
     \right|
     + \e^n \left|
       \begin{array}{ccc}
           s_{i,1} & s_{i,2} & 1 \\
           s_{j,1} & s_{j,2} & 1 \\
           s_{k,1} & s_{k,2} & 1 \\
       \end{array}
     \right|,
\label{eq.incircle_SOS}
\end{multline}
where the second equality follows from the multilinearity of the determinant. In the last expression of the above right-hand side, the factor of $\e^n$ does not vanish since according to \cref{eq.Gen1} the seed points $\s_i, \s_j, \s_k$ are not aligned; hence, using the sign of the leading term in the above expansion in powers of $\e$ results in a modified predicate which never returns $0$ and takes consistent combinatorial decisions in the presence of degeneracies.
\par\medskip

\noindent \textit{Step 2: Computation of $\bVsp$ by truncation of $\bLag$.}

For each $i=1,\ldots,N$, the $i^{\text{th}}$ cell $V_i(\s,\bpsi)$ of the modified diagram $\bVsp$ is obtained by computing the intersection between its classical counterpart $\Lag_i(\s,\bpsi)$ with the ball centered at $\s_i$ with radius $\sqrt{\psi_i}$, as depicted in \cref{fig.VoroDelau} (c). According to \cref{propdef.Lag}, the boundary of $V_i(\s,\bpsi)$ is composed of line segments -- some of them pertaining to the boundary of the computational domain $D$, possibly bearing particular references inherited from the latter -- and, occasionally, circular arcs, see \cref{fig.Vipsi,fig.POT_KMT}.

As we shall see more precisely in the next \cref{sec.mechcomp}, the mechanical computations 
involved in the evaluation of the shape functionals of our optimization problem \cref{eq.sopb} rely on a discretization of $\Omega$ made of convex polytopes. 
The latter is constructed via the procedure described in \cite{levy2010p}:
a number $\narc$ is chosen, and each circular arc is discretized into a polygonal line made of $\narc$ line segments and $(\narc+1)$ vertices.
The calculation of the positions of these vertices from the datum of the seed points $\s$ is detailed in \cref{app.verseeds}. 


\subsubsection{Computation of the gradient and of the Hessian matrix of the Kantorovic functional}

\noindent In Steps 2 and 4 of \cref{algo.Newton}, the diagram $\mathbf{V}(\s,\bpsi^n)$ is used to evaluate the entries of the gradient $\bF(\s,\bnu,\bpsi^n)=\nabla_{\bpsi} K(\s,\bnu,\bpsi^n)$ and of the Hessian matrix $\nabla_{\bpsi}\bF(\s,\bnu,\bpsi^n) =\nabla^2_{\bpsi} K(\s,\bnu,\bpsi^n)$ of the Kantorovich functional $K(\s,\bnu,\cdot)$ at $\bpsi^n$. 

These components depend on the measures of
the cells $|V_i(\s,\bpsi^n)|$ of the diagram $\mathbf{V}(\s,\bpsi^n)$ and on the lengths of their edges $\be_{ij}$.
Their expressions have been calculated in e.g. \cite{DBLP:journals/corr/KitagawaMT16,levy2018notions}; 
for the sake of completeness, they are recalled in \cref{app.derF}, where an intuitive proof is provided under simplifying assumptions.

\subsubsection{Solution mechanism for the linear systems}

\noindent
The calculation of the Newton step $\p^n$ in Step 5 of \cref{algo.Newton} is based on the solution of the following linear system with size $N \times N$: 
\begin{equation}\label{eq.Newteqssys} \left[ \nabla_{\bpsi}\bF(\s,\bnu,\bpsi^n)\right] \bp^n = - \bF(\s,\bnu,\bpsi^n).
\end{equation}
When two-dimensional applications are concerned, featuring a relatively small number $N$ of seed points (say, $N < 10^5$),
direct methods such as Gaussian elimination or $LDL^T$ factorization are available. 

Larger configurations raise the need to use an iterative solver, that takes advantage of the sparse nature of the matrix involved in \cref{eq.Newteqssys}, see again its expression in \cref{app.derF}. The Conjugate Gradient algorithm \cite{Hestenes1952} combined with the simple Jacobi preconditionner (which involves division by the diagonal coefficients) is an efficient candidate in this perspective. 

When very large problems are considered (where the number $N$ of seed points exceeds $10^6$), yet another strategy is handful. The latter is based on the fact that the linear system \cref{eq.Newteqssys} corresponds to the discretization of a Poisson equation with $\mathbb{P}_1$ finite elements. This calls for a multigrid method, that can solve such a Poisson system in linear time. In our case, the mesh characterizing this discretization is irregular and we recommend the use of an algebraic multigrid method \cite{Demidov2019}. Roughly speaking, such a multi-grid method creates a hierarchy of operators by merging nodes in the graph that corresponds to the non-zero entries of the sparse matrix.
This technique made it possible to solve 3d problems featuring up to $10^8$ points in our previous works \cite{levy2024monge,levy2024largescalesemidiscreteoptimaltransport}.

\subsubsection{Stopping criterion}

\noindent 
The components of the gradient $\bF(\s,\bnu,\bpsi^n) = \nabla_{\bpsi}K(\s,\bnu,\bpsi^n)$ of the Kantorovic functional express the difference between the target cell measures $\nu_i$ and the measures $|V_i(\s,\bpsi^n)|$ of the cells of the actual version $\mathbf{V}(\s,\bpsi^n)$ of the diagram. 
Hence, the supremum norm of this gradient appraises the largest error on the fulfillment of the cell measure constraint. This geometric interpretation of $\bF(\s,\bnu,\bpsi^n)$ suggests a natural convergence criterion for \cref{algo.Newton}: the Newton procedure is terminated as soon as the largest measure error is smaller than 1 \% of the smallest measure prescription, i.e. when the following inequality is satisfied:
$$  \lvert \bF(\s,\bnu,\bpsi^n)\lvert_\infty < \e_{\text{Newt}}, \text{ where } \e_{\text{Newt}} = 0.01 \left(\min\limits_{i=1,\ldots,N} \nu_i \right). $$

\subsubsection{Determination of the descent parameter}

\noindent The selection of a suitable descent parameter $\alpha^n$ for the update of the weight vector $\bpsi$ in Step 6 of \cref{algo.Newton} deserves a particular attention. Indeed, the invertibility of the Hessian matrix $\bF(\s,\bnu,\bpsi^n)$ of the Kantorovic functional involved in the linear system \cref{eq.Newteqssys} for the Newton step $\bp^n$ crucially requires that none of the cells $V_i(\s,\bpsi^n)$ is empty, $n=0,\ldots,N$, see again the expressions in \cref{app.derF}. Hence, the choice of $\alpha^n$ must guarantee that no cell $V_i(\s,\bpsi^{n+1})$ in the updated diagram $\mathbf{V}(\s,\bpsi^{n+1})$ is empty. 

To ensure this property, we rely on the Kitagawa-M\'erigot-Thibert strategy, described in \cref{algo.KMT} and analyzed in \cite{DBLP:journals/corr/KitagawaMT16}: starting from the Newton step $\alpha^{n,0} = 1$, the descent parameter is halved until 
the magnitude of the gradient $\bF(\s,\bnu,\bpsi^n)$ has sufficiently decreased while 
the size of the smallest cell in $\mathbf{V}(\s,\bpsi^{n+1})$ is larger than a certain threshold $\nu_{\text{min}}$; in practice, we take:
\begin{equation}
\label{eq.numin}
\nu_{\text{min}} = \frac{1}{2}\min\left(\min\limits_{i=0,1,\ldots,N}|V_i(\s,\bm{0})| , \min\limits_{i=0,1,\ldots,N} \nu_i \right).
\end{equation}
Using the so-computed descent parameter, the convergence of the Newton algorithm \cref{algo.Newton} is proven, see \cite{DBLP:journals/corr/KitagawaMT16} for the details. 

\begin{algorithm}[!ht]
\caption{Determination of a suitable descent parameter for the update of the weight vector}
\label{algo.KMT}
\begin{algorithmic}[0]
\STATE \textbf{Input:} The current Newton iterate $n$ in \cref{algo.Newton}, characterized by the datum of:
\begin{itemize}
    \item The collection of seed points $\bs = \{s_1,\ldots,s_N\}$;
    \item The current weight vector $\bpsi^n$;
    \item The current Newton step vector $\bp^n$;
    \item The initial guess about the descent parameter $\alpha^{n,0}=1$.
\end{itemize}
\FOR{$k=0,...,$ until convergence}
\STATE
\begin{enumerate}
\item \textbf{if} $\min\limits_{i=0,\ldots,N} | V_i(\bpsi^n + \alpha^{n,k} \bp^n) | >\nu_{\text{min}}$ \mbox{and} $\lvert \bF(\s,\bnu,\bpsi^n + \alpha^{n,k} \bp^n) \lvert \leq (1 - \frac{\alpha^{n,k} }{2}) \lvert \bF(\s,\bnu,\bpsi^n) \lvert$ \textbf{exit loop}
\item Update the descent parameter as: $\alpha^{n,k+1} = \alpha^{n,k} / 2$. 
\item Compute the new diagram $\mathbf{V}(\s,\bpsi^n + \alpha^{n,k+1} \bp^n)$
\end{enumerate}
\ENDFOR
\RETURN Descent parameter $\alpha^n := \alpha^{n,k}$.
\end{algorithmic}
\end{algorithm}

\subsection{Smoothing of a diagram by a variant of Lloyd's algorithm}\label{sec.lloyd}

\noindent As we shall describe more extensively in \cref{sec.mechcomp}, the polygonal mesh induced by the diagram $\bVsp$ representing the shape $\Omega$ in our optimal design \cref{algo.lagevol} is used, in particular, as the numerical support for the solution of boundary value problems posed on $\Omega$, by means of the Virtual Element Method. The accuracy of the latter, like that of any Galerkin-based solution strategy, is tightly related to the ``quality'' of this polygonal mesh; intuitively, this notion appraises how close its polygons are from being regular \cite{sorgente2022role,sorgente2024mesh}.  Unfortunately, \cref{algo.lagevol} does not offer any guarantee about the quality of the cells of the diagrams $\mathbf{V}(\s^n,\bpsi^n)$ at play, and it is necessary to periodically stop the process, say every $3$, $4$ iterations, to improve their quality.

One possibility to achieve this goal borrows from the famous Lloyd's algorithm for Voronoi diagrams \cite{lloyd1982least}. Roughly speaking, the latter produces a ``well-shaped'' centroidal Voronoi tessellation (i.e. where the seed point of each cell coincides with its centroid) from an arbitrary initial diagram $\bVor(\s)$; it proceeds within a series of iterations, by replacing the seed point $\s_i$ of each cell $i=1,\ldots,N$ by the centroid $\c_i$ of the latter, and then computing the new diagram $\bVor(\c)$.
Often, a relaxed version of this procedure is used, whereby each seed point $\s_i$ is moved towards $\c_i$ for a short pseudo-time step $\alpha \in (0,1)$. 

In our strategy, we rely on a natural extension of this procedure to the present context of Laguerre diagrams in which the measure of each cell is imposed.
This was recently proposed in \cite{de2012blue,de2015power,xin2016centroidal}, see also \cite{bourne2015centroidal};  it is sketched in \cref{algo.lloydcapa}.  

\begin{algorithm}[!ht]
\caption{Smoothing of a diagram with constrained cell measures by a variant of Lloyd's algorithm}
\label{algo.lloydcapa}
\begin{algorithmic}[0]
\STATE \textbf{Inputs:} 
\begin{itemize}
\item Diagram $\mathbf{V}(\s^0,\bpsi^0)$ induced by:
\begin{itemize}
 \item[-] An initial collection of seed points $\s^0 \in \R^N$; 
\item[-] A given vector $\bnu \in \R^N$ of cell measures;
\item[-] The unique weight vector $\bpsi^0 \in \R^N$ such that each cell $V_i(\s^0,\bpsi^0)$ has measure $\nu_i$. 
\end{itemize}
\item Relaxation parameter $\alpha \in (0,1)$.
\end{itemize}
\FOR{$n=0,...,$ until convergence}
\STATE
\begin{enumerate}
\item Calculate the collection $\mathbf{c}^n := \left\{ \mathbf{c}^n_i \right\}_{i=1,\ldots,N}$ of the centroids $\mathbf{c}^n_i \in \R^d$ of the cells $V_i(\s^n,\bpsi^n)$.
\item Update the seed points as: $\s^{n+1} = (1-\alpha) \s^n + \alpha\mathbf{c}^n$.
\item Calculate the weight vector $\bpsi^{n+1}$ guaranteeing that $\lvert V_i(\s^{n+1},\bpsi^{n+1}) \lvert = \nu_i$ for $i=1,\ldots,N$. 
\item Compute the new diagram $\mathbf{V}(\s^{n+1},\bpsi^{n+1})$.
\end{enumerate}
\ENDFOR
\RETURN Diagram $\mathbf{V}(\s^n,\bpsi^n)$ where each cell  $V_i(\s^n,\bpsi^n)$ is ``well-shaped'' and has measure $\nu_i$
\end{algorithmic}
\end{algorithm}


\subsection{Additional operations on diagrams}\label{sec.resample}

\noindent In this section, we outline two geometric operations on diagrams of the form \cref{eq.decompOm,eq.defVipsi} that significantly improve the efficiency of the shape optimization \cref{algo.Newton}.
\subsubsection{Resampling of a diagram}

\noindent As will be exemplified in the numerical examples of \cref{sec.num}, the optimization of the shape $\Omega^n$ generally entails significant changes in its volume through the iterations $n=0,\ldots$ of the process; it is often relevant to dynamically adjust the number $N$ of cells of the defining diagram $\mathbf{V}(\s^n,\bpsi^n)$ according to these changes. 

To achieve this, we periodically carry out a simple resampling procedure, at the end of every one iteration over, say, $3$ or $4$ of \cref{algo.lagevol}: a desired average value $\overline{\nu}$ for the measures of the cells of the diagram of the shape $\Omega^n$ is given, and we infer the suitable number $N^n$ of cells in the latter via the following relation: 
$$ N^n = \frac{\Vol(\Omega^n)}{\overline\nu}.$$
We then add or delete seed points from the collection $\s^n$ to attain this number, according to the following rules:
\begin{itemize}
\item Seed points $\s_i$ are added inside the regions of $\Omega^n$ lying far from the boundary of $\Omega^n$ (so that this entity is unchanged in the process), where the local cell measures are largest. The measures $\nu_i$ of the corresponding added cells are set to $\overline \nu$, and the (larger) measures of the neighboring cells are decreased so that the volume of $\Omega^n$ is unaltered by this operation.
\item The deleted seed points $\s_i$ lie far from $\partial \Omega^n$, and they correspond to cells $V_i(\s^n,\bpsi^n)$ with ``small'' measures. Their mass is redistributed to the neighboring cells, so that $\Vol(\Omega^n)$ is unchanged in the process. 
\end{itemize}

\subsubsection{Elimination of material ``islands'' disconnected from the main structure}\label{sec.islands}

\noindent As described in \cref{sec.compLaguerre}, some of the edges of the cell in the diagrams $\mathbf{V}(\s^n,\bpsi^n)$ produced by \cref{algo.lagevol} may bear particular labels, inherited from the boundary $\partial D$ of the computational domain. These notably serve to identify the regions of $\partial \Omega^n$ bearing particular boundary conditions in the physical problem at play. For instance, in the settings of \cref{sec.conduc,sec.elas}, they characterize the regions $\Gamma_D$ and $\Gamma_N$ supporting homogeneous Dirichlet and inhomogeneous Neumann boundary conditions. 
In practice, it may be desirable to remove the components of $\Omega^n$ that are not connected to any cell bearing such a label. 
For instance, in the situation of \cref{sec.elas}, 
the regions of $\Omega^n$ that are not connected to $\Gamma_D$ cause the problem \cref{eq.elas} to be ill-posed, 
since the elastic displacement $\u_{\Omega^n}$ is only characterized by this problem up to a rigid-body motion inside each such material ``islands''.

The elimination of these regions is realized thanks to a simple algorithmic procedure, based on the connectivity of $\mathbf{V}(\s^n,\bpsi^n)$. We start by storing into a pile all the indices $i \in \left\{1,\ldots,N\right\}$ belonging to cells having one edge bearing the desired label. We then travel the cells of the diagram by propagating through the neighbors of the elements in the pile, and we thus obtain all the indices of the connected component of $\Omega^n$ attached to the label at stake.
We eventually discard all the cells (i.e. the associated seed points) in $\bVsp$ that have not been visited in this process.

\section{Mechanical computations on $\Omega$ via the Virtual Element Method}\label{sec.mechcomp}

\noindent In this section, we describe -- in 2d for simplicity -- the resolution of physical boundary value problems of the form \cref{eq.conduc,eq.elas} on a polygonal mesh $\calT$ of the shape $\Omega$.
Different numerical methods are available to achieve this goal. 
Although they are slightly unusual in the physical contexts at stake in this article, let us notably mention finite volume methods -- see \cite{Cebula2014} for the treatment of heat conduction problems, or \cite{jasak2000application} for linear elasticity --, discontinuous Galerkin methods \cite{cangiani2014hp}, or the recent Network Element Method \cite{coatleven2021principles,coatleven2023network}.
Here, we rely on the Virtual Element Method, which is an elegant variant of the well-known Finite Element Method adapted to the solution of boundary value problems on arbitrary polygonal meshes, see e.g. \cite{antonietti2022virtual,beirao2014hitchhiker,da2013virtual,gain2014virtual,sutton2017virtual}, or \cite{dassi2023vem++} for a recent open-source implementation.

After setting notations in \cref{sec.notVEM}, we present this method in \cref{sec.VEMlap} in the context of the 2d conductivity equation \cref{eq.conduc}, where its salient features can be exposed in a relatively non technical manner. 
This presentation is not intended to be minimal; rather, it prepares the ground for the treatment of the linear elasticity system which is discussed next in \cref{sec.VEMelas}. 
Practical implementation details are deferred to \cref{app.VEM}. 

\subsection{Notations}\label{sec.notVEM}

\noindent Throughout this section, $\Omega$ is a two-dimensional shape equipped with a polygonal mesh $\calT$. 
The latter is composed of $N$ (closed) elements $E_i$, $i=1,\ldots,N$, and 
$M$ vertices, denoted by $\q_j$, $j=1,\ldots,M$. 
Each polygon $E \in \calT$ is characterized by its $n^E$ vertices, which form a subset $\left\{\q^E_{j}\right\}_{j=1, \ldots, n^E}$ of $\left\{ \q_j \right\}_{j=1,\ldots,M}$; 
these are numbered in a counterclockwise fashion and we set $\q^E_0 := \q^E_{n^E}$ and $\q_{n^E+1}^E := \q_1^E$, see \cref{fig.eltE} (a). We also introduce the following notation:
\begin{itemize}
\item The diameter of $E$ is denoted by $h^E := \sup_{x,y \in E} \lvert x - y\lvert$. 
\item For $i=1,\ldots,n^E-1$, $\be_i$ is the edge with endpoints $\q^E_i$ and $\q^E_{i+1}$; $\be_0 = \be_{n^E}$ stands for the edge between $\q^E_{n^E}$ and $\q^E_1$. 
\item For $i=0,\ldots,n^E$, we denote by $\n_{\be_i}$ the unit normal vector to $\be_i$, pointing outward $E$.
\item For $i=1,\ldots,n^E$, we denote by $\hat \be_i$ the edge between $\q^E_{i-1}$ and $\q^E_{i+1}$, 
and by $\n_{\hat \be_i}$ the unit normal vector to $\hat \be_i$, oriented so that $\langle \n_{\hat {\be}_i} \, \n_{\be_j}\rangle \geq 0$ for $j\in \left\{i-1,i\right\}$. 
A simple calculation shows that
\begin{equation}\label{eq.nehat}
\lvert \be_{i-1} \lvert \n_{\be_{i-1}} + \lvert \be_i \lvert \n_{\be_i} = \lvert \hat \be_i  \lvert \n_{\hat \be_i}. 
\end{equation}
\item For any continuous function $v : E \to \R$, we denote by
$$ \bar v = \frac{1}{n^E} \sum\limits_{j=1}^{n^E} v(\q^E_j)$$
the average of $v$ over the vertices of $E$; the same notation is used when $v$ is replaced by a vector field or a tensor field on $E$.
\item Accordingly, $\overline{\q^E} = (\overline{q^E_1} , \overline{q^E_2})$ is the average position of the vertices of $E$, i.e. 
$$ \overline{\q^E} = \frac{1}{n^E} \sum\limits_{j=1}^{n^E} \q^E_j;$$
note that $\overline{\q^E}$ differs from the centroid of $E$. 
\item For any integrable function $v : E \to \R$, we denote the average of $v$ over $E$ by 
$$ \langle v \rangle = \frac{1}{|E|} \int_E v(\x) \:\d \x;$$
 the same notation is used when $v$ is replaced by a vector field or a tensor field on $E$.
\end{itemize}

\begin{figure}[!ht]
\centering
\begin{tabular}{cc}
\begin{minipage}{0.44\textwidth}
\begin{overpic}[width=1.0\textwidth]{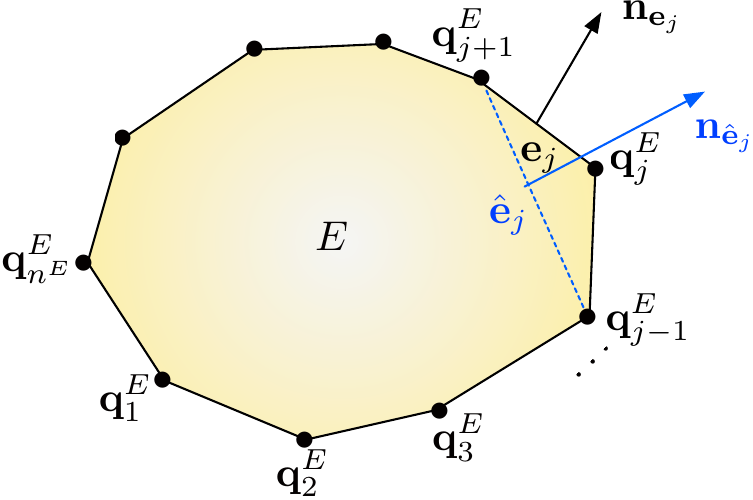}
\put(0,3){\fcolorbox{black}{white}{$a$}}
\end{overpic}
\end{minipage}
 & 
 \begin{minipage}{0.57\textwidth}
\begin{overpic}[width=1.0\textwidth]{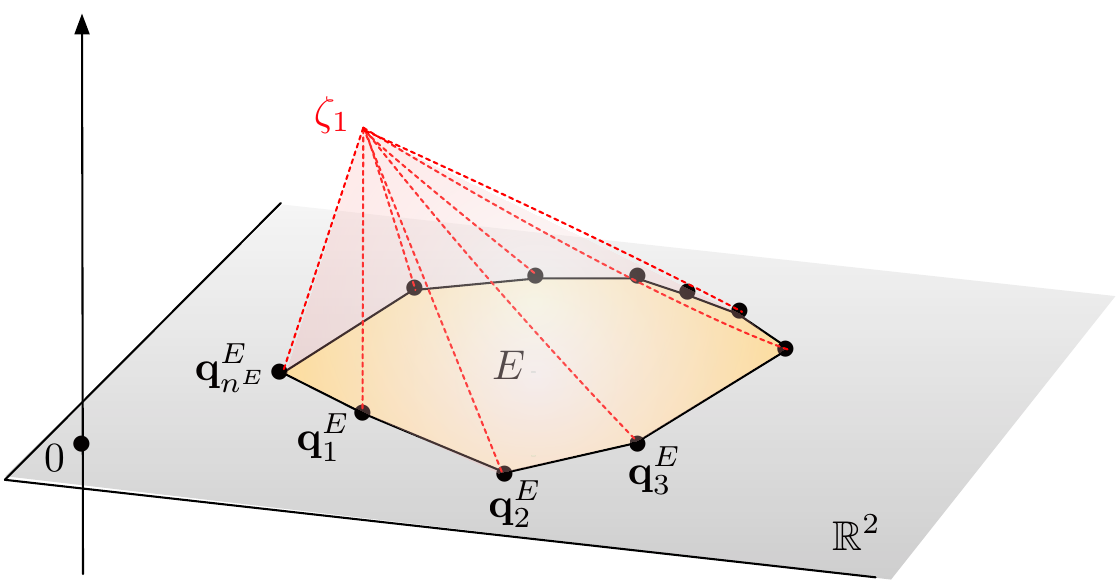}
\put(0,3){\fcolorbox{black}{white}{$b$}}
\end{overpic}
\end{minipage}
\end{tabular}
\caption{\it (a) One polygonal element $E \in \calT$; (b) Graph of the basis function $\zeta_1$, attached to vertex $\q^E_1$ (see \cref{sec.localspaceVEMlap}).} 
\label{fig.eltE}
\end{figure}

\subsection{The virtual element method for the 2d conductivity equation}\label{sec.VEMlap}

\noindent This section deals with the implementation of the Virtual Element Method for the solution of the conductivity equation equipped with homogeneous Dirichlet boundary conditions: 
\begin{equation}\label{eq.lapVEM}
\tag{\textcolor{gray}{Cond}}
\left\{
\begin{array}{cl}
- \dv(\gamma \nabla u) = f & \text{in } \Omega, \\
u = 0 & \text{on } \partial \Omega, 
\end{array}
\right.
\end{equation}
where the sought function $u \in H^1_0(\Omega)$ is denoted without reference to the (fixed) domain $\Omega$ in this section, 
the right-hand side $f \in L^2(\Omega)$ stands for a source term, and the conductivity $\gamma$ is assumed to be a positive constant. 
Our presentation relies on the pedagogic article \cite{sutton2017virtual}, see also \cite{beirao2014hitchhiker}. We provide a general overview of the method in \cref{sec.genpresVEMlap} before focusing on the definition of the local space of discrete functions in \cref{sec.localspaceVEMlap}.

\subsubsection{General presentation of the method}\label{sec.genpresVEMlap}

\noindent 
Like the Finite Element Method (see e.g. \cite{ciarlet2002finite}), 
the Virtual Element Method leverages the variational formulation \cref{eq.varfcond} of \cref{eq.lapVEM}, which reads: 
\begin{equation}\label{eq.varflap}
 \text{Search for } u \in V \text{ s.t. } \forall v \in V, \quad a(u,v) = \ell(v).
 \end{equation}
Here, the functional space $V$ is $H^1_0(\Omega)$, and we have posed:
\begin{equation}\label{eq.aelllap}
\forall u,v \in V, \quad a(u,v) = \gamma \int_\Omega \nabla u \cdot \nabla v \:\d \x \:\text{ and }\: \ell(v) = \int_\Omega f v \:\d \x.
\end{equation}
The discretization of \cref{eq.varflap} relies on a finite-dimensional subspace $\calW_{\calT}$ of $V$;
the numerical approximation $u_{\calT} \in \calW_{\calT}$ of $u$ is sought as the solution to the problem:
\begin{equation}\label{eq.discVEMformulation}
 \text{Search for } u_{\calT} \in \calW_{\calT} \text{ s.t. } \forall v \in \calW_{\calT}, \quad a(u_{\calT},v) = \ell(v). 
 \end{equation}
The space $\calW_{\calT}$ is defined by:
\begin{equation}\label{eq.defcalWcalT}
 \calW_{\calT} = \left\{ u \in \calC(\overline\Omega) \text{ and } u \lvert_E \:\in\: \calW(E)  \text{ for all } E \in \calT\right\},
 \end{equation}
where for each element $E \in \calT$, $\calW(E)$ is a local space of functions on $E$ whose definition is the topic of the next \cref{sec.localspaceVEMlap}. 
For the moment, let us solely mention that the dimension of $\calW_{\calT}$ coincides with the number $M$ of vertices in the mesh $\calT$, 
and that the values of functions $u \in \calW_{\calT}$ at the vertices $\q_j$, $j=1,\ldots,M$ are the degrees of freedom of the discretization. 
We introduce a basis $\left\{ \varphi_k \right\}_{k = 1,\ldots,M}$ of $\calW_{\calT}$ as follows:
for $k=1,\ldots,M$, the function $\varphi_k \in \calW_{\calT}$ is attached to the vertex $\q_k$ in the sense that: 
\begin{equation}\label{eq.vphikql} 
\forall l = 1,\ldots,M, \quad\varphi_k(\q_l) = \left\{
\begin{array}{cl}
1 & \text{if } k =l , \\
0 & \text{otherwise}.
\end{array}
\right.
\end{equation}

By writing the decomposition $u_{\calT} = \sum_{l=1}^{M} u_l \varphi_l$ of the sought function $u_{\calT}$ on this basis, and using $v = \varphi_k$ as test function in \cref{eq.discVEMformulation} for $k=1,\ldots,M$, 
the problem \cref{eq.varflap,eq.aelllap} boils down to the following $M \times M$ linear system:
\begin{equation}\label{eq.FEVEM}
 K_{\calT} U_{\calT} = F_{\calT}, 
 \end{equation}
where $U_{\calT} := (u_1,\ldots,u_{M}) \in \R^{M}$, and the entries of the stiffness matrix $K_{\calT} \in \R^{M \times M}$ and force vector $F_{\calT} \in \R^{M}$ are given by:
\begin{equation}\label{eq.KTgen}
K_{\calT,kl} = a(\varphi_l,\varphi_k) , \text{ and } F_{\calT,k} = \ell(\varphi_k), \quad k,l=1,\ldots,M.
\end{equation}
The integrals defining these quantities are sums of contributions from each element $E \in \calT$, i.e.
\begin{multline}\label{eq.defaE}
 a(\varphi_l,\varphi_k) = \sum\limits_{E \in \calT} a^E(\varphi_l,\varphi_k) , \text{ where } a^E(u,v) := \gamma\int_E \nabla u \cdot \nabla v \:\d \x, \text{ and } \\
 F_{\calT,k} = \sum\limits_{E \in \calT}\ell^E(\varphi_k), \text{ where } \ell^E(v) := \int_E f v \:\d \x.
 \end{multline}
Hence, the assembly of \cref{eq.FEVEM} is realized by calculating the local, element-wise contributions $a^E(\varphi_l,\varphi_k)$, $\ell^E(\varphi_k)$ to the entries $ K_{\calT,kl} $ of the global stiffness matrix 
and $F_{\calT,k}$ of the global force vector, respectively. 
 These computations are discussed in the next \cref{sec.localspaceVEMlap,app.VEM}. 

\subsubsection{The local space $\calW(E)$ attached to an element $E \in \calT$}\label{sec.localspaceVEMlap}

\noindent Throughout this section, $E$ is a given polygon in $\calT$;
we define $\calW(E)$ as the space of functions $u : E \to \R$ satisfying the following properties:
\begin{itemize}
\item The restriction $u\vert_{\be} : \be \to \R$ of $u$ to any edge $\be$ of $E$ is affine; 
\item The differential operator of \cref{eq.lapVEM} cancels $u$, i.e. $-\dv(\gamma \nabla u) = -\gamma \Delta u = 0$ in $E$.
\end{itemize}
Obviously, the functions $u \in \calW(E)$ are uniquely determined by their values at the vertices $\q_j^E$ of $E$, $j=1,\ldots,n^E$;
in particular, $\calW(E)$ has dimension $n^E$ and it contains the space $\calP(E)$ of affine functions on $E$:
$$ \calP(E) := \left\{ u(\x) = a + \langle \textbf{b} , \x\rangle, \:\: a \in \R ,\: \: \textbf{b} \in \R^2\right\}.$$
We then define a basis $\left\{ \zeta_i \right\}_{i=1,\ldots,n^E}$ of $\calW(E)$ as follows: for $i=1,\ldots,n^E$, $\zeta_i$ is the unique function in $\calW(E)$ such that:
\begin{equation}\label{eq.defzetai}
 \zeta_i(\q_j^E) =
 \left\{
 \begin{array}{cl}
 1 & \text{if } i = j, \\
 0 & \text{otherwise}.
 \end{array}
 \right.
 \end{equation} 
 Observe that the dependency of the functions $\zeta_i$ on the actual element $E$ is omitted for notational simplicity.
 
 \begin{remark}
\noindent \begin{itemize}
\item  Several definitions of $\calW(E)$ are actually possible, all of them leading to discrete functions $u \in \calW_{\calT}$ whose degrees of freedom are their values at the vertices of $\calT$.
In fact, a careful selection of $\calW(E)$ may significantly ease the numerical implementation,
see for instance \cref{sec.L2proj} for an alternative definition of the behavior of functions $u \in \calW(E)$ inside $E$ which lends itself to an easier evaluation of integrals of the form $\int_E uv \:\d x$, $u,v \in \calW(E)$, or the appendix in \cite{gain2014virtual}  for another definition which is better adapted to the evaluation of quadrature formulas in the 3d elasticity setting.
\item Higher-order versions of $\calW(E)$, featuring higher degree polynomials, could also be considered, leading to higher-order approximations $u_{\calT}$ of $u$. 
\end{itemize}
\end{remark}
 
 With these notations, the global basis functions $\varphi_k$ defined in \cref{eq.vphikql} can be expressed in terms of the local functions $\zeta_i$; more precisely:
 $$\forall k = 1,\ldots,M, \:\: E\in \calT,\quad \varphi_k\lvert_E =  \left\{
 \begin{array}{cl}
 \zeta_i& \text{if there exists } i \in \left\{1,\ldots,n^E\right\} \text{ s.t. } \q_i^E = \q_k, \\
 0& \text{otherwise}. \\
 \end{array}\right.
 $$ 
 As a result, the computation of the entries of the stiffness matrix $K_{\calT}$ and force vector $F_{\calT}$ in \cref{eq.defaE} boils down to that of the local building blocks
 \begin{equation}\label{eq.KEij}
 K^E_{ij} := a^E(\zeta_i, \zeta_j) \text{ and } F^E_i :=\ell^E(\zeta_i), \quad i,j=1,\ldots,n^E.
 \end{equation}
Unfortunately, no explicit expression of the $\zeta_i$ is available to realize these computations; the key idea of the Virtual Element Method is to take advantage of the definition of $\calW(E)$ to approximate them with a computation which uses the sole property \cref{eq.defzetai}.
To achieve this, let us introduce: 
\begin{itemize} 
\item The subspace $\calW_R(E) \subset \calP(E)$ of constant functions on $E$:
$$ \calW_R(E) := \span \left\{m_1 \right\}, \text{ where } m_1(\x) = 1.$$
\item The subspace $\calW_C(E) \subset \calP(E)$ of functions with zero mean value over the vertices of $E$:
$$ \calW_C(E) := \left\{  u  \in \calP(E), \:\: \overline u= 0 \right\},$$
which is spanned by the functions
$$ m_2(\x) = x_1 - \overline{q_1^E}, \text{ and } m_3(\x) = x_2 - \overline{q^E_2}.$$
\end{itemize}
It is easily seen that $\calP(E)$ is the direct sum of $\calW_R(E)$ and $\calW_C(E)$:
$$ \calP(E) = \calW_R(E) \oplus \calW_C(E).$$
We next introduce projection operators $\pi_R$ and $\pi_C$ onto $\calW_R(E)$ and $\calW_C(E)$, respectively:
for any (smooth enough) function $v : E \to \R$, $\pi_R v \in \calW_R(E)$ and $\pi_C v \in \calW_C(E)$ are respectively defined by:
$$ \pi_R v(\x) = \overline v , \text{ and } \pi_C v(\x) = \langle \nabla v \rangle \cdot (\x - \overline{\q^E}).$$
With these notations, the following relations are easily verified:
$$ \forall c \in \calW_C(E), \:\: \pi_R c = 0, \:\: \pi_C c = c \text{ and }  \forall r \in \calW_R(E), \:\: \pi_C r = 0, \:\: \pi_R r = r.$$
We finally define the projection $\pi_{\calP} : \calW(E) \to \calP(E)$ onto affine functions by:
\begin{equation}\label{eq.piPlap}
 \pi_{\calP} v = \pi_{C} v + \pi_R v.
 \end{equation}

The following lemma draws elementary, albeit crucial properties of these objects.
\begin{lemma}\label{lem.NRJorthosca}
The following facts hold true: 
\begin{enumerate}[(i)]
\item The operators $\pi_R$ and $\pi_C$ are the ``orthogonal projections'' from $\calW(E)$ onto $\calW_R(E)$ and $\calW_C(E)$, respectively, in the sense of the bilinear form $a^E(\cdot,\cdot)$: 
$$ \forall r \in \calW_R(E), \:\: c \in \calW_C(E), \:\: v \in \calW(E), \quad a^E(r,v) = 0 \text{ and } a^E (c , v - \pi_C v) = 0.$$
\item The following decomposition of $a^E(\cdot,\cdot)$ holds true:
\begin{equation}\label{eq.decomposaE}
 \forall u,v \in \calW(E), \quad a^E(u,v) =  a^E(\pi_C u, \pi_C v ) + a^E(u-\pi_{\calP} u , v-\pi_{\calP} v).
 \end{equation}
\end{enumerate}
\end{lemma}
\begin{proof}
$(i)$ The first relation is obvious since functions in $\calW_R(E)$ are constant on $E$. 
The second one arises from the fact that the gradient $\nabla c$ of any function $c \in \calW_C(E)$ is constant over $E$: 
$$
 a^E(c, v - \pi_C v) = \gamma \nabla c \cdot \left(\int_E (\nabla v - \nabla (\pi_C v))\:\d \x \right)
= \gamma \nabla c \cdot \Big( |E| \langle \nabla v \rangle - |E| \langle \nabla v \rangle \Big) 
= 0.
$$
 
\noindent $(ii)$ A straightforward calculation based on the identities in $(i)$ yields:
$$\begin{array}{>{\displaystyle}cc>{\displaystyle}l}
a^E(u,v) &=& a^E(\pi_R u + \pi_C u + (u-\pi_{\calP} v) , \pi_R v + \pi_C v + (v-\pi_{\calP} v)) \\
&=& a^E(\pi_C u + (u-\pi_{\calP} v) , \pi_C v + (v-\pi_{\calP} v)) \\
&=&  a^E(\pi_C u, \pi_C v ) + a^E(u-\pi_{\calP} v , v-\pi_{\calP} v).
\end{array} $$
\end{proof}

Let us now note that \cref{eq.decomposaE} gives rise to a decomposition of the local stiffness matrix $K^E \in \R^{n^E \times n^E}$ in \cref{eq.KEij} into two terms:
$$ K^E = P^E + S^E.$$
Here, the matrices $P^E$ and $S^E \in \R^{n^E \times n^E}$ are defined by:
$$ \forall i,j=1,\ldots, n^E, \quad P^{E}_{ij} = a^E(\pi_C\zeta_i, \pi_C \zeta_j) \text{ and } S^E_{ij} = a^E(\zeta_i-\pi_{\calP}\zeta_i,\zeta_j-\pi_{\calP}\zeta_j).$$
The entries of the matrix $P^E$ can be calculated exactly; indeed, the projection $\pi_C u$ of a function $u\in \calW(E)$ has coefficients over the basis $\left\{ m_\alpha \right\}_{\alpha=1,2,3}$ 
that can be calculated solely from the values of $u$ at the vertices of $E$, see \cref{app.VEMlap} about this fact. 

On the other hand, the second term $a^E(u-\pi_{\calP} u , v-\pi_{\calP} v)$ in the right-hand side of \cref{eq.decomposaE} (associated to $S^E$) appraises the behavior of $a^E(\cdot,\cdot)$
on those functions in $\calW(E)$ that are not affine. It is not easily computed, and it is therefore replaced by another bilinear form $\widetilde{s}^E$ (and matrix $\widetilde{S}^E$) which can be calculated 
just from the values of $u$ and $v$ at the vertices of $E$ (see \cref{app.VEMlap} about this calculation):
$$ \widetilde{s}^E(u-\pi_{\calP} u , v-\pi_{\calP} v), \text{ where } \widetilde{s}^E (w,z) := \alpha^E \sum\limits_{i=1}^{n^E} w(\q_i^E) z(\q_i^E),$$
 and $\alpha^E$ is a constant which is chosen so that $\widetilde{s}^E$ scale as $a^E$ upon refinement of the mesh $\calT$;
in this simple case, we take $\alpha^E=1$.

This replacement gives rise to the modified bilinear form $\widetilde{a}^E(\cdot,\cdot)$ given by:
$$\forall u,v \in \calW(E), \quad \widetilde{a}^E(u,v) = a^E(\pi_{\calC} u, \pi_{\calC} v) +  \widetilde{s}^E(u-\pi_{\calP} u , v-\pi_{\calP} v).$$
It guarantees the so-called polynomial consistency of $\widetilde{a}^E$ with $a^E$ (or ``patch test'' in engineering): 
$\widetilde a^E(u,v)$ and $a^E(u,v)$ coincide when at least one of the functions $u$ or $v$ belongs to $\calP(E)$. Indeed, 
$$\begin{array}{>{\displaystyle}cc>{\displaystyle}l}
 \forall p \in \calP(E), \:\: v \in \calW(E), \quad \widetilde{a}^E(p,v) &=& a^E(\pi_C p,\pi_C v) \\
 &=& a^E(\pi_C p, v),\\
 &=& a^E(p,v).
 \end{array}$$
This method can be proved to be convergent, see \cite{da2013virtual}. Intuitively, it behaves at least as well as if only linear polynomials were the only basis functions on each element.

\begin{remark}\label{rem.vemev}
The strategy illustrated in this section can be adapted to the solution of eigenvalue problems of the form \cref{eq.conducev}, see
\cite{gardini2018virtual}. The main additional operation required to achieve this purpose is the assembly of the mass matrix, which is described in \cref{sec.L2proj}. 
\end{remark}

\subsection{Resolution of the linear elasticity system with the Virtual Element Method}\label{sec.VEMelas}

\noindent In this section, whose presentation is based on \cite{berbatov2021guide,gain2014virtual},
 we briefly describe how the strategy presented in the previous \cref{sec.VEMlap} can by adapted to deal with the solution of the 2d linear elasticity system. To emphasize the similarities between the treatments of both situations, we use the same notations for corresponding quantities when no confusion is possible.
 
Let us recall the setting of \cref{sec.elas}: the boundary of the shape $\Omega$ is decomposed into three disjoint regions $\Gamma_D$, $\Gamma_N$, $\Gamma$, 
 where $\Gamma_D$ is a clamping area, $\Gamma_N$ is the support of external loads $\g: \Gamma_N \to \R^2$ and $\Gamma$ is effort-free. 
 Assuming the presence of body forces $\f : \Omega \to \R^2$, the elastic displacement $\u : \Omega \to \R^2$ is the unique solution in $H^1_{\Gamma_D}(\Omega)^2$ to the boundary value problem:  
\begin{equation}\label{eq.VEMelassys}
\tag{\textcolor{gray}{Elas}}
\left\{
\begin{array}{cl}
-\dv(Ae(\u)) = \f & \text{in } \Omega, \\
\u = \bzero & \text{on } \Gamma_D, \\
Ae(\u)\n = \g & \text{on } \Gamma_N, \\
Ae(\u)\n = \bzero & \text{on } \Gamma.
\end{array}
\right. 
\end{equation}
The associated variational formulation reads, see \cref{eq.varfelas}: 
\begin{multline*}
 \text{Search for } \u \in V \text{ s.t. } \forall \v \in V, \quad a(\u,\v) = \ell(\v), \text{ where } V = H^1_{\Gamma_D}(\Omega)^2, \\
 a(\u,\v) = \int_\Omega A e(\u): e(\v) \:\d \x, \text{ and } \ell(\v) = \int_\Omega \f \cdot \v \:\d \x + \int_{\Gamma_N} \g \cdot \v \:\d s.
 \end{multline*}

As described in \cref{sec.genpresVEMlap}, the cornerstone of the development of the Virtual Element Method
lies in the definition of the local space of discrete functions $\calW(E)$ and the assembly of the local stiffness matrix $K^E \in \R^{2n^E \times 2n^E}$ attached to a given polygon $E$ in $\calT$. The local space $\calW(E)$ reads:
$$ \calW(E) = \Big\{ \u: E \to \R^2, \:\: \u \text{ is affine on each edge of } \partial E \text{ and } -\dv(Ae(\u)) = 0 \Big\}.$$
Again, an element $\u \in \calW(E)$ is uniquely determined by its values at the vertices $\q_j^E$, $j=1,\ldots,n^E$ of $E$ and 
$\calW(E)$ has dimension $2n^E$. 
Besides, $\calW(E)$  contains the space $\calP(E)$ of affine functions, given by:
$$ \calP(E) := \left\{ \u(\x) = \a + B \x, \:\: \a \in \R^2, \: \: B \in \R^{2\times 2}\right\}.$$
Let us introduce the natural basis $\left\{\bzeta_j\right\}_{j=1,\ldots,2n^E}$ of $\calW(E)$, defined by, for $i=1,\ldots,n^E$:
\begin{equation}\label{eq.defvphi1}
 \bzeta_{2i-1}(\q_i^E) = \left(\begin{array}{c}
1\\
0
\end{array}
\right), \text{ and }   \bzeta_{2i-1}(\q_j^E) = \left(\begin{array}{c}
0\\
0
\end{array}
\right)\text{ for } j \neq i;
\end{equation}
\begin{equation}\label{eq.defvphi2}
 \bzeta_{2i}(\q_i^E) =\left(\begin{array}{c}
0\\
1
\end{array}
\right), \text{ and }   \bzeta_{2i}(\q_j^E) = \left(\begin{array}{c}
0\\
0
\end{array}
\right) \text{ for } j \neq i. 
\end{equation} 
In a similar state of mind as in \cref{sec.VEMlap}, among the elements in $\calP(E)$, we shall distinguish the rigid-body motions $\u \in \calW_R(E)$, whose strain $e(\u)$ equals $0$, 
and the fields $\u \in \calW_C(E)$ with constant strain $e(\u)$; precisely:
\begin{itemize}
\item The subspace $\calW_R(E) \subset \calP(E)$ of rigid-body motions is defined by
$$ \calW_R(E) := \left\{ \a + M (\x-\bar \q^E), \:\: \a \in \R^2, \:\: M \in \R^{2\times 2}, \:\: M^T = -M \right\};$$
it is spanned by the three vector fields:
\begin{equation}\label{eq.defri}
\r_1 (\x) = \left(\begin{array}{c}
1\\
0
\end{array}
\right), \quad \r_2 (\x) = \left(\begin{array}{c}
0\\
1
\end{array}
\right) \text{ and } \r_3 (\x) = \left(\begin{array}{c}
-(x_2-\overline{q_2^E})\\
x_1 - \overline{q_1^E}
\end{array}
\right).
\end{equation}
\item The subspace $\calW_C(E) \subset \calP(E)$ of vector fields with constant strain is defined by
$$ \calW_C(E) := \left\{ S(\x-\overline{\q^E}), \:\: S \in \R^{2\times 2}, \:\: S^T = S \right\};$$
it is spanned by the three vector fields: 
\begin{equation}\label{eq.defci}
 \c_1 (\x) = \left(\begin{array}{c}
x_1 - \overline{q^E_1}\\
0
\end{array}
\right), \quad \c_2 (\x) = \left(\begin{array}{c}
0\\
x_2 - \overline{q^E_2}
\end{array}
\right) \text{ and } \c_3 (\x) = \left(\begin{array}{c}
x_2 - \overline{q_2^E}\\
x_1 - \overline{q_1^E}
\end{array}
\right).
\end{equation}
\end{itemize}
Again, it is easily verified that $\calP(E)$ is the direct sum of $\calW_R(E)$ and $\calW_C(E)$:
$$ \calP(E) = \calW_R(E) \oplus \calW_C(E).$$
We next define the projections $\pi_R$ and $\pi_C$ from $\calW(E)$ onto $\calW_R(E)$ and $\calW_C(E)$, respectively. 
\begin{itemize}
\item For any (smooth enough) vector field $\u : E \to \R^2$, $\pi_R \u \in \calW_R(E)$ is defined by:
$$ \pi_R \u(\x) = \bar \u + \langle \omega(\u) \rangle
 \left(\begin{array}{c}
-(x_2 - \overline{q^E_2})\\
x_1 - \overline{q^E_1}
\end{array}
\right)
, \quad \x \in E;$$
where $\omega(\u) := \frac{1}{2}\left( \frac{\partial u_2}{\partial x_1} - \frac{\partial u_1}{\partial x_2}\right)$ is (half) the scalar curl of  $\u$.
\item For any (smooth enough) vector field $\u : E \to \R^2$, $\pi_C \u \in \calW_C(E)$ is defined by:
 $$ \pi_C \u(\x) = \langle e(\u) \rangle (\x-\overline{\q^E}).$$
\end{itemize}
The following relations are then easily verified:
$$ \forall \r \in \calW_R(E), \:\: \pi_C \r = \bzero, \:\: \pi_R \r = \r, \text{ and }  \forall \c \in \calW_C(E), \:\: \pi_R \c = \bzero, \:\: \pi_C \c = \c.$$
From these data, we define the projection operator $\pi_{\calP}$ onto affine functions in $\calP(E)$ by:
\begin{equation}\label{eq.piP}
 \pi_{\calP} \u = \pi_{C} \u + \pi_R \u.
 \end{equation}
The next lemma is the counterpart of \cref{lem.NRJorthosca} in the present linear elasticity setting; its very similar proof is omitted for brevity.
\begin{lemma}\label{lem.NRJorthoelas}
The following facts hold true:
\noindent \begin{enumerate}[(i)]
\item The operators $\pi_R$ and $\pi_C$ are the ``orthogonal projections'' from $\calW(E)$ onto $\calW_R(E)$ and $\calW_C(E)$, respectively, in the sense of the bilinear form $a^E(\cdot,\cdot)$: 
$$ \forall \r \in \calW_R(E), \:\: \c \in \calW_C(E), \:\: \v \in \calW(E), \quad a^E(\r,\v) = 0 \text{ and } a^E (\c , \v - \pi_C \v) = 0.$$
\item The following decomposition of $a^E(\cdot,\cdot)$ holds:
\begin{equation}\label{eq.decomposaEelas}
 \forall \u,\v \in \calW(E), \quad a^E(\u,\v) =  a^E(\pi_C \u, \pi_C \v ) + a^E(\u-\pi_{\calP} \u , \v-\pi_{\calP} \v).
 \end{equation}
\end{enumerate}
\end{lemma}

Let us now discuss the calculation of the local stiffness matrix 
$$K^E \in \R^{2n^E \times 2n^E}, \quad K^E_{ij} = a(\bzeta_i,\bzeta_j), \:\: i,j=1,\ldots,2n^E.$$ 
Like in \cref{sec.localspaceVEMlap}, \cref{eq.decomposaEelas} induces a convenient decomposition $K^E$ into two blocks:
$$ K^E = P^E + S^E, \quad P^E, S^E \in \R^{2n^E \times 2n^E}.$$
Here,
\begin{itemize}
\item The first term in the right-hand side of \cref{eq.decomposaEelas} gives rise to the matrix $P^E \in \R^{2n^E \times 2n^E}$ defined by: 
$$ \forall i,j=1,\ldots, 2n^E, \quad P^E_{ij} = a^E(\pi_C\bzeta_i, \pi_C \bzeta_j),$$
which roughly speaking corresponds to the block of $K^E$ over $\calP(E)$.
The entries of $P^E$ can be calculated exactly, since the coefficients of the projection $\pi_C \u$ of a vector field $\u \in \calW(E)$ in the basis $\left\{ \c_\alpha \right\}_{\alpha=1,2,3}$ can be calculated from the values of $\u$ at the vertices of $E$, see \cref{app.VEMelas}.
\item  The second term in the right-hand side of \cref{eq.decomposaEelas} appraises the behavior of the bilinear form $a^E(\cdot,\cdot)$
over those functions in $\calW(E)$ that are not affine. Again, this second block is replaced by a suitable quantity which can be calculated 
just from the values of $\u$ and $\v$ at the vertices of $E$.
\end{itemize}
As the result of these considerations, we replaced the local bilinear form $a^E(\cdot , \cdot)$ with 
$$ \widetilde{a}^E(\u,\v) = a^E(\pi_C \u , \pi _C \v) + \widetilde{s}^E(\u-\pi_{\calP}\u, \v- \pi_{\calP} \v), \text{ where } \widetilde{s}^E (\u,\v) = \alpha^E \sum\limits_{i=1}^n \u(\q_i) \cdot \v(\q_i),$$
 and the coefficient $\alpha^E$ depends on the element $E$ so that $\widetilde{s}^E$ scale as $a^E$ upon refinement of the mesh $\calT$.
The implementation details of this methodology are presented in \cref{app.VEMelas}. 

\section{Sensitivity of a shape functional with respect to seed points and volume fractions}\label{sec.deropt}

\noindent In this section, we return to our main purpose of solving the shape and topology optimization problem
\begin{equation}\label{eq.sopbs6} 
\tag{\textcolor{gray}{P}}
\min \limits_{\Omega} J(\Omega) \:\: \text{ s.t. } \G(\Omega) = 0,
\end{equation}
where $J(\Omega)$ and $\G(\Omega)$ depend on $\Omega$ via the solution $u_\Omega$ to a boundary value problem such as \cref{eq.conduc} or \cref{eq.elas}.
We describe how the shape derivative of a general objective function $J(\Omega)$ (or similarly, that of a general constraint function $\G(\Omega)$) can be expressed in terms of the defining variables of a diagram $\bVsp$ representing $\Omega$ in the sense that \cref{eq.decompOm} holds. 
In this perspective, in \cref{sec.derver}, we first deal with the discretization of $J(\Omega)$ into a function $J(\q)$ of the vertices $\q \in \R^{dM}$ of the diagram, 
and we expose how the theoretical expression of the shape derivative $J^\prime(\Omega)(\btheta)$ allows to calculate the gradient $\nabla_{\q} J(\q)$ of this discrete function.
As we have mentioned, $J(\q)$ (resp. $\G(q)$) in turn induces a function $\widetilde{J}(\s,\bnu)$ (resp. $\widetilde{G}(\s,\bnu)$) of the seed points $\s \in \R^{dN}$ and cell measures $\bnu \in \R^N$ of the diagram $\bVsp$.
This gives rise to our discrete shape and topology optimization problem:
\begin{equation}\label{eq.sopbdisc} 
\tag{\textcolor{gray}{P-disc}}
\min \limits_{(\s, \bnu)} \widetilde J(\s,\bnu) \:\: \text{ s.t. } \widetilde{\G}(\s,\bnu) = 0.
\end{equation}
We explain in \cref{sec.vertoseeds} how the sensitivities $\nabla_{\s} \widetilde J(\s,\bnu)$ and $\nabla_{\bnu} \widetilde J(\s,\bnu)$ of $J(\Omega)$ with respect to the parameters $(\s,\bnu)$ of the diagram for $\Omega$ can be inferred from the gradient $\nabla_\q J(\q)$.
In \cref{sec.velext}, we recall the Hilbertian trick, which allows among other things to smooth these sensitivities in a consistent manner. This procedure is also a key algorithmic component of the constrained optimization algorithm used to tackle \cref{eq.sopbs6}, that we finally describe in \cref{sec.nullspace}.
\par\medskip

\noindent \textbf{Notation.} In this section, the shape $\Omega \subset \R^d$ is described by a modified diagram $\bVsp$ of the form \cref{eq.decompOm,eq.defVipsi}. We consider a generic function $J(\Omega)$ of $\Omega$, standing for either the objective or a constraint in \cref{eq.sopbs6}. 
To simplify the presentation and without loss of generality, we shall sometimes assume that the space dimension $d$ equals $2$,
and that $J(\Omega)$ is of the form
\begin{equation}\label{eq.JOmgen}
 J(\Omega) = \int_\Omega j(u_\Omega) \:\d \x,
\end{equation}
where $j : \R \to \R$ is a smooth function, satisfying adequate growth conditions, and $u_\Omega$ is the solution to the conductivity equation \cref{eq.conduc}.  
As we have hinted at, $J(\Omega)$ induces several discrete versions, one depending on the vertices $\q$ of the diagram, that we shall denote by $J(\q)$ with a small abuse of notation, and in turn, another function $\widetilde{J}(\s,\bnu)$, depending on the seed points $\s$ and the cell measures $\bnu$. 

\subsection{Discretization and differentiation of $J(\Omega)$ in terms of the vertices of a diagram for $\Omega$}\label{sec.derver}

\noindent In this section, we discuss the discretization of the shape functional $J(\Omega)$ as a function $J(\q)$ of the vertices $\q \in \R^{dM}$ of the diagram 
$\bVsp$ representing $\Omega$ as \cref{eq.decompOm} and the calculation of the gradient $\nabla_\q J(\q)$.

As we have mentioned, in our applications, $J(\Omega)$ depends on $\Omega$ via the solution $u_\Omega$ 
to a boundary value problem posed on $\Omega$, see e.g. \cref{eq.JOmgen}. 
The latter is solved along the lines of \cref{sec.mechcomp} by applying the Virtual Element Method on the polygonal mesh $\calT$ resulting from the discretization of the diagram $\bVsp$, see \cref{sec.geomcomp}. The discrete solution is a function of the vertices $\q$ of the mesh $\calT$, and it gives rise in turn to a discrete objective function $J(\q)$ once the formula \cref{eq.JOmgen} is discretized by suitable quadrature formulas. 
In the present work, we rely on very basic expressions, which are exemplified in the forthcoming discussion. 
More elaborate rules are available in the context of general polygonal elements \cite{chin2021scaled,sommariva2009gauss}; 
for brevity, we do not emphasize on this classical, albeit technical issue.

The calculation of the gradient $\nabla_\q J(\q)$ is a slightly more delicate issue. It falls under the ``discretize-then-optimize'' setting, in which  
the derivatives of the discrete quantities of interest are used \cite{pironneau1982optimal}; several methods can be used to achieve this goal. 
One convenient practice relies on automatic differentiation, see e.g. \cite{mohammadi2010applied} for an introduction in the shape optimization context. In a nutshell, the gradient of the mapping $\q \mapsto J(\q)$ can be calculated concurrently with its evaluation by leveraging adequate data structures and by implementing all the constituent operations of the evaluation pipeline (including the assembly and inversion of the stiffness matrix involved in calculations by the Virtual Element Method) together with their derivatives. 
This elegant framework is unfortunately very stringent, and it is not always compatible with realistic applications where the mechanical solver is used in a black-box fashion. 

When automatic differentiation of $J(\q)$ is not available, the computation of $\nabla_\q J(\q)$ can be realized ``by hand''.
In fact, the derivatives of the aforementioned building blocks of the evaluation of $\q \mapsto J(\q)$ can be calculated explicitly, in principle. 
Even though this strategy could be adopted in our context, it is still very intrusive as it requires a complete knowledge of the mechanical solver, notably.\par\medskip

For this reason, we consider yet another option, borrowing from the ``optimize-then-discretize'' paradigm in optimal design, which mainly relies on the continuous formula for the shape derivative $J^\prime(\Omega)(\btheta)$.
More precisely, we leverage an elegant remark from \cite{berggren2009unified}: the sensitivity $\frac{\partial J}{\partial \q_i}(\q)(\h_i)$ 
of $J(\q)$ to a perturbation of the $i^{\text{th}}$ vertex $\q_i$ in the direction $\h_i \in \R^d$ can be approximated by evaluating (a discretization of) the volume form of the continuous shape derivative $\btheta \mapsto J^\prime(\Omega)(\btheta)$ with the particular deformation field $\btheta = \varphi_i \h_i $, where $\varphi_i$ is the unique function in the discrete space $\calW_{\calT}$ given by \cref{eq.defcalWcalT} taking the value $1$ at vertex $i$ and $0$ elsewhere, see \cref{eq.vphikql}: 
$$ 
\frac{\partial J}{\partial \q_i}(\q)(\h_i) = J^\prime(\Omega)(\varphi_i \h_i), \text{ where } \varphi_i \in \calW_{\calT} \text{ is defined by: } \varphi_i(\q_j) = \left\{
\begin{array}{cl}
1 & \text{if } i = j, \\
0 & \text{otherwise.}
\end{array}
\right.
$$
In a nutshell, this practice is legitimated by the fact that the theoretical formula for the volume form of $J^\prime(\Omega)(\btheta)$ can be obtained by the exact same adjoint-based trail as that for the discrete derivative $\frac{\partial J}{\partial \q_i}(\q)(\h_i)$.
We refer to e.g. \cite{glowinski1998shape} and $\S 2.9$ in \cite{gunzburger2002perspectives} for related discussions about the (non) commutation of discretization and differentiation in optimal design. 

In order to describe this possibility more precisely, let us assume that the shape derivative $J^\prime(\Omega)(\btheta)$ is available under the following volume form \cref{eq.Fpvol}:
$$  J^\prime(\Omega)(\btheta) = \int_\Omega (\bt_\Omega \cdot \btheta + S_\Omega : \nabla \btheta )\:\d \x , \text{ for some } \bt_\Omega : \Omega \to \R^d \text{ and } S_\Omega: \Omega \to \R^{d\times d},$$
and let us suppose for a moment that the fields $\bt_\Omega$ and $S_\Omega$ are given approximate, discrete counterparts $\bt_{\calT}$ and $S_{\calT}$ that are constant in restriction to each element $E \in\calT$: 
\begin{equation}\label{eq.JpcalT}
J^\prime(\Omega)(\btheta) \approx J^\prime_{\calT}(\btheta) := \int_\Omega (\bt_{\calT} \cdot \btheta + S_{\calT} : \nabla \btheta )\:\d \x.
\end{equation}
The gradient of $J(\q)$ with respect to the coordinates of the $i^{\text{th}}$ vertex of the mesh is then approximated as:   
$$ \nabla_{\q_i} J (\q) \approx \left( 
\begin{array}{c}
J^\prime_{\calT}(\btheta_i^1)\\
\vdots \\
J^\prime_{\calT}(\btheta_i^d)
\end{array} \right),$$
where for $k=1,\ldots,d$, $\btheta_i^k \in \calW_{\calT}^d$ is the unique vector-valued discrete function such that:
\begin{equation}\label{eq.btheta01}
\btheta_i^k(\q_j) = \be_k  \text{ if } i=j \text{ and } \bz \text{ otherwise}.
\end{equation}

The aforementioned approximations of $\bt_\Omega$, $ S_\Omega$ by $\bt_{\calT}$ and $S_{\calT}$ can be realized in various ways. For instance one could use the simple formulas:
$$ \bt_{\calT}\lvert_E = \frac{1}{n^E}\sum\limits_{j=1}^{n^E} \bt_\Omega(\q_j^E), \text{ and } S_{\calT}\lvert_E  = \frac{1}{n^E} \sum\limits_{j=1}^{n^E} S_\Omega(\q_j^E),$$
where we have used the notations of \cref{sec.notVEM}. 

When $S_\Omega$ involves the gradient of the state function $u_\Omega$, say for instance
$ S_\Omega = \gamma \nabla u_\Omega \otimes \nabla u_\Omega$, one could also make the approximation: 
$$ S_{\calT}\lvert_E  = \gamma_E \nabla (\pi_{\calP}u) \otimes \nabla (\pi_{\calP} u), \text{ where } \gamma_E := \frac{1}{n^E}\sum\limits_{j=1}^{n^E} \gamma(\q_j^E),$$ 
and $\pi_{\calP}$ is the projector over first-order polynomial functions in $\calP(E)$, defined in \cref{eq.piP}. 

In any event, the integrals featured in the quantities $J^\prime_{\calT}(\btheta_i^k)$ can be naturally decomposed as sums of integrals running over the elements $E$ of the mesh $\calT$:
 $$ J_{\calT}^\prime(\btheta_i^k) = \sum\limits_{E\in \calT} \Big(  \bt_{\calT}\lvert_E \cdot \bt_E(\btheta_i^k)   + S_{\calT} \lvert_E : S_E(\btheta_i^k) \Big),$$
 where the vectors $\bt_E(\btheta_i^k)$ and matrices $S_E(\btheta_i^k)$ are defined by:
$$ \bt_E(\btheta_i^k) = \int_E \btheta_i^k \:\d \x \text{ and of the matrices } S_E(\btheta_i^k) = \int_E \nabla \btheta_i^k \:\d \x.$$
 These basic ingredients can be calculated or approximated in  explicit terms of the physical characteristics of the element $E$. 
 For instance, in the case where the space dimension $d$ equals $2$, we approximate $\bt_E(\theta_i^1)$ and $\bt_E(\theta_i^2)$ as: 
 $$ \bt_E(\btheta_i^1) \approx \frac{1}{n^E} \lvert E \lvert \left( 
 \begin{array}{c}
 1\\
 0
 \end{array}
 \right) \text{ and }  \bt_E(\btheta_i^2) \approx \frac{1}{n^E} \lvert E \lvert \left( 
 \begin{array}{c}
 0\\
 1
 \end{array}
 \right) .$$
As far as the matrices $ S_E(\btheta_i^1) $, $ S_E(\btheta_i^2)$ are concerned, similar integrations by parts to those used in \cref{app.VEMlap} yield: 
 $$ S_E(\btheta_i^1) = \frac12 \left(
\begin{array}{cc} 
\lvert \widehat{\be_j} \lvert (\n_{\widehat{\be_j}})_1 &  \lvert \widehat{\be_j} \lvert (\n_{\widehat{\be_j}})_2 \\
0 & 0
\end{array}
\right), \text{ and } S_E(\btheta_i^2) = \frac12 \left(
\begin{array}{cc} 
0 & 0\\
\lvert \widehat{\be_j} \lvert (\n_{\widehat{\be_j}})_1 &  \lvert \widehat{\be_j} \lvert (\n_{\widehat{\be_j}})_2
\end{array}
\right).$$

\begin{remark}
In principle, the continuous shape derivative $J^\prime(\Omega)(\btheta)$ depends only
on the values of the deformation $\btheta$ on the boundary $\partial \Omega$, see \cref{sec.sotuto} about this general feature. 
On the contrary, it is not exactly true that the approximation $J^\prime_{\calT}(\btheta_i^k)$ in \cref{eq.JpcalT} vanishes when the deformation $\btheta_i^k$ is associated to an internal vertex $\q_i$. This is a manifestation of the lack of commutation between the operations of discretization and differentiation in optimal design, see again \cite{glowinski1998shape} and $\S 2.9$ in \cite{gunzburger2002perspectives}.
However, we expect (and indeed verify) that if the approximation \cref{eq.JpcalT} is ``good enough'', the entries of $\nabla_\q J(\q)$ attached to internal vertices should be ``very small'' when compared to those related to boundary vertices. For this reason, it is enough to compute the quantities $J^\prime_{\calT}(\btheta_j^k)$ for the indices $j=1,\ldots,M$ of the boundary vertices of $\Omega$.
\end{remark} 

\subsection{Calculation of the derivative of a functional with respect to the positions of the seeds and the cell measures}\label{sec.vertoseeds}

\noindent 
We have seen in \cref{sec.derver} how a general function of the domain $J(\Omega)$ 
can be discretized into a function $J(\q)$ depending on the positions $\q \in \R^{dM}$ of the vertices of the mesh $\calT$, and how to calculate the gradient $\nabla_{\q} J(\q)$.
In turn, $\q$ depends on the seed points $\s$ and the weights $\bpsi$ of the diagram $\bVsp$ of $\Omega$, and, again, on the seed points $\s$ and the cell measures $\bnu$, which are the design variables of our discrete version of the shape and topology optimization problem \cref{eq.sopbs6}, as explained in \cref{sec.Lagalgo}.
This raises the need to express the derivative of a general function $J(\q)$ of the vertices $\q$ of the diagram $\bVsp$ in terms of $\s$ and $\bnu$.\par\medskip 

To formulate this issue rigorously, let us write the relation between the seeds points $\s \in \R^{dN}$, the weights $\bpsi \in \R^N$ and the vertices $\q$ of the diagram $\bVsp$ under the abstract form:
\begin{equation}\label{eq.qQsp}
 \q = \bQ(\s,\bpsi^*(\s,\bnu)),
\end{equation}
where $\bQ:\R^{dN}_{\s} \times \R^N_{\bpsi} \to \R^{dM}$ produces the collection of vertices of the mesh $\calT$ 
induced by $\bVsp$, and, as in \cref{sec.Lagrefvol}, $\bpsi^*(\s,\bnu)$ is the unique weight vector such that each cell $V_i(\s,\bpsi^*(\s,\bnu))$ has measure $\nu_i$, $i=1,\ldots,N$. We then wish to calculate the derivatives of the functional $\widetilde J: \R^{dN}_\s \times \R^N_\bnu \to \R$ defined by:
\begin{equation}\label{eq.JtJQ}
 \widetilde J(\s,\bnu) := J(\bQ(\s,\bpsi^*(\s,\bnu)).
 \end{equation}
This is the purpose of the next result. 

\begin{proposition}\label{prop.derJqs}
Let $J: \R^{dM} \to \R$ be a differentiable function. 
Then the composite function $\widetilde J : \R^{dN}_\s \times \R^N_\bnu \to \R$ in \cref{eq.JtJQ} is differentiable almost everywhere, and its derivatives read:
\begin{equation}\label{eq.NsJt}
 \nabla_\s \widetilde J(\s,\bnu) = \Big[ \nabla_\s \bQ(s,\bpsi^*(\s,\bnu))\Big]^T \nabla_\q J(\q)  +  \Big[\nabla_\s \bF(\s, \bnu,\bpsi^*(\s,\bnu)) \Big]^T \p(\s,\bnu), 
 \end{equation}
 and 
\begin{equation}\label{eq.NnuJt}
 \nabla_\bnu  \widetilde{J}(\s,\bnu) = -\p(\s,\bnu).
 \end{equation}
Here, the function $\bF: \R^{dN}_\s \times \R^N_\bnu \times \R^N_\bpsi \to \R^N$ is the derivative $\nabla_\bpsi K(\s,\bnu,\bpsi)$, see \cref{eq.implicitPsi}; 
the adjoint vector $\p(\s,\bnu) \in \R^N$ is defined as the solution to the $N \times N$ system:
\begin{equation}\label{eq.adjQs}
 \Big[\nabla_{\bpsi} \bF(\s, \bnu, \bpsi^*(\s,\bnu))\Big] ^T  \p(\s,\bnu) = -  \Big[\nabla_\bpsi \bQ  (s,\bpsi^*(\s,\bnu))\Big]^T  \nabla_\q J(\q).
\end{equation}
In these formulas, 
\begin{itemize}
\item The gradient $\nabla_\q J(\q)$ is a vector in $\R^{dM}$, whose expression is assumed to be known.
\item The matrices $\left[\nabla_\s \bQ(\s,\bpsi^*(\s,\bnu))\right] $ and $\left[\nabla_\bpsi \bQ(\s,\bpsi^*(\s,\bnu))\right] $  have respective sizes $dM \times dN$ and $dM \times N$; 
their explicit expressions (in 2d) are the subject of \cref{app.verseeds}.
\item The matrix $ \left[\nabla_{\bpsi} \bF(\s, \bnu,\bpsi^*(\s,\bnu))\right] ^T$ has size $N \times N$, and its calculation (in 2d) is provided in \cref{app.derF}. 
\end{itemize}
\end{proposition}

\begin{proof}
At first, \cref{prop.derver} below states that the mapping $\bQ$ is differentiable at any point $(\s,\bpsi) \in \R^{dN} \times \R^N$ satisfying the assumptions of \cref{prop.psi}. These points form an open subset of $\R^{dN} \times \R^N$ whose complement has null Lebesgue measure.
Moreover, for any such point $(\s,\bnu) \in \R^{dN} \times \R^N$, we have seen that $\bpsi^*(\s,\bnu)$ is the unique solution to the equation \cref{eq.implicitPsi}. By the implicit function theorem and the regularity of $\bF$ expressed in \cref{prop.psi}, it follows that $\bpsi^*$ is differentiable at such a point $(\s,\bnu)$, and
by composition, $\widetilde J$ is differentiable at a.e. point $(\s, \bnu)$ in $\R^{dN} \times \R^N$.\par\medskip

\noindent \textit{Let us now prove \cref{eq.NsJt}.} At any point $(\s,\bnu)$ where $\widetilde J$ is differentiable, the chain rule shows that: 
$$\forall \h \in \R^{dN}, \quad \frac{\partial \widetilde J}{\partial \s}(\s,\bnu)(\h) = \frac{\partial J}{\partial \q}(\q) \left( \frac{\partial \bQ}{\partial \s}(\s,\bpsi^*(\s,\bnu))(\h) + \frac{\partial \bQ}{\partial \bpsi}(\s,\bpsi^*(\s,\bnu))\left(\frac{\partial \bpsi^*}{\partial \s}(\s,\bnu)(\h)\right)\right),$$
where we have set $\q := \bQ(\s,\bpsi^*(\s,\bnu))$. This rewrites, in terms of gradients and matrices:
\begin{equation}\label{eq.nablajt1}
\nabla_\s \widetilde J(\s,\bnu) = \Big(\Big[ \nabla_\s \bQ(\s,\bpsi^*(\s,\bnu))\Big]^T + \Big[ \nabla_\s \bpsi^*(\s,\bnu) \Big] ^T\Big[\nabla_\bpsi \bQ  (\s,\bpsi^*(\s,\bnu))\Big]^T \Big) \nabla_\q J(\q).
\end{equation}
All the terms are explicit in this formula, except for the matrix $\left[\nabla_\s \bpsi^*(\s,\bnu)\right]$. 
The calculation of the latter is difficult, 
since the function $\s \mapsto \bpsi^*(\s,\bnu)$ is only known implicitly, through the relation \cref{eq.implicitPsi}.
To overcome this difficulty, we use the classical adjoint technique for the differentiation of a quantity depending on the solution to a parametrized implicit equation, 
see again \cite{lions1971optimal,plessix2006review}.
Differentiation with respect to $\s$ in \cref{eq.implicitPsi} yields:
\begin{equation}\label{eq.diffFpsi}
\underbrace{ \Big[\nabla_\s \bF(\s,\bnu, \bpsi^*(\s,\bnu)) \Big]}_{N \times dN \text{ matrix}} + \underbrace{ \Big[\nabla_{\bpsi} \bF(\s, \bnu,\bpsi^*(\s,\bnu)) \Big]}_{N \times N \text{ matrix}} \underbrace{ \Big[\nabla_\s \bpsi^*(\s,\bnu)\Big]}_{N \times dN \text{ matrix}} = 0.
 \end{equation}
 
We then introduce the adjoint state $\p(\s,\bnu) \in \R^N$, defined as the solution to the $N\times N$ linear system:
$$ \Big[\nabla_{\bpsi} \bF(\s,\bnu, \bpsi^*(\s,\bnu))\Big] ^T  \p(\s,\bnu) = -  \Big[\nabla_\bpsi \bQ  (\s,\bpsi^*(\s,\bnu))\Big]^T  \nabla J_\q(\q).$$
Injecting this identity into \cref{eq.nablajt1}, we obtain: 
$$ \nabla_\s \widetilde J(\s,\bnu) = \Big[ \nabla_\s \bQ(\s,\bpsi^*(\s,\bnu))\Big]^T \nabla_{\q} J(\q)  -  \Big[ \nabla_\s \bpsi^*(\s,\bnu) \Big] ^T \Big[\nabla_{\bpsi} \bF(\s, \bnu,\bpsi^*(\s,\bnu))\Big] ^T \p(\s,\bnu) .$$
Eventually, using \cref{eq.diffFpsi} to transform the last expression in the above right-hand side, we obtain:
$$ \nabla_\s \widetilde J(\s,\bnu) = \Big[ \nabla_\s \bQ(\s,\bpsi^*(\s,\bnu))\Big]^T \nabla_{\q} J(\q)  +  \Big[\nabla_\s \bF(\s, \bpsi^*(\s,\bnu)) \Big]^T \p(\s,\bnu),$$
which is the desired formula.\par\medskip

\noindent \textit{Let us then consider the second formula \cref{eq.NnuJt}.}
Again, a simple use of the chain rule yields:
\begin{equation}\label{eq.NJnuproof}
\begin{array}{ccl}
\nabla_{\bnu} \widetilde{J}(\s,\bnu) &=& \nabla_{\bnu} \Big(J(\bQ(\s,\bpsi^*(\s,\bnu))) \Big) \\
&=& \Big[\nabla_\bnu \bpsi^*(\s,\bnu)\Big]^T \Big[\nabla_\bpsi \bQ(\s,\bpsi^*(\s,\bnu))\Big]^T \nabla_\q J(\q),
\end{array}
\end{equation}
where $\q := \bQ(\s,\bpsi^*(\s,\bnu))$.
As in the proof of \cref{eq.NsJt}, we rely on the adjoint method to eliminate the difficult term $\left[\nabla_\bnu \bpsi^*(\s,\bnu)\right]^T$ from this expression.
Recalling the defining system \cref{eq.adjQs} for $\p(\s,\bnu) \in \R^N$, we obtain immediately:
$$
\begin{array}{>{\displaystyle}cc>{\displaystyle}l}
\nabla_{\bnu} \widetilde{J}(\s,\bnu)   &=& -  \Big[\nabla_\bnu \bpsi^*(\s,\bnu)\Big]^T \Big[ \nabla_\bpsi \bF(\s,\bnu,\bpsi^*(\s,\bnu)) \Big]^T \p(\s,\bnu) \\
 &=& - \p(\s,\bnu),
\end{array}
$$
which is the desired expression.
\end{proof}

\begin{remark}
When the optimized shape $\Omega$ is only one phase within a large computational domain $D$ 
and is represented as a subcollection of the cells of a ``true'' Laguerre diagram $\bLag$ of $D$, of the form \cref{eq.defLaguerre,eq.defLagi},
similar considerations show that formula \cref{eq.NnuJt} reads instead:
$$ \nabla_\bnu  \widetilde{J}(\s,\bnu) = -\left( \p(\s,\bnu) - \frac{1}{N} \sum\limits_{i=1}^N p_i(\s,\bnu)\right) .$$
\end{remark}

\begin{remark}
The geometric calculations needed in the present section for the expressions of the matrices $\Big[ \nabla_{\s} \bF(\s,\bnu,\bpsi)\Big]$, $\Big[ \nabla_{\bpsi} \bF(\s,\bnu,\bpsi)\Big]$, $\Big[ \nabla_{\s} \bQ(\s,\bnu,\bpsi)\Big]$ and $\Big[ \nabla_{\bpsi} \bQ(\s,\bnu,\bpsi)\Big]$ are provided in \cref{app.verseeds,app.derF}. They need only be implemented once and for all.
In particular, this implementation is independent of the particular objective function. 
\end{remark}
 
\subsection{The Hilbertian method}\label{sec.velext}

\noindent We have described in the previous \cref{sec.derver,sec.vertoseeds} how to calculate the sensitivities $\nabla_\s \widetilde J(\s,\bnu)$ and $\nabla_\bnu \widetilde J(\s,\bnu)$ of a function $J$, 
depending on the domain $\Omega$ via the seeds $\s \in \R^{dN}$ and the cell measures $\bnu \in \R^N$ of a diagram $\bVsp$ of the form \cref{eq.decompOm}.
Unfortunately, a direct use of these derivatives in the solution of our discrete shape and topology optimization problem \cref{eq.sopbdisc} often proves awkward in practice:
on the one hand, they may show very irregular variations in space (i.e. they may have very different components along the indices $i=1,\ldots,N$ of neighboring cells), thus causing numerical instabilities. On the other hand, in multiple applications,
it is desirable to handle sensitivities that take into account certain requirements, for instance that some of their components associated to cells located in non optimizable regions of space should vanish. 
In this section, we explain how the so-called Hilbertian method allows to achieve both purposes with rigorous guarantees that the resulting modified sensitivities retain descent properties. The application in shape optimization of this classical procedure in the context of gradient flows is exposed in e.g. \cite{allaire2020survey,burger2003framework,de2006velocity,mohammadi2010applied}.

\subsubsection{The general strategy}

\noindent
Let us describe the use of the Hilbertian method in the illustrative perspective of regularizing the gradient $\nabla_\s \widetilde{J}(\s,\bnu)$ of a discrete objective function $\widetilde J(\s,\bnu)$ at a given point $(\s,\bnu) \in \R^{dN} \times \R^N$. \par\medskip

To this end, let us recall that $\nabla_\s \widetilde{J}(\s,\bnu)$ is obtained by identifying the derivative $\h \mapsto \frac{\partial \widetilde J}{\partial \s}(\s,\bnu)(\h)$ 
with an element in the space $\R^{dN}$ via the following relation:
$$ \forall \h \in \R^{dN}, \quad  \frac{\partial \widetilde J}{\partial \s}(\s,\bnu)(\h) = \left\langle \nabla_\s \widetilde{J}(\s,\bnu), \h \right\rangle.$$
The main idea of the Hilbertian method is to introduce a subspace $V$ of the space $\R^{dN}$ of seed points, 
equipped with an inner product $a_\s(\cdot,\cdot)$ different from the canonical one $\langle \cdot, \cdot \rangle$,
and to search for the gradient $\h_J \in V$ of $ \widetilde{J}(\cdot,\bnu)$ with respect to this inner product. In other terms, we solve the following variational problem:
\begin{equation}\label{eq.defhJ}
 \text{Search for } \h_J \in V \text{ s.t. for all } \h \in V, \quad a_\s(\h_J,\h) = \frac{\partial \widetilde J}{\partial \s}(\s,\bnu)(\h).
 \end{equation}
Depending on the nature of the space $V$ and its inner product $a_\s(\cdot,\cdot)$, the gradient $\h_J$ may for instance be more regular than its counterpart $\nabla_\s \widetilde{J}(\s,\bnu)$ associated to the canonical inner product, some of its components may vanish, etc. 
 
Our choices of $V$ and $a_\s(\cdot,\cdot)$ in the present context rely on the smoothing effect of the discrete Laplace operator acting on mass--spring networks, 
see for instance \cite{san2011designing} or  \cite{nealen2006physically} for reviews about the role of such discrete Laplace operators in the field of computer graphics.
We take $V = \R^{dN}$ and search for the regularized gradient of $\widetilde{J}(\cdot,\bnu)$ at $\s$ as the solution $\h_J := \left\{ \h_{J,i} \right\}_{i=1,\ldots,N} \in \R^{dN}$ to the following minimization problem: 
\begin{equation}\label{eq.minEHilbert}
 \min\limits_{\h \in \R^{dN}} E(\h) , \text{ where } E(\h) := \frac{\alpha^2}{2} \sum\limits_{(i,j) \in \calE \atop j > i} \lvert \h_i - \h_j \lvert^2  + \frac{1}{2} \sum\limits_{i=1}^N \lvert \h_i \lvert^2 - \sum\limits_{i=1}^N \left\langle \nabla_{\s_i} \widetilde{J}(\s,\bnu) , \h_i \right\rangle,
 \end{equation}
 where $\alpha$ is a user-defined parameter encoding a regularization length scale. Intuitively, the solution $\h_J$ to \cref{eq.minEHilbert} is a version of $\nabla_{\s} \widetilde{J}(\s,\bnu)$ whose spatial variations are averaged over a region with radius $\alpha$.
The strictly convex optimization problem \cref{eq.minEHilbert} has a unique solution, which is characterized by the following first-order optimality condition: 
$$\text{For }\: i=1,\ldots,N,\quad \alpha^2 \sum\limits_{j \in \calN_i} (\h_i - \h_j) + \h_i   =  \nabla_{\s_i} \widetilde{J}(\s,\bnu).$$
This rewrites as a $2N \times 2N$ matrix system:
$$ (\alpha^2 A_\s + \I) \h  = \nabla_\s \widetilde{J}(\s,\bnu),$$
where the matrix $A_\s \in \R^{2N \times 2N}$ is given by the entries:
\begin{multline}\label{eq.As} 
A_{\s,2(i-1)+1,2(j-1)+1} = A_{\s,2(i-1)+2,2(j-1)+2}  = \left\{ 
\begin{array}{cl}
\# \calN_i & \text{if } j=i, \\
-1 & \text{if } j \in \calN_i, \\
0 & \text{otherwise},
\end{array}
\right.
\text{ and }  \\
A_{\s,2(i-1)+1,2(j-1)+2} = A_{\s,2(i-1)+2,2(j-1)+1} = 0, \quad i,j=1,\ldots,N.
\end{multline}
The inner product $a_\s(\cdot,\cdot)$ on $V= \R^{dN}$ associated to this strategy is:
\begin{equation}\label{eq.defas}
\forall \:\h^1,\h^2 \in V, \quad a_\s(\h^1,\h^2) = \Big\langle (\alpha^2 A_\s + \I) \h^1, \h^2 \Big\rangle.
\end{equation}

\begin{remark}
A similar strategy can be applied to impose smooth spatial variations of the gradient $\nabla_{\bnu} \widetilde J(\s,\bnu)$ of $\widetilde J$ with respect to the measures of cells. One then relies on the choice $V=\R^N$, equipped with the inner product:
\begin{equation}\label{eq.defanu}
a_{\bnu}(\bnu^1,\bnu^2) = \Big\langle A_\nu \bnu^1, \bnu^2 \Big\rangle , \text{ where } A_{\bnu,ij} =  \left\{ 
\begin{array}{cl}
\# \calN_i & \text{if } j=i, \\
-1 & \text{if } j \in \calN_i, \\
0 & \text{otherwise},
\end{array}
\right.
\end{equation}
\end{remark}

\subsubsection{Imposing that seeds stay inside the computational domain}

\noindent The Hilbertian method naturally allows to enforce certain constraints on the regularized sensitivity $\h_J$ of $J$ with respect to the seed points $\s$ in \cref{eq.defhJ}. For instance $\h_J$ can be imposed to have null components over the indices $i$ inside a given subset $I \subset \left\{ 1,\ldots,N \right\}$ of the seed points by simply taking the following Hilbert space $V$ in \cref{eq.defhJ}:
$$ V := \left\{ \h = (\h_1,\ldots,h_N) \in \R^{dN} \text{ s.t. } \h_i = \bz \text{ for } i \in I \right\}.$$
In this section, we exemplify how the seed points $\s$ of the diagram $\bVsp$ featured in \cref{algo.lagevol} can be constrained  to stay inside the computational domain $D$ during the optimization process thanks to this practice, see also \cref{rem.outcells} about this point.\par\medskip

Let $(\s,\bnu) \in \R^{dN} \times \R^N$ be a given stage of the process of \cref{algo.lagevol}, and let $\bpsi = \bpsi^*(\s,\bnu)$. 
To simplify the presentation, we assume that the seed points of the diagram $\bVsp$ are numbered in such a way that the first $B \leq N$ ones are those adjacent to the external boundary $\partial D$.
For notational simplicity and without loss of generality, we also assume that for each $i=1,\ldots,B$ at most one edge of each cell $V_i(\s,\bpsi)$ belongs to the external boundary $\partial D$, and we denote by $\n_i$ the unit normal vector to this edge, pointing outward $D$. 
We then search for a gradient $\h$ of $\widetilde J$ at $(\s,\bnu)$ with the property that, for each index $i=1,\ldots,B$, the normal component $\h_i \cdot \n_i$ of the displacement $\h_i$ should vanish:
\begin{equation}\label{eq.hdotnB}
 \forall i =1,\ldots,B, \quad \h_i \cdot \n_i = 0. 
 \end{equation}
One possibility to enforce this requirement is to proceed by exact penalization, i.e. to take $V = \R^{dN}$ and replace the inner product \cref{eq.defas} with: 
\begin{equation*}\label{eq.defasTGV}
a_\s(\h^1,\h^2) = \Big\langle (\alpha^2 A_\s + \I) \h^1, \h^2 \Big\rangle + \frac{1}{\e} \sum\limits_{i=1}^B(\h^1_i \cdot \n_i)(\h^2_i \cdot \n_i),
\end{equation*}
where $A_\s \in \R^{2N \times 2N}$ is the matrix defined in \cref{eq.As} and $\e$ is a ``sufficiently small'' parameter. 

Another strategy consists in setting:
$$V = \left\{\h \in \R^{dN}, \:\: \h_i \cdot \n_i = 0 \text{ for } i=1,\ldots,B \right\}, \text{ with the inner product \cref{eq.defas}}. $$ 
The identification problem \cref{eq.defhJ} for the gradient $\h_J$ amounts to solving the optimization problem \cref{eq.minEHilbert} under the additional constraint \cref{eq.hdotnB}. 
This is again a convex optimization problem, whose first-order optimality conditions read: 
$$\exists \blambda \in \R^B \text{ s.t. } \left\{
\begin{array}{>{\displaystyle}ccc>{\displaystyle}l}
\alpha^2 \sum\limits_{j \in \calN_i} (\h_i - \h_j) + \h_i   + \lambda_i \n_i &=&  \nabla_{\s_i} \widetilde{J}(\s) & \text{for } i=1,\ldots,B, \\
\alpha^2 \sum\limits_{j  \in \calN_i} (\h_i - \h_j) + \h_i   &= &  \nabla_{\s_i} \widetilde{J}(\s) & \text{for } i=B+1,\ldots,N, \\
\h_i \cdot \n_i &=&  0 & \text{for } i =1,\ldots,B. 
\end{array}
\right. $$
The search for the pair $(\h ,\blambda)\in \R^{2N} \times \R^B$ rewrites as the following $(2N + B) \times (2N+B)$ matrix system:
$$ \left( 
\begin{array}{c|c}
\:\: \: \alpha^2 A_\s + \I \:\:\: & \:\:\: L \:\:\:   \\[0.4em]
\hline \\[-0.6em]
L^T & \bigzero
\end{array}
\right) \left( 
\begin{array}{c}
\h \\[0.4em]
\hline \\[-0.6em]
\blambda
\end{array}
\right)
 = \left( 
\begin{array}{c}
\nabla_\s \widetilde{J}(\s) \\[0.4em]
\hline \\[-0.6em]
\bigzero
\end{array}
\right),
$$
where the matrix $L \in \R^{2N \times B}$ is defined by:
\begin{multline*}
L_{2(i-1)+1,j} = 
\left\{ 
\begin{array}{cl}
n_{i,1} & \text{if } j = i, \\
0 & \text{otherwise},
\end{array} 
\right.  \text{ and } \\
L_{2(i-1)+2,i} = 
\left\{ 
\begin{array}{cl}
n_{i,2} & \text{if } j = i, \\
0 & \text{otherwise},
\end{array} 
\right. 
n_{i,2} \text{ for } i=1,\ldots,N, \:\: j=1,\ldots,B.
\end{multline*}
%

\subsection{Exploitation of these derivatives with a constrained optimization algorithm}\label{sec.nullspace}

\noindent In this section, we describe how the derivatives of the objective and constraint functions $J(\Omega)$, $\G(\Omega)$ with respect to the seeds $\s$ and measures $\bnu$ 
calculated in the previous \cref{sec.derver,sec.vertoseeds,sec.velext} are exploited to solve our discrete constrained optimization problem:
\begin{equation}\label{eq.sopbdiscs64} 
\tag{\textcolor{gray}{P-disc}}
\min \limits_{(\s, \bnu)} \widetilde J(\s,\bnu) \:\: \text{ s.t. } \widetilde{\G}(\s,\bnu) = 0.
\end{equation}
This task relies on a constrained optimization algorithm developed in one of our previous works \cite{feppon2020null} which is sketched briefly in \cref{sec.NSdesc}.
We then describe its application to our particular context in \cref{sec.NSapp}.  

\subsubsection{The null-space optimization algorithm in a nutshell}\label{sec.NSdesc}

\noindent Let $V$ be a Hilbert space, equipped with the inner product $a(\cdot,\cdot)$. 
We consider an optimization problem of the form: 
\begin{equation}\label{eq.NSoptpb}
\min\limits_{h \in V} \: J(h) \text{ s.t. } \G(h) = 0,
\end{equation}
where $J : V \to \R$ is a smooth objective function and $\G : V \to \R^p$ accounts for a collection of $p$ smooth equality constraints.
We denote by $\theta_J$, $\theta_{G_i} \in V$, $i=1,\ldots,p$ the gradients of $J$ and $G_i$ at a particular given value $h$ of the optimized variable, that is: 
$$ \forall \zeta \in V, \quad J^\prime(h)(\zeta) = a(\theta_J, \zeta), \text{ and } G_i^\prime(h)(\zeta) = a(\theta_{G_i}, \zeta). $$

Let $h \in V$ stand for an arbitrary stage of the solution process for \cref{eq.NSoptpb}. We search for an update of $h$ of the form $h + \Delta t \xi$, where $\Delta t >0$ is a suitable descent step, and the direction $\xi \in V$ is of the form:
\begin{equation}\label{eq.defxiNS}
\xi = - (\alpha_J \xi_J + \alpha_G \xi_G).
\end{equation}
The two contributions $\xi_J$ and $\xi_G \in V$ are orthogonal, and they are characterized by the following requirements: 
\begin{itemize}
\item The null-space step $\xi_J$ is the best descent direction for $J(h)$ which leaves the values of the constraint functional $\G(h)$ unaltered at first order. Formally, $\xi_J$ is the projection of the gradient $\theta_J$ of $J(h)$ onto the null space $\Ker(\G^\prime(h)) = \span\left\{ \theta_{G_1}, \ldots,\theta_{G_p} \right\}^\perp$ of the constraints. It is of the form 
\begin{equation}\label{eq.defxiJ}
 \xi_J = \theta_J + \sum\limits_{j=1}^p \lambda_j \theta_{G_j} \text{ for some } \blambda \in \R^p.
 \end{equation}
The requirement that $a(\xi_J, \theta_{G_j}) = 0$ for $j=1,\ldots,p$ shows that the vector $\blambda \in \R^p$ is the solution to the following $p \times p$ system:
\begin{multline}\label{eq.defSNS}
 S \blambda = \bm{b}, \text{ where } S \in \R^{p\times p} \text{ and } \b \in \R^p \text{ are given by the entries:} \\ 
 S_{ij} = a(\theta_{G_i}, \theta_{G_j}), \text{ and } b_i = -a(\theta_J, \theta_{G_i}).
 \end{multline}
\item The range space step $\xi_G$ aims to decrease the violation of constraints. It belongs to $\Ker(\G^\prime(h))^\perp = \span \left\{ \theta_{G_i} \right\}_{i=1,\ldots,p}$, and so it can be expressed as
\begin{equation}\label{eq.defxiG}
 \xi_G = \sum\limits_{j=1}^p \beta_j \theta_{G_j}, \text{ for some coefficient vector } \bbeta \in \R^p.
 \end{equation}
The components $\beta_i$, $i=1,\ldots,p$ are determined by imposing the following decrease in the value of $G$ in the course of one iteration:
$$ \G(h + \Delta t \xi) = (1-\Delta t \alpha_G) \G(h),$$ 
which we approximate by
$$ \G^\prime(h)(\xi_G) \approx \G(h).$$
This requirement is discretized as:
\begin{equation}\label{eq.defc}
 S \bbeta = \bm{c}, \text{ where } S \in \R^{p\times p} \text{ is the matrix \cref{eq.defSNS} and } \\
  \bm{c} \in \R^p \text{ has entries } c_i = G_i(h), \:\: i=1,\ldots,p.
  \end{equation}
\end{itemize}

As far as the choice of the descent step $\Delta t$ is concerned, we rely on a merit function: the update $h + \Delta t \xi$ of $h$ should result in a decrease in the value of the function
$$ M(h) := \alpha_{J} \Big( J(h) - \blambda \cdot \G(h) \Big) + \frac{\alpha_{G}}{2} S^{-1}\G(h) \cdot \G(h), $$
whose negative gradient is $\xi$ by construction.
\subsubsection{Application to the context of interest}\label{sec.NSapp}

\noindent Let us now apply the general optimization strategy of \cref{sec.NSdesc} to our discrete shape and topology optimization problem \cref{eq.sopbdiscs64}.
Since the optimized variables $\s \in \R^{dN}$ and $\bnu \in \R^N$ are of very independent natures, we treat optimization with respect to each of them separately, via an alternating descent strategy: 
\begin{itemize}
\item When it comes to updating the seed points $\s$, we apply the Hilbertian method of \cref{sec.NSdesc} with the space $V = \R^{dN}$ and the inner product $a_\s (\cdot,\cdot)$ defined in \cref{eq.defas}.
The gradients $\btheta_{J,\s}$, $\btheta_{G_i,\s}$, $i=1,\ldots,p$, are calculated by solving the identification problem \cref{eq.defhJ} from the derivatives of $\s \mapsto \widetilde J(\s,\bnu)$ and $\s \mapsto \widetilde{\G}(\s,\bnu)$ supplied by \cref{prop.derJqs}. 
The coefficients $\blambda_\s \in \R^p$ and $\bbeta_\s \in \R^p$  involved in the expressions \cref{eq.defxiJ,eq.defxiG} of the steps $\bxi_{J,\s}$ and $\bxi_{G,\s}$ are then easily calculated from these data, see \cref{eq.defSNS,eq.defc},
and the resulting descent direction $\bxi_\s \in \R^{dN}$ for the seed points $\s$ reads:
\begin{equation*}
 \bxi_\s  = -\alpha_{J,\s} \bxi_{J,\s} - \alpha_{G,\s} \bxi_{G,\s}, \text{ where }
 \bxi_{J,\s} := \btheta_{J,\s} - \sum\limits_{i=1}^p\lambda_{\s,i} \btheta_{G_i,\s} \text{ and } \bxi_{G,\s} = \sum\limits_{i=1}^p \beta_{\s,i} \btheta_{G_i,\s},
 \end{equation*}
 with the weights $\alpha_{J,\s}$ and $\alpha_{G,\s}$ for the relative contributions of the null space and range space steps to the descent direction.
\item When it comes to updating the vector $\bnu$, we use the space $V = \R^{N}$ and the inner product $a_\bnu (\cdot,\cdot)$ in \cref{eq.defanu}. Similar considerations yield the following descent direction $\bxi_\bnu \in \R^{N}$ for the measures $\bnu$:
\begin{equation*}
 \bxi_\bnu  = -\alpha_{J,\bnu} \bxi_{J,\bnu} - \alpha_{G,\bnu} \bxi_{G,\bnu}, \text{ where }
 \bxi_{J,\bnu} := \btheta_{J,\bnu} - \sum\limits_{i=1}^p\lambda_{\bnu,i} \btheta_{G_i,\bnu} \text{ and } \bxi_{G,\bnu} = \sum\limits_{i=1}^p \beta_{\bnu,i} \btheta_{G_i,\bnu},
 \end{equation*}
where again $\blambda_{\bnu},\bbeta_{\bnu} \in \R^p$ are calculated via \cref{eq.defSNS,eq.defc}, and $\alpha_{J,\bnu}$ and $\alpha_{G,\bnu}$ are suitable weights.
\end{itemize}

Moreover, as suggested in \cite{feppon2020null}, instead of considering fixed weights $\alpha_{J\s}$, $\alpha_{G,\s}$ for the contributions of the null space and range space steps, we consistently select these coefficients as:
$$ \alpha_{J,\s} = \frac{ A_{J,\s} \texttt{h}_{\texttt{av}}}{| \bxi_{J,\s} |_\infty} \text{ and }  \alpha_{G,\s} = \frac{A_{G,\s} \texttt{h}_{\texttt{av}}}{| \bxi_{G,\s} |_\infty}$$
where $A_{J,\s}$, and $A_{G,\s} \in [0,1]$ are fixed parameters controlling the decrease rates of the objective function and the violation of the constraint, and
$\texttt{h}_{\texttt{av}}$ is the typical distance between two neighboring seed points. 
This allows to bound the maximum displacement of a seed induced by the update $\bxi_\s$ in this step by $(A_{J,\s} + A_{G,\s}) \texttt{h}_{\texttt{av}}$. 
Likewise, we take
$$ \alpha_{J,\bnu} = \frac{A_{J,\bnu} \texttt{V}_{\texttt{av}}}{| \bxi_{J,\bnu} |_\infty}, \text{ and }  \alpha_{G,\bnu} = \frac{A_{G,\bnu} \texttt{V}_{\texttt{av}}}{| \bxi_{V,\bnu} |_\infty},$$
where $A_{J,\bnu}$, $A_{G,\bnu} \in [0,1]$ and $\texttt{V}_{\texttt{av}}$ is the average volume of cells. Doing so ensures that the maximum change in a local volume fraction induced in the update of $\bnu$ by $\xi_\bnu$ is $(A_{J,\bnu} + A_{G,\bnu}) \texttt{V}_{\texttt{av}}$.

\begin{remark}\label{rem.volcst}
In the optimization problems featured in the examples of the next \cref{sec.num}, only $p=1$ equality constraint is imposed, which is related to the volume of the shape $\Omega$ as: $G(\Omega) = \Vol(\Omega) - V_T$.
In such a situation, the foregoing formulas conveniently simplify into: 
$$ \btheta_{G,\s} = 0, \text{ and } \btheta_{G,\bnu} = \left(
\begin{array}{c}
1\\
\vdots\\
1
\end{array}
\right) \in \R^N,$$
so that the optimization with respect to the seed points $\s$ is unconstrained, 
and the coefficients featured in the above scheme \cref{eq.defxiJ,eq.defxiG,eq.defc} read
$$S = a_\bnu( \btheta_{G,\bnu},\btheta_{G,\bnu} ) = N, \quad \lambda = \frac{1}{N} a_{\bnu}( \btheta_{J,\bnu},\btheta_{G,\bnu} ) \text{ and } \beta = \frac{1}{N}  \Big(\widetilde{\Vol}(\s,\bnu) - V_T \Big).$$
Note that if the volume constraint is satisfied at the current stage $(\s,\bnu)$ of the optimization process, then $\xi_G =0$. In particular, if the vector $\bnu$ of cell measures satisfies $\sum_{i=1}^N \nu_i = V_T$, the update \cref{eq.defxiNS} of the null-space optimization algorithm for \cref{eq.NSoptpb} coincides (up to a positive coefficient) with that of the unconstrained gradient algorithm for the minimization of $\widetilde{J}(\cdot,\bnu)$.
\end{remark}

\section{Numerical examples}\label{sec.num}

\noindent In this section, we present various 2d numerical examples to illustrate the main features of our shape optimization \cref{algo.lagevol}. 
After starting in \cref{sec.nummcf} with a ``simple'' geometric optimization problem, we revisit in \cref{sec.exev} the 
classical subject of optimization of the eigenvalues of the Laplace operator with our new framework. 
We then turn to more concrete physical applications in the contexts of thermal devices in \cref{sec.numconduc2phase} and of mechanical structures in \cref{sec.numcanti,sec.exbr,sec.stbr}. 

\subsection{Optimization of the perimeter of shapes}\label{sec.nummcf}

\noindent Our first example deals with a well-known shape optimization problem, featuring only geometric functionals of the domain, i.e. no physical boundary value problem is involved. 
We minimize the perimeter of the shape $\Omega \subset \R^2$ under an equality constraint on its volume: 
\begin{equation}\label{eq.pbmcf}
\min\limits_{\Omega} \: \Per(\Omega) \text{ s.t. } \Vol(\Omega) = V_T,
\end{equation}
where $\Per(\Omega)$ and $\Vol(\Omega)$ are defined by \cref{eq.volper} and $V_T$ is a volume target.

Among the numerous motivations for \cref{eq.pbmcf}, let us mention that the flow induced by its solution from an initial shape $\Omega^0$ is the well-known mean curvature flow, which proceeds by steadily blurring its sharp features (i.e. its bumps and creases with high mean curvature) while preserving its total volume, see e.g. \cite{mantegazza2011lecture} about the latter.

From the mathematical viewpoint, such problems have been extensively considered in the literature, see e.g. \cite{osserman1978isoperimetric,polya1951isoperimetric}. 
In particular, according to the famous isoperimetric inequality, the unique global minimizers to \cref{eq.pbmcf} are 2d disks with radius $\sqrt{V_T/\pi}$.

\begin{remark}
In this example, the discrete versions $\Vol(\q)$ and $\Per(\q)$ of the volume and perimeter $\Vol(\Omega)$ and $\Per(\Omega)$ of the shape $\Omega$ as well as their derivatives can be explicitly calculated from the polygonal mesh $\calT$ induced by the diagram $\bVsp$ representing $\Omega$. Indeed, simple calculations show that:
$$ \Vol(\q) = \sum\limits_{E \in \calT} \lvert E \lvert,$$
where the measure of an element $E$ with vertices $\q_1^E,\ldots,\q_n^E$ can be calculated by the following formula, where we use the notations of \cref{sec.notVEM}: 
$$ \lvert E \lvert = \frac{1}{2} \sum\limits_{j=1}^n \lvert \hat \be_j \lvert \n_{\hat \be_j} \cdot \q_j^E.$$
Taking derivatives, then summing over the elements $E \in \calT$, an elementary calculation shows that the gradient of the volume functional equals:
$$
 \nabla_{\q_j} \Vol(\q) = \left\{
\begin{array}{cl}
\frac{1}{2} \lvert \hat \be_j \lvert \n_{\hat \be_j}& \text{if } \q_j \text{ is a vertex on the boundary } \partial\Omega \\
\bzero & \text{otherwise}.
\end{array}
\right.
$$ 
Likewise,
assuming for notational simplicity that the vertices $\q_j$ are enumerated in such a way that the border $\partial \Omega$ is discretized by the $n_b < M$ first vertices $\q_1,\ldots,\q_{n_b} = \q_0$ of the collection $\q$ and the associated line segments $\q_{j-1}\q_{j}$, $j=1,\ldots,n_b$, 
one has
$$ \Per(\q) = \sum\limits_{j=1}^{n_b} |\q_j \q_{j+1}|, $$
and so
$$ \nabla_{\q_j} \Per(\q) = \left\{
\begin{array}{cl}
\frac{\q_j - \q_{j-1}}{|\q_j - \q_{j-1}|} + \frac{\q_j - \q_{j+1}}{|\q_j - \q_{j+1}|}  & \text{if } j=1,\ldots,n_b, \\
\bzero & \text{otherwise}.
\end{array}
\right.$$
\end{remark}

We apply the ``free boundary'' version of our shape and topology optimization \cref{algo.lagevol} to the solution of this problem. The initial shape $\Omega^0$ is represented by a diagram $\mathbf{V}(\s^0,\bpsi^0)$ made of a random collection of $N=500$ cells with equal measures $\overline{\nu} := V_T/N$. Throughout the optimization process, we impose that the cells $V_i(\s^n,\bpsi^n)$, $i=1,\ldots,N$, satisfy this measure constraint, so that our optimization procedure becomes constraint free, as noted in \cref{rem.volcst}.
A few snapshots of the solution process are depicted on \cref{fig.MC}. As expected, the optimized shape is close to a disk. Remarkably, the shape undergoes drastic topological changes during the process: our algorithm is able to detect that the only means to decrease the total perimeter of the shape $\Omega$ by  moving the defining seed points $\s$ without altering the measures $\bnu$ of the cells is to aggregate the cells. This operation -- and the underlying shape sensitivity of the objective functional -- is completely different from that of a more classical boundary variation algorithm as in e.g. \cite{allaire2004structural,allaire2014shape,pironneau1982optimal}, whereby each individual cell would shrink, regardless of the others, without capturing this collective behavior. We shall retrieve this trend in the mechanical example of \cref{sec.numcanti}.
 
\begin{figure}[!ht]
\centering
\begin{tabular}{ccc}
\begin{minipage}{0.3\textwidth}
\begin{overpic}[width=1.0\textwidth]{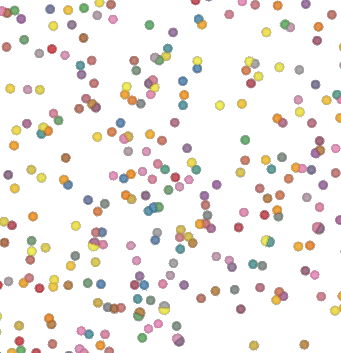}
\put(0,3){\fcolorbox{black}{white}{$n=0$}}
\end{overpic}
\end{minipage}
 & 
 \begin{minipage}{0.3\textwidth}
\begin{overpic}[width=1.0\textwidth]{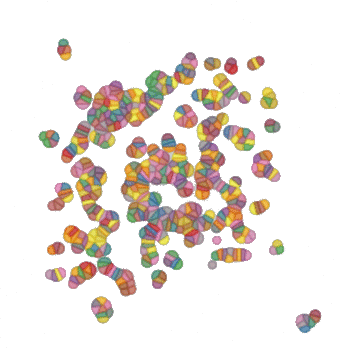}
\put(0,3){\fcolorbox{black}{white}{$n=10$}}
\end{overpic}
\end{minipage} 
&
\begin{minipage}{0.3\textwidth}
\begin{overpic}[width=1.0\textwidth]{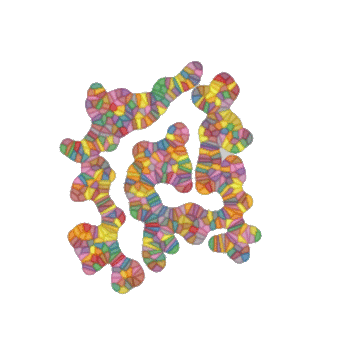}
\put(0,3){\fcolorbox{black}{white}{$n=25$}}
\end{overpic}
\end{minipage}
 \\
 \\
 \begin{minipage}{0.3\textwidth}
\begin{overpic}[width=1.0\textwidth]{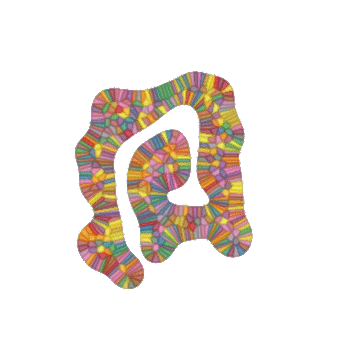}
\put(0,3){\fcolorbox{black}{white}{$n=100$}}
\end{overpic}
\end{minipage}
 & 
 \begin{minipage}{0.3\textwidth}
\begin{overpic}[width=1.0\textwidth]{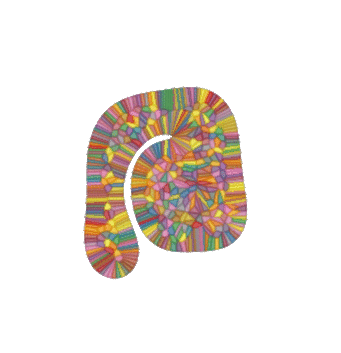}
\put(0,3){\fcolorbox{black}{white}{$n=220$}}
\end{overpic}
\end{minipage}
&
 \begin{minipage}{0.3\textwidth}
\begin{overpic}[width=1.0\textwidth]{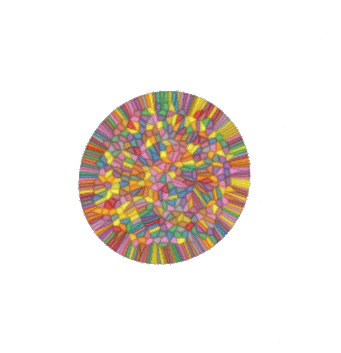}
\put(0,3){\fcolorbox{black}{white}{$n=260$}}
\end{overpic}
\end{minipage}

\end{tabular}
\caption{\it A few intermediate shapes $\Omega^n$ produced in the course of the solution of the perimeter minimization problem \cref{eq.pbmcf} in \cref{sec.nummcf}. The cells $V_i(\s^n,\bpsi^n)$ of the diagram, $i=1,\ldots,N$, are represented with different colors to help vizualization.}
\label{fig.MC}
\end{figure}

\subsection{Optimization of the eigenvalues of the Dirichlet Laplace operator}\label{sec.exev}

\noindent Our first example involving boundary value problems concerns the minimization of some of the eigenvalues of the Laplace operator equipped with Dirichlet boundary conditions. We consider the problem:
\begin{equation}\label{eq.sopbevnum}
\min \limits_\Omega \:\lambda_\Omega^{(k)} \text{ s.t. } \Vol (\Omega) = V_T,
\end{equation}
where we recall from \cref{sec.conduc} that $\lambda_\Omega^{(k)}$ is the $k^{\text{th}}$ lowest real number for which there exists a non trivial function $u \in H^1_0(\Omega)$ such that
$$ -\Delta u = \lambda_\Omega^{(k)} u, \text{ with the normalization } \int_\Omega u_\Omega^2 \:\d x = 1.$$
These eigenvalues arise in multiple applications, such as acoustics (where they are related to the properties of sound propagation), thermal conduction (where they encode the dissipation rate of a peak of temperature over time), and structure mechanics (where they represent the vibration modes of a thin membrane), to name a few. Spectral shape optimization problems, of the form \cref{eq.sopbevnum}, have been extensively studied in the literature, from both theoretical \cite{belhachmi2006shape,grebenkov2013geometrical,henrot2006extremum} and numerical \cite{antunes2012numerical,oudet2004numerical} viewpoints.

We address the numerical solution of the problem \cref{eq.sopbevnum} thanks to the ``free boundary'' strategy described in \cref{sec.Lagalgo}. The shape $\Omega$ is described by a diagram $\bVsp$ attached to collections $\s$ and $\bpsi$ of $N$ seed points and weights, respectively. As in the previous \cref{sec.nummcf}, the measures of the cells are consistently set to the common value $\overline{\nu} := V_T / N$, so that the solution of \cref{eq.sopbevnum} boils down to the unconstrained minimization of $\lambda_\Omega^{(k)}$ with respect to the positions $\s$ of the seed points, see again \cref{rem.volcst}.
On a different note, the sensitivity of the discrete version of $\lambda_\Omega^{(k)}$ with respect to the vertices $\q$ of $\bVsp$ is calculated by automatic differentiation before being interpreted in terms of the positions $\s$ of the seed points, see \cref{sec.derver,sec.vertoseeds}.\par\medskip


In this context, we conduct two numerical experiments.
At first, we address the minimization of the first Dirichlet eigenvalue, i.e. $k=1$ in \cref{eq.sopbevnum}. It is well-known since the seminal conjecture of Lord Rayleigh \cite{rayleigh1896theory} that disks with radius $\sqrt{V_T/\pi}$ are the unique minimizers of \cref{eq.sopbevnum}, see \cite{bucur2015free} for a rigorous proof of the associated Faber-Krahn inequality. 
The initial shape $\Omega^0$ is depicted in \cref{fig.eig1}. It is discretized with $N = 1000$ cells with equal measures $\overline{\nu} = V_T/N$. We apply $200$ iterations of our optimization \cref{algo.lagevol} to solve this problem; a few intermediate shapes are depicted on \cref{fig.eig1,fig.ueig1}, and the associated convergence history is reported in \cref{fig.histeig} (a); as expected, the optimized shape is a disk. 
Note that disconnected structures naturally appear during the optimization process. While they are harmless to the resolution, we explicitly enforce the connectedness of the shape at each iteration by removing and randomly resampling the cells that are disconnected from the main region.

\begin{figure}[!ht]
\centering
\begin{tabular}{ccc}
\begin{minipage}{0.3\textwidth}
\begin{overpic}[width=1.0\textwidth]{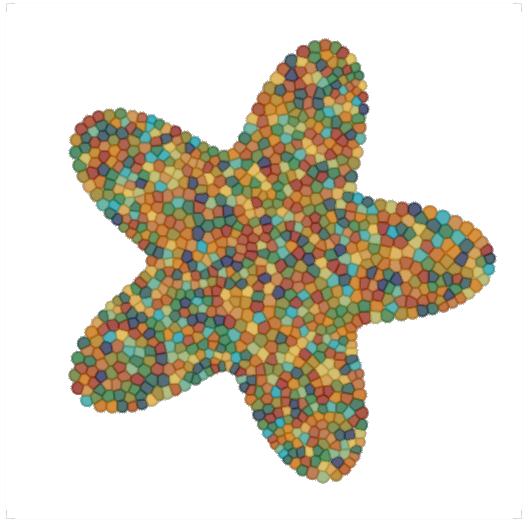}
\put(0,3){\fcolorbox{black}{white}{$n=0$}}
\end{overpic}
\end{minipage}
 & 
 \begin{minipage}{0.3\textwidth}
\begin{overpic}[width=1.0\textwidth]{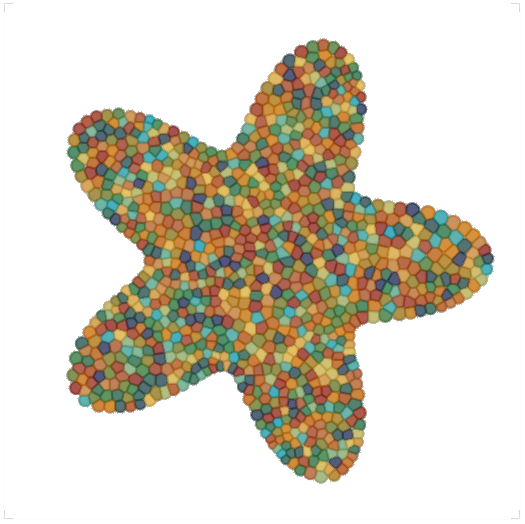}
\put(0,3){\fcolorbox{black}{white}{$n=20$}}
\end{overpic}
\end{minipage} 
&
\begin{minipage}{0.3\textwidth}
\begin{overpic}[width=1.0\textwidth]{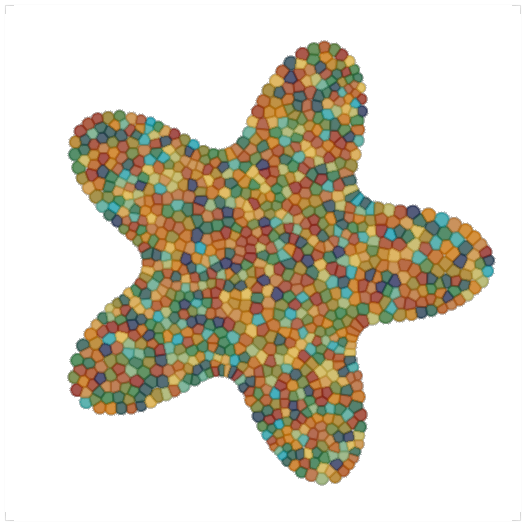}
\put(0,3){\fcolorbox{black}{white}{$n=50$}}
\end{overpic}
\end{minipage}
 \\ \\
 \begin{minipage}{0.3\textwidth}
\begin{overpic}[width=1.0\textwidth]{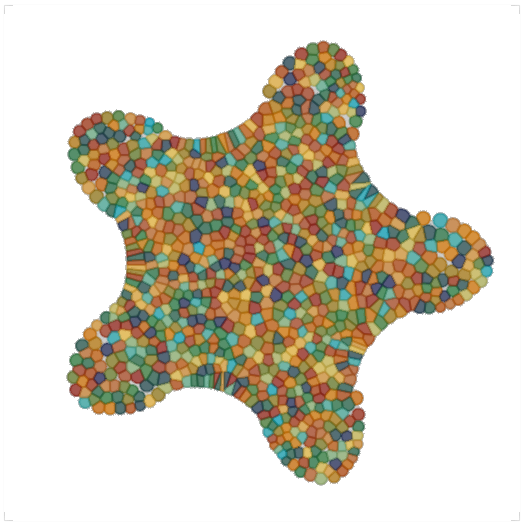}
\put(0,3){\fcolorbox{black}{white}{$n=100$}}
\end{overpic}
\end{minipage}
 & 
 \begin{minipage}{0.3\textwidth}
\begin{overpic}[width=1.0\textwidth]{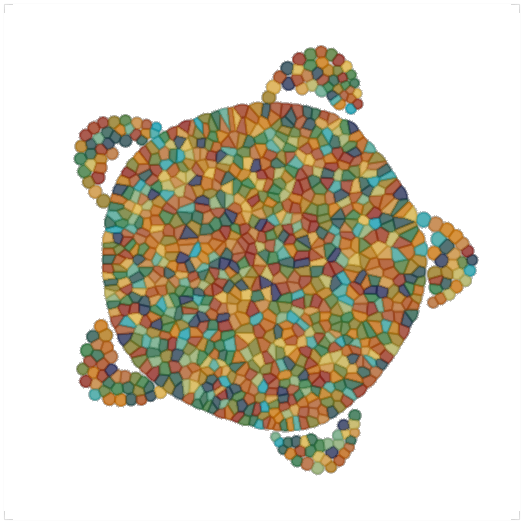}
\put(0,3){\fcolorbox{black}{white}{$n=150$}}
\end{overpic}
\end{minipage}
&
 \begin{minipage}{0.3\textwidth}
\begin{overpic}[width=1.0\textwidth]{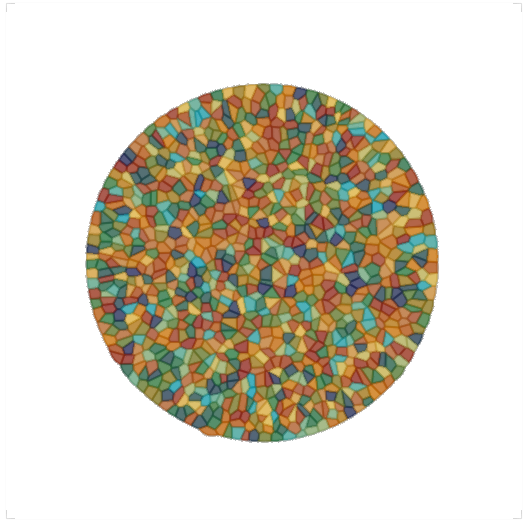}
\put(0,3){\fcolorbox{black}{white}{$n=200$ (final)}}
\end{overpic}
\end{minipage}

\end{tabular}
\caption{\it Intermediate shapes $\Omega^n$ produced in the course of the minimization of the first eigenvalue of the Dirichlet Laplace operator in \cref{sec.exev}.}
\label{fig.eig1}
\end{figure}

\begin{figure}[!ht]
\centering
\begin{tabular}{cccc}
\begin{minipage}{0.3\textwidth}
\begin{overpic}[width=1.0\textwidth]{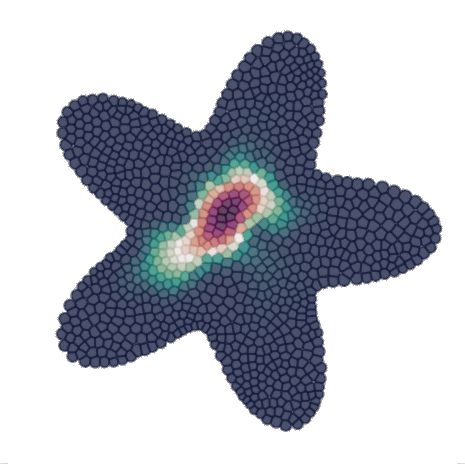}
\put(0,10){\fcolorbox{black}{white}{$n=0$}}
\end{overpic}
\end{minipage}
 & 
 \begin{minipage}{0.3\textwidth}
\begin{overpic}[width=1.0\textwidth]{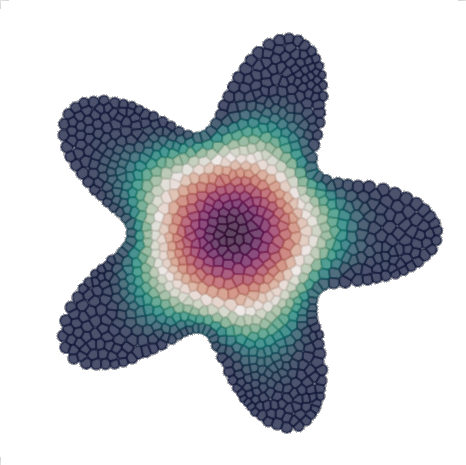}
\put(0,10){\fcolorbox{black}{white}{$n=20$}}
\end{overpic}
\end{minipage} 
&
\begin{minipage}{0.3\textwidth}
\begin{overpic}[width=1.0\textwidth]{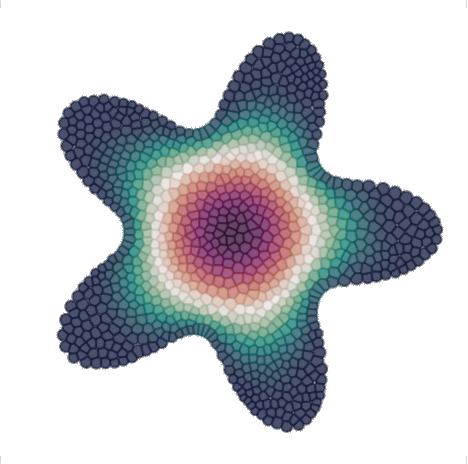}
\put(0,10){\fcolorbox{black}{white}{$n=50$}}
\end{overpic}
\end{minipage}
 &
\begin{minipage}{0.03\textwidth}
 \includegraphics[width=\textwidth]{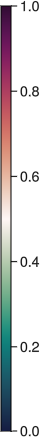}
\end{minipage}
 \\ & & & \\
 \begin{minipage}{0.3\textwidth}
\begin{overpic}[width=1.0\textwidth]{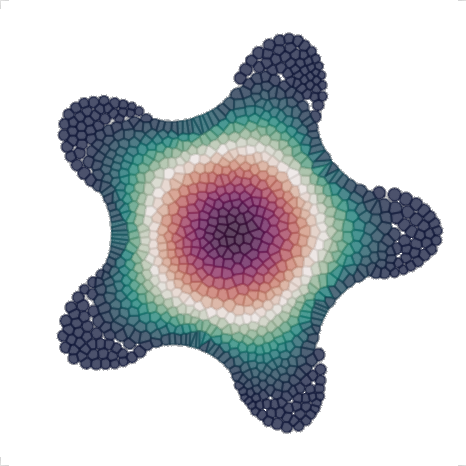}
\put(0,10){\fcolorbox{black}{white}{$n=100$}}
\end{overpic}
\end{minipage}
 & 
 \begin{minipage}{0.3\textwidth}
\begin{overpic}[width=1.0\textwidth]{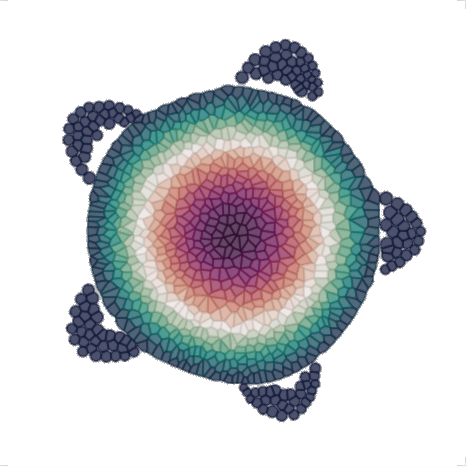}
\put(0,10){\fcolorbox{black}{white}{$n=150$}}
\end{overpic}
\end{minipage}
&
 \begin{minipage}{0.3\textwidth}
\begin{overpic}[width=1.0\textwidth]{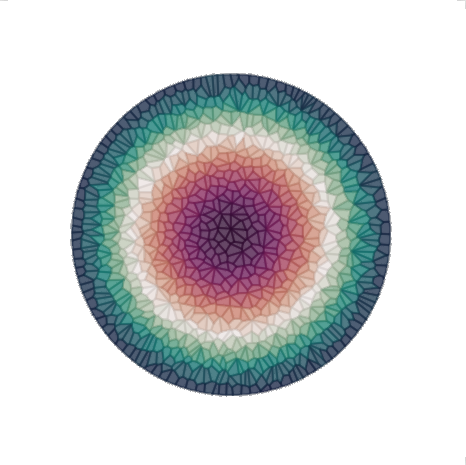}
\put(0,10){\fcolorbox{black}{white}{$n=200$}}
\end{overpic}
\end{minipage} 
\end{tabular}
\caption{\it Graph of the first eigenfunction of the Dirichlet Laplace operator at various iterations of the minimization process of the first Dirichlet eigenvalue in \cref{sec.exev}.}
\label{fig.ueig1}
\end{figure}

We next turn to a more challenging problem of the form \cref{eq.sopbevnum}, featuring the $5^{\text{th}}$ Dirichlet eigenvalue, i.e. $k=5$ in \cref{eq.sopbevnum}. The fifth eigenmode is indeed suspected to be the lowest order eigenvalue that is minimized by a shape which is neither a disk, nor a reunion of disks. Moreover, it is expected that, at the optimum, this eigenvalue is multiple, leading to a non differentiable behavior of this function of the domain. 
Starting from an initial shape $\Omega^0$ made of $N=1000$ cells with equal measures, we apply $1000$ iterations of our shape and topology optimization \cref{algo.lagevol} to the solution of this problem. A few intermediate shapes $\Omega^n$ arising in the course of the process are depicted on \cref{fig.eig5}, and the associated convergence history is reported in \cref{fig.histeig} (b). The resulting optimized shape is qualitatively very similar to the candidate for minimizer evidenced in \cite{oudet2004numerical}. 

\begin{figure}[!ht]
\centering
\begin{tabular}{cccc}
\begin{minipage}{0.3\textwidth}
\begin{overpic}[width=1.0\textwidth]{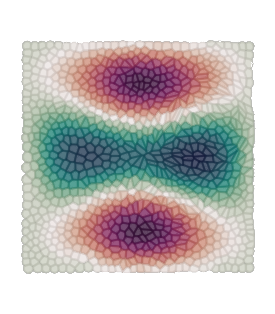}
\put(0,10){\fcolorbox{black}{white}{$n=0$}}
\end{overpic}
\end{minipage}
 & 
 \begin{minipage}{0.3\textwidth}
\begin{overpic}[width=1.0\textwidth]{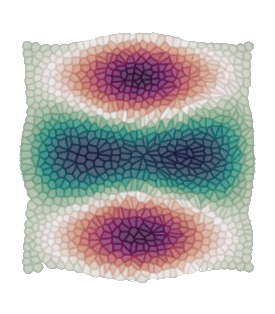}
\put(0,10){\fcolorbox{black}{white}{$n=20$}}
\end{overpic}
\end{minipage} 
&
\begin{minipage}{0.3\textwidth}
\begin{overpic}[width=1.0\textwidth]{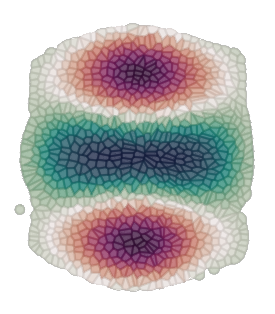}
\put(0,10){\fcolorbox{black}{white}{$n=50$}}
\end{overpic}
\end{minipage}
&
\begin{minipage}{0.03\textwidth}
 \includegraphics[width=\textwidth]{figures/WBKG/fig_2u_colorbar.png}
\end{minipage}

\\ 
 \begin{minipage}{0.3\textwidth}
\begin{overpic}[width=1.0\textwidth]{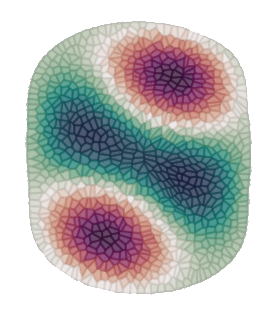}
\put(0,10){\fcolorbox{black}{white}{$n=500$}}
\end{overpic}
\end{minipage}
 & 
 \begin{minipage}{0.3\textwidth}
\begin{overpic}[width=1.0\textwidth]{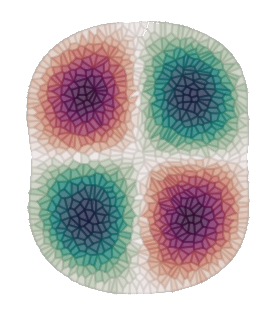}
\put(0,10){\fcolorbox{black}{white}{$n=700$}}
\end{overpic}
\end{minipage}
&
 \begin{minipage}{0.3\textwidth}
\begin{overpic}[width=1.0\textwidth]{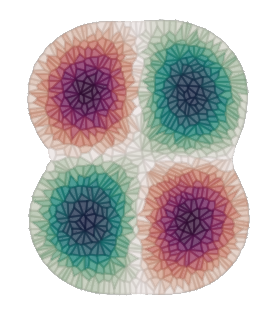}
\put(0,10){\fcolorbox{black}{white}{$n=1000$ (final)}}
\end{overpic}
\end{minipage}
\end{tabular}
\caption{\it Intermediate shapes arising in the course of the minimization of the $5^{\text{th}}$ Dirichlet eigenvalue $\lambda_\Omega^{(5)}$ in \cref{sec.exev}; the colors refer to the values of an associated eigenfunction.}
\label{fig.eig5}
\end{figure}

\begin{figure}[!ht]
\centering
\begin{tabular}{cc}
 \begin{minipage}{0.49\textwidth}
\begin{overpic}[width=1.0\textwidth]{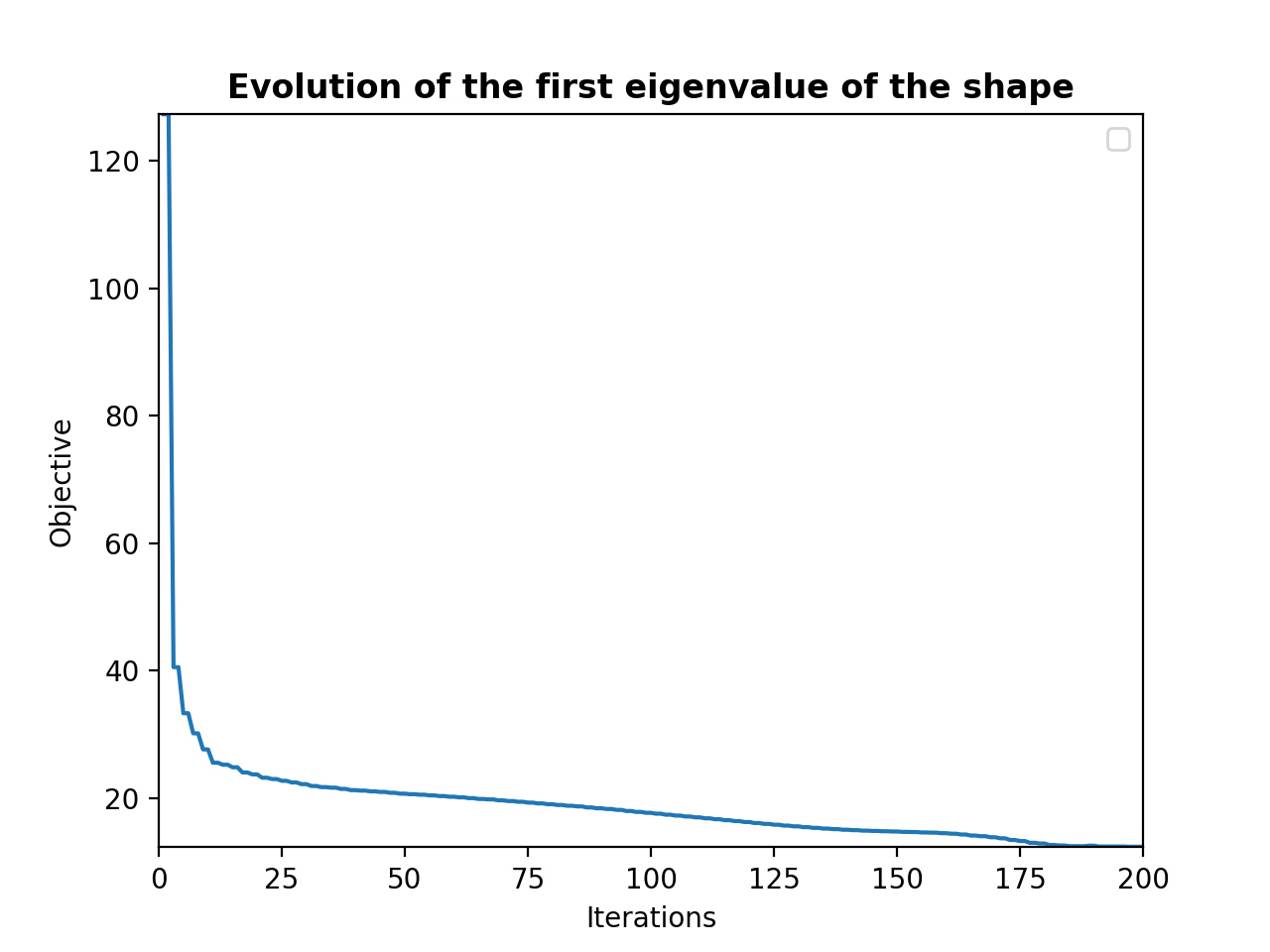}
\put(0,-3){\fcolorbox{black}{white}{a}}
\end{overpic}
\end{minipage} 
&
\begin{minipage}{0.49\textwidth}
\begin{overpic}[width=1.0\textwidth]{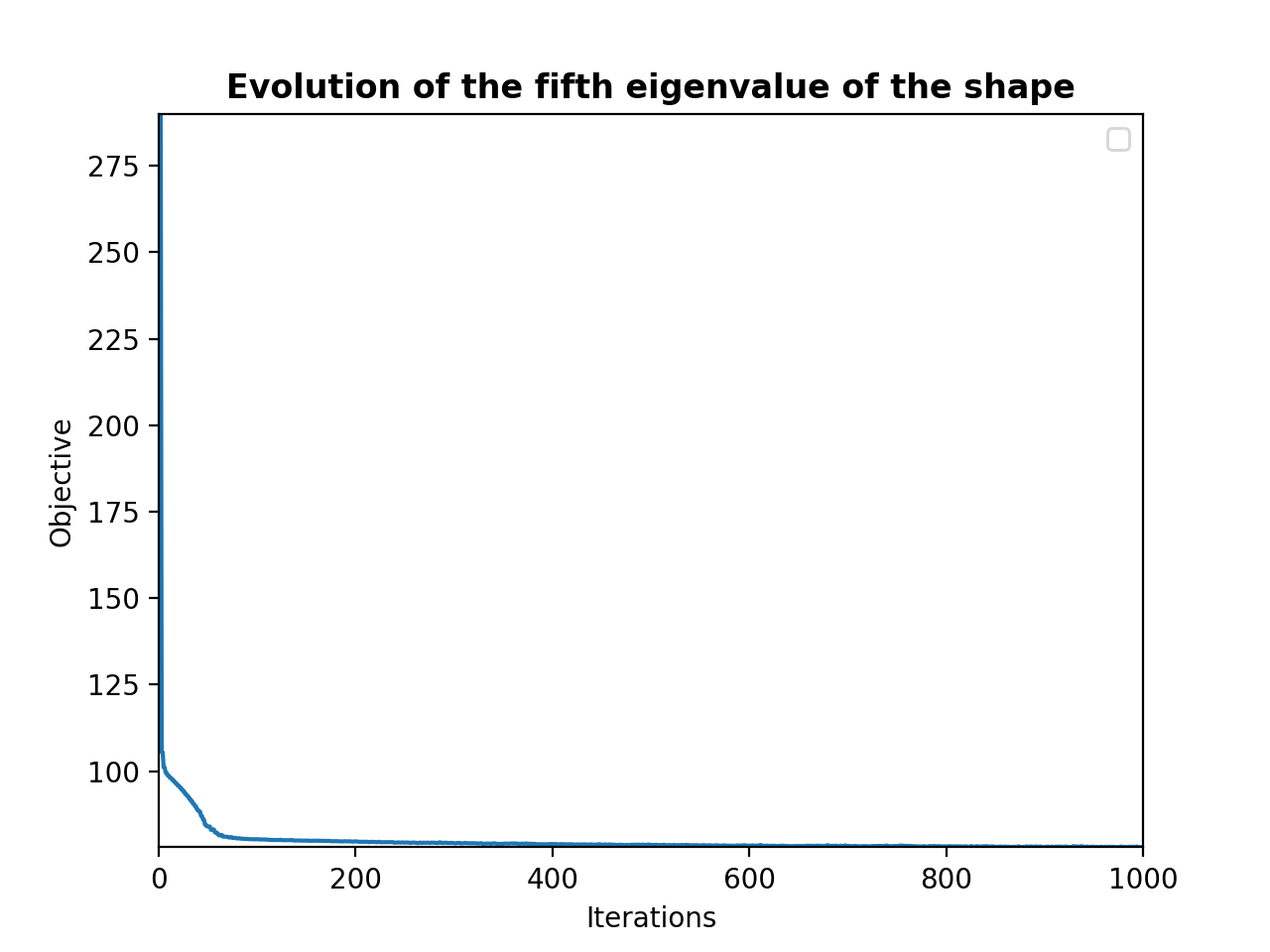}
\put(0,-3){\fcolorbox{black}{white}{b}}
\end{overpic}
\end{minipage}
\end{tabular}
\caption{\it (a) Evolution of the first eigenvalue $\lambda_\Omega^{(1)}$ of the shape in the course of the first experiment of \cref{sec.exev}; (b) Evolution of the fifth eigenvalue in the course of the second experiment.}
\label{fig.histeig}
\end{figure}

\subsection{Shape optimization in two-phase conductivity}\label{sec.numconduc2phase}

\noindent From this section on, we deal with shape optimization problems which are more directly related to mechanical applications. We first consider the setting of the conductivity equation presented in \cref{sec.conduc}, and notably its two-phase version, evoked in \cref{rem.2phaseconduc}. 

The situation under scrutiny is that depicted in \cref{fig.setex} (a): inside a fixed computational domain $D$ with size $1 \times 1$,
the shape $\Omega \subset D$ accounts for one phase $\Omega_1 := \Omega$ occupied by a material with high conductivity $\gamma_1 = 10$, 
the complementary phase $\Omega_0 := D \setminus \overline{\Omega}$ being filled by a material with low conductivity $\gamma_0 = 1$. 
\begin{figure}[!ht]
\centering
\begin{tabular}{cc}
\begin{minipage}{0.5\textwidth}
\vspace{-0.4cm}
\centering
\begin{overpic}[width=0.7\textwidth]{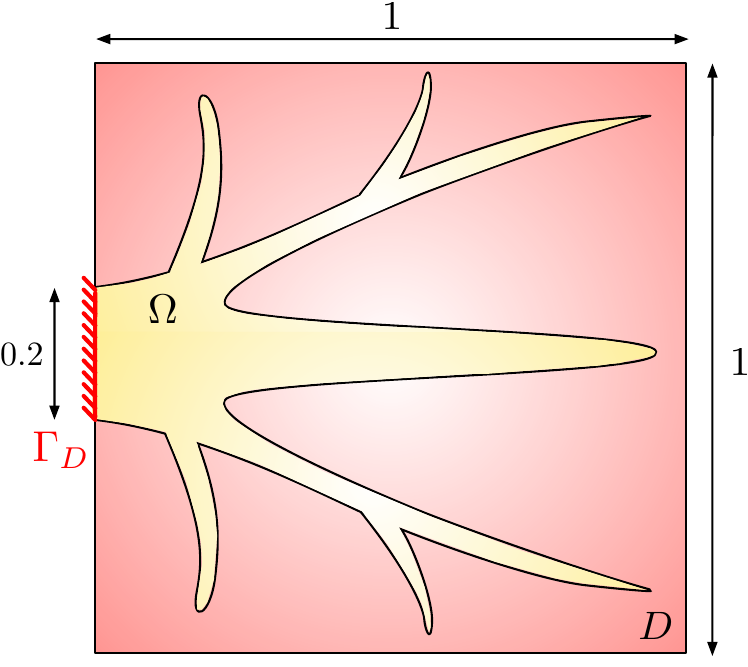}
\put(0,3){\fcolorbox{black}{white}{$a$}}
\end{overpic}
\end{minipage}
 & 
 \begin{minipage}{0.5\textwidth}
\begin{overpic}[width=1.0\textwidth]{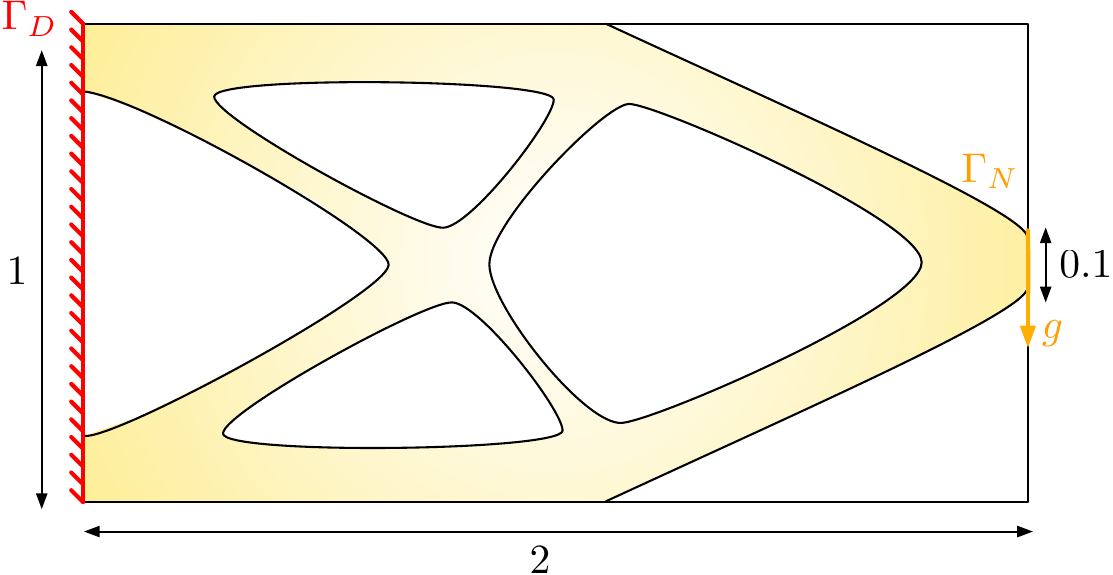}
\put(0,3){\fcolorbox{black}{white}{$b$}}
\end{overpic}
\end{minipage} \\
\\
\begin{minipage}{0.5\textwidth}
\centering
\begin{overpic}[width=0.8\textwidth]{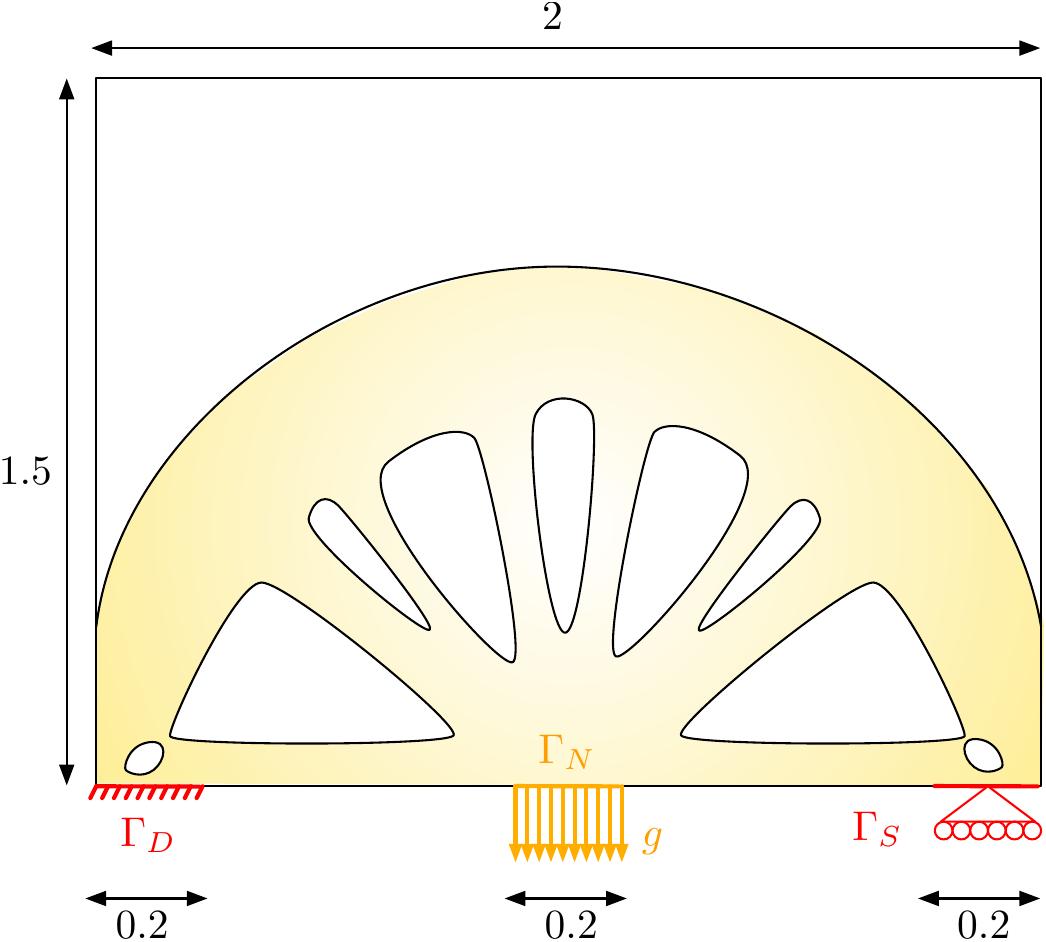}
\put(0,3){\fcolorbox{black}{white}{$c$}}
\end{overpic}
\end{minipage}
 & 
 \begin{minipage}{0.5\textwidth}
 \centering
\begin{overpic}[width=0.65\textwidth]{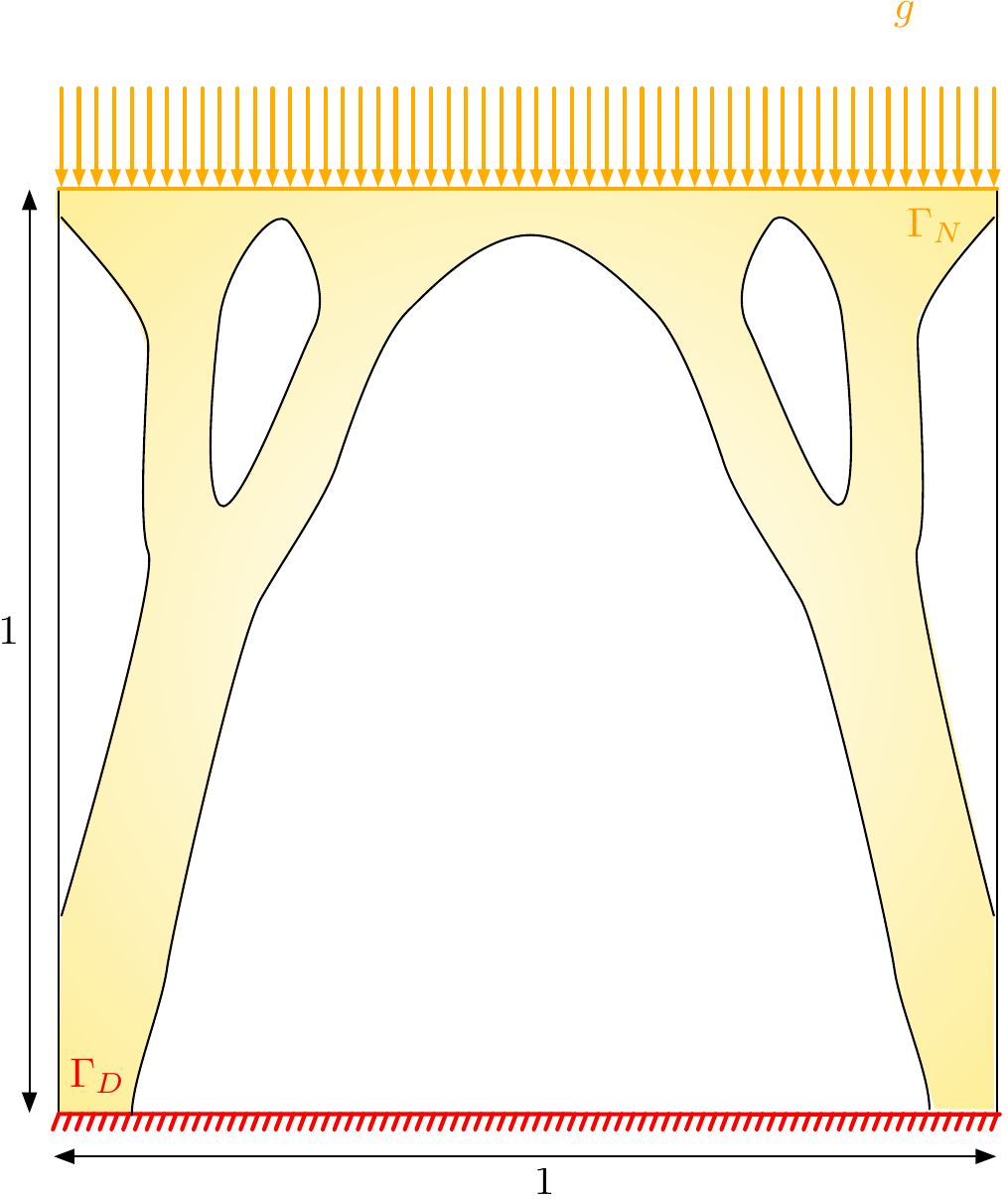}
\put(0,3){\fcolorbox{black}{white}{$d$}}
\end{overpic}
\end{minipage}
\end{tabular}
\caption{\it Settings of the physical numerical examples of \cref{sec.num}: (a) The heat diffuser test case of \cref{sec.numconduc2phase}; (b) The elastic cantilever example of \cref{sec.numcanti}; (c) The optimal bridge example of \cref{sec.exbr}; (d) The self-sustaining bridge of \cref{sec.stbr}.}
\label{fig.setex}
\end{figure}
The temperature $u_\Omega$ within $D$ is set to $0$ on a region $\Gamma_D \subset \partial D$ located at the left-hand side of $\partial D$, while the remaining part $\Gamma_N := \partial D \setminus \overline{\Gamma_D}$ is insulated from the outside. Assuming the presence of a constant source $f \equiv 1$ within $D$, $u_\Omega$ is the solution to the following two-phase conductivity problem:
\begin{equation}\label{eq.2phaseconducnum}
\left\{
\begin{array}{cl}
-\dv(\gamma_\Omega \nabla u_\Omega) = f & \text{in } D, \\
u_\Omega = 0 & \text{on } \Gamma_D, \\
\gamma_\Omega \frac{\partial u_\Omega}{\partial n} = 0 & \text{on } \Gamma_N, 
\end{array}
\right. 
\text{ where } \gamma_\Omega(\x) = \left\{
\begin{array}{cl}
\gamma_1 & \text{if } \x \in \Omega_1, \\
\gamma_0 & \text{otherwise.}
\end{array}
\right.
\end{equation}

We aim to minimize the mean temperature $T(\Omega)$ within $D$, i.e. we solve:
\begin{equation}\label{eq.pbminTOm}
\min\limits_{\Omega \subset D} \: T(\Omega) \text{ s.t. } \Vol(\Omega) = V_T,
\end{equation}
where $V_T$ is a volume target, and we have defined:
$$ T(\Omega) = \frac{1}{\Vol(D)} \int_D u_\Omega \:\d \x.$$

In this context, the shape $\Omega$ is represented via a classical Laguerre diagram $\bLag$ of $D$: complementary subcollections of cells of the latter account for $\Omega$ and $D \setminus \overline\Omega$, in the sense that \cref{eq.decompOmSubset} holds true.
Moreover, contrary to our practice in the previous \cref{sec.nummcf,sec.exev}, we rely on the non intrusive approach of \cref{sec.derver} to calculate the derivative of $T(\Omega)$ with respect to the seed points and weights of the representing diagram. In this perspective, we recall in the following proposition the expression of the shape derivative of $T(\Omega)$; since the result is fairly classical in the literature, we omit the proof for brevity, see e.g. \cite{allaire2020survey}.

\begin{proposition}\label{prop.sdthermics}
The functional $T(\Omega)$ is shape differentiable at any shape $\Omega \subset D$, and its derivatives reads for any deformation field $\btheta$ vanishing on $\overline{\Gamma_D}$:
$$ T^\prime(\Omega)(\btheta) = \frac{1}{\Vol(D)} \int_D \dv(\btheta) u_\Omega \:\d \x + \int_D \gamma_\Omega (\dv\btheta \I - \nabla \btheta - \nabla \btheta^T) \nabla u_\Omega \cdot \nabla p_\Omega \:\d \x - \int_D \dv(f\btheta) p_\Omega \:\d \x,$$
where the adjoint state $p_\Omega$ is the unique $H^1(D)$ solution to the boundary value problem:
$$
\left\{
\begin{array}{cl}
-\dv(\gamma_\Omega \nabla p_\Omega) = -\frac{1}{\Vol(D)} & \text{in } D, \\
p_\Omega = 0 & \text{on } \Gamma_D, \\
\gamma_\Omega \frac{\partial p_\Omega}{\partial n} = 0 & \text{on } \Gamma_N.
\end{array}
\right. 
$$
\end{proposition}

\begin{remark}\label{rem.selfadj}
Since in our context the source $f$ is identically equal to $1$, we immediately see that the adjoint state $p_\Omega$ equals $p_\Omega = -\frac{1}{\Vol(D)} u_\Omega$. This self-adjoint property of the problem \cref{eq.pbminTOm} conveniently allows to avoid the solution of an extra boundary value problem at each stage of the optimization process.  
\end{remark}

We perform 200 iterations of our shape and topology optimization \cref{algo.lagevol} to solve this problem. On average, the considered diagrams contain about 3,500 cells, and the total computation takes about 45 min on a regular \texttt{Macbook Pro} laptop with 2 GHz Quad-Core Intel Core i5 and 16 Gb of memory.
A few snapshots of the optimization path are presented in \cref{fig.conduc2phaseres}, 
and the associated convergence histories are reported in \cref{fig.conduc2phasehisto}.
Again, the shape undergoes dramatic deformations in the course of the process, while being consistently equipped with an explicit discretization, which is updated in a Lagrangian manner.

\begin{figure}[!ht]
\centering
\begin{tabular}{ccc}
\begin{minipage}{0.33\textwidth}
\begin{overpic}[width=1.0\textwidth]{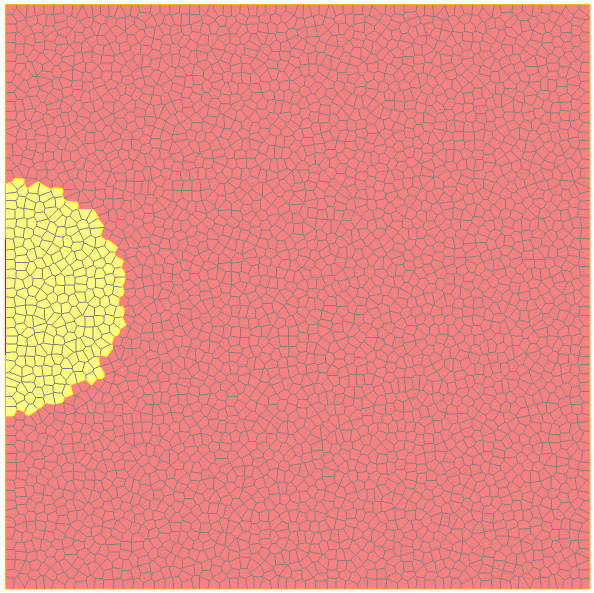}
\put(0,3){\fcolorbox{black}{white}{$n=0$}}
\end{overpic}
\end{minipage}
 & 
 \begin{minipage}{0.33\textwidth}
\begin{overpic}[width=1.0\textwidth]{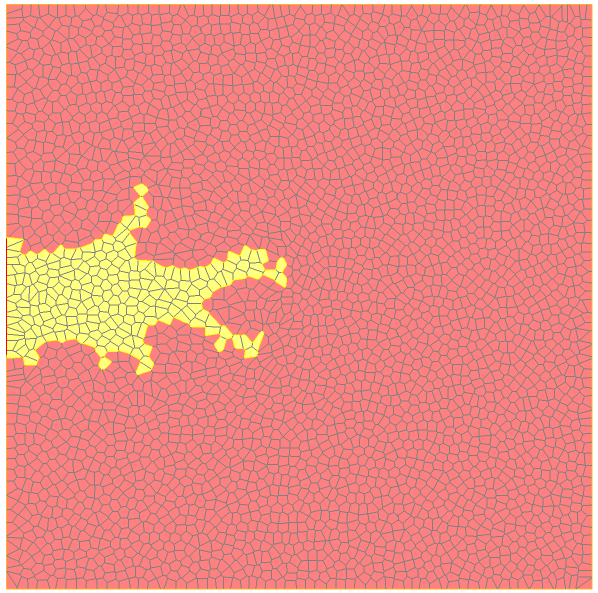}
\put(0,3){\fcolorbox{black}{white}{$n=12$}}
\end{overpic}
\end{minipage} &
\begin{minipage}{0.33\textwidth}
\begin{overpic}[width=1.0\textwidth]{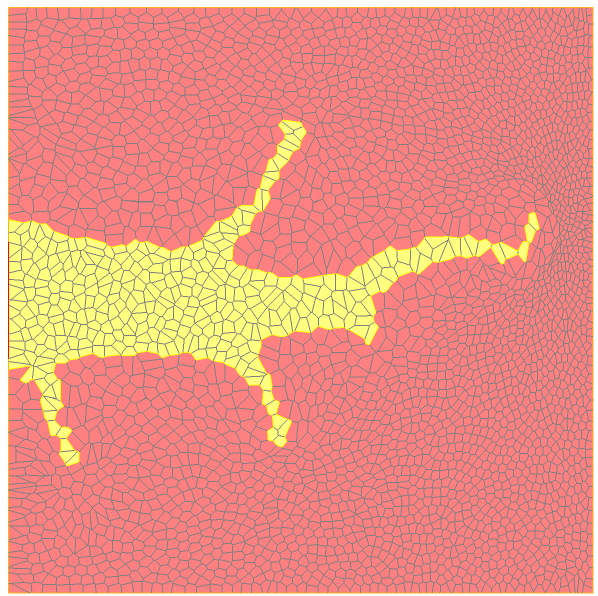}
\put(0,3){\fcolorbox{black}{white}{$n=29$}}
\end{overpic}
\end{minipage}
 \\
 \begin{minipage}{0.33\textwidth}
\begin{overpic}[width=1.0\textwidth]{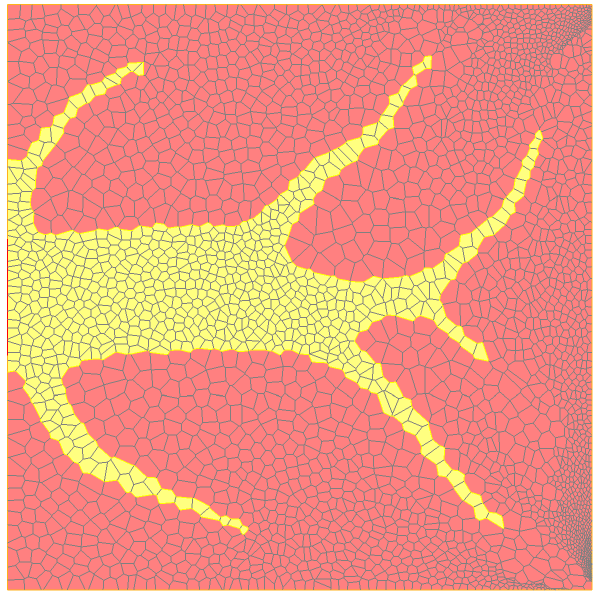}
\put(0,3){\fcolorbox{black}{white}{$n=51$}}
\end{overpic}
\end{minipage}
 & 
 \begin{minipage}{0.33\textwidth}
\begin{overpic}[width=1.0\textwidth]{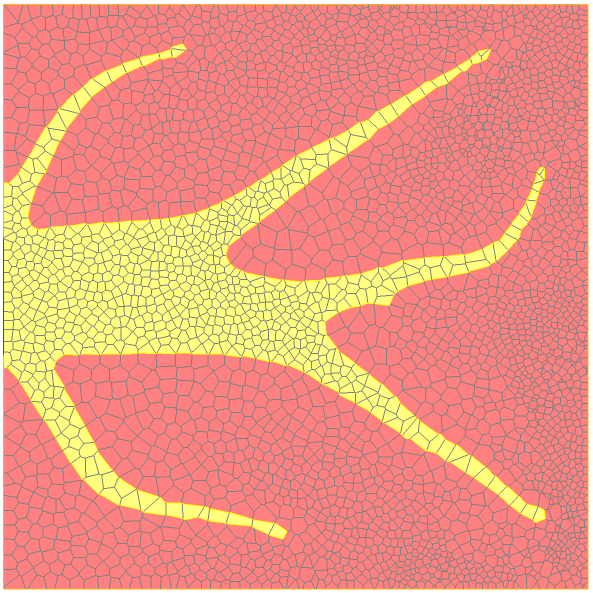}
\put(0,3){\fcolorbox{black}{white}{$n=151$}}
\end{overpic}
\end{minipage} &
\begin{minipage}{0.33\textwidth}
\begin{overpic}[width=1.0\textwidth]{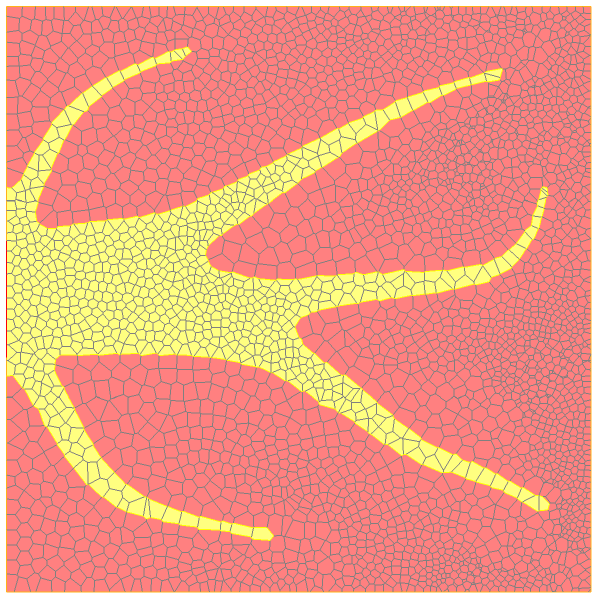}
\put(0,3){\fcolorbox{black}{white}{$n=250$}}
\end{overpic}
\end{minipage}
\end{tabular}
\caption{\it A few intermediate design $\Omega^n$ in the solution of the two-phase conductivity optimization problem of \cref{sec.numconduc2phase}.}
\label{fig.conduc2phaseres}
\end{figure}

\begin{figure}[!ht]
\centering
\includegraphics[width=0.8\textwidth]{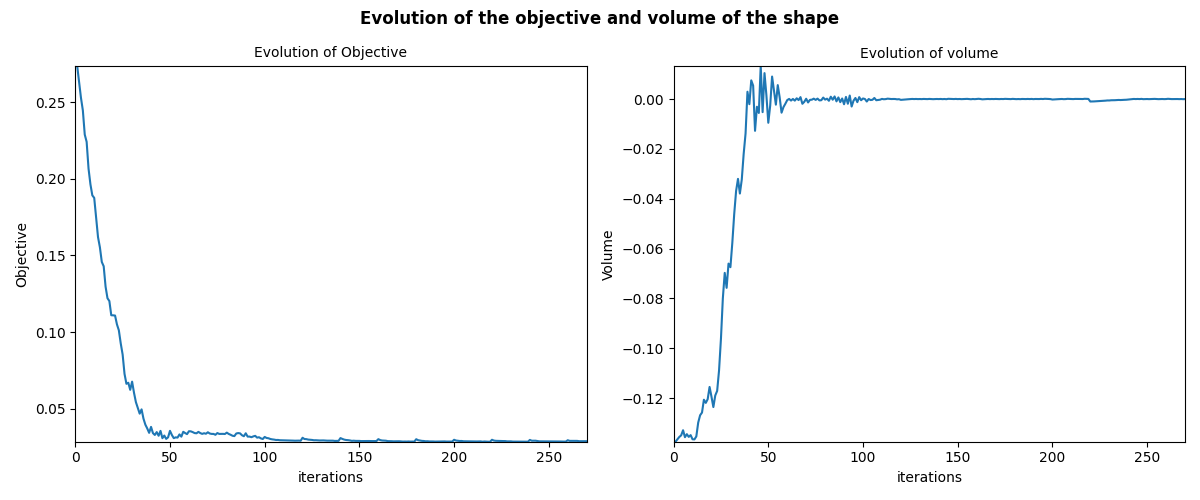}
\caption{\it Evolution of the objective $T(\Omega)$ and constraint $G(\Omega) = \Vol(\Omega)-V_T$ during  the solution of the two-phase conductivity optimization problem of \cref{sec.numconduc2phase}.}
\label{fig.conduc2phasehisto}
\end{figure}

\subsection{Optimization of the shape of an elastic cantilever beam}\label{sec.numcanti}

\noindent The examples in this section take place in the physical context of mechanical structures presented in \cref{sec.elas}. The considered shapes $\Omega$ are enclosed in a fixed computational box $D$ with size $2 \times 1$; they are fixed on the left-hand side $\Gamma_D$ of the boundary $\partial D$, 
and surface loads $\g : \Gamma_N \to \R^2$ are applied on a region $\Gamma_N$ at the right-hand side of $\partial D$. The remaining region $\Gamma := \partial \Omega \setminus (\overline{\Gamma_D} \cup \overline{\Gamma_N})$ is the only one which is subject to optimization, see \cref{fig.setex} (b). Omitting body forces for simplicity, the displacement $\u_\Omega : \Omega \to \R^d$ of the shape is the solution to the linear elasticity system: 
\begin{equation}\label{eq.elasnum}
\left\{
\begin{array}{cl}
- \dv(Ae(\u_\Omega)) = \bz & \text{in }\Omega, \\
\u_\Omega = \bz & \text{on } \Gamma_D, \\
Ae(\u_\Omega) \n = \g & \text{on } \Gamma_N, \\
Ae(\u_\Omega) \n = \bz& \text{on } \Gamma. 
\end{array}
\right.
\end{equation}
In this situation, we consider the following shape optimization problem: 
\begin{equation}\label{eq.mincompelas}
\min\limits_{\Omega \subset D} \:C(\Omega) \text{ s.t. }\Vol(\Omega) = V_T,
\end{equation}
where $C(\Omega)$ is the compliance \cref{eq.complianceElas} of $\Omega$ and the volume target $V_T$ is set to $0.7$.\par\medskip

To address this problem, we consider both shape optimization strategies presented in \cref{sec.Lagrefvol}. 
In a first experiment, we rely on the ``free boundary'' approach: the shape is discretized as a diagram $\bVsp$, see \cref{def.modLag}. The $N=3000$ cells of this diagram are consistently endowed with the same measure $\Vol(\Omega)/N$. At each stage of the optimization process, the derivative of the objective function $C(\Omega)$ with respect to the vertices $\q$ of this diagram is calculated via automatic differentiation, and the sensitivity of this function with respect to the seed points $\s$ of the diagram is inferred along the lines of \cref{sec.derver}, while the cell measures $\bnu$ are kept fixed. Concurrently, we periodically decrease the measures of all cells by a fixed amount until the desired volume constraint $V_T$ is attained.  
The optimized design and a few intermediate shapes resulting from this procedure are depicted in \cref{fig.cantiresfree}.  Remarkably, the topology of the shape changes dramatically in the course of the evolution, while no topological derivative is involved in the process. Holes spontaneously appear in mechanically relevant locations inside the structure, pretty much in the same manner as the drastic topological changes occurring in the example of \cref{sec.nummcf} were nevertheless relevant from the optimization viewpoint. This observation confirms that the notion of sensitivity with respect to the domain associated to our representation of shapes via diagrams of the form $\bVsp$ contains a much richer amount of information when compared to more classical boundary variation algorithms.
One drawback of this approach is the difficulty to control the connectedness of the shape when optimizing the volume of the cells by means of the strategy of \cref{sec.vertoseeds}, which is why the heuristic volume update strategy outlined above is preferred to the use of the formulas for derivative with respect to volume fractions. 

\begin{figure}[!ht]
\centering
\begin{tabular}{ccc}
\begin{minipage}{0.49\textwidth}
\begin{overpic}[width=1.0\textwidth]{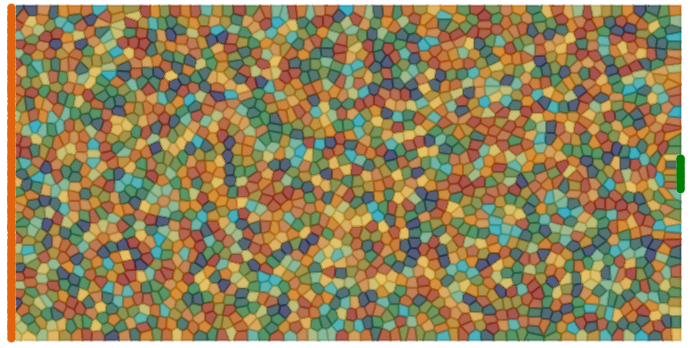}
\put(0,-3){\fcolorbox{black}{white}{$n=0$}}
\end{overpic}
\end{minipage}
 & 
 \begin{minipage}{0.49\textwidth}
\begin{overpic}[width=1.0\textwidth]{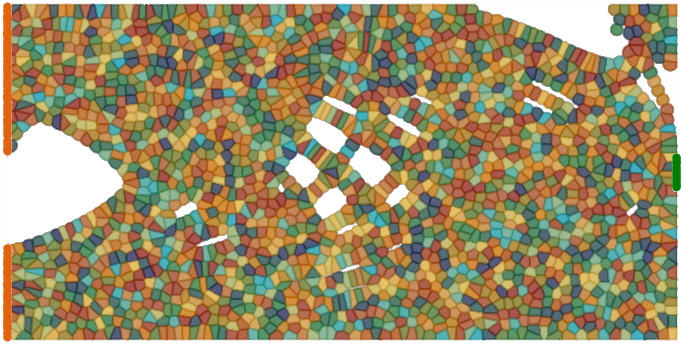}
\put(0,-3){\fcolorbox{black}{white}{$n=200$}}
\end{overpic}
\end{minipage} 
\\
\\
\begin{minipage}{0.49\textwidth}
\begin{overpic}[width=1.0\textwidth]{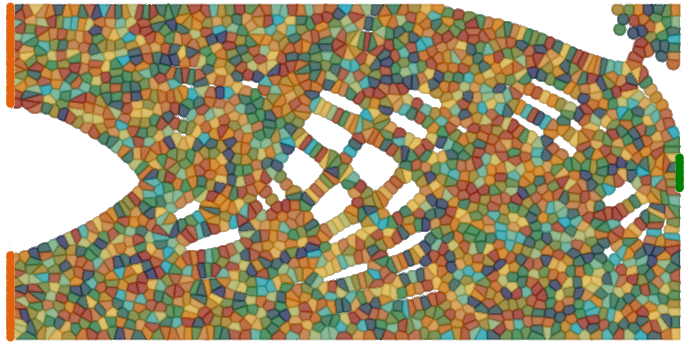}
\put(0,-3){\fcolorbox{black}{white}{$n=600$}}
\end{overpic}
\end{minipage} & 
\begin{minipage}{0.49\textwidth}
\begin{overpic}[width=1.0\textwidth]{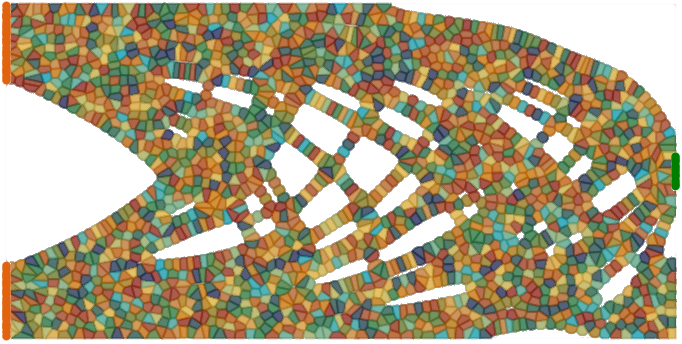}
\put(0,-3){\fcolorbox{black}{white}{$n=1000$}}
\end{overpic}
\end{minipage}
 \\ 
 \\
 \begin{minipage}{0.49\textwidth}
\begin{overpic}[width=1.0\textwidth]{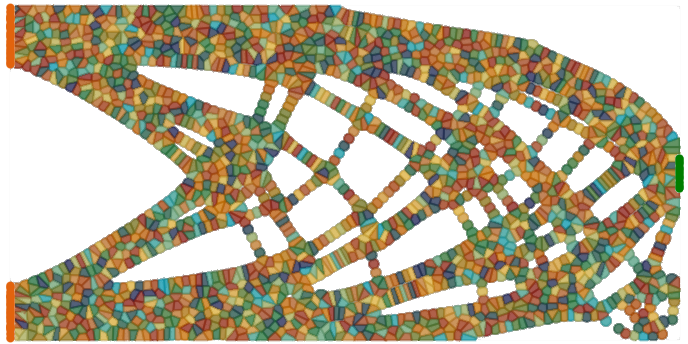}
\put(0,-3){\fcolorbox{black}{white}{$n=1500$}}
\end{overpic}
\end{minipage} 
&
\begin{minipage}{0.49\textwidth}
\begin{overpic}[width=1.0\textwidth]{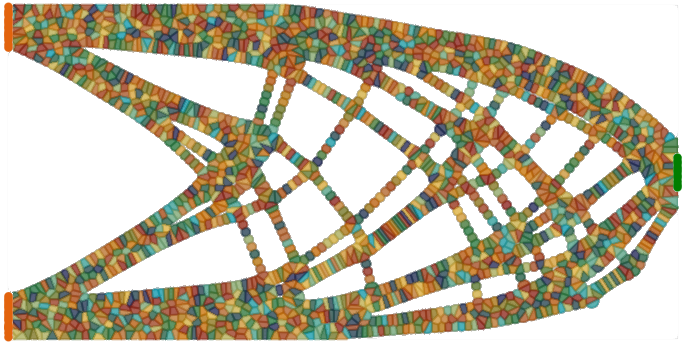}
\put(0,-3){\fcolorbox{black}{white}{$n=2000$}}
\end{overpic}
\end{minipage}
\end{tabular}
\caption{\it Various iterates in the solution of the cantilever optimization problem of \cref{sec.numcanti} using a discretization of shapes by modified diagrams \cref{eq.decompOm,eq.defVipsi}.}
\label{fig.cantiresfree}
\end{figure}

We conduct a second experiment in the same physical context, where we rely on the two-phase diagram approach for the representation of the shape $\Omega$: the total computational domain $D$ is equipped with a classical Laguerre diagram $\bLag$, and the shape $\Omega$ is defined as a subcollection of the cells of this diagram, see \cref{def.classLag,eq.decompOmSubset}. 
In this setting, we follow the ``optimize-then-discretize'' approach of \cref{sec.derver} to calculate the sensitivity of $C(\Omega)$ with respect to the vertices $\q$ of the diagram (although automatic differentiation could be used as in the previous examples), before expressing this information in terms of the seed points $\s$ and cell measures $\bnu$ of the diagram, along the lines of \cref{sec.vertoseeds}. This practice relies on the continuous formula for the shape derivative of $C(\Omega)$. This result is fairly classical in the literature (again, see e.g. \cite{allaire2020survey}) and we limit ourselves with the statement of the result.
 
 \begin{proposition}
 The shape derivative of the compliance $C(\Omega)$ reads, for any vector field $\btheta$ vanishing on $\overline{\Gamma_D} \cup \overline{\Gamma_N}$:
 \begin{equation*}
  C^\prime(\Omega)(\btheta) = -\int_\Omega \dv (\btheta) Ae(\u_\Omega) : e(\u_\Omega) \:\d \x  
  +4 \mu \int_\Omega ( \nabla \u_\Omega \nabla \btheta ) : e(\u_\Omega) \:\d \x 
 + 2\lambda \int_\Omega \tr(\nabla \u_\Omega \nabla \btheta ) \dv(\u_\Omega) \:\d \x.
\end{equation*}
 \end{proposition}

 Note that the compliance $C(\Omega)$ has the convenient property to make the adjoint state $\bp_\Omega$ equal to the state function $\u_\Omega$ (up to a sign).
 
 We apply our optimization \cref{algo.lagevol} to the solution of this problem. Actually, acknowledging that the retained discretization of shapes causes the boundary to be a little rough, we add a very small penalization to the objective function $C(\Omega)$ of the problem by the perimeter functional $\Per(\Omega)$. The featured Laguerre diagrams contain on average 7000 cells, and the computation proceeds in 200 iterations, for a total time of about 80 min. A few iterates of the optimization process are depicted on \cref{fig.cantires}, and the associated convergence histories are represented on \cref{fig.cantihisto}. The optimization path is much smoother in this second experiment as in the former one, which can be explained by the fact that variations of the measures of the cell do not incur non differentiability of the objective function so easily as in the ``free boundary'' context. Indeed, holes cannot emerge naturally inside the bulk structure; this leaves much less room for one of the conditions \cref{eq.Gen0,eq.Gen1,eq.Gen2,eq.Gen3,eq.Gen4,eq.Gen5,eq.Gen6} for the differentiability of the vertices of the diagram with respect to seed points and cell measures to become violated. However, one drawback of this approach is admittedly that the total computational domain $D$ has to be equipped with a diagram, thus resulting in an increase in computational burden, especially in the perspective of the 3d extension of this work.

\begin{figure}[!ht]
\centering
\begin{tabular}{ccc}
\begin{minipage}{0.49\textwidth}
\begin{overpic}[width=1.0\textwidth]{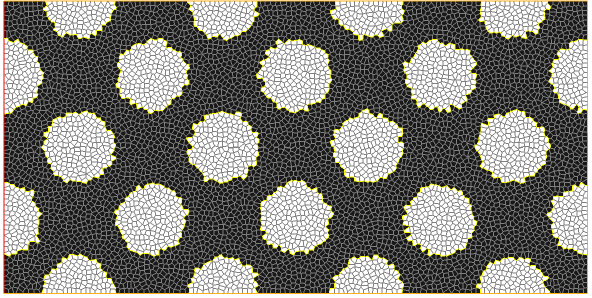}
\put(0,-3){\fcolorbox{black}{white}{$n=0$}}
\end{overpic}
\end{minipage}
 & 
 \begin{minipage}{0.49\textwidth}
\begin{overpic}[width=1.0\textwidth]{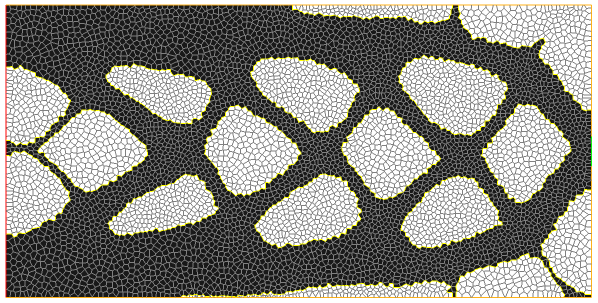}
\put(0,-3){\fcolorbox{black}{white}{$n=20$}}
\end{overpic}
\end{minipage} 
\\
\\
\begin{minipage}{0.49\textwidth}
\begin{overpic}[width=1.0\textwidth]{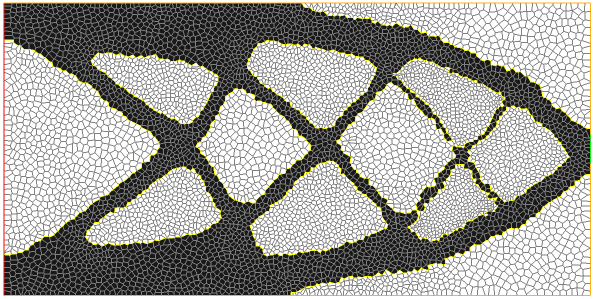}
\put(0,-3){\fcolorbox{black}{white}{$n=41$}}
\end{overpic}
\end{minipage} & 
\begin{minipage}{0.49\textwidth}
\begin{overpic}[width=1.0\textwidth]{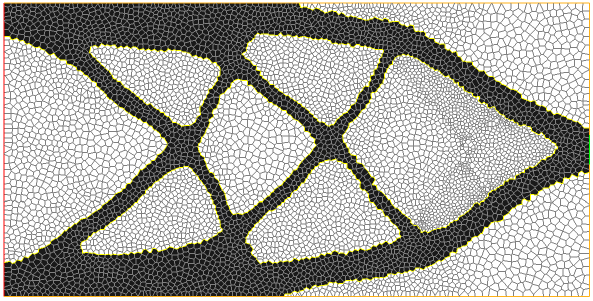}
\put(0,-3){\fcolorbox{black}{white}{$n=55$}}
\end{overpic}
\end{minipage}
 \\ 
 \\
 \begin{minipage}{0.49\textwidth}
\begin{overpic}[width=1.0\textwidth]{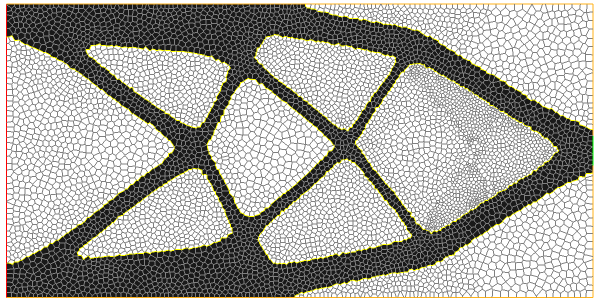}
\put(0,-3){\fcolorbox{black}{white}{$n=75$}}
\end{overpic}
\end{minipage} 
&
\begin{minipage}{0.49\textwidth}
\begin{overpic}[width=1.0\textwidth]{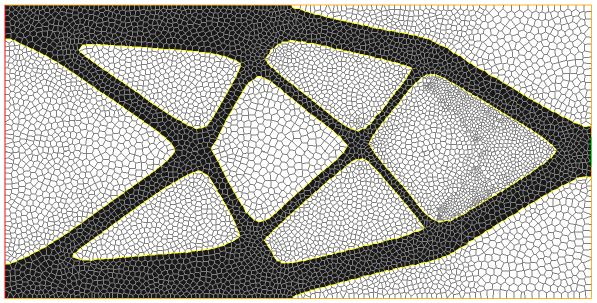}
\put(0,-3){\fcolorbox{black}{white}{$n=150$}}
\end{overpic}
\end{minipage}
\end{tabular}
\caption{\it Various iterates in the solution of the cantilever optimization problem of \cref{sec.numcanti} via a discretization of shapes using two-phase Laguerre diagrams of the computational domain $D$.}
\label{fig.cantires}
\end{figure}

\begin{figure}[!ht]
\centering
\includegraphics[width=0.8\textwidth]{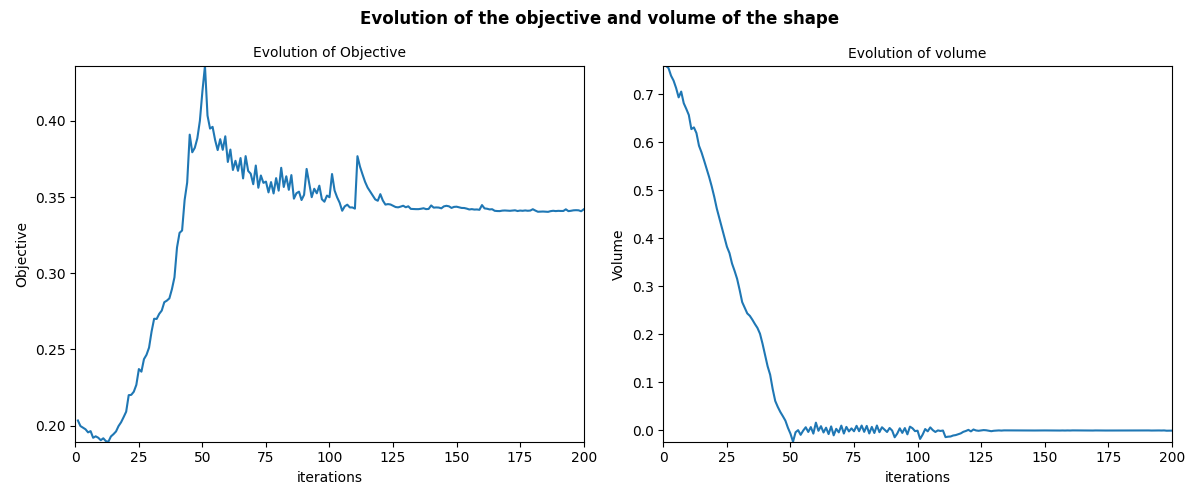}
\caption{\it Evolution of the objective (mechanical compliance) and volume during the solution of the cantilever optimization problem of \cref{sec.numcanti}.}
\label{fig.cantihisto}
\end{figure}

\subsection{Optimization of the shape of a bridge by combination of shape and topological derivatives}\label{sec.exbr}

\noindent This section deals with the optimization of another type of linearly elastic structures, namely a two-dimensional bridge.
The considered shapes $\Omega$ are enclosed in a box $D$ with size $2\times 1.5$; they are clamped on a region $\Gamma_D$ around their lower-left corner, and the vertical displacement is prevented on another region $\Gamma_S$ around their lower-right corner. 
A unit vertical load $\g = (0,-1)$ is applied on a region $\Gamma_N$ at the middle of their lower side, see \cref{fig.setex} (c). 

In this example, we again minimize the compliance of the structure under a volume constraint, i.e. \cref{eq.mincompelas} is solved, with the volume target $V_T = 0.7$.
We rely on a discretization of the shape $\Omega$ by two-phase Laguerre diagrams of $D$, see \cref{def.classLag,eq.decompOmSubset}, and we calculate the sensitivity of the compliance $C(\Omega)$ with respect to the defining parameters $\s$ and $\bnu$ of the diagram by means of the ``optimize-then-discretize'' approach presented in \cref{sec.derver}. 
We also add another ingredient with respect to the workflow of the previous section: we periodically use the topological derivative to try and drill a tiny hole inside $\Omega$ in an optimal way, see \cref{rem.topder}. In this perspective, let us recall the following result about the topological derivative of the compliance functional in two space dimensions, see \cite{garreau2001topological,novotny2012topological}. 

\begin{theorem}
Let $d = 2 $, and let $\Omega \subset \R^2$ be a bounded, Lipschitz domain. The elastic compliance $C(\Omega)$ defined in \cref{eq.complianceElas} has a topological derivative at every point $\x \in \Omega$, which reads: 
$$ \d_T C(\Omega)(\x)  = \frac{\pi(\lambda + 2\mu)}{2\mu(\lambda+\mu)} \Big( 4\mu Ae(\u_\Omega) : e(\u_\Omega) + (\lambda-\mu) \tr(Ae(\u_\Omega)) \tr(e(\u_\Omega)) \Big)(\x). $$ 
\end{theorem}

Starting from an initial shape $\Omega^0$ with trivial topology, see \cref{fig.brres}, we apply $200$ iterations of our shape and topology optimization \cref{algo.lagevol} (augmented with the aforementioned use of topological derivatives). On average, the considered Laguerre diagrams contain about 7000 cells, and the total computation takes about 80 min. A few intermediate shapes arising during the resolution are depicted in \cref{fig.brres}, and the convergence history is provided in \cref{fig.brhisto}. Although the initial shape has a very poor topology, the shape develops very non trivial features in the course of the optimization process, thanks to the use of topological derivatives. The final, optimized design is very reminiscent of those obtained in the same context in e.g. \cite{dapogny2022tuto}.  

\begin{figure}[!ht]
\centering
\begin{tabular}{ccc}
\begin{minipage}{0.33\textwidth}
\begin{overpic}[width=1.0\textwidth]{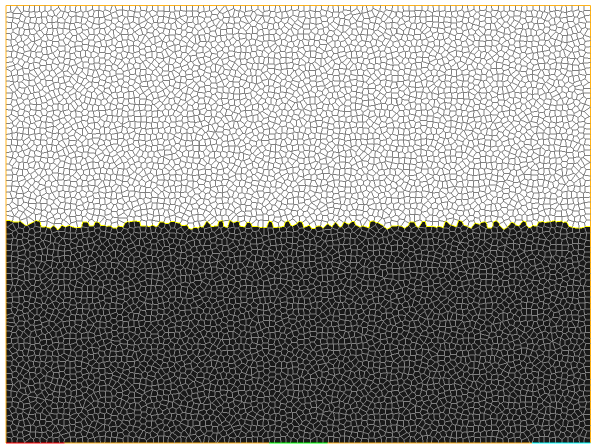}
\put(0,-5){\fcolorbox{black}{white}{$n=0$}}
\end{overpic}
\end{minipage}
 & 
 \begin{minipage}{0.33\textwidth}
\begin{overpic}[width=1.0\textwidth]{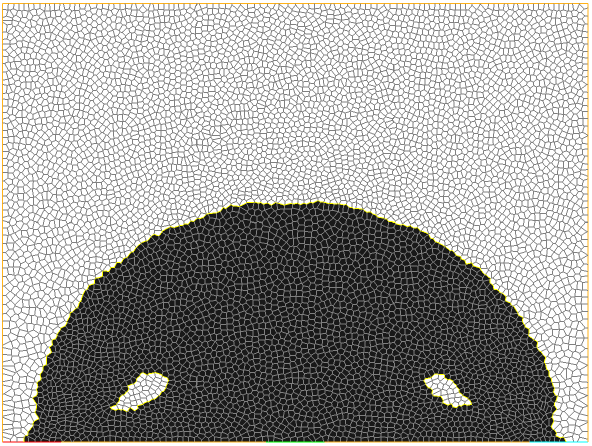}
\put(0,-5){\fcolorbox{black}{white}{$n=22$}}
\end{overpic}
\end{minipage} 
&
\begin{minipage}{0.33\textwidth}
\begin{overpic}[width=1.0\textwidth]{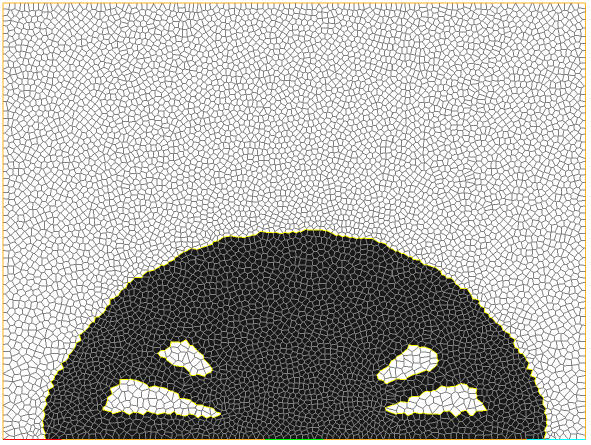}
\put(0,-5){\fcolorbox{black}{white}{$n=39$}}
\end{overpic}
\end{minipage}\\ 
\\
\begin{minipage}{0.33\textwidth}
\begin{overpic}[width=1.0\textwidth]{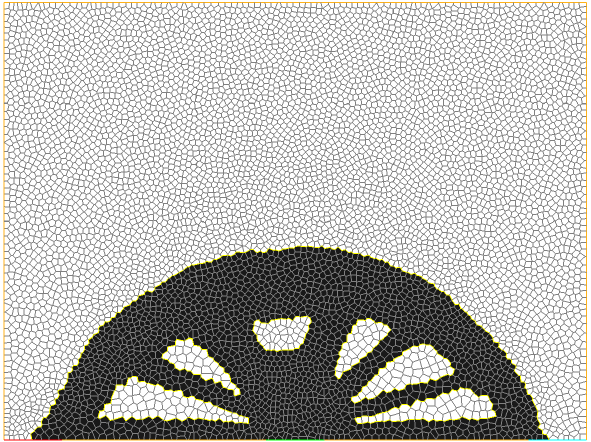}
\put(0,-5){\fcolorbox{black}{white}{$n=55$}}
\end{overpic}
\end{minipage}
 & 
 \begin{minipage}{0.33\textwidth}
\begin{overpic}[width=1.0\textwidth]{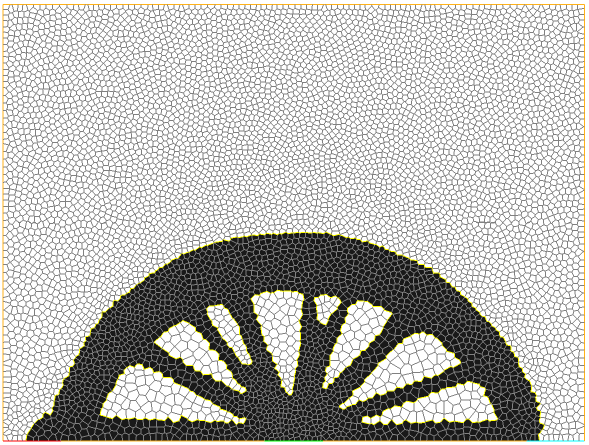}
\put(0,-5){\fcolorbox{black}{white}{$n=75$}}
\end{overpic}
\end{minipage} 
&
\begin{minipage}{0.33\textwidth}
\begin{overpic}[width=1.0\textwidth]{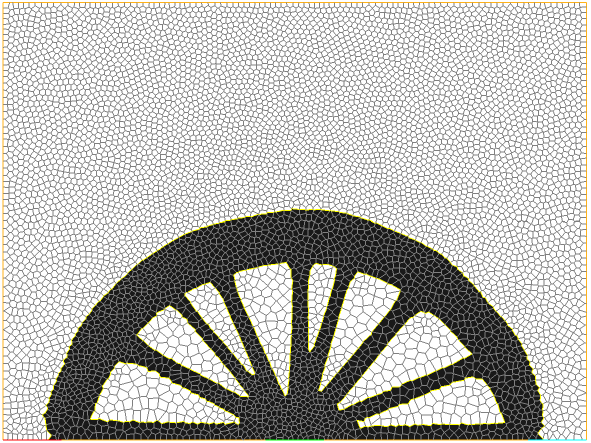}
\put(0,-5){\fcolorbox{black}{white}{$n=200$}}
\end{overpic}
\end{minipage}\\

\end{tabular}
\caption{\it Various iterates in the solution of the bridge optimization problem of \cref{sec.exbr}.}
\label{fig.brres}
\end{figure}

\begin{figure}[!ht]
\centering
\includegraphics[width=0.8\textwidth]{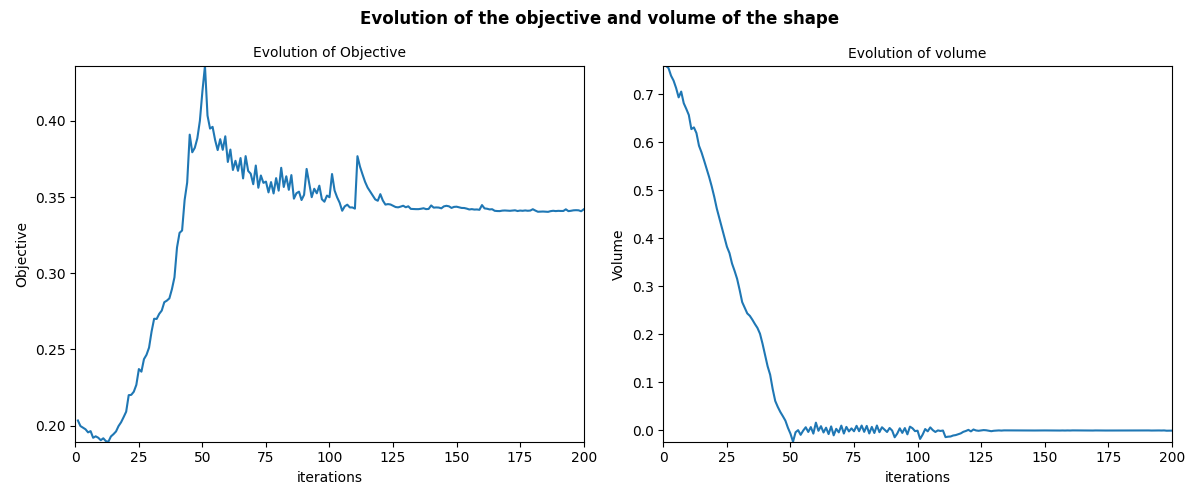}
\caption{\it Evolution of the objective (mechanical compliance) and volume during the solution of the bridge optimization problem of \cref{sec.exbr}.}
\label{fig.brhisto}
\end{figure}

\subsection{Minimization of the stress within a bridge structure}\label{sec.stbr}

\noindent Our last numerical example deals with the minimization of the stress within a two-dimensional self-sustaining bridge, which is depicted in \cref{fig.setex} (d): the shapes are contained in a square-shaped box $D$ with size $1 \times 1$; they are clamped on their bottom side $\Gamma_D$, and a uniform vertical load $\g = (0,-1)$ is applied on their upper side $\Gamma_N$, which stands for the deck of the structure. 
Omitting body forces, the displacement $\u_\Omega : \Omega \to \R^2$ of the structure is characterized by the boundary value problem \cref{eq.elasnum}. 

In this setting, we aim to minimize an integral quantity of the stress within the structure, i.e. we solve: 
\begin{equation}\label{eq.pbminSOm}
\min\limits_\Omega \: S(\Omega) \text{ s.t. }\Vol(\Omega) = V_T.
\end{equation}
In this formulation, the stress function $S(\Omega)$ is defined by:
$$ S(\Omega) = \int_\Omega  \lvert\lvert \sigma(\u_\Omega) \lvert\lvert^2 \:\d \x, \text{ where } \sigma(\u_\Omega) := Ae(\u_\Omega),$$
and the volume target $V_T$ is set to $0.8$. 

We address the solution of \cref{eq.pbminSOm} with the two-phase Laguerre diagram discretization presented in \cref{sec.Lagrep} and used in the previous examples \cref{sec.numcanti,sec.exbr}. Here again, the sensitivities of the objective function with respect to the seed points $\s$ and cell measures $\bnu$ defining the diagrams at stake are calculated by the optimize-then-discretize approach presented in \cref{sec.derver}.
From the implementation viewpoint, the main difference between this example and those addressed in the previous \cref{sec.numcanti,sec.exbr} is that the optimization problem is no longer self-adjoint, i.e. the shape derivative of $S(\Omega)$ involves an adjoint state $\p_\Omega$ which is not easily related to $\u_\Omega$,
as revealed by the following result, about which we refer again to e.g. \cite{allaire2020survey}.

\begin{proposition}
The function $S(\Omega)$ is shape differentiable at any smooth enough shape $\Omega$, and its derivative reads, for any vector field $\btheta$ vanishing on $\overline{\Gamma_D} \cup \overline{\Gamma_N}$: 
\begin{multline}\label{eq.pOmSOm}
 S^\prime(\Omega)(\btheta) = \int_\Omega \dv(\btheta) Ae(\u_\Omega) : e(\p_\Omega) \:\d \x - 2\mu \int_\Omega \Big( (\nabla \u_\Omega \nabla \btheta) : e(\p_\Omega) + (\nabla \p_\Omega \nabla \btheta) : e(\u_\Omega) \Big) \:\d \x\\
  -\lambda \int_\Omega \Big(\tr (\nabla \u_\Omega \nabla \btheta ) \dv(\p_\Omega) +  \tr (\nabla \p_\Omega \nabla \btheta ) \dv(\u_\Omega)   \Big) \:\d \x,
 \end{multline}
where the adjoint state $\p_\Omega$ is the solution to the following variational problem:
\begin{equation*}
\text{Search for } \p_\Omega \in H^1_{\Gamma_D}(\Omega)^2 \text{ s.t. } \forall \v \in H^1_{\Gamma_D}(\Omega)^2, \quad
\int_\Omega Ae(\p_\Omega) : e(\v) \:\d \x = - 2 \int_\Omega Ae(\u_\Omega) : Ae(\v) \:\d \x.
\end{equation*}
\end{proposition}

We apply 250 iterations of our shape and topology optimization \cref{algo.lagevol} to solve the problem \cref{eq.pbminSOm}. The average number $N$ of cells of the considered Laguerre diagrams is 8000, and the total computation takes about 3 hours. We note that the calculation of the adjoint state $\p_\Omega$ requires the solution to the boundary value problem \cref{eq.pOmSOm} for $\p_\Omega$ at each iteration of the process, in addition to that of \cref{eq.elasnum}, which explains the significant increase in computational burden observed in this case.
A few snapshots of the optimization process are reported on \cref{fig.stbrres}. Again, the shape dramatically changes topology in the course of the solution process.  

\begin{figure}[!ht]
\centering
\begin{tabular}{cc}
 \begin{minipage}{0.45\textwidth}
\begin{overpic}[width=1.0\textwidth]{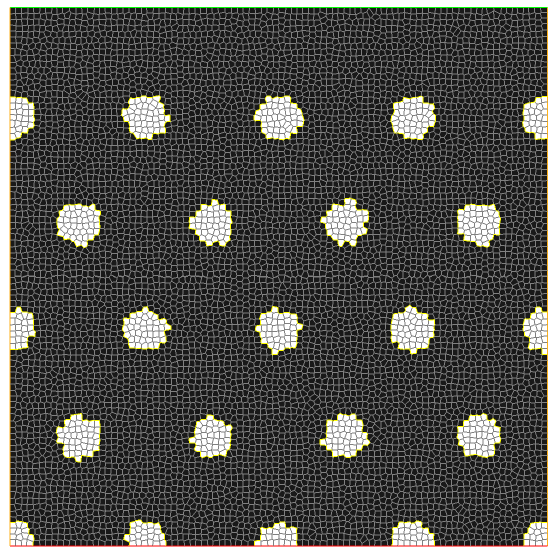}
\put(0,-3){\fcolorbox{black}{white}{$n=0$}}
\end{overpic}
\end{minipage} 
&
\begin{minipage}{0.45\textwidth}
\begin{overpic}[width=1.0\textwidth]{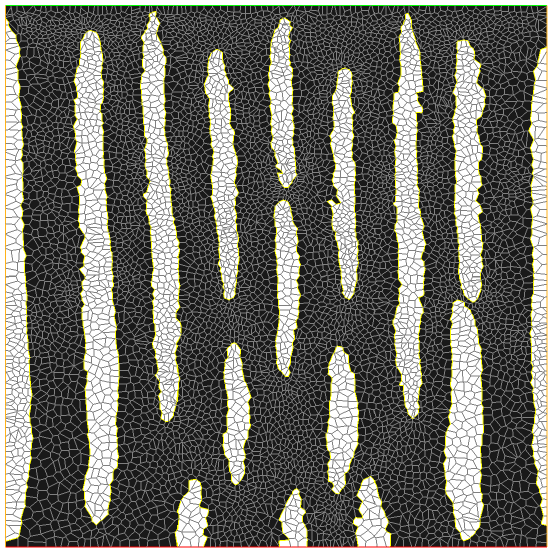}
\put(0,-3){\fcolorbox{black}{white}{$n=34$}}
\end{overpic}
\end{minipage}\\ 
\\
\begin{minipage}{0.45\textwidth}
\begin{overpic}[width=1.0\textwidth]{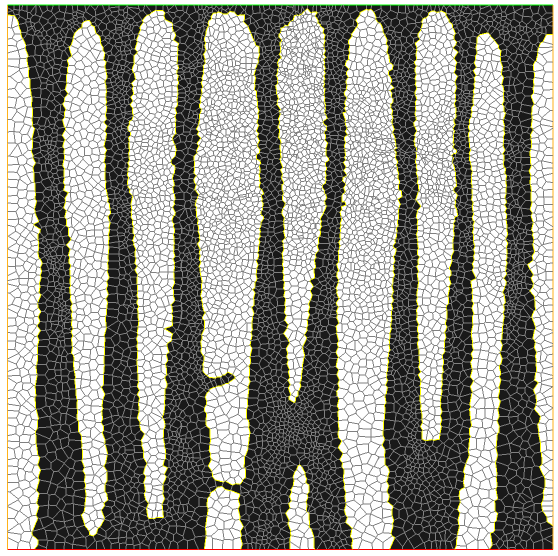}
\put(0,-3){\fcolorbox{black}{white}{$n=81$}}
\end{overpic}
\end{minipage}
 & 
 \begin{minipage}{0.45\textwidth}
\begin{overpic}[width=1.0\textwidth]{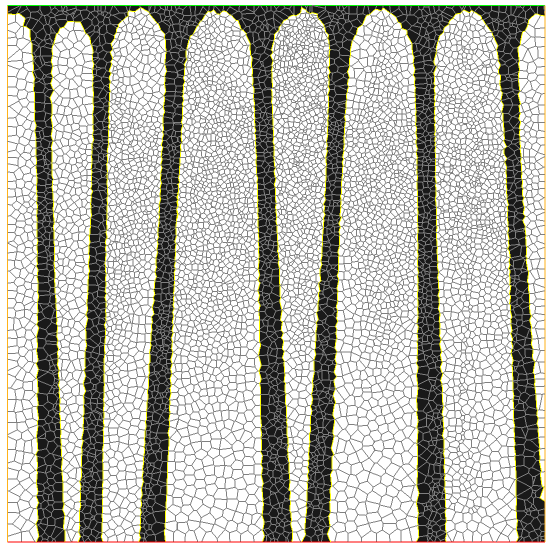}
\put(0,-3){\fcolorbox{black}{white}{$n=250$}}
\end{overpic}
\end{minipage} \\ 
\end{tabular}
\caption{\it Various iterates in the solution of the stress minimization problem within a self-sustaining bridge of \cref{sec.stbr}.}
\label{fig.stbrres}
\end{figure}

\section{Conclusions and perspectives}\label{sec.concl}

\noindent In this article, we have presented a novel numerical framework for optimal design, inspired by modern concepts from fields so diverse as optimal transport, computational geometry, and shape and topology optimization. The key feature of our method is a consistent body-fitted representation of the shape, in terms of a (modified version of a) Laguerre diagram. The latter is parametrized in terms of the seed points and the measures of its cells, as permitted by deep results from optimal transport theory. The evolution of the shape through the iterations of the process is driven by the shape gradients of the objective and constraint functionals, which are suitably expressed in terms of these defining variables of Laguerre diagrams. Our numerical strategy is Lagrangian in nature, since the evolution of the shape is tracked via the motion of the seed points and weights of the diagram. However, dramatic updates can be accounted for in a robust fashion, including topological changes, since the diagram is never deformed itself -- it is rediscovered from the updated seed points and cell measures. 
This framework can generally handle complex physical situations, featuring multiple phases, and we hope that it can serve to other applications. Remarkably, it features a description of a wide variety of shapes by a relatively small number of parameters, which pleads in favor of the use of such representation in connection with reduced basis methods, or of the use of neural networks for a parametrization of shapes by such a diagram, and the realization of its optimization with machine learning techniques. 

For the sake of simplicity, we have focused our presentation and numerical illustrations on the case of two space dimensions, although the underlying theory holds true in 3d, and the critical numerical ingredients of the method (notably, the computation of Laguerre diagrams) are already available in this context. In future work, we ambition to adapt this numerical methodology to this theoretically similar, albeit numerically much more involved (and much richer) 3d context.

On the longer term, and from a more theoretical viewpoint, it would be fascinating to appraise more rigorously the 
new type of shape evolution featured by this method. Indeed, the examples of \cref{sec.nummcf,sec.numcanti} have demonstrated how different the induced notion of sensitivity with respect to the design is from more ``classical'' concepts such as shape and topological derivatives: we suspect that these issues are somehow related to the idea of linearized optimal transport, see \cref{rem.linearOT}.
\par\bigskip

\noindent \textbf{Acknowledgements.} The work of C.D. and E.O. is partially supported by the projects ANR-18-CE40-0013 SHAPO and ANR-22-CE46-0006 StableProxies, financed by the French Agence Nationale de la Recherche (ANR).
This work was completed while C.D. was visiting the Laboratoire Jacques-Louis Lions from Sorbonne Universit\'es, whose hospitality is gratefully acknowledged.
E.O. was also supported by the ``Institut Universitaire de France".
The work of B.L. is supported by the COSMOGRAM-launchpad Inria Action Exploratoire grant.

\appendix

\section{Calculation and differentiation of the coordinates of the vertices of a diagram with respect to its seed points and weights}\label{app.verseeds}

\noindent In this appendix, we detail the calculation of the positions of the vertices $\q = \left\{\q_j \right\}_{j=1,\ldots,M} \in \R^{dM}$ of the diagram $\bVsp$ in \cref{eq.decompOm} 
from the datum of its seed points $\s$ and weights $\bpsi$. We also investigate the calculation of their derivatives with respect to $\s$ and $\bpsi$, 
which is a central ingredient in the procedure of \cref{sec.vertoseeds} for expressing differentiation of functions with respect to the positions $\q$ of the vertices of the diagram in terms of the generating seed points $\s$ and weights $\bpsi$. 
These formulas are admittedly not new, but their rigorous exposition and establishment are not so easily found in the literature.
Our presentation focuses on the case $d=2$, but similar analyses could be conducted in three space dimensions.\par\medskip

Let us start with the following general result. 
\begin{proposition}\label{prop.derver}
Let $\s \in \R^{dN}$ and $\bpsi \in \R^N$ be sets of seed points and weights satisfying the genericity assumptions \cref{eq.Gen0,eq.Gen1,eq.Gen2,eq.Gen3,eq.Gen4,eq.Gen5,eq.Gen6} and
let  $\q = \left\{\q_j \right\}_{j=1,\ldots,M} \in \R^{dM}$ denote the vertices of the diagram $\bVsp$.
Then for $\widehat \s \in \R^{dN}$ and $\widehat \bpsi \in \R^N$ small enough, the perturbed diagram $\mathbf{V}(\s + \widehat \s, \bpsi + \widehat \bpsi)$ has the same neighbour relations as $\bVsp$. 

Moreover, there exists a mapping $\m : \R^{dN}_{\s} \times \R^N_{\bpsi} \to \Winfty$ which is differentiable in a neighborhood of $(\bz,\bz)$ such that:
$$ \forall i =1,\ldots, N, \quad \m(\widehat \s, \widehat \bpsi) \Big( V_i(\s,\bpsi) \Big) = V_i(\s + \widehat \s, \bpsi + \widehat \bpsi),$$
and 
\begin{equation*}
 \text{The vertices of the diagram } \mathbf{V}(\s + \widehat \s, \bpsi + \widehat \bpsi)
  \text{ are exactly the points } \m(\widehat \s, \widehat \bpsi) (\q_j), \:\: j=1,\ldots,M.  
  \end{equation*}
\end{proposition}

\begin{proof}
At first, we prove that $i\neq j \in\left\{ 1,\ldots,N\right\}$ are the indices of neighbor cells in the diagram $\bVsp$ if and only if 
the cells $i$ and $j$ are neighbors in $\mathbf{V}(\s + \widehat \s, \bpsi + \widehat \bpsi)$ for small enough perturbations $\widehat \s \in \R^{dN}$ and $\widehat \bpsi \in \R^N$.
For simplicity, we only deal with the case where the edge $\be_{ij}$ between the cells $V_i(\s,\bpsi)$ and $V_j(\s,\bpsi)$ in $\bVsp$ lies entirely in the open set $D$, the treatment of the situation where it intersects $\partial D$ being similar.  

To achieve our purpose, let us first observe that $V_i(\s,\bpsi)$ and $V_j(\s,\bpsi)$ are neighbors in the diagram $\bVsp$ if and only if the following property holds true: 
$$\exists \x \in D \text{ s.t. } \left\{\begin{array}{cl}
\lvert \x - \s_i \lvert^2 - \psi_i =  \lvert \x - \s_j \lvert^2 - \psi_j,&\\
\lvert \x - \s_i \lvert^2 - \psi_i <  \lvert \x - \s_k \lvert^2 - \psi_k, & \text{for each } k \notin \left\{i,j\right\}. 
\end{array}
\right.$$
In turn, this is equivalent to the positivity of the optimal value
$$\max\limits_{\x} \left\{ \min\limits_{k \notin \left\{i,j\right\}} \left\{ \lvert \x - \s_k \lvert^2 - \psi_k - \Big(  \lvert \x - \s_i \lvert^2 - \psi_i \Big) \right\}, \: \x \in D, \:\: \lvert \x - \s_i \lvert^2 - \psi_i =  \lvert \x - \s_j \lvert^2 - \psi_j\right\}.$$
We now invoke the continuity of the optimal value of a constrained optimization program with respect to perturbations of its parameters, see e.g. Prop. 4.4 in \cite{bonnans2013perturbation}, whose assumptions are satisfied in the present case.
Hence, if the cells indexed by $i \neq j$ are neighbors in the diagram $\bVsp$, then they are also neighbors in the diagram $\mathbf{V}(\s + \widehat \s, \bpsi + \widehat \bpsi)$ for small enough $\widehat \s \in \R^{dN}$, $\widehat \bpsi\in \R^N$.
 Note that a similar argument shows that if $\bVsp$ satisfies \cref{eq.Gen0,eq.Gen1,eq.Gen2,eq.Gen3,eq.Gen4,eq.Gen5,eq.Gen6}, then so does   $\mathbf{V}(\s + \widehat \s, \bpsi + \widehat \bpsi)$.
 
Conversely, let us now prove that there exists $\e >0$ such that, for any perturbations $\widehat{\s}\in \R^{dN}$, $\widehat{\bpsi}\in \R^N$ with $\lvert \widehat{\s} \lvert + \lvert \widehat{\bpsi} \lvert < \e$, the diagram $\mathbf{V}(\s + \widehat{\s}, \bpsi + \widehat{\bpsi)})$ cannot have neighbor cells that are not already neighbors in $\bVsp$. 
To achieve this, we proceed by contradiction, assuming that there exist sequences $\widehat{\s}^n \in \R^{dN}$, $\widehat{\bpsi}^n \in \R^N$ converging to $\bzero$ and index sequences $i^n \neq j^n \in \left\{1,\ldots,N \right\}$ such that the cells with indices $i^n$, $j^n$ are neighbors in $\mathbf{V}(\s + {\widehat \s}^n, \bpsi + \widehat{\bpsi}^n)$ but not in $\bVsp$. Actually, since the number $N$ of seed points is finite, we may extract a subsequence (still labeled by $^n$) along which these indices $i$, $j$ are independent of $n$.
Let then $\a^n$ and $\b^n$ denote the vertices of the edge between the cells $i$ and $j$ in $\mathbf{V}(\s + \widehat \s^n, \bpsi + \widehat \bpsi^n)$, 
and let $k^n,l^n \notin \left\{i,j \right\}$ denote the (distinct) indices such that $\a^n$ (resp. $\b^n$) is at the intersection between the cells $i$, $j$ and $k^n$ (resp. $l^n$) in $\mathbf{V}(\s + \widehat \s^n, \bpsi + \widehat \bpsi^n)$. 
Up to extraction of another subsequence, we may assume that these indices do not depend on $n$, and we denote them by $k\neq l$. 
We may also assume that $\a^n$ and $\b^n$ converge to limiting positions $\a^*$ and $\b^*$, respectively. Both points belong to the boundaries $\partial V_i(\s,\bpsi)$ and $\partial V_j(\s,\bpsi)$. Since $V_i(\s,\bpsi)$ and $V_j(\s,\bpsi)$ are not neighbors in $\bVsp$, one must then have $\a^* = \b^*$ and this single vertex is at the intersection between the four cells $V_i(\s,\bpsi)$, $V_j(\s,\bpsi)$, $V_k(\s,\bpsi)$ and $V_l(\s,\bpsi)$. This contradicts Assumption \cref{eq.Gen1}, and the statement that $i$ and $j$ are not neighbors in $\bVsp$ is absurd.

At this point, we have proved that for $\widehat \s \in \R^{dN}$, $\widehat \bpsi\in \R^N$ small enough, the perturbed diagram  $\mathbf{V}(\s + \widehat \s, \bpsi + \widehat \bpsi)$ has the same neighbor relations as $\bVsp$. 
Now, each cell $V_i(\s + \widehat \s, \bpsi + \widehat \bpsi)$ is a closed, convex subset of $\overline D$ whose non degenerate edges are completely characterized by their endpoints.
The calculations below show that all the vertices of this diagram behave as differentiable functions in a neighborhood of $(\s,\bpsi)$. 
We may then define $\m(\widehat \s,\widehat \bpsi)$ as an affine function in restriction to every edge of the cells $V_i(\s,\bpsi)$, 
and then extend the latter into $D$ as a whole by barycentric coordinates extension. This induces a $\Winfty$ mapping $\m(\widehat \s,\widehat \bpsi)$ which is a differentiable of the variables $(\s,\bpsi)$ in the neighborhood of $(\bz,\bz) \in \R^{dN} \times \R^N$.
\end{proof}

We provide the missing ingredients of the above proof by showing that the vertices of the diagram $\bVsp$ are smooth functions of the seed points $\s$ and weights $\bpsi$ when they are slight perturbations of reference values satisfying \cref{eq.Gen0,eq.Gen1,eq.Gen2,eq.Gen3,eq.Gen4,eq.Gen5,eq.Gen6}. Meanwhile, we provide  explicit characterizations of their derivatives, which are useful in particular in the numerical implementation of the methods presented in \cref{sec.deropt}. 
We proceed in the case $d=2$ for simplicity, but a similar analysis could be conducted when $d=3$, up to an increased level of tediousness. We distinguish between several types of vertices, which are illustrated on \cref{fig.Lagvert}.
\par\medskip 

\begin{figure}
\centering
\begin{tabular}{ccc}
\begin{minipage}{0.3\textwidth}
\begin{overpic}[width=1.0\textwidth]{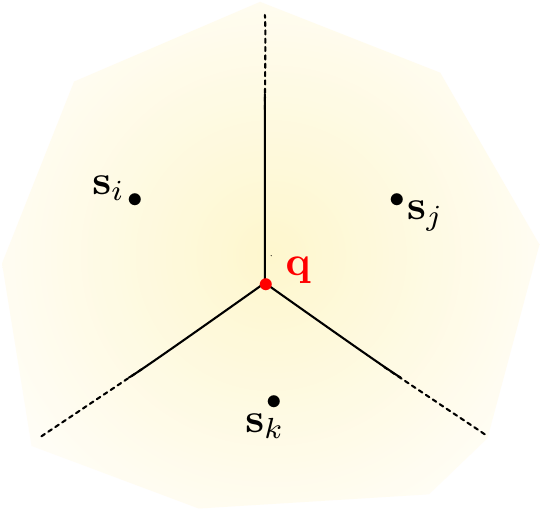}
\put(0,3){\fcolorbox{black}{white}{$a$}}
\end{overpic}
\end{minipage}&
\begin{minipage}{0.39\textwidth}
\begin{overpic}[width=1.0\textwidth]{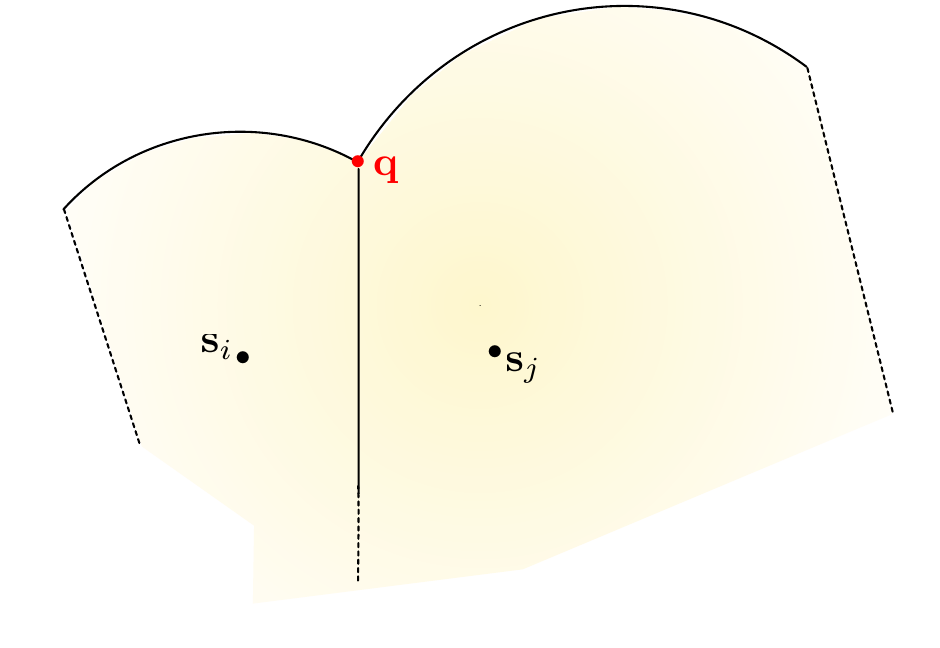}
\put(0,3){\fcolorbox{black}{white}{$b$}}
\end{overpic}
\end{minipage}& 
\begin{minipage}{0.345\textwidth}
\begin{overpic}[width=1.0\textwidth]{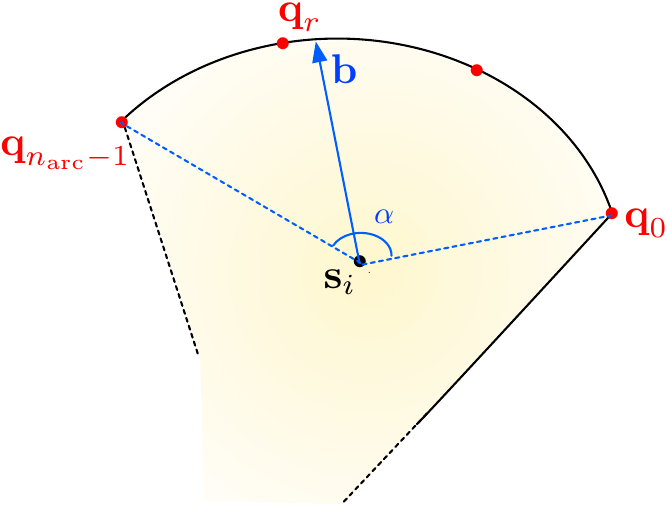}
\put(0,3){\fcolorbox{black}{white}{$c$}}
\end{overpic}
\end{minipage}
\end{tabular}\par\medskip

\begin{tabular}{cc}
\begin{minipage}{0.3\textwidth}
\begin{overpic}[width=1.0\textwidth]{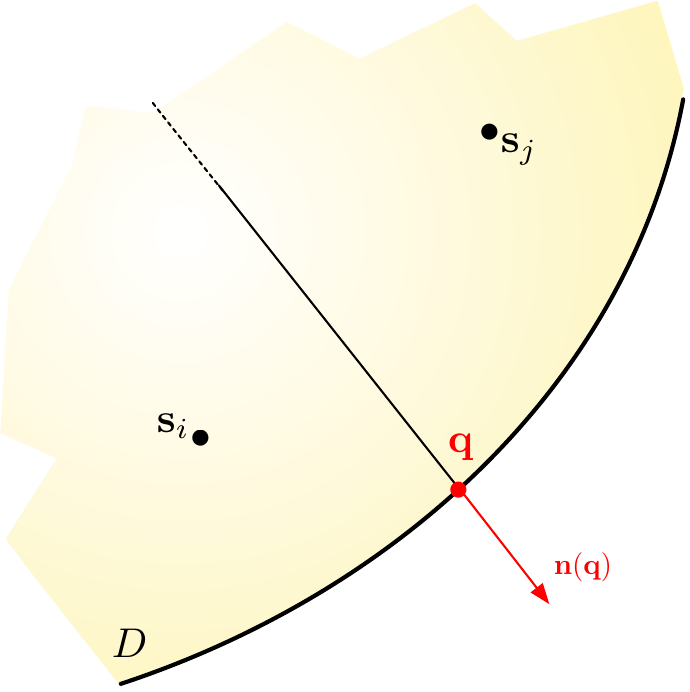}
\put(0,3){\fcolorbox{black}{white}{$d$}}
\end{overpic}
\end{minipage}&
\begin{minipage}{0.39\textwidth}
\begin{overpic}[width=1.0\textwidth]{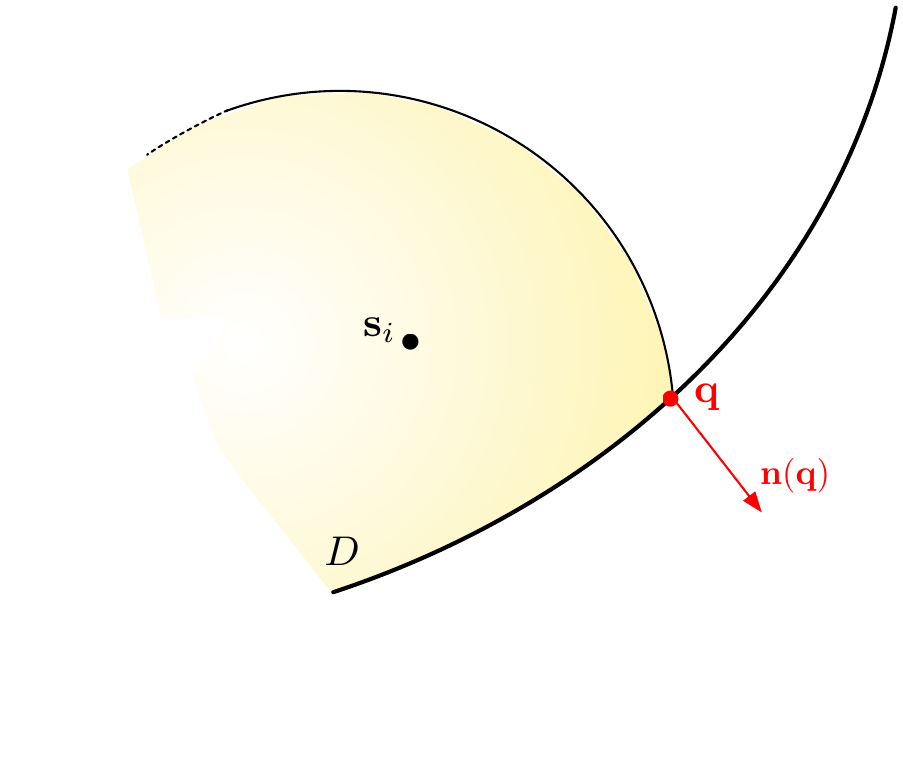}
\put(0,3){\fcolorbox{black}{white}{$e$}}
\end{overpic}
\end{minipage}
\end{tabular}
\caption{\it Different categories of vertices in a diagram $\bVsp$: (a) Case of a vertex $\q$ at the intersection between $3$ cells; (b) Case of a vertex between $2$ cells and the void phase; (c) Case of a vertex resulting from the discretization of an exterior circle; (d) Case of a vertex at the intersection between $2$ cells and the boundary $\partial D$ of the computational domain $D$; (e) Case of a vertex at the intersection between one cell, the void phase, and $\partial D$.}
\label{fig.Lagvert}
\end{figure}

\noindent \textit{Case 1: Vertices at the intersection of three cells.} 
Let $\q \in \R^2$ be the vertex at the intersection between three given cells $V_i(\s,\bpsi)$, $V_j(\s,\bpsi)$ and $V_k(\s,\bpsi)$, for some distinct indices $i,j,k \in \left\{1,\ldots,N\right\}$.  
This vertex is characterized by the following relations:
$$ 
\left\{
\begin{array}{ccc}
| \q - \s_i| ^2 - \psi_i &=& |\q - \s_j|^2 - \psi_j \\[0.4em]
| \q - \s_i| ^2 - \psi_i &=& |\q - \s_k|^2 - \psi_k, 
\end{array}
\right.
$$
which rewrites: 
$$ 
\left\{
\begin{array}{ccc}
 \langle \q , \s_j - \s_i \rangle &=& \frac12(|\s_j|^2 - |\s_i|^2 - (\psi_j-\psi_i))\\[0.4em]
 \langle \q , \s_k - \s_i \rangle &=& \frac12(|\s_k|^2 - |\s_i|^2 - (\psi_k-\psi_i)). 
\end{array}
\right.
$$
Invoking Assumptions \cref{eq.Gen1,eq.Gen2}, this system uniquely determines $\q$ as a smooth function of $\s_i$, $\s_j$, $\s_k$ and $\psi_i$, $\psi_j$, $\psi_k$, 
a dependence which is not made explicit for notational simplicity. 
Taking derivatives in the previous system then leads to the following $2 \times 2$ linear systems for the sensitivities $\left[ \nabla_{\s_i} \q \right]\in \R^{2\times 2}$ and $\frac{\partial \q}{\partial \psi_i} \in \R^2$:
$$ \left( \begin{array}{c}
 \s_j-\s_i \\
 \hline 
\s_k-\s_i
\end{array}
\right)\left[ \nabla_{\s_i} \q\right] = \left( \begin{array}{c}
\q-\s_i \\
 \hline 
\q-\s_i
\end{array}
\right), \text{ and }
 \left( \begin{array}{c}
 \s_j-\s_i \\
 \hline 
\s_k-\s_i
\end{array}
\right)
\frac{\partial \q}{\partial \psi_i}= \left( \begin{array}{c}
\frac12\\[0.2em]
\frac12
\end{array}
\right),
$$
which are invertible on account of Assumption \cref{eq.Gen1}. Similar expressions hold true for the derivatives $\left[ \nabla_{\s_j} \q \right]$, $\left[ \nabla_{\s_k} \q \right]$ and $\frac{\partial \q}{\partial \psi_j}$, $\frac{\partial \q}{\partial \psi_k}$. \par\medskip

\noindent \textit{Case 2: Vertices at the intersection of two cells and void.} 
Let now $\q \in \R^2$ be one vertex at the intersection between the cells $V_i(\s,\bpsi)$, $V_j(\s, \bpsi)$ ($i\neq j$) and the void phase $V_0(\s,\bpsi)$.  
Then, $\q$ satisfies the following system of equations:
$$ 
\left\{
\begin{array}{ccc}
 \langle \q , \s_j - \s_i \rangle &=& \frac12(|\s_j|^2 - |\s_i|^2 - (\psi_j-\psi_i))\\ 
|\q - \s_i|^2 &=& \psi_i.
\end{array}
\right.
$$
Note that there may be several such vertices, and so the above system may accordingly have multiple solutions.
Assumption \cref{eq.Gen2} and the implicit function theorem together imply that $\q$ behaves as a smooth function of $\s_i$, $\s_j$, $\psi_i$ and $\psi_j$ -- actually, any solution to this system induces a smooth branch of solutions for small perturbations of these parameters.
Taking derivatives, it follows:
$$ \left( \begin{array}{c}
 \s_j-\s_i \\
 \hline 
\q-\s_i
\end{array}
\right)\left[\nabla_{\s_i} \q\right] = \left( \begin{array}{c}
\q-\s_i \\
 \hline 
\q - \s_i
\end{array}
\right), \text{ and }
 \left( \begin{array}{c}
 \s_j-\s_i \\
 \hline 
\q-\s_i
\end{array}
\right)\frac{\partial \q}{\partial \psi_i}= \left( \begin{array}{c}
\frac12\\[0.2em]
\frac12
\end{array}
\right).
$$
These systems are invertible on account of Assumption \cref{eq.Gen2}. Again, similar expressions hold for the derivatives $\left[\nabla_{\s_j}\q \right]$ and $\frac{\partial \q}{\partial \psi_j}$.\par\medskip

\noindent \textit{Case 3: Vertices corresponding to the discretization of a circular arc.} 
Consider a circular arc pertaining to the boundary of the cell $V_i(\s,\bpsi)$, and let $\q_r$, $r=0, \ldots, \narc$ be the vertices corresponding to its discretization, enumerated counterclockwise, see \cref{sec.geomcomp}.
Let $\alpha \in (0,2\pi)$ be the angle from $\s_i \q_0$ to $\s_i \q_{\narc}$, which is a smooth function of $\s_i$, $\q_0$ and $\q_{\narc}$, 
and, in turn, as revealed by the previous item, a smooth function of the seed points $\s$ and weights $\bpsi$. 
More precisely, it holds: 
$$ \cos \alpha = \left\langle \frac{\s_i\q_0}{\lvert \s_i \q_0\lvert} , \frac{\s_i\q_{\narc}}{\lvert \s_i \q_{\narc}\lvert} \right\rangle, \:\: \sin \alpha = \det \left( \frac{\s_i\q_0}{\lvert \s_i \q_0\lvert} , \frac{\s_i\q_{\narc}}{\lvert\s_i \q_{\narc}\lvert}\right),$$
whence we infer $\alpha\in (0,2\pi)$ via the formula 
$$ \alpha =  2\text{acot} \left( \frac{1+\cos \alpha}{\sin \alpha}\right)  = \pi - 2\text{atan} \left( \frac{1+\cos \alpha}{\sin \alpha}\right).$$
Then, let us introduce the vector $\b = \s_i \q_0^\perp := (-(s_{i,2}-q_{0,2}),s_{i,1}-q_{0,1})$, 
where $\u^\perp := (-u_2,u_1)$ denotes the $90^{\circ}$ counterclockwise rotate of a vector $\u = (u_1,u_2) \in \R^2$. 
Each vertex $\q_r$ ($r=1,\ldots,\narc-1$) is determined by the following relation: 
$$ \s_i \q_r = \cos (t_r \alpha) \s_i \q_0 + \sin (t_r \alpha) \b, \text{ where } t_r= \frac{r}{\narc}.$$
The differentiability of $\q_r$ and the closed form expressions of its derivatives with respect to $\s$ and $\bpsi$ are straightforward consequences of this formula.\par\medskip

Since the cells of the diagram $\bVsp$ are restricted to the bounded computational domain $D \subset \R^2$, two additional situations should be considered, which are depicted on \cref{fig.Lagvert} (e) (f).
To address them, it is convenient to introduce a level set function $\phi : \R^2 \to \R$ for $D$, that is
$$ \forall \x \in \R^2, \quad \left\{
\begin{array}{cl}
\phi(\x) < 0 & \text{if } \x \in D, \\
\phi(\x) = 0 & \text{if } \x \in \partial D, \\
\phi(\x) > 0 & \text{otherwise}.
\end{array}\right.
 $$
Then, the unit normal vector $\n(\x)$ to $\partial D$ pointing outward $D$ is given by $\n(\x) = \frac{\nabla \phi(\x)}{\lvert \nabla \phi(\x)\vert}$ for $\d s$ a.e. $\x \in \partial D$,
see e.g. \cite{osher2006level,sethian1999level} about general features of such a level set representation of domains. 
 \par\medskip
 \noindent \textit{Case 2b: Vertices corresponding to the intersection between two cells and $\partial D$.} 
Let $\q \in \R^2$ be one vertex at the intersection between the cells $V_i(\s,\bpsi)$, $V_j(\s,\bpsi)$ and $\partial D$.
Then the following system of equations holds: 
\begin{equation*}\label{eq.2ptdD} 
\left\{
\begin{array}{ccc}
\lvert  \q - \s_i \lvert ^2 - \psi_i &=& \lvert \q - \s_j \lvert ^2 - \psi_j \\ 
\phi(\q)&=& 0. 
\end{array}
\right.
\end{equation*}
As in the previous items, the combination of Assumptions \cref{eq.Gen4,eq.Gen5} with the implicit function theorem shows that this $\q$ behaves as a smooth function of the parameters $\s$ and $\bpsi$ when the latter are slightly perturbed from reference values satisfying \cref{eq.Gen0,eq.Gen1,eq.Gen2,eq.Gen3,eq.Gen4,eq.Gen5,eq.Gen6}. 
In the meantime, taking derivatives in the above equations yields immediately the following systems:
$$ \left( \begin{array}{c}
 \s_j-\s_i \\
 \hline 
\n(\q)
\end{array}
\right)\left[ \nabla_{\s_i} \q\right] = \left( \begin{array}{c}
\q-\s_i \\
 \hline 
\bz
\end{array}
\right), \text{ and }
 \left( \begin{array}{c}
 \s_j-\s_i \\
 \hline 
\n(\q)
\end{array}
\right)\frac{\partial \q}{\partial \psi_i}= \left( \begin{array}{c}
\frac12\\
0
\end{array}
\right),
$$
which are invertible on account of \cref{eq.Gen4,eq.Gen5}. Similar relations allow to characterize the derivatives $\left[\nabla_{\s_j}\q \right]$ and $\frac{\partial \q}{\partial \psi_j}$.\par\medskip

 \noindent \textit{Case 3b: Vertices at the intersection between one exterior circle and $\partial D$.} 
Let $\q \in \R^2$ be one vertex lying at the intersection between the cell $V_i(\s,\bpsi)$ ($i=1,\ldots,N$), the void phase $V_0(\s,\bpsi)$ and $\partial D$.
Then,  
$$ 
\left\{
\begin{array}{ccl}
| \q - \s_i| ^2 - \psi_i &=& 0 \\ 
\phi(\q)&=& 0 ,\\ 
\end{array}
\right.
$$
Using \cref{eq.Gen0,eq.Gen1,eq.Gen2,eq.Gen3,eq.Gen4,eq.Gen5,eq.Gen6} and the implicit function theorem, we see that this system characterizes a smooth function of $\s$ and $\bpsi$. In the meantime,
taking derivatives into these equations, we obtain the following invertible systems:
$$ \left( \begin{array}{c}
 \q-\s_i \\
 \hline 
\n(\q)
\end{array}
\right)\left[ \nabla_{\s_i} \q\right] = \left( \begin{array}{c}
\q-\s_i \\
 \hline 
\bz
\end{array}
\right), \text{ and }
 \left( \begin{array}{c}
 \q-\s_i \\
 \hline 
\n(\q)
\end{array}
\right)\frac{\partial \q_r}{\partial \psi_i}= \left( \begin{array}{c}
\frac12\\
0
\end{array}
\right).
$$



\section{Second order derivatives of the Kantorovic functional with respect to seeds and weights}\label{app.derF}

\noindent Let $\s \in \R^{dN}$, $\bpsi \in \R^N$ be generic seed points and weights, in the sense that \cref{eq.Gen0,eq.Gen1,eq.Gen2,eq.Gen3,eq.Gen4,eq.Gen5,eq.Gen6} hold true. Let also $\bnu = \left\{ \nu_1,\ldots, \nu_N \right\}$ be a vector of prescribed measures. Restricting to the case of two space dimensions $d=2$ for simplicity, this section details the calculation of the partial
derivatives of the derivative $\bF : \R^{dN}_{\s} \times \R^N_{\bnu} \times \R^N_{\bpsi} \to \R^N$ of the Kantorovic functional introduced in \cref{prop.psi}: 
$$ \bF(\s,\bnu,\bpsi) =  \Big(\nu_i- |V_i(\s,\bpsi) | \Big)_{i=1,\ldots,N}.$$
Let us recall that this result is useful for a variety of purposes in our framework: 
on the one hand, it is the key ingredient of the Newton-Raphson algorithm implemented for the resolution of \cref{eq.volPsi}, see \cref{sec.seedstoLag}.
On the other hand, it is needed in the procedure of \cref{sec.deropt} for expressing the derivative of a quantity with respect to the vertices $\q$ of the diagram in terms of the defining seed points $\s$ and weights $\bpsi$.  
Again, the results of this appendix are not completely new, see 
e.g.  \cite{de2012blue,merigot2021optimal,nivoliers2013approximating} where fairly similar calculations are conducted in a formal fashion.
The calculation of the derivatives of $\bF$ with respect to the weights $\bpsi$ is also achieved in \cite{merigot2021optimal}, and in \cite{de2019differentiation} in a more general context. Here, we present an elementary mixture of arguments used in the aforementioned references. \par\medskip

Let us introduce a few additional notations, which are illustrated on \cref{fig.Lagmovseed}:
\begin{itemize}
\item For any index $i=1,\ldots,N$ and any neighbor $j \in \calN_i$ of $i$, 
we recall that $\be_{ij}$ is the line segment $V_i(\s,\bpsi) \cap  V_j(\s,\bpsi)$, and we denote by $\m_{ij}$ the midpoint of $\be_{ij}$.
\item For $i=1,\ldots,N$, we denote by $\calC_{i,r}$, $r=1,\ldots,n_{c,i}$ the circular arcs composing the boundary of the cell $V_i(\s,\bpsi)$; the value $n_{c,i}=0$ indicates that the cell is not on the boundary of the associated shape $\Omega$. We also let $\calC_i := \bigcup_{r=1}^{n_{c,i}} \calC_{i,r}$. 
\item For $i=1,\ldots,N$ and $r= 1,\ldots,n_{c,i}$, the endpoints of the circular arc $\calC_{i,r}$ are denoted by $\q_{i,r}^0$, $\q_{i,r}^1$ when the latter is oriented counterclockwise, and $\theta_{i,r} \in (0,2\pi)$ is its angular aperture.
\item For $i=1,\ldots,N$, we denote by $\n_i$ the unit normal vector field to $\partial V_i(\s,\bpsi)$, pointing outward $V_i(\s,\bpsi)$. This latter is given for $\d s$-almost every point $\x \in \partial V_i(\s,\bpsi)$ by the following expressions:
$$\n(\x) = \left\{\begin{array}{cl}
\frac{\s_j-\s_i}{\lvert \s_j-\s_i\lvert} & \text{if }\x \in \be_{ij} \text{ for some } j \in \calN_i,\\
\frac{\x-\s_i}{\lvert\x-\s_i\lvert} & \text{if } \x \in \calC_i.
\end{array}
\right.$$
\end{itemize}
The result of interest in this appendix is the following.

\begin{proposition}\label{prop.dFi}
Let $\s = \left\{ \s_i \right\}_{i=1,\ldots,N} \in \R^{dN}$ and $\bpsi = \left\{ \psi_i\right\}_{i=1,\ldots,N} \in \R^N$ be sets of seed points and weights satisfying the genericity Assumptions \cref{eq.Gen0,eq.Gen1,eq.Gen2,eq.Gen3,eq.Gen4,eq.Gen5,eq.Gen6}. 
Then, for all $i=1,\ldots,N$, the function $\bF$ is differentiable at $(\s,\bnu,\bpsi)$.
It holds, for $i=1,\ldots,N$:
$$ \nabla_{\s_i} F_i(\s,\bnu,\bpsi) = -\sum\limits_{j \in N_i} \frac{|\be_{ij}|}{|\s_j - \s_i |}( \m_{ij} - \s_i) - \sum\limits_{r=1}^{n_{c,i}}  (\q_{i,r}^0-\q_{i,r}^1)^\perp , \:\: \text{ and for } j \neq i \:\:\:  \nabla_{\s_j} F_i(\s,\bnu,\bpsi) = \frac{|\be_{ij}|}{|\s_j - \s_i |}( \m_{ij} - \s_i),$$
where $\u^\perp := (-u_2,u_1)$ is the 90$^\circ$ counterclockwise rotate of a vector $\u = (u_1,u_2)\in \R^2$.
Besides,
$$ \frac{\partial F_i} {\partial \psi_i} (\s,\bnu,\bpsi) =   -\frac{1}{2} \left ( \sum\limits_{j \in N_i}  \frac{\lvert \be_{ij} \lvert }{|\s_j - \s_i|}\right)  - \sum\limits_{r=1}^{n_{c,i}} \theta_{i,r}, \:\: \text{ and for } j \neq i \:\:\:  \frac{\partial F_i} {\partial \psi_j} (\s,\bnu,\bpsi)  =  \frac{1}{2} \frac{\lvert \be_{ij}\lvert}{|\s_j - \s_i |}.$$
\end{proposition}

\begin{figure}
\centering
\begin{tabular}{cc}
\begin{minipage}{0.59\textwidth}
\begin{overpic}[width=1.0\textwidth]{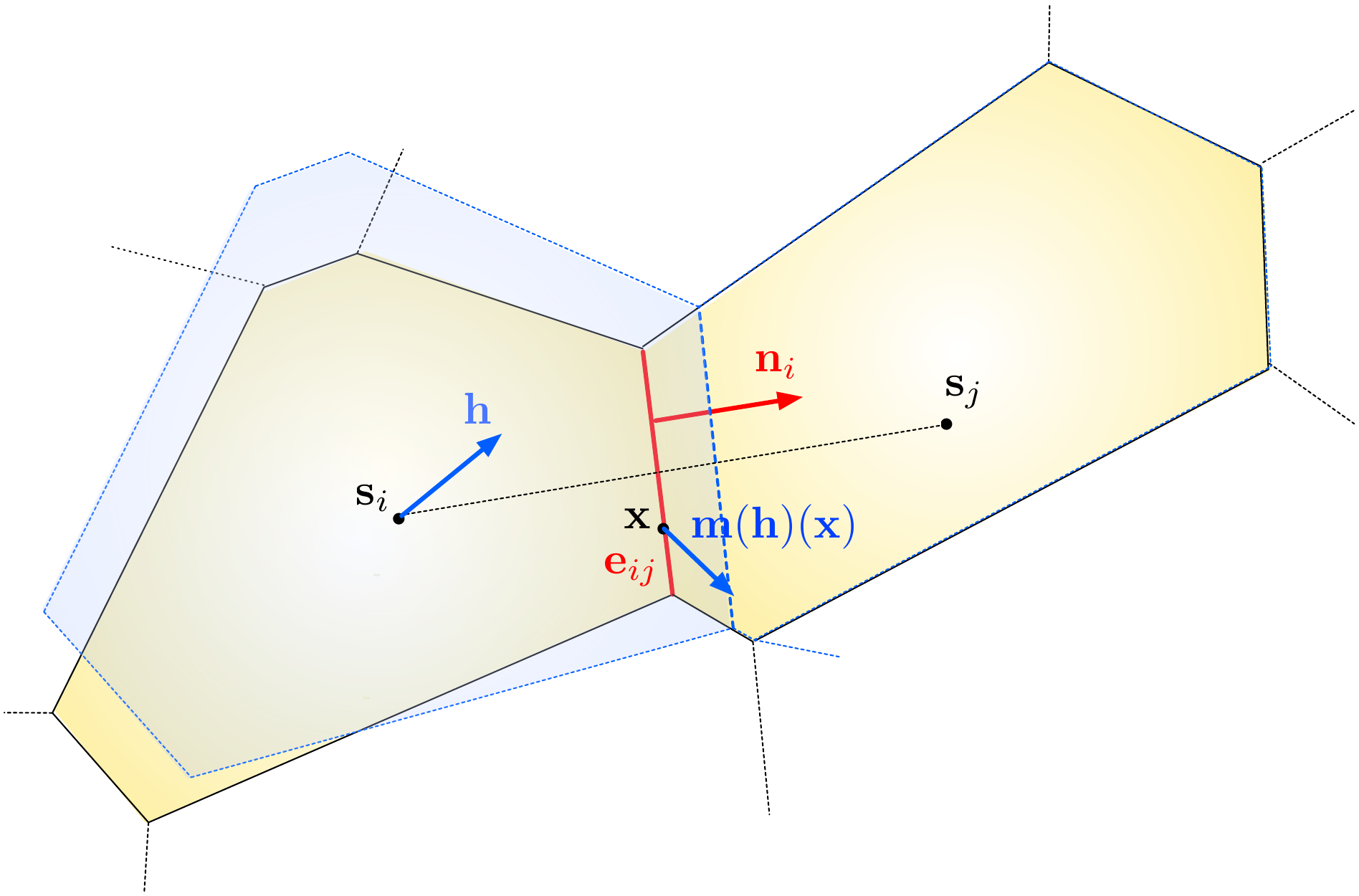}
\put(0,3){\fcolorbox{black}{white}{$a$}}
\end{overpic}
\end{minipage}&
\begin{minipage}{0.41\textwidth}
\begin{overpic}[width=1.0\textwidth]{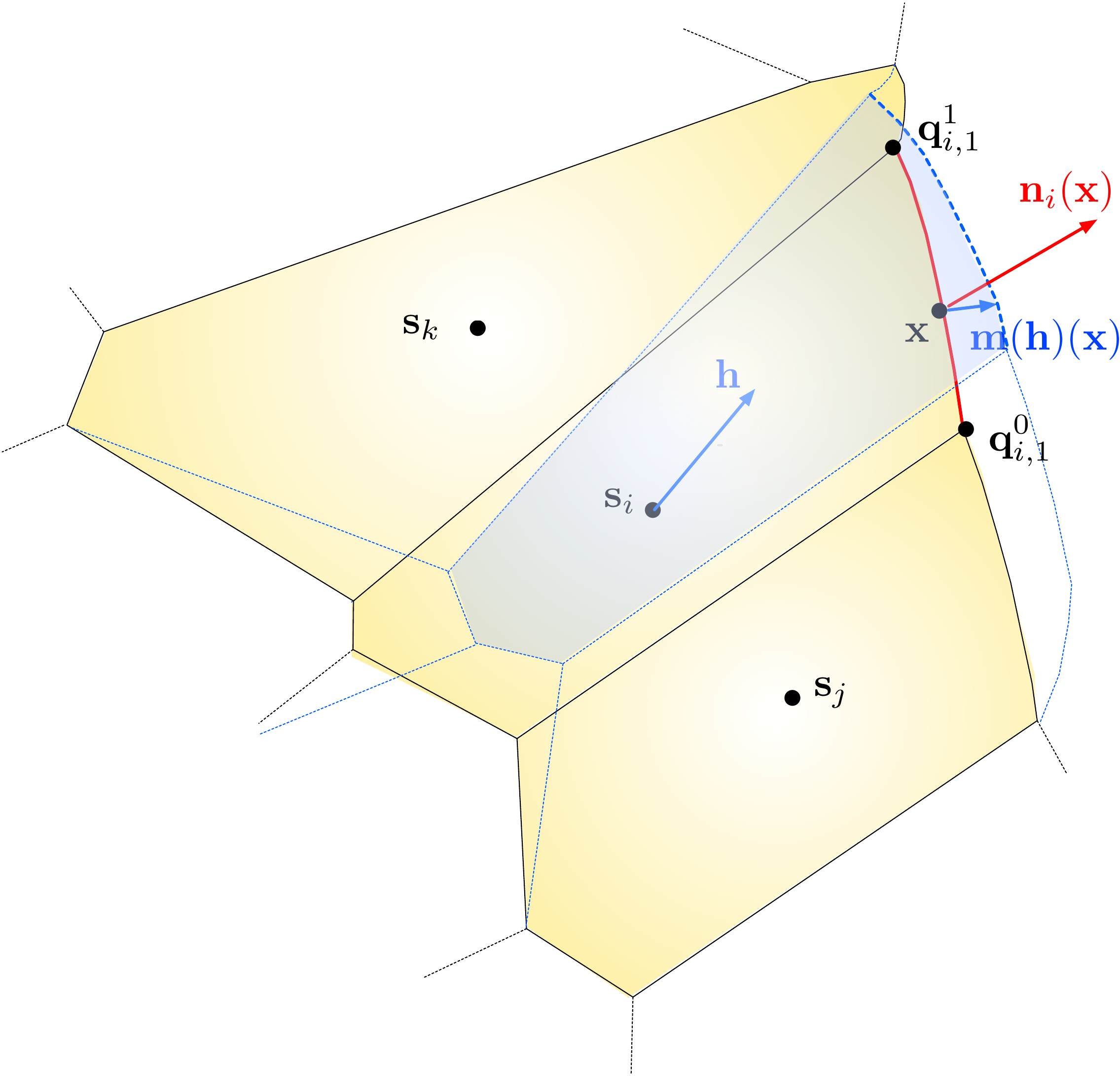}
\put(0,3){\fcolorbox{black}{white}{$b$}}
\end{overpic}
\end{minipage}
\end{tabular}
\caption{\it Deformation of (a) An interior cell; (b) A cell containing at least one circular arc in the diagram $\bVsp$, induced by a small perturbation of the corresponding seed point, as considered in \cref{app.derF}.}
\label{fig.Lagmovseed}
\end{figure}

\begin{proof}
Judging from the definition \cref{eq.implicitPsi} of the function $\bF(\s,\bnu,\cdot)$, the proof boils down to the calculation of the derivatives of the measures 
$$A_i(\s,\bpsi) := \lvert V_i(\s,\bpsi) \lvert, \quad i=1,\ldots,N$$ 
with respect to the seed points $\s \in \R^{dN}$ and weights $\bpsi \in \R^N$ of the diagram $\bVsp$. 
The main idea builds on the conclusion of \cref{prop.derver}, whereby their exists $\e >0$ such that for any perturbation $(\widehat{\s},\widehat{\bpsi}) \in \R^{dN}_{\s} \times \R^N_{\bpsi}$ with $\lvert \widehat{\s}\lvert + \lvert\widehat{\bpsi}\lvert < \e$, the diagram $\mathbf{V}(\s + \widehat{\s}, \bpsi + \widehat{\bpsi})$ is obtained from the reference diagram $\bVsp$ by application of a ``smooth'' deformation 
$$\Id +\m(\widehat{\s},\widehat{\bpsi}), \text{ where } \m(\widehat{\s},\widehat{\bpsi}) \in W^{1,\infty}(\R^2,\R^2).$$
The derivative of $(\widehat{\s},\widehat{\bpsi}) \mapsto A_i(\s+\widehat{\s},\bpsi + \widehat{\bpsi})$ thus results from the combination of the chain rule with the formula \cref{eq.Volprime} for the derivative of the volume of a domain with respect to arbitrary perturbations of its boundary.\par\medskip

We first apply this strategy to calculate the derivative of $A_i(\s,\bpsi)$ with respect to $\s_i$, for a given index $i = 1,\ldots, N$. For any vector $\h \in \R^2$, let $\s(\h) := (\s_1,\ldots,\s_{i-1} , \s_i + \h , \s_{i+1} , \ldots, \s_N)$ be the collection of seed points resulting from a translation $\h$ of the $i^{\text{th}}$ element of $\s$. 
As we have recalled, when $\lvert \h \lvert < \e$, the cell $V_i(\s(\h),\bpsi )$ is of the form 
\begin{equation}\label{eq.pertVish}
 V_i(\s(\h), \bpsi) = (\Id + \m(\h)) \Big( V_i(\s,\bpsi) \Big),
 \end{equation}
for a certain differentiable mapping $B(\bzero,\e) \ni \h \mapsto \m(\h) \in W^{1,\infty}(\R^2,\R^2)$. 
The combination of the chain rule with formula \cref{eq.Volprime} yields: 
\begin{equation}\label{eq.dAidsi}
\begin{array}{>{\displaystyle}cc>{\displaystyle}l}
 \frac{\partial A_i}{\partial \s_i}(\s,\bpsi)(\h) &=&  \int_{\partial V_i(\s,\bpsi)} \left\langle \frac{\partial \m}{\partial \h}(\bzero)(\h)(\x), \n_i(\x)\right\rangle \:\d s(\x) \\[1em]
 &=& \sum\limits_{j \in N_i} \int_{\be_{ij}} \left\langle \frac{\partial \m}{\partial \h}(\bzero)(\h)(\x) , \n_i(\x)\right\rangle \:\d s(\x) + \sum\limits_{r=1}^{n_{c,i}} \int_{\calC_{i,r}} \left\langle\frac{\partial \m}{\partial \h}(\bzero)(\h)(\x), \n_i(\x)\right\rangle  \:\d s(\x), \\[1em]
 \end{array}
\end{equation}
and we now proceed to glean information about the derivative $\R^2 \ni \h \mapsto \frac{\partial \m}{\partial \h}(\bzero)(\h) \in W^{1,\infty}(\R^2,\R^2)$.

Let $j \in \calN_i$, and $\x$ be an arbitrary point on the edge $\be_{ij}$. For any perturbation $\h \in B(\bzero,\e)$, the point $\m(\h)(\x)$ belongs to the intersection of the cells $V_i(\s(\h),\bpsi)$ and $V_j(\s(\h),\bpsi)$, which implies that:
$$ \lvert \x + \m(\h)(\x) - \s_i - \h \lvert^2 -\psi_i =   \lvert \x + \m(\h)(\x) - \s_j \lvert^2  -\psi_j.$$  
A simple calculation then yields:
$$ 2 \langle \m(\h)(\x), \s_j - \s_i \rangle = 2 \langle \x - \s_i,\h \rangle +  \lvert \x - \s_j \lvert ^2 - \lvert \x-\s_i \lvert^2 -(\psi_j - \psi_i) + \o(\h), $$
and so: 
$$ \left\langle \frac{\partial \m}{\partial \h}(\bz)(\h)(\x) , \s_j - \s_i \right\rangle =   \langle \x - \s_i , \h \rangle.$$
As a result, we have proved the expression
\begin{equation}\label{eq.dmdhjNi}
\forall \x \in \be_{ij}, \quad \left\langle \frac{\partial \m}{\partial \h}(\bz)(\h)(\x) ,\n_{i}  \right\rangle =   \left\langle  \frac{\x - \s_i}{\vert \s_j - \s_i \lvert} , \h \right\rangle.
\end{equation}

Let now $\x$ be a point on one of the circular arcs $\calC_{i,r}$, $r=1,\ldots,n_{c,i}$. Again, the definition of $\calC_{i,r}$ readily implies that, for any perturbation vector $\lvert \h \lvert < \e$, we have:
$$\lvert \x + \m(\h)(\x) - \s_i - \h \lvert^2 = \psi_i.$$
Hence, we obtain 
$$ \left\langle \m(\h)(\x), \x- \s_i \rangle =   \langle \x- \s_i , \h \right\rangle - \frac{1}{2} |\x-\s_i|^2 + \frac{1}{2} \psi_i + \o(\h),$$
and so:
\begin{equation}\label{eq.dmdhjCir}
\forall \x \in \calC_{i,r}, \quad \left\langle \frac{\partial \m}{\partial \h}(\bz)(\h)(\x) , \n_i(\x) \right\rangle =  \left\langle \frac{\x - \s_i}{|\x- \s_i|}, \h \right\rangle.
\end{equation}

Eventually, combining \cref{eq.dAidsi} with \cref{eq.dmdhjNi,eq.dmdhjCir}, we obtain:
$$\begin{array}{>{\displaystyle}cc>{\displaystyle}l}
 \frac{\partial A_i}{\partial \s_i}(\s,\bpsi)(\h) 
 &=&  \left\langle \sum\limits_{j \in \calN_i} \int_{\be_{ij} }\frac{\x - \s_i}{|\s_j - \s_i|} \:\d s(\x), \h \right\rangle + \left\langle \sum\limits_{r=1}^{n_{c,i}} \int_{\calC_{i,r}}\frac{\x - \s_i}{|\x- \s_i|} \:\d s(\x), \h \right\rangle \\[1em]
 &=&  \left\langle \sum\limits_{j \in \calN_i} \frac{\lvert\be_{ij} \lvert}{|\s_j - \s_i|} (\m_{ij} - \s_i), \h \right\rangle + \left\langle \sum\limits_{r=1}^{n_{c,i}} (\q_{i,r}^0-\s_i)^\perp - (\q_{i,r}^1-\s_i)^\perp , \h \right\rangle,
 \end{array}$$
where the second line follows from an elementary calculation. This yields the desired formula. \par\medskip

Let us now apply the same strategy to calculate the derivative of $A_i(\s,\bpsi)$ with respect to the weight $\psi_i$. A perturbation of $\psi_i$ of the form $\psi_i+h$ for $\lvert h \lvert < \e$ incurs a deformation of the $i^{\text{th}}$ cell of the diagram of the form: 
$$ V_i(\s, \bpsi + h\be_i) = (\Id + \m(h))\Big(V_i(\s,\bpsi)\Big),$$
for some differentiable mapping $B(0,\e) \ni h \mapsto \m(h) \in W^{1,\infty}(\R^2,\R^2)$. Again, the combination of \cref{eq.Volprime} with the chain rule yields: 
\begin{equation}\label{eq.dAidpsii}
\begin{array}{>{\displaystyle}cc>{\displaystyle}l}
 \frac{\partial A_i}{\partial \psi_i}(\s,\bpsi)(h) &=&  \int_{\partial V_i(\s,\bpsi)} \left\langle \frac{\partial \m}{\partial h}(0)(h)(\x), \n_i(\x)\right\rangle \:\d s(\x) \\[1em]
 &=& \sum\limits_{j \in N_i} \int_{\be_{ij}} \left\langle \frac{\partial \m}{\partial h}(0)(h)(\x) , \n_i(\x)\right\rangle \:\d s(\x) + \sum\limits_{r=1}^{n_{c,i}} \int_{\calC_{i,r}} \left\langle\frac{\partial \m}{\partial h}(0)(h)(\x), \n_i(\x)\right\rangle  \:\d s(\x), \\[1em]
 \end{array}
\end{equation}
and we now aim to characterize the derivative $\R \ni h \mapsto \frac{\partial \m}{\partial h}(0)(h) \in W^{1,\infty}(\R^2,\R^2)$.

To this end, let us first consider an index $j \in \calN_i$, and let $\x \in \be_{ij}$ be given. For any $h \in B(0,\e)$, it holds:
$$ |\x + \m(h)(\x) - \s_i |^2 -\psi_i - h =   |\x + \m(h)(\x) - \s_j |^2  -\psi_j,$$
and so 
$$ \langle  \m(h)(\x) , \s_j - \s_i \rangle = \frac{h}{2} + |\x - \s_j|^2 - |\x-\s_i|^2 -(\psi_j - \psi_i) + \o(h).$$  
We then obtain the following formula involving the derivative of the mapping $m$:
$$ \forall \x \in \be_{ij}, \quad \left\langle \frac{\partial \m}{\partial h}(0)(h)(\x) , \s_j - \s_i \right\rangle =  \frac{h}{2},$$
which rewrites:
\begin{equation}\label{eq.dmdhjNipsi}
\forall \x \in \be_{ij}, \quad \left\langle \frac{\partial \m}{\partial h}(0)(h)(\x) , \n_i(\x) \right\rangle =  \frac{h}{2} \frac{1}{\lvert \s_j - \s_i\lvert}.
\end{equation}

Let us now consider a point $\x$ located on one of the circular arcs $\calC_{i,r}$, $r=1,\ldots,n_{c,i}$. It holds, for all $h \in B(0,\e)$: 
$$ |\x + m(h)(\x) - \s_i  |^2 = \psi_i + h,$$
and so: 
\begin{equation}\label{eq.dmdhjCirpsi}
\left\langle \frac{\partial \m}{\partial h}(0)(h), \n_i(\x) \right\rangle = \frac{1}{2|\x-\s_i |} = \frac{1}{2} \psi_i^{-1/2}.
\end{equation}

Eventually, combining \cref{eq.dAidpsii} with \cref{eq.dmdhjNipsi,eq.dmdhjCirpsi}, we arrive at:
$$\begin{array}{>{\displaystyle}cc>{\displaystyle}l}
 \frac{\partial A_i}{\partial \psi_i}(\s,\bpsi)(h) 
 &=&  \left\langle \sum\limits_{j \in \calN_i} \int_{\be_{ij} }\frac{1}{2} \frac{1}{\lvert \s_j - \s_i\lvert} \:\d s(\x), h \right\rangle + \left\langle \sum\limits_{r=1}^{n_{c,i}} \int_{\calC_{i,r}} \frac{1}{2} \psi_i^{-1/2} \:\d s(\x), h \right\rangle \\[1em]
 &=&  \left\langle \frac12\sum\limits_{j \in \calN_i} \frac{\lvert\be_{ij} \lvert}{|\s_j - \s_i|}, h \right\rangle + \left\langle \frac12\sum\limits_{r=1}^{n_{c,i}} \theta_{i,r}, h \right\rangle,
 \end{array}$$
which yields the desired formula.\par\medskip

The expressions of the partial derivatives of the mapping $(\s,\bpsi) \mapsto A_i(\s,\bpsi)$ with respect to the position $\s_j$ and weight $\psi_j$ of a neighboring index $j \in \calN_i$ are obtained in a completely similar fashion, and we omit the details for brevity. The proof of \cref{prop.dFi} is therefore complete.
\end{proof}

\begin{remark}\label{rem.diffverArea}
The differentiability of the mapping $\bpsi \mapsto K(\s,\bnu,\bpsi)$ -- which essentially involves the areas $\lvert V_i(\s,\bpsi) \lvert$ of the cells of the diagram $\bVsp$ -- holds true under weaker assumptions than the collection \cref{eq.Gen0,eq.Gen1,eq.Gen2,eq.Gen3,eq.Gen4,eq.Gen5,eq.Gen6}, which implies the differentiability of all the individual vertices of the diagram, see for instance \cite{de2019differentiation,merigot2021optimal}. To keep the presentation elementary and self-contained, we limit ourselves with the presented results, which are sufficient for our purpose.
\end{remark}

\begin{corollary}\label{cor.irred}
Let $\s \in \R^{dN}$ and $\bpsi \in \R^N$ be collections of seed points and weights 
satisfying \cref{eq.Gen0,eq.Gen1,eq.Gen2,eq.Gen3,eq.Gen4,eq.Gen5,eq.Gen6}. 
Then, $\left[\nabla_{\bpsi} \bF(\s,\bnu,\bpsi) \right]$ is a negative definite $N \times N$ matrix, which is therefore invertible. 
\end{corollary}
\begin{proof}
The above assumptions imply in particular that the diagram $\bVsp$ contains no empty cell, so that the weights $\psi_i$ are positive for $i=1,\ldots,N$ (see Step 2 in the next \cref{app.existunique} where the same argument is used).
Then, the matrix $\left[\nabla_{\bpsi} \bF(\s,\bnu,\bpsi) \right]$ is strictly diagonal dominant, and by a classical argument, it is invertible. 
Since its diagonal entries are negative, the claim follows. 
\end{proof}

\begin{remark}
In the case where the void phase $V_0(\s,\bpsi)$ is empty, $\left[\nabla_{\bpsi} \bF(\s,\bnu,\bpsi) \right]$ has a one-dimensional kernel, 
which reflects the fact that the collection of weights $\bpsi$ realizing a given measure constraint for each cell is unique up to addition of a common real number to all the $\psi_i$. 
\end{remark}

\section{Proof of \cref{prop.psi}}\label{app.existunique}

\noindent The proof is decomposed into four steps. \par\medskip

\noindent \textit{Step 1: We prove that the function $\R^N \ni \bpsi \mapsto K(\s,\bnu,\bpsi) \in \R$ defined by \cref{eq.defK} is concave, differentiable, and we calculate its derivative.} \par\medskip

\noindent 
For $\x \in X$, we denote for convenience $\s_0 \equiv \s_0(\x) = \x$. 
For a given weight vector $\bpsi \in \R^N$, let us define the measurable mappings $I_{\bpsi}: X \to \left\{0,\ldots,N\right\}$ and $T_{\bpsi} : X \to X$ by:
$$ I_{\bpsi}(\x) = \argmin\limits_{i=0,\ldots,N} \Big( |\x - \s_i|^2 - \psi_i\Big), \text{ and } T_{\bpsi}(x) = \s_{I_{\bpsi}(\x)},$$
where we recall the convention $\psi_0 = 0$. 
Note that the points $\x \in X$ where the above minimum is not uniquely attained form a set of Lebesgue measure $0$, so that $I_{\bpsi}$ and $T_{\bpsi}$ are well-defined a.e. on $X$.
This definition immediately implies that, for any measurable mapping $I : X \to \left\{0,\ldots,N\right\}$, it holds:
$$ |\x - \s_{I_{\bpsi}(\x)}|^2 - \psi_{I_{\bpsi}(\x)} \leq |\x - \s_{I(\x)}|^2 - \psi_{I(\x)} \text{ for a.e. } \x \in X.$$
It follows that, for any weight vectors $\bpsi, \bvphi \in \R^N$: 
\begin{equation}\label{eq.ineqKproof}
\begin{array}{>{\displaystyle}cc>{\displaystyle}l}
 K(\s , \bnu, \bvphi) &=& \int_D \Big( |\x - \s_{I_\bvphi(\x)}(\x)| - \varphi_{I_\bvphi(\x)} \Big) \:\d \x + \sum\limits_{i=1}^N \nu_i \varphi_i \\
  &\leq& \int_D \Big( |\x - \s_{I_\bpsi(\x)}(\x)| - \varphi_{I_{\bpsi(\x)}} \Big) \:\d \x + \sum\limits_{i=1}^N \nu_i \varphi_i \\
 &=& \sum\limits_{i=0}^N \int_{V_i(\s,\bpsi)}   \left(|\x - \s_i |^2 - \varphi_i  \right) \d \x + \sum\limits_{i=1}^N \nu_i \varphi_i \\[1.5em]
 &=& K(\s,\bnu, \bpsi) + \langle \bF(\s,\bnu,\bpsi) , \bvphi - \bpsi \rangle, 
 \end{array} 
 \end{equation}
 where we recall that the components of $\bF(\s, \bnu,\bpsi) \in \R^N$ are defined by:
 \begin{equation}\label{eq.defFsphi}
  F_i(\s,\bnu,\bpsi) :=  \nu_i- |V_i(\s,\bpsi) |, \quad i=1,\ldots,N.
  \end{equation}
Hence, the superdifferential $\partial_{\bpsi} K(\s,\bnu,\bpsi)$ of the mapping $\bvphi \mapsto K(\s,\bnu,\bvphi)$ at an arbitrary weight vector $\bpsi \in \R^N$ contains $\bF(\s,\bnu,\bpsi) \in \R^N$,
and it is in particular non empty. 
On the other hand, the inequality \cref{eq.ineqKproof} expresses $K(\s,\bnu,\cdot)$ as the minimum of a collection of affine functions:
$$ K(\s,\bnu,\bvphi) = \min\limits_{\bpsi \in \R^N} \Big( K(\s,\bnu,\bpsi) +  \langle \bF(\s,\bnu,\bpsi ),  \bvphi - \bpsi \rangle \Big), \quad \bvphi \in \R^N.$$
This discussion shows that the mapping $\bpsi \mapsto K(\s,\bnu,\bpsi)$ is concave, and thus differentiable at a.e. $\bpsi \in \R^N$; besides, its gradient at any point $\bpsi$ where it is differentiable equals $\nabla_\bpsi K(\s,\bnu,\bpsi) = \bF(\s,\bnu,\bpsi)$, see \cite{rockafellar1970convex}, Th. 25.4.  \par\medskip

The $\calC^1$ character of $\bpsi \mapsto K(\s,\bnu,\bpsi)$ follows from the same argument as in the proof of Th. 40 in \cite{merigot2021optimal}.
At first, a simple application of the Lebesgue dominated convergence theorem shows that the mapping $\bpsi \to \bF(\s,\bnu,\bpsi)$ is continuous on $\R^N$. 
On the other hand, according to Th. 25.6 in \cite{rockafellar1970convex} the superdifferential $\partial_\bpsi K(\s,\bnu,\bpsi)$ at any vector $\bpsi \in \R^N$ reads: 
$$ \partial_\bpsi K(\s,\bnu,\bpsi) = \left\{ \lim\limits_{n \to \infty} \g^n, \:\: \g^n = \nabla_\bpsi K (\s,\bnu,\bpsi^n), \:\: \bpsi^n \xrightarrow{n\to \infty} \bpsi \text{ and } K(\s,\bnu,\cdot) \text{ is differentiable at } \bpsi^n \right\}.$$
Now, the continuity of $\bF(\s,\bnu,\cdot)$ implies that, for any sequence $\bpsi^n \to \bpsi$ such that $K(\s,\bnu,\cdot)$ is differentiable at $\bpsi^n$, the unique element $\bF(\s,\bnu,\bpsi^n)= \nabla_\bpsi K (\s,\bnu,\bpsi^n)$ in $\partial_{\bpsi} K(\s,\bnu,\bpsi^n)$ converges to $\bF(\s,\bnu,\bpsi)$ as $n \to \infty$. 
The superdifferential $\partial_\bpsi K(\s,\bnu,\bpsi)$ is thus reduced to a single point, and so $K(\s,\bnu,\cdot)$ is differentiable on $\R^N$, with continuous derivative $\bF(\s,\bnu,\cdot)$. 
This completes the first step in the proof of \cref{prop.psi}; in passing, we have proved that the maximizers $\bpsi \in \R^N$ in \cref{eq.maxK} are characterized by the equation 
$$ \bF(\s,\bnu,\bpsi) = \bzero.$$
\par\medskip

 \noindent \textit{Step 2: We prove the existence of a weight vector $\bpsi^*$ satisfying \cref{eq.volPsi}.}\par\medskip
 
 \noindent To this end, we adapt the argument of the proof of Corollary 39 in \cite{merigot2021optimal}. 
 This starts with a few remarks: 
 \begin{itemize}
 \item For any $\bpsi \in \R^N$ and $i=1,\ldots,N$, the mapping $ \R \ni t \mapsto \lvert V_i(\s,\bpsi + t \be_i) \lvert$ is non decreasing. 
 \item For any $\bpsi \in \R^N$, and $i,j= 1,\ldots,N$, $i\neq j$, the mapping $ \R \ni t \mapsto \lvert V_j(\s,\bpsi + t \be_i) \lvert$ is non increasing. 
 \item If the cell $V_i(\s,\bpsi)$ is non empty, then the weight $\psi_i$ must satisfy
 $ \psi_i \geq 0$, since any point $\x \in V_i(\s,\bpsi)$ satisfies in particular $ \lvert \x-\s_i \lvert^2 - \psi_i \leq 0$. 
 \item In the same spirit, if $V_i(\s,\bpsi)$ has positive Lebesgue measure, one has $\psi_i >0$.
 \item If the cell $V_0(\s,\bpsi)$ in \cref{eq.V0} is non empty, it holds: 
 $$ \forall i =1,\ldots, N, \quad \psi_i \leq  \diam(D)^2.$$
 \end{itemize}

Now, let $\bnu \in \R^N$ be a vector satisfying \cref{eq.assumnu}, and let $\calK \subset \R^N$ be the set defined by
$$ \calK = \Big\{ \bpsi \in \R^N \text{ s.t. } |V_i(\s,\bpsi)| \leq \nu_i \text{ for all } i =1,\ldots, N \Big\}.$$ 
According to the previous observations, for all $\bpsi \in \calK$, it holds $\lvert V_0(\s ,\bpsi) \lvert  \: \geq \nu_0 >0$, and so
$ \calK \subset [0,\diam(D)^2]^N$.  Moreover, since the mapping $\bF(\s,\bnu,\cdot)$ is continuous, $\calK$ is closed, and it is therefore compact. 
Let us then consider one solution $\bpsi^*$ to the following maximization problem:
\begin{equation}\label{eq.maxW}
 \max\limits_{\bpsi \in \calK}\: W(\bpsi), \text{ where } W(\bpsi) = \sum\limits_{i=1}^N \psi_i.
 \end{equation}
 Assume that there exists $i\in \left\{1,\ldots,N\right\}$ such that $\lvert V_i(\s,\bpsi^*)\lvert < \nu_i$. Then by the continuity of $F(\s,\bnu,\cdot)$, there exists $t >0$ such that 
 $ \lvert V_i(\s,\bpsi^* + t \be_i)\lvert < \nu_i$, and by the foregoing observations, one also has:
$$ \lvert V_j(\s,\bpsi^* + t \be_i) \lvert \: \leq  \:  \lvert V_j(\s,\bpsi^* ) \lvert \: \leq \: \nu_j \text{ for } j = 1,\ldots,N, \:\:  j \neq i. $$
 Hence, the vector $(\bpsi^* + t \be_i)$ also belongs to $\calK$ and $W(\bpsi^* + t \be_i) > W(\bpsi^*)$,
 which contradicts the definition of $\bpsi^*$. We have thus proved that any maximizer $\bpsi^*$ in \cref{eq.maxW} satisfies the desired property:
 $$ \lvert V_i (\s, \bpsi^*) \lvert = \nu_i, \text{ for } i =1,\ldots,N.$$
 \par\medskip 
 
 \noindent \textit{Step 3: We prove the uniqueness of $\bpsi^*$ under the genericity assumptions \cref{eq.Gen0,eq.Gen1,eq.Gen2,eq.Gen3,eq.Gen4,eq.Gen5,eq.Gen6}.}  \par\medskip

\noindent We have seen in \cref{cor.irred} that under the assumptions \cref{eq.Gen0,eq.Gen1,eq.Gen2,eq.Gen3,eq.Gen4,eq.Gen5,eq.Gen6}, 
the mapping $\bpsi \mapsto \bF(\s,\bnu,\bpsi)$ is differentiable at $\bpsi^*$, and that the derivative $\left[\nabla_\bpsi \bF(\s,\bnu,\bpsi^*) \right]$ is an invertible $N \times N$ matrix.

Let us assume that there exists another weight vector $\bvphi^*$ satisfying \cref{eq.volPsi}. Then $\bpsi^*$ and $\bvphi^*$ are both maximizers of the concave Kantorovic functional $K(\s,\bnu,\cdot)$, and so, for any $t\in (0,1)$, it holds:
$$ \bF(\s, \bnu,(1-t)\bpsi^* + t\bvphi^*) = \bF(\s,\bnu,\bpsi^*) = \bzero.$$ 
Rearranging the above equality, dividing both sides by $t$, then letting $t$ tend to $0$, we obtain:
$$\left[\nabla_\bpsi \bF(\s,\bnu,\bpsi^*) \right](\bvphi^* - \bpsi^*) = \bzero, $$
which proves that $\bpsi^*=\bvphi^*$, as expected.
\par\medskip

\noindent \textit{Step 4: We use duality to conclude that \cref{eq.Tpsi} is the optimal transport mapping between $\mu$ and $\nu$.}  \par\medskip

\noindent Let us recall the well-known expression of the dual maximization problem attached to the Kantorovic minimization problem \cref{eq.OTKanto}:
\begin{equation}\label{eq.OTD}
\tag{\textcolor{gray}{OT-D}}
 \max\limits_{\psi \in \calC(X)} \left( \int_{X} \min\limits_{\y \in X} \Big( c(\x,\y) - \psi(\y)\Big) \:\d \mu(\y) + \nu_0 \frac{1}{|D|} \int_X \psi(\y) \:\d \y + \sum\limits_{i=1}^N \nu_i \psi_i\right),
 \end{equation}
 see e.g. \cite{santambrogio2015optimal}.
The weak duality phenomenon reads, in the present situation:
\begin{equation}\label{eq.wkD}
 \max \: \cref{eq.OTD} \leq \min \: \cref{eq.OTKanto}.
 \end{equation}
Note that actually, strong duality holds, i.e. $\max \: \cref{eq.OTD} =\min \: \cref{eq.OTKanto}$, see Th. 1.39 in \cite{santambrogio2015optimal}. We shall not need this fact, which will actually turn out to be a by-product of our analysis. 

 Let $\bpsi \in \R^N$  be the weight vector supplied by the second step, guaranteeing the equality \cref{eq.volPsi}, 
and let $T_\bpsi$ be the mapping defined in \cref{eq.Tpsi}; we recall from Step 2 that $\psi_i >0$ for all $i=1,\ldots, N$ and we retain the convention $\psi_0 = 0$. By construction, $(T_\bpsi)_\# \mu = \nu$, so that $(\Id, T_\bpsi)_\# \mu$ is an admissible transport plan in the problem \cref{eq.OTKanto}; hence:
$$ \min \: \cref{eq.OTKanto} \leq \int_X c(\x,T_\bpsi(\x)) \:\d \mu(\x).$$
On the other hand, let us define the (discontinuous) function $\widetilde\psi : X \to \R$ by 
$$ \forall \y \in X, \quad \widetilde \psi(\y) = \left\{
\begin{array}{cl}
\psi_i & \text{if } \y = \s_i, \: i=1,\ldots,N, \\
\psi_0 = 0 & \text{otherwise}.
\end{array}
\right.$$
It follows from the definition of $T_\bpsi$ that:
$$\forall \x \in X, \quad c(\x,T_\bpsi(\x)) - \widetilde \psi(T_\bpsi(\x))= \min\limits_{\y \in X} \Big(c(\x,\y) - \widetilde \psi(\y)\Big).$$
Integrating this identity, we obtain:
\begin{equation}\label{eq.idproofopOTmap}
 \int_X c(\x,T_\bpsi(\x)) \:\d \mu(\x) - \sum\limits_{i=1}^N \nu_i \psi_i = \int_X \min\limits_{\y \in X} \Big(c(\x,\y) - \widetilde \psi(\y)\Big) \:\d \mu(\x).
 \end{equation}
We now approximate the discontinuous function $\widetilde \psi$ in the last term in the above right-hand side.  
To this end, let us first introduce the step function $\chi_\e$ defined for $\e >0$ small enough by
$$ \forall \y \in X, \quad \chi_\e(\y) = \left\{
\begin{array}{cl}
\psi_i & \text{if } \lvert \y-\s_i \lvert < \e \text{ for some } i=1,\ldots, N, \\
0 & \text{otherwise}.
\end{array}
\right.$$ 
We also introduce one function $\zeta_\e \in \calC(X)$ satisfying the properties:
$$  \left\{
\begin{array}{cl}
\zeta_\e(\y) =\psi_i & \text{if } \lvert \y-\s_i \lvert < \frac{\e}{2} \text{ for } i=1,\ldots, N, \\
0 \leq \zeta_\e(\y) \leq \psi_i & \text{if } \frac{\e}{2} \leq \lvert \y-\s_i \lvert < \e \text{ for } i=1,\ldots, N, \\
0 & \text{otherwise},
\end{array}
\right.
$$
i.e. $\zeta_\e$ is a smoothed version of $\widetilde \psi$, whose support is contained in a collection of balls of radius $\e$ around the $\s_i$, $i=1,\ldots,N$; 
the latter can easily be constructed thanks to mollifiers.  
Then, it holds:
$$ \min\limits_{\y \in X} \Big(c(\x,\y) - \chi_\e(\y) \Big) \leq \min\limits_{\y \in X} \Big(c(\x,\y) - \zeta_\e(\y) \Big) \leq \min\limits_{\y \in X} \Big(c(\x,\y) - \widetilde\psi(\y) \Big).$$
On the other hand, the first term in the above left-hand side can be estimated as:
$$ \begin{array}{ccl}
\min\limits_{\y \in X} \Big(c(\x,\y) - \chi_\e(\y) \Big) & \geq & \min\Big(0,\min\limits_{i=1,\ldots,N \atop \h \in B(0,\e)} \Big( c(\x, \s_i + \h) - \psi_i\Big)  \Big) \\
&=&  \min\Big(0,\min\limits_{i = 1,\ldots,N \atop \h \in B(0,\e)} \Big( c(\x - \h, \s_i ) - \psi_i\Big)\Big) \\
&\geq &  \min\Big(0,\min\limits_{i = 1,\ldots,N } \Big( c(\x, \s_i ) - \psi_i\Big)\Big)  - L\e \\
& = & \min\limits_{\y \in X} \Big(c(\x,\y) - \widetilde\psi(\y) \Big)  - L\e,
\end{array}$$
where $L$ is a Lipschitz constant for the cost $c(\x,\y) = \lvert \x- \y\lvert^2$ on $X\times X$, and we have used the particular difference structure of the latter when passing from the first to the second line. 
It follows that:
$$ \min\limits_{\y \in X} \Big(c(\x,\y) - \widetilde\psi(\y) \Big)   \leq \min\limits_{\y \in X} \Big(c(\x,\y) -  \zeta_\e(\y)\Big) + L\e, $$
and so, returning to \cref{eq.idproofopOTmap}:
$$ \begin{array}{>{\displaystyle}cc>{\displaystyle}l}
 \int_X c(\x,T_\bpsi(\x)) \:\d \mu(\x)  &\leq& \int_X \min\limits_{\y \in X} \Big(c(\x,\y) -  \zeta_\e(\y)\Big) \:\d \mu(\x)  + \sum\limits_{i=1}^N \nu_i \psi_i + L \e \\
 &\leq& \int_X \min\limits_{\y \in X} \Big(c(\x,\y) -  \zeta_\e(\y)\Big) \:\d \mu(\x)  + \nu_0 \frac{1}{|D|} \int_X \zeta_\e(\y) \:\d \y + \sum\limits_{i=1}^N \nu_i \psi_i + L \e \\
 &\leq& \max \: \cref{eq.OTD}+ L\e.
 \end{array} $$
 Since $\e$ is arbitrary, we may let $\e$ tend to $0$ in the above inequality; recalling the weak duality inequality \cref{eq.wkD}, we have thus proved that $T_\bpsi$ is indeed an optimal transport mapping.
 
 The uniqueness of such an optimal transport mapping is a classical fact in optimal transport theory, see e.g. Th. 1.17 in \cite{santambrogio2015optimal}.  
This concludes the proof of \cref{prop.psi}.

\section{Implementation details of the Virtual Element Method}\label{app.VEM}

\noindent In this section, we describe the practical implementation of the Virtual Element Method for the solution of the conductivity equation \cref{eq.lapVEM} and of the linear elasticity system \cref{eq.VEMelassys}.

\subsection{The case of the conductivity equation}\label{app.VEMlap}

\noindent Slipping into the notation of \cref{sec.VEMlap} and closely following \cite{sutton2017virtual}, we now provide a little details about the computation of the entries of the local stiffness matrix $K^E \in \R^{n^E \times n^E}$ defined by
$$ \forall i,j = 1,\ldots,n, \quad K^E_{ij} = \widetilde{a}^E(\zeta_i,\zeta_j),$$
 where the discrete bilinear form $\widetilde{a}^E(u,v)$ reads
 $$\forall u ,v \in \calW(E), \quad \widetilde{a}^E(u,v) = a^E(\pi_C u,\pi_C v) + \widetilde{s}^E(u-\pi_{\calP} u , v- \pi_{\calP} v),$$
 and the basis $\left\{ \zeta_i \right\}_{i=1,\ldots,n^E}$ of $\calW(E)$ is characterized by \cref{eq.defzetai}. 
As discussed in \cref{sec.VEMlap}, this formula induces a decomposition of $K^E$ as the sum of two contributions:
\begin{equation*}
 K^E = P^E + \widetilde{S}^E, \text{ where }
P^E_{ij} := a^E(\pi_C \zeta_i, \pi_C \zeta_j) \text{ and } \widetilde{S}^E_{ij} = \widetilde{s}^E(\zeta_i-\pi_{\calP} \zeta_i , \zeta_j- \pi_{\calP} \zeta_j).
\end{equation*}

The cornerstone of the calculation of both matrices is therefore the calculation of the matrix $\Pi \in \R^{3\times n^E}$ of the projection operator $\pi_{\calP}$ onto affine functions, expressed in the bases $\left\{ \zeta_i \right\}_{i=1,\ldots,n^E}$ of $\calW(E)$ and $\left\{m_\alpha \right\}_{\alpha=1,2,3}$ of $\calP(E)$, i.e. 
the $i^{\text{th}}$ column of $\Pi$ gathers the entries of $\pi_{\calP}\zeta_i$ in the basis $m_\alpha$: 
$$ \pi_\calP\zeta_i(\x) = \sum\limits_{\alpha = 1}^3 \Pi_{\alpha i} m_\alpha(\x), \quad i=1,\ldots,n^E, \quad \x \in E.$$
By using the following characterization of $\pi_{\calP}$, 
$$\overline{\pi_{\calP} \zeta_i} = \overline{\zeta_i} = \frac{1}{n^E} \text{ and } a^E(\pi_\calP \zeta_i, m_\alpha) = a^E(\zeta_i,m_\alpha), \:\: \alpha = 2,3,$$
we obtain the identity:
\begin{equation}\label{eq.GPiB}
 \widetilde G \Pi = B,
 \end{equation}
where
\begin{itemize}
\item The matrix $\widetilde G \in \R^{3\times 3}$ is defined by: 
$$\widetilde G_{11} = 1 \text{ and } \widetilde G_{12} = \widetilde G_{13} = 0,$$
as well as: 
$$\widetilde G_{\alpha,\beta} = \int_E \nabla m_\alpha \cdot \nabla m_\beta \:\d \x \text{ for } \alpha=2,3 \text{ and } \beta=1,2,3.$$
\item The matrix $B \in \R^{3\times n^E}$ is given by:
$$ B_{1i} = \frac{1}{n^E}, \text{ and } B_{\alpha i} = \int_E \nabla \zeta_i \cdot \nabla m_\alpha \:\d \x  \text{ for } i=1,\ldots,n^E \text{ and } \alpha= 2,3.$$
The entries of $B$ can be calculated in closed form by using the facts that $\nabla m_\alpha$ is constant on $E$ and that the basis functions $\zeta_i$ are affine on each edge of $\partial E$. Indeed, Green's formula implies that, for $\alpha=2,3$: 
$$ B_{\alpha, i} = \left( \int_E \nabla \zeta_i \:\d \x \right) \cdot \nabla m_\alpha = \left( \int_{\partial E} \zeta_i \n \:\d \ell \right)  \cdot \nabla m_\alpha = \frac12 \lvert \widehat \be_i \lvert \n_{\widehat \be_i} \cdot \nabla m_\alpha.$$
\end{itemize}

Hence, the system \cref{eq.GPiB} features the explicit matrix $\widetilde G$ and right-hand side $B$, and it can be solved for the matrix $\Pi$.
The practical implementation is further simplified thanks to the introduction of
the matrix $D \in \R^{n^E\times 3}$ corresponding to the expression of the polynomials $m_\alpha$, $\alpha =1,2,3$ in the basis $\left\{\zeta_i\right\}_{i=1,\ldots,n^E}$ of $\calW(E)$, that is: 
$$ m_\alpha(\x) = \sum\limits_{i=1}^{n^E} D_{i \alpha} \zeta_i(\x), \quad \x \in E.$$
The coefficients of this matrix read, by definition of the $\zeta_i$:
$$ D_{i\alpha} = m_\alpha(\q_i^E),$$
and so
$$ D_{i1} = 1, \text{, } D_{i2} = q^E_{i,1}-\overline{q_1^E}, \text{ and } D_{i3} = q^E_{i,2}-\overline{q_2^E}  \text{ for } i=1,\ldots,n^E.$$
The definitions of the matrices $\widetilde G$, $B$ and $D$ yield the following relation:
$$ \widetilde{G} = BD.$$
Indeed, elementary calculations lead to:
$$ \widetilde G_{11} = 1 = \sum\limits_{i=1}^{n^E} B_{1i}D_{i1}, \:\: \widetilde G_{1\beta} = 0 = \sum\limits_{i=1}^{n^E} B_{1i}D_{i\beta} \text{ for } \beta = 2,3, $$
and likewise, for $\alpha=2,3$, 
$$ \widetilde G_{\alpha\beta} = \int_{E} \nabla m_\alpha \cdot \nabla m_\beta \:\d \x = \sum\limits_{i=1}^{n^E} D_{i\beta} \int_{E} \nabla m_\alpha \cdot \nabla \zeta_i \:\d \x = \sum\limits_{i=1}^{n^E} D_{i\beta} B_{\alpha i}.$$

Once the matrix $\Pi$ is calculated, $P^E$ is inferred from the formula:
\begin{equation}\label{eq.calcPE} P^E = \gamma \Pi^T G \Pi, \text{ where } G \text{ is the matrix with entries } G_{\alpha \beta} = \int_E \nabla m_\alpha \cdot \nabla m_\beta \:\d \x, \quad \alpha, \beta = 1,2,3.
\end{equation}
Likewise, $\widetilde{S}^E$ is obtained as:
\begin{equation}\label{eq.calcStE}
\widetilde{S}^E= \alpha^E \Big( \I - D\Pi \Big) ^T \Big( \I - D\Pi \Big) , \text{ where } \alpha^E = 1,
\end{equation}
and $\I$ is the identity matrix with size $n^E \times n^E$.\par\medskip

Summarizing, the computations of the matrices $P^E$ and $S^E$ (and thus $K^E$) proceed as follows: 
\begin{enumerate}
\item Assemble the matrices $B$ and $D$; 
\item Calculate $\widetilde G = BD \in \R^{3\times 3}$; 
\item Compute $\Pi \in \R^{3\times n^E}$ as the solution to the matrix system 
$$ \widetilde G \Pi = B.$$
\item The matrice $P^E$ and $\widetilde{S}^E$ are inferred from the formulas \cref{eq.calcPE,eq.calcStE}.
\end{enumerate}

Eventually, the entries $F^E_i$ of the local force vector $F^E \in \R^{n^E}$ are simply approximated as: 
$$ F^E_{i} \:\: = \:\: \int_E f \zeta_i \:\d \x \:\: \approx\:\: f(\overline{\q^E})\int_E \zeta_i \:\d \x  \:\: \approx \:\: f(\overline{\q^E}) \frac{\lvert E\lvert}{n^E}, \quad i= 1,\ldots,n^E. $$

\begin{remark}\label{rem.dpdproj}
An important consequence of the above analysis is that the projection $\pi_{\calP} \zeta$ of a function $\zeta \in \calW_{\calT}$ onto affine functions depends only on the values $\zeta(\q_i^E)$ ($i=1,\ldots,n^E$) of $\zeta$ at the vertices of the considered element $E$
\end{remark}

\subsection{The case of the linear elasticity system}\label{app.VEMelas}

\noindent The present section deals with the linear elasticity system \cref{eq.VEMelassys}. 
Relying on \cite{berbatov2021guide,gain2014virtual}, we briefly describe its numerical solution by the Virtual Element Method.

As emphasized in \cref{sec.mechcomp,app.VEMlap}, the most critical operation in this perspective is the calculation of the local stiffness matrix $K^E \in \R^{2n^E \times 2n^E}$ defined by
$$ \forall i,j = 1,\ldots,2n^E, \quad K^E_{ij} = \widetilde{a}^E(\bzeta_i,\bzeta_j),$$ 
where $\widetilde{a}^E$ is given by:
$$\forall u,v \in \calW(E), \quad \widetilde{a}^E(\u,\v) = a^E(\pi_C \u,\pi_C \v) + \widetilde{s}^E(\u-\pi_{\calP} \u , \v- \pi_{\calP} \v),$$
and the basis functions $\bzeta_i$, $i=1,\ldots, 2n^E$ are characterized by \cref{eq.defvphi1,eq.defvphi2}. 
Again, we decompose $K^E$ as:
\begin{multline*}
 K^E = P^E + \widetilde{S}^E, \text{ where } P^E, \widetilde{S}^E\in \R^{2n^E \times 2n^E} \text{ are the matrices }\\ 
P^E_{ij} := a^E(\pi_C \bzeta_i, \pi_C \bzeta_j) \text{ and } \widetilde{S}^E_{ij} := \widetilde{s}^E(\bzeta_i-\pi_{\calP} \bzeta_i , \bzeta_j- \pi_{\calP} \bzeta_j).
\end{multline*}
Let $W_C \in \R^{2n^E \times 3}$ be the transpose of the matrix of $\pi_{C}$ in the bases $\left\{ \bzeta_i \right\}_{i=1,\ldots,2n^E}$ and $\left\{ \c_\alpha \right\}_{\alpha=1,2,3}$, i.e. the $i^{\text{th}}$ line of $W_C$ gathers the coordinates of the projection $\pi_C \bzeta_i$ over the functions 
$\c_1,\c_2,\c_3$ in \cref{eq.defci}:
$$ \forall i = 1,\ldots , 2n^E, \quad \pi_C \bzeta_i (\x)= \sum\limits_{\beta=1}^3 ({W_C})_{i \beta} \c_\beta(\x), \quad \x \in E.$$
Likewise, let $W_R \in \R^{2n^E \times 3}$ be the transpose of the matrix of $\pi_{R}$ in the bases $\left\{ \bzeta_i \right\}_{i=1,\ldots,2n}$ and $\left\{ \r_\alpha \right\}_{\alpha=1,2,3}$:
$$ \forall i = 1,\ldots , 2n^E, \quad \pi_R \bzeta_i (\x)= \sum\limits_{\beta=1}^3 ({W_R})_{i\beta} \r_\beta(\x).$$
With these notations, the first block $P^E \in \R^{2n^E\times 2n^E}$ of $K^E$ reads,
$$P^E_{ij} = a^E(\pi_C \bzeta_i, \pi_C\bzeta_j) = \sum\limits_{\alpha, \beta=1}^{3} (W_C)_{i\alpha} (W_C)_{j\beta}\:a^E(\c_\alpha,\c_\beta) ,$$
and so, in matrix form:
\begin{equation}\label{eq.PEelas}
P^E= W_C D W_C^T, \text{ where } D \in \R^{3 \times 3} \text{ is defined by } D_{\alpha \beta} = a^E (\c_\alpha,\c_\beta) = \int_E Ae(\c_\alpha) : e(\c_\beta) \:\d \x.
\end{equation}
An elementary calculation yields:
$$ D  =  |E|
\left( 
\begin{array}{ccc}
2\mu + \lambda & \lambda & 0 \\
\lambda & 2\mu + \lambda & 0 \\
0 & 0 &4\mu
\end{array}
\right) ,$$
and we shall detail below how to calculate the matrices $W_C$ and $W_R$ in practice. \par\medskip

As far as the stabilizing block $\widetilde{S}^E$ of $K^E$ is concerned, we need to calculate the matrices $P_C$ and $P_R \in \R^{2n^E \times 2n^E}$ of the projections $\pi_C : \calW(E) \to \calW(E)$ and $\pi_R: \calW(E) \to \calW(E)$
in the basis $\left\{\bzeta_i\right\}_{i=1,\ldots,2n^E}$. 
To this end, let us introduce the matrix $N_C \in \R^{2n^E\times 3}$, whose $\alpha^{\text{th}}$ column ($\alpha=1,2,3$) contains the coordinates of $\c_\alpha$ over the basis $\left\{\bzeta_i\right\}$, that is:
$$ \forall \alpha =1,2, 3, \quad \c_\alpha(\x) = \sum\limits_{i=1}^{2n^E} (N_C)_{i\alpha} \bzeta_i(\x), \quad \x \in E.$$  
With these notations, it holds, for $j=1,\ldots,2n^E$:
$$\begin{array}{>{\displaystyle}cc>{\displaystyle}l}
 \pi_C \bzeta_j &=& \sum\limits_{\alpha = 1}^3 (W_C)_{j\alpha} \c_\alpha\\
 &=&  \sum\limits_{i=1}^{2n^E}  \left( \sum\limits_{\alpha = 1}^3 (N_C)_{i\alpha}  (W_C)_{j\alpha}\right) \bzeta_i
 \end{array} 
 $$
 and so, $P_C  \in \R^{2n^E \times 2n^E}$ reads:
$$ P_C = N_C W_C^T.$$
Likewise, let $N_R \in \R^{2n^E\times 3}$ be the matrix whose $\alpha^{\text{th}}$ column ($\alpha=1,2,3$) contains the coordinates of $\r_\alpha$ in the basis $\bzeta_i$, that is:
$$ \forall \alpha =1,2, 3, \quad \r_\alpha(\x) = \sum\limits_{i=1}^{2n^E} (N_R)_{i\alpha} \bzeta_i(\x), \quad \x \in E;$$
it holds:
\begin{equation}\label{eq.PRelas}
P_R = N_R W_R^T.
\end{equation}
Finally, the matrix $P_P \in \R^{2n^E\times 2n^E}$ of the projection $\pi_{\calP}$ over the space  $\calP(E)$ of affine functions reads:
\begin{equation}\label{eq.PPelas}
 P_P = P_R + P_C.
 \end{equation}
With these notations, $\widetilde{S}^E$ can be computed via the following formula:
\begin{equation}\label{eq.SEelas}
S^E  = \alpha^E (\I - P_P)^T (\I-P_P),
\end{equation}
where $\I$ is the $2n^E \times 2n^E$ identity matrix. 

It follows from the formulas \cref{eq.PEelas,eq.PRelas,eq.PPelas,eq.SEelas} that the basic ingredients of the practical implementation of the Virtual Element Method dedicated to the solution of the linear elasticity system \cref{eq.elas} are the assembly of the matrices  $N_C, N_R, W_C, W_R$, and the calculation of the coefficient $\alpha^E$; we now detail these operations. \par\medskip

\noindent \textit{Assembly of the matrices $W_C, W_R$} \par\medskip

\noindent For $i=1,\ldots, n^E$, let us introduce the vector $\a_i \in \R^2$ defined by
$$ \a_i = \frac{|\hat \be_i | }{2 |E| } \n_{\hat \be_i}, $$
where we recall the notation $\n_{\hat \be_i}$ in \cref{eq.nehat}.

Let us first consider the matrix $W_R \in \R^{2n^E \times 3}$; for $i=1,\ldots,n^E$ and $\x \in E$, we have
$$ \pi_R \bzeta_{2i-1}(\x) = \overline{\bzeta_{2i-1}} + \langle \omega(\bzeta_{2i-1}) \rangle \left( \begin{array}{c}
-(x_2 - \overline{q_2^E})\\
x_1 - \overline{q_1^E}
\end{array}
\right),
$$
where obviously, $\overline{\bzeta_{2i-1}} = \frac{1}{n^E} \r_1$ and a simple calculation shows that:
$$\begin{array}{>{\displaystyle}cc>{\displaystyle}l}
 \langle \omega(\bzeta_{2i-1}) \rangle &=&- \frac{1}{2|E|} \int_E  \frac{\partial}{\partial x_2} (\bzeta_{2i-1})_1\:\d \x \\[1em]
 &=&- \frac{1}{2|E|} \int_{\partial E}  (\bzeta_{2i-1})_1 n_2 \:\d \ell \\[1em]
 &=& -\frac{1}{4|E|} (|\be_{i-1} | n_{\be_{i-1}} + |\be_i| n_{\be_i} )_2 \\[1em]
 &=&- \frac{1}{|E|} \frac{|\hat \be_i| }{4} (n_{\widehat{\be_{i}}})_2 \\[1em]
 &=& -\frac12 (a_i)_2.
 \end{array}
 $$
Hence, it follows:
$$  \pi_R \bzeta_{2i-1}(\x) =  \frac{1}{n^E} \r_1(\x) -  \frac12 (a_i)_2 \: \r_3(\x).$$
A similar calculation yields:
$$  \pi_R \bzeta_{2i}(\x) =  \frac{1}{n^E} \r_2(\x) +  \frac12 (a_i)_1  \: \r_3(\x).$$

Likewise, considering the entries of $W_C$, we obtain:
$$  \pi_C \bzeta_{2i-1}(\x) =  (a_i)_1 \c_1(\x) +  \frac12 (a_i)_2  \:  \c_3(\x), \text{ and }  \pi_C \bzeta_{2i}(\x) =  (a_i)_2 \c_2(\x)  + \frac12 (a_i)_1\: \c_3(\x).$$
Summarizing, the matrices $W_R$ and $W_C$ read, in the respective bases $\left\{ \r_\alpha \right\}_{\alpha=1,2,3}$, \:\: $\left\{\bzeta_{i}\right\}_{i=1,\ldots,2n^E}$ and $\left\{ \c_\alpha \right\}_{j=1,2,3}$, \:\: $\left\{\bzeta_{i}\right\}_{i=1,\ldots,2n^E}$:
$$
W_R = \left( \begin{array}{ccc}
\vdots & \vdots & \vdots \\
\\
\frac{1}{n^E}\:\:\: &\:\:\: 0 \:\:\:& \:\:\: -\frac{1}{2}(a_i)_2 \\[2mm]
0 \:\:\: & \:\:\: \frac{1}{n^E}\:\:\:& \:\:\:  \frac{1}{2}(a_i)_1 \\
\\
\vdots & \vdots & \vdots 
\end{array} 
\right), \text{ and } 
W_C = \left( \begin{array}{ccc}
\vdots & \vdots & \vdots \\
\\
(a_i)_1\:\:\: & \:\:\: 0  \:\:\:&  \:\:\: \frac{1}{2}(a_i)_2 \\[2mm]
0  \:\:\: &  \:\:\: (a_i)_2  \:\:\:&  \:\:\:  \frac{1}{2}(a_i)_1 \\
\\
\vdots & \vdots & \vdots 
\end{array} 
\right).
$$\par\medskip

\noindent \textit{Assembly of the matrices $N_C, N_R$} \par\medskip

It follows from the definition \cref{eq.defvphi1,eq.defvphi2} of the functions $\bzeta_i$ that:
$$ \c_\alpha(\x) = \sum\limits_{i=1}^{n^E} (c_\alpha (\q_i^E))_1 \: \bzeta_{2i-1}(\x) +  \sum\limits_{i=1}^{n^E} (c_\alpha(\q_i^E))_2 \: \bzeta_{2i}(\x), \quad \x \in E.$$
Hence, a simple calculation based on the explicit form \cref{eq.defci} of the basis functions $\c_j$ reveals that: 
$$
N_C = \left( \begin{array}{ccc}
\vdots & \vdots & \vdots \\
\\
(c_1(\q_i^E))_1\:\:\: &\:\:\: (c_2(\q_i^E))_1 \:\:\:& \:\:\:(c_3(\q_i^E))_1 \\[2mm]
(c_1(\q_i^E))_2 \:\:\: & \:\:\: (c_2(\q_i^E))_2\:\:\:& \:\:\:  (c_3(\q_i^E))_2 \\
\\
\vdots & \vdots & \vdots 
\end{array} 
\right) =   
\left( \begin{array}{ccc}
\vdots & \vdots & \vdots \\
\\
q_{i,1}^E - \overline{q_1^E}\:\:\: & \:\:\: 0 \:\:\:& \:\:\:q_{i,2}^E - \overline{q_2^E} \\[2mm]
0 \:\:\: & \:\:\: q_{i,2}^E - \overline{q_2^E}\:\:\:& \:\:\: q_{i,1}^E -\overline{q_1^E} \\
\\
\vdots & \vdots & \vdots 
\end{array} 
\right) .
$$ 
Similarly, the following expression of $N_R$ is derived: 
$$N_R = \left( \begin{array}{ccc}
\vdots & \vdots & \vdots \\
\\
1\:\:\: &\:\:\: 0 \:\:\:& \:\:\: -q_{i,2}^E - \overline{q_2^E} \\[2mm]
0 \:\:\: & \:\:\: 1\:\:\:& \:\:\: q_{i,1}^E - \overline{q_1^E} \\
\\
\vdots & \vdots & \vdots 
\end{array} 
\right) . $$\par\medskip

\noindent \textit{Calculation of $\alpha^E$}\par\medskip
\noindent As we have mentioned, the coefficient $\alpha^E$ is chosen so as to ensure the stability of the method. 
The idea consists in taking $\alpha^E$ such that the block $\left\{ s^E(\c_k, \c_l)\right\}_{k,l=1,2,3}$, 
corresponding to the application of the stabilizing form $\widetilde{s}^E$ to constant strain fields,
scale like $\left\{a^E(\c_k,\c_l)\right\}_{k,l=1,2,3}$, which is exactly the matrix $D$, when the mesh $\calT$ is refined. 
Hence, we require that
$ \widetilde{s}^E(\c_k,\c_l) = \alpha^E (N_C^T N_C)_{kl}$ be comparable with $D_{kl}$,
which motivates the choice
$$ \alpha^E = \frac{\tr(D)}{\tr(N_C^TN_C)}.$$

Eventually, let us mention that the approximation of the local force vector $F^E \in \R^{2n^E}$ whose entries equal
$$ F^E_i = \int_E \f \cdot \bzeta_i \:\d \x + \int_{\partial E \cap \Gamma_N} \g \cdot \bzeta_i \:\d \ell , \quad i =1,\ldots,2n^E$$ 
is realized in a completely similar fashion as in the case of the conductivity equation, see \cref{app.VEMlap}.

\subsection{Addition of a term of order $0$}\label{sec.L2proj}

\noindent For the sake of simplicity and without loss of generality, let us return to the setting of the conductivity equation \cref{eq.lapVEM} of \cref{sec.conduc}.
In multiple applications, starting from the solution of eigenvalue problems of the form \cref{eq.evlap} (see \cref{sec.exev} for a numerical example), the numerical implementation requires to calculate the  mass matrix $M_{\calT} \in \R^{M \times M}$ in addition to the stiffness matrix $K_{\calT}$ in \cref{eq.KTgen}. This matrix $M_{\calT}$ is defined by:
$$ \forall k,l = 1,\ldots,M, \quad M_{\calT,kl} = \int_\Omega \varphi_k \varphi_l \:\d \x,$$
where the basis $\left\{ \varphi_k \right\}_{k=1,\ldots,M}$ of the virtual element space $\calW_{\calT}$ is that characterized by \cref{eq.vphikql}.
Arguing as in \cref{sec.genpresVEMlap}, $M_{\calT}$ is assembled from the local contributions $M^E \in \R^{n^E \times n^E}$ attached to the individual elements $E \in \calT$, defined by:
\begin{equation}\label{eq.ME}
M^E_{ij} = \int_E \zeta_i \zeta_j \:\d \x.
\end{equation}
To achieve this computation, we introduce the $L^2$ projector $\pi_{\calP}^0 : \calW(E) \to \calP(E)$ over $\calP(E)$, defined by
$$ \forall m \in \calP(E), \quad \int_E (\pi_{\calP}^0 \zeta) m \:\d \x = \int_E \zeta m \:\d \x.$$
Unfortunately, this operator cannot be expressed in closed form, as the last term in the above right-hand side cannot be readily computed from the degrees of freedom of the function $\zeta$.

To overcome this issue, we leverage an elegant observation from \cite{ahmad2013equivalent} whereby a slight modification of the local space $\calW(E)$ in \cref{sec.localspaceVEMlap}, which is completely transparent in the numerical implementation, enables this computation at no additional cost. 
Let us first define the larger space $\widetilde\calW(E)$ of local functions attached to $E$ by: 
$$\widetilde\calW(E) = \Big\{ \zeta: E \to \R, \:\: \zeta \text{ is affine on each edge of } \partial E, \:\: - \Delta \zeta \text{ is an affine function on } E \Big\}.$$
We note that the projection operator $\pi_{\calP} : \calW(E) \to \calP(E)$ introduced in \cref{eq.piPlap} can actually be extended as $\pi_{\calP} : \widetilde{\calW}(E) \to \calP(E)$; 
actually, for a given function $\zeta \in \widetilde{\calW}(E)$, the same calculations as in the previous \cref{app.VEMlap} reveal that $\pi_{\calP} \zeta$ can still be calculated from the sole values of $\zeta$ at the vertices of $E$. 
We now introduce the subspace $\calZ(E) \subset \widetilde{\calW}(E)$ defined by:
\begin{multline}\label{eq.defZE}
\calZ(E) = \Big\{ \zeta: E \to \R, \:\: \zeta \text{ is affine on each edge of } \partial E, \:\: - \Delta \zeta \text{ is affine on } E, \text{ and } \\
\forall m \in \calP(E), \:\: \int_E \zeta m \:\d \x = \int_E (\pi_{\calP} \zeta ) m  \: \d \x\Big\}.
 \end{multline}

The key remark about $\calZ(E)$ is that its elements can be characterized by the exact same degrees of freedom as those of the space $\calW(E)$. Before proceeding, let us provide a simple technical fact:
\begin{lemma}\label{lem.isoLapP1}
The mapping $L:= \pi_{\calP}^0 \circ \Delta^{-1} : \calP(E) \to \calP(E)$ is an isomorphism,
where $\Delta$ stands for the Laplace operator on $E$ with homogeneous Dirichlet boundary conditions.
\end{lemma}
\begin{proof}
Since $\calP(E)$ is a finite-dimensional space, it is enough to prove that $L$ is injective. 
Let then $m \in \calP(E)$ be such that $Lm = 0$, and let $q = \Delta^{-1} m \in H^1_0(E)$ be the function defined by
$$ \forall w \in H^1_0(E), \quad \int_E \nabla q \cdot \nabla w \:\d \x = \int_E mw \:\d \x.$$
It follows:
$$ 0 =\int_E Lm \: m \:\d \x = \int_E q \: m \:\d \x = \int_E \lvert \nabla q \lvert^2 \:\d \x.$$
Hence, $q =0$, and so $m=0$, as expected.
\end{proof}

The result of interest is now the following.
\begin{lemma}\label{lem.basisZE}
For all $i=1,\ldots,n^E$, there exists a unique function $\zeta_i \in \calZ(E)$ such that
\begin{equation}\label{eq.zetacalZ}
\zeta_i(\q_j^E) =
 \left\{
 \begin{array}{cl}
 1 & \text{if } i = j, \\
 0 & \text{otherwise}.
 \end{array}
 \right.
 \end{equation}
\end{lemma}
\begin{proof}
Let $i \in \left\{1,\ldots,n^E\right\}$ be an arbitrary index;
we first assume that there exists one function $\zeta \in \calZ(E)$
satisfying \cref{eq.zetacalZ}, and we prove that it is necessarily unique. To this end, let $m := \Delta \zeta \in \calP(E)$, and let us introduce the unique function $\widetilde{\zeta}_i\in \calW(E)$ such that $\widetilde{\zeta}_i(\q_j^E) = 1$ if $i=j$ and $0$ otherwise, see \cref{sec.localspaceVEMlap}. 
As noted in \cref{rem.dpdproj}, the functions $\pi_{\calP}\zeta$ and $\pi_{\calP}\widetilde{\zeta}_i$ coincide.
Then, introducing the basis $\left\{m_\alpha\right\}_{\alpha=1,2,3}$ of $\calP(E)$ defined in \cref{sec.localspaceVEMlap}, the difference $r = \zeta - \widetilde{\zeta}_i$ satisfies:
$$ 
\begin{array}{>{\displaystyle}cc>{\displaystyle}l}
r=0 \text{ on } \partial E, \:\: \Delta r = m, \text{ and for } \alpha=1,2,3, \quad \int_E r m_\alpha \:\d \x &=& \int_E \zeta m_\alpha \:\d \x - \int_E \widetilde{\zeta_i} m_\alpha \:\d \x,\\[1em]
&=& \int_E (\pi_{\calP}\zeta - \pi^0_{\calP}\widetilde{\zeta_i}) m_\alpha \:\d \x \\[1em]
&=& \int_E (\pi_{\calP}\widetilde{\zeta_i} - \pi^0_{\calP}\widetilde{\zeta_i}) m_\alpha \:\d \x,
\end{array}$$
where the second line follows from the definition \cref{eq.defZE} of $\calZ(E)$ and the final line is a consequence of the equality between $\pi_{\calP}\zeta$ and $\pi_{\calP}\widetilde{\zeta}_i$.
It follows from this discussion and the previous \cref{lem.isoLapP1} that:
$$ \Delta\zeta = L^{-1}\Big(\pi_{\calP}\widetilde{\zeta_i} - \pi^0_{\calP}\widetilde{\zeta_i} \Big).$$
Hence, $\zeta$ is the solution to the boundary value problem:
$$ 
\left\{
\begin{array}{cl}
-\Delta \zeta = -L^{-1}(\pi_{\calP}\widetilde{\zeta_i} - \pi^0_{\calP}\widetilde{\zeta_i}) & \text{in } E, \\
\zeta = \widetilde{\zeta}^i & \text{on } \partial E,
\end{array}
\right.
$$
so that the function $\zeta$ satisfying \cref{eq.zetacalZ} is necessarily unique. 

Conversely, defining $\zeta$ by the above formula immediately yields one function $\zeta \in \calZ(E)$ satisfying \cref{eq.zetacalZ}, which terminates the proof.
\end{proof}

Let us summarize the foregoing discussion:
\begin{itemize}
\item The functions pertaining to $ \calW(E)$ and in $\calZ(E)$ can both be characterized by their values at the vertices of $E$:
for any collection of values $\left\{v_i \right\}_{i=1,\ldots,n^E} \in \R^{n^E}$, there exist unique functions $w \in  \calW(E)$ and $z \in \calZ(E)$ such that:
$$ w(\q_i^E) = v_i \text{ and } z(\q_i^E) = v_i.$$
Both functions $w$ and $z$ are different, but, as noted in \cref{rem.dpdproj}, their projections onto affine functions coincide:
$$ \pi_{\calP} z = \pi_{\calP}w.$$
\item By definition, the projectors $\pi_{\calP}$ and $\pi_{\calP}^0$ coincide on the space $\calZ(E)$:
$$ \forall \zeta \in \calZ(E), \quad \pi_{\calP}^0 \zeta = \pi_{\calP} \zeta.$$
\end{itemize}

Introducing the basis $\left\{ \zeta_i \right\}_{i=1,\ldots,n^E}$ of the local space $\calZ(E)$ defined in \cref{lem.basisZE}, the local mass matrix $M^E$ expressed in terms of this basis reads as in \cref{eq.ME}. 
Like in the previous sections, we may approximate this matrix as:
$$ \forall i,j = 1,\ldots,n^E \quad M^E_{ij} = \int_E (\pi_{\calP}^0 \zeta_i )( \pi_{\calP}^0 \zeta_j )\:\d \x + \int_E (\zeta_i - \pi_{\calP}^0 \zeta_i )(\zeta_j - \pi_{\calP}^0 \zeta_j) \:\d \x  .$$
Here, the first contribution can be calculated exactly on account of the previous observations; 
as for the second one, it can be  replaced by a stabilization term, in the same spirit as in \cref{sec.localspaceVEMlap}.

\bibliographystyle{siam}
\bibliography{./genbib.bib}

\end{document}